\documentclass[11pt]{article}
\makeatletter
\newcommand*{\rom}[1]{\expandafter\@slowromancap\romannumeral #1@}

\usepackage{amsmath,amsthm,amssymb,comment,fullpage}

\usepackage[T1]{fontenc}
\usepackage{latexsym}
\usepackage{amsmath,amsfonts,amsthm,amssymb,amscd}
\input amssym.def
\input amssym.tex
\usepackage{color}
\usepackage{hyperref}
\usepackage{url}

\newcommand{\ol}[1]{\overline{#1}}
\newcommand{\calb}{\mathcal{B}}
\newcommand{\calp}{\mathcal{P}}

\numberwithin{equation}{section}

\newtheorem{thm}{Theorem}[section]

\newtheorem{cor}[thm]{Corollary}

\newtheorem{prop}[thm]{Proposition}

\newtheorem{que}[thm]{Question}

\newtheorem{cla}[thm]{Claim}

\newtheoremstyle{TheoremNum}
{\topsep}{\topsep}              
{\itshape}                      
{}                              
{\bfseries}                     
{.}                             
{ }                             
{\thmname{#1}\thmnote{ \bfseries #3}}
\theoremstyle{TheoremNum}
\newtheorem{thmrep}{Theorem}

\newtheoremstyle{TheoremNum}
{\topsep}{\topsep}              
{\itshape}                      
{}                              
{\bfseries}                     
{.}                             
{ }                             
{\thmname{#1}\thmnote{ \bfseries #3}}
\theoremstyle{TheoremNum}
\newtheorem{proprep}{Proposition}

\newtheoremstyle{TheoremNum}
{\topsep}{\topsep}              
{\itshape}                      
{}                              
{\bfseries}                     
{.}                             
{ }                             
{\thmname{#1}\thmnote{ \bfseries #3}}
\theoremstyle{TheoremNum}

\theoremstyle{plain}

\newtheorem{definition}[thm]{Definition}

\newtheorem{lemma}[thm]{Lemma}
\newtheorem{proposition}[thm]{Proposition}
\newtheorem{theorem}[thm]{Theorem}

\newtheorem{remark}[thm]{Remark}

\newcommand\be{\begin{equation}}
	\newcommand\ee{\end{equation}}
\newcommand\bea{\begin{eqnarray}}
	\newcommand\eea{\end{eqnarray}}
\newcommand\bi{\begin{itemize}}
	\newcommand\ei{\end{itemize}}
\newcommand\ben{\begin{enumerate}[(a)]}
	\newcommand\een{\end{enumerate}}
\newcommand\bc{\begin{center}}
	\newcommand\ec{\end{center}}
\def\ba#1\ea{\begin{align*}#1\end{align*}}

\newcommand{\R}{\ensuremath{\mathbb{R}}}
\newcommand{\C}{\ensuremath{\mathbb{C}}}
\newcommand{\Z}{\ensuremath{\mathbb{Z}}}
\newcommand{\Q}{\mathbb{Q}}
\newcommand{\N}{\mathbb{N}}

\newcommand{\ve}{\varepsilon}

\newcommand{\lr}[1]{\left\langle#1\right\rangle}
\newcommand{\inv}[1]{{#1}^{-1}}

\newcommand{\norm}[1]{\left\vert\left\vert {#1}\right\vert\right\vert}

\newcommand{\cavg}[2]{\mathop{\mathbb{E}}_{{#1}\in{#2}}\ }
\newcommand{\lpf}[1]{\mathrm{lpf}({#1})}

\newcommand{\zn}{{\mathbb{Z}_N}}

\newcommand{\rcavg}[1]{\mathop{\mathbb{E}}_{#1\in R}}
\newcommand{\lcm}{\ensuremath{\mathrm{lcm}}}

\usepackage{enumerate}

\date{\today}
\title{On the Polynomial Szemer\'edi Theorem in Finite Commutative Rings}
\author{Vitaly Bergelson and Andrew Best}

\begin{document}
	\maketitle
	\begin{abstract}
		The polynomial Szemer\'{e}di theorem from \cite{bl} implies that, for every $\delta \in (0,1)$, every family $\{P_1,\ldots, P_m\} \subset \Z[y]$ of nonconstant polynomials with constant term zero, and all sufficiently large $N \in \N$, every subset of $\{1,\ldots, N\}$ with cardinality at least $\delta N$ contains a nontrivial configuration $\{x,x+P_1(y),\ldots, x+P_m(y)\}$. When the polynomials are assumed independent, one can expect a sharper result to hold over finite fields, which was established by Peluse in \cite{peluse}.
		
		One goal of this article is to explain this and similar results as the consequence of joint ergodicity in the presence of \emph{asymptotic total ergodicity}, introduced in \cite{bb}. Guided by this concept, we establish, over finite commutative rings, a version of the polynomial Szemer\'{e}di theorem for independent polynomials $\{P_1,\ldots, P_m\} \subset \Z[y_1,\ldots, y_n]$, deriving new combinatorial consequences. Recall that the characteristic of a ring $R$ with unity is the least positive integer $N$ (possibly composite) such that $N \cdot 1_R = 0_R$.
		\begin{enumerate}
			\item Let $\mathcal R$ be one of the collections of finite commutative rings indicated in Corollary~\ref{cor: zero density} below (for example, finite fields, $\Z/N\Z$, Galois rings). There exists $\gamma \in (0,1)$ such that, for every $R \in \mathcal R$ such that the least prime factor of its characteristic is sufficiently large, every subset $A \subset R$ with $|A| \geq |R|^{1-\gamma}$ contains a nontrivial configuration $\{x,x+P_1(y),\ldots, x+P_m(y)\}$.  
			\item 	Let $\delta \in (0,1)$. For every finite commutative ring $R$ such that the least prime factor of its characteristic is sufficiently large and every subset $A \subset R$ with $|A| \geq \delta |R|$, the number of pairs $(x,y) \in R \times R^n$ such that $\{x,x+P_1(y),\ldots, x+P_m(y)\}$ is a subset of $A$ with cardinality $m+1$ is roughly the expected number, $(|A|/|R|)^{m+1} \cdot |R|^{n+1}$. 
		\end{enumerate}
		The fact that rings have zero divisors is the source of many obstacles, which we overcome; for example, by studying character sums, we develop a new bound on the number of roots of an integer polynomial over a general finite commutative ring, a result which is of independent interest.
	\end{abstract}

	\section{Introduction}
	
	\subsection{A version of the polynomial Szemer\'{e}di theorem over finite commutative rings}
	We are interested in \emph{nontrivial configurations} that may be found in subsets of finite rings with various properties. Let $S$ be a set, and let $f_1(x,y),\ldots, f_m(x,y) : S\times S \to S$ be some functions. We say that a configuration $\{f_1(x,y),\ldots, f_m(x,y)\}$ is \emph{nontrivial} for given $x,y \in S$ if the set has cardinality $m$ and \emph{degenerated} if the cardinality is less than $m$.
	
	A special case of the polynomial Szemer\'{e}di theorem over $\N = \{1,2,\ldots, \}$ obtained in \cite{bl} and \cite{bm unif} has the following finitary form.
	\begin{theorem}\label{finitary poly Sz}
		Let $\mathbf P = \{P_1(y),\ldots, P_m(y)\} \subset \Z[y]$ be a family of nonconstant polynomials, each with constant term zero, and let $\delta \in (0,1)$. There exists $C > 0$ such that, for every integer $N \geq C$, every subset of the interval $\{1,\ldots, N\}$ of cardinality at least $\delta N$ contains a nontrivial configuration $\{x,x+P_1(y),\ldots, x+P_m(y)\}$ for some $x,y \in \N$. 
	\end{theorem}
	\begin{remark} For example, if $\mathbf P = \{y,y^2\}$, the configuration $\{x,x+y,x+y^2\} \subset \Z$ is nontrivial as long as $y \neq 0,1$, and the value of $C$ guaranteed to exist by Theorem~\ref{finitary poly Sz} accounts for these degenerating values of $y$. Using the fact that the number of real roots of a univariate real polynomial is bounded by its degree, it is clear that a finite family of polynomials $\mathbf P \subset \Z[y]$ will have finitely many values of $y$ which degenerate a configuration $\{x,x+P_1(y),\ldots, x+P_m(y)\}$ over integers. The same bound does not in general hold over finite rings, so ensuring nontriviality of configurations in rings is more involved, as we will see later.
	\end{remark}
	From Theorem~\ref{finitary poly Sz}, it is not difficult to derive a corresponding theorem over finite commutative rings, a result which seems to have gone unnoticed in the literature. In this article, rings have unity, denoted $1_R$ or $1$, and $1_R \neq 0_R$. We recall that the \emph{characteristic} of a finite commutative ring $R$ is the additive order of unity, i.e., the smallest positive integer $N$ such that $N \cdot 1 = 0$. All finite fields have prime characteristic, but finite rings can have composite characteristic; for example, $\Z/N\Z$, denoted $\Z_N$ in this article, has characteristic $N$.
	\begin{theorem}\label{ring poly Sz}
		Let $\mathbf P = \{P_1(y),\ldots, P_m(y)\} \subset \Z[y]$ be a family of nonconstant polynomials, each with constant term zero, and let $\delta \in (0,1)$. There exists $C > 0$ such that for any finite commutative ring $R$ with characteristic at least $C$, any subset $A \subset R$ of cardinality\footnote{We denote the cardinality of a finite set $S$ by $|S|$ or $\#S$.} at least $\delta|R|$ contains a nontrivial configuration $\{x,x+P_1(y),\ldots, x+P_m(y)\}$ for some $x,y \in R$. 
	\end{theorem}
	\begin{proof}
		Applying Theorem~\ref{finitary poly Sz}, let $C > 0$ be such that, for every integer $N \geq C$, every subset of the interval $\{1,\ldots, N\}$ of cardinality at least $\delta N$ contains a nontrivial configuration $\{x,x+P_1(y),\ldots, x+P_m(y)\}$ for some $x,y \in \N$.
		
		Let $R$ be a finite commutative ring with characteristic $N \geq C$, and let $A$ be a subset of cardinality at least $\delta|R|$. As an additive group, $R$ contains the subgroup $\Z_N$ generated by $1$; hence we may decompose $R$ as a finite disjoint union of cosets $J_j$ of $\Z_N$. Since $|A| \geq \delta |R|$, it follows that $A \cap J_j$ has cardinality at least $\delta |J_j| = \delta N$ for at least one $j$. Thus, there exists $r \in R$ and a subset $I$ of $\{1,\ldots, N\}$ such that $A \cap J_j = \{r + i : i \in I\}$ and $I$ has cardinality at least $\delta N$. By choice of $C$, the set $I$ contains a nontrivial configuration $\{x',x'+P_1(y'),\ldots, x'+P_m(y')\}$ for some $x',y' \in \N$. A nontrivial configuration over the interval $\{1,\ldots, N\}$ remains nontrivial when interpreted modulo $N$. Thus, $A \cap J_j$ contains the nontrivial configuration $\{x,x+P_1(y),\ldots, x+P_m(y)\}$, where $x = r + x'$ and $y = y'$.
	\end{proof}
	We collect some remarks on Theorem~\ref{ring poly Sz} below:
	\begin{enumerate}
		\item As can be seen by considering the family $\mathbf P = \{y,2y,\ldots, ky\}$ for a positive integer $k$, the restriction of the conclusion of the theorem to only those rings with sufficiently large characteristic is clearly necessary, since the configuration $\{x,x+y,\ldots, x+ky\}$ is degenerated for every $x,y \in R$ if $R$ has characteristic $k$ or less.
		\item With the conclusion restricted to finite fields, Theorem~\ref{ring poly Sz} was shown in the course of proving \cite[Theorem 5.16]{blm}, and the idea of slicing by cosets comes from there. Moreover, the full strength of the cited theorem allows one to only require that the cardinality of the field, instead of the characteristic, be sufficiently large. It is not possible to strengthen Theorem~\ref{ring poly Sz} in the same way because of local obstruction. To give a trivial counterexample, in a finite commutative ring of characteristic 3, which may be of arbitrarily large finite cardinality, every configuration $\{x,x+3y\}$ is degenerated. One may ask whether there is a similar counterexample in which there actually are nontrivial configurations $\{x,x+3y\}$ for at least some $x,y \in R$ that may in principle be found in a large set $A$, but $A$ nonetheless avoids them. Indeed, there is one: Over the ring $R_m$ defined as the direct product of $m$ copies of $\Z_3$ and one copy of $\Z_9$ with coordinatewise addition and multiplication, the set $A = \Z_3 \times \cdots \times \Z_3 \times \{0,1,2\}$ has cardinality $|R_m|/3$ and contains no nontrivial configurations $\{x,x+3y\}$; the positive integer $m$ was arbitrary.
		\item It is known (see \cite{bll}) that, for a family $\mathbf P = \{P_1,\ldots, P_m\}$ of nonconstant integer-valued polynomials\footnote{A polynomial $P \in \R[y]$ is \emph{integer-valued} if $P(y) \in \Z$ for every $y \in \Z$. Every integer-valued polynomial has rational coefficients.}, Theorem~\ref{finitary poly Sz} holds if and only if $\mathbf P$ is \emph{jointly intersective}, i.e., for every positive integer $k$, there exists 
		an integer $y \in \Z$ such that $P_j(y) \equiv 0 \bmod k$ for each $j$. When $m = 1$ and $\mathbf P = \{P\}$, we usually say $P$ is intersective rather than $\mathbf P$ is jointly intersective. Recalling that integer-valued polynomials may have non-integer rational coefficients, we observe that some jointly intersective families, such as $\mathbf P_1 = \{\frac{y(y-1)}{2}, 3 \cdot \frac{y(y-1)}{2}\}$,\footnote{It follows by \cite[Proposition 6.1]{bll} that $\mathbf P_1$ is jointly intersective.} are composed of such polynomials. Theorem~\ref{ring poly Sz} is by the same proof also true for jointly intersective families, provided that we further restrict the conclusion of the theorem to only those rings with characteristic coprime with any denominator of any coefficient of any polynomial in the family. Consider $\mathbf P_1$. There is no well-defined way to interpret division by 2 in a ring of even characteristic, so there is nothing to assert about such rings in connection with configurations determined by this family. To avoid disclaimers like this, throughout this article we will generally only work with \emph{integer polynomials}, i.e., polynomials with integer coefficients, which are meaningful in any commutative ring.
		\item For a family $\mathbf P$ of nonconstant integer polynomials, $\mathbf P$ is jointly intersective if and only if Theorem~\ref{ring poly Sz} holds. The forward implication follows from the previous remark. It is not hard to see that the backward implication holds by contraposition. Indeed, suppose that $\mathbf P$ is not jointly intersective, and let $k$ be a positive integer such that, for every integer $y \in \Z$, there exists $j$ such that $P_j(y) \not\equiv 0 \bmod k$. Fix a positive integer $m$, and consider the set $A := \{n \in \Z_{km} : n \equiv 0 \bmod k\}$, which has cardinality $m = \frac{1}{k} |\Z_{km}|$. It is clear that $A$ does not contain any configurations of the form $\{x,x+P_1(y),\ldots, x+P_m(y)\}$ for $x,y \in \Z_{km}$, whether nontrivial or degenerated, because the presence of any such configuration would imply that $P_j(y) \equiv 0 \bmod k$ for all $j$. The ring $\Z_{km}$ has arbitrarily large characteristic since $m$ was arbitrary, so the conclusion of Theorem~\ref{ring poly Sz} is false for $\delta < 1/k$. 
	\end{enumerate}
	Some of the previous remarks together imply the following equivalence, which we record here for clarity:
	\begin{theorem}\label{ring poly Sz equiv versions}
		Let $\mathbf P = \{P_1,\ldots, P_m\} \subset \Z[y]$ be a family of nonconstant integer polynomials. The following are equivalent:
		\begin{enumerate}
			\item $\mathbf P$ is jointly intersective, i.e., for every positive integer $k$, there exists an integer $y \in \Z$ such that $P_j(y) \equiv 0 \bmod k$ for each $j \in \{1,\ldots, m\}$.
			\item For each $\delta \in (0,1)$, there exists $C > 0$ such that, for every integer $N > C$, every subset of the interval $\{1,\ldots, N\}$ of cardinality at least $\delta N$ contains a nontrivial configuration $\{x,x+P_1(y),\ldots, x+P_m(y)\}$ for some $x,y \in \N$.
			\item For each $\delta \in (0,1)$, there exists $C > 0$ such that, for any finite commutative ring $R$ with characteristic at least $C$, any subset $A \subset R$ of cardinality at least $\delta|R|$ contains a nontrivial configuration $\{x,x+P_1(y),\ldots, x+P_m(y)\}$ for some $x,y \in R$.
		\end{enumerate}
	\end{theorem}
	\subsection{Motivation: Can we strengthen Theorem~\ref{ring poly Sz} to a ``zero density'' version?}
	In number theory and, more specifically, additive combinatorics, there has been an enduring interest in questions regarding ``zero density'' phenomena in various settings.
	
	To explain what we mean, we recall a classical example of a ``positive density'' phenomenon. In \cite{et}, Erd\H{o}s and Tur\'{a}n conjectured that, for any positive integer $k \geq 2$, for any $\delta \in (0,1)$, for any sufficiently large integer $N$, any subset of $\{1,\ldots, N\}$ of cardinality at least $\delta N$ contains a nontrivial configuration $\{x,x+y,\ldots, x+ky\}$. This was famously proven when $k = 2$ by Roth in \cite{roth} and when $k \geq 3$ by Szemer\'{e}di in \cite{sz1,sz2}.
	
	As is perhaps well known, before any of these proofs appeared, a stronger statement had been ``widely conjectured,''\footnote{See \cite{sasp}. On p. 200 of the much later article \cite{sz2}, Szemer\'{e}di attributes the stronger conjecture to Erd\H{o}s and Tur\'{a}n.} namely, that, for any positive integer $k \geq 2$, for some $\gamma > 0$, for any sufficiently large integer $N > 0$, any subset of the interval $\{1,\ldots, N\}$ of cardinality at least $N^{1-\gamma}$ contains a nontrivial configuration $\{x,x+y,\ldots,x+ky\}$. Salem and Spencer showed this was untrue in \cite{sasp} by constructing sets $A_N \subset \{1,\ldots, N\}$ free of such configurations with cardinality large enough for a disproof; Behrend gave a quantitatively better counterexample in \cite{behr}.
	
	Being harder to deal with, zero density phenomena have remained of significant interest; for example, see the refinements of Szemer\'{e}di's theorem by Gowers in \cite{g-sz1} and \cite{g-sz2}, of Roth's theorem by Bloom and Sisask in \cite{bloomsisask}, and of the $\mathbf{P} = \{y,y^2\}$ case of Theorem~\ref{finitary poly Sz} by Peluse and Prendiville in \cite{pelprend2} and \cite{pelprend1}. See also \cite{eg}, \cite{pel20}, and \cite{prend} for bounds on sets avoiding specific arithmetic or polynomial configurations.
	
	Another relevant aspect of research in this area is the tradeoff in quantitative strength versus generality of results. In some articles, authors restrict themselves to finding specific configurations and/or working in specific settings, and this specificity helps them to obtain quantitative savings that may not be so straightforward in a more general setup. For example, consider the Furstenberg--S\'{a}rk\"{o}zy theorem; given an integer-valued polynomial $P(y)$, let $r_{P}(N)$ denote the cardinality of the largest subset of $\{1,\ldots, N\}$ that does not contain a nontrivial configuration $\{x,x+P(y)\}$. What is traditionally called the Furstenberg--S\'{a}rk\"{o}zy theorem asserts that $r_P(N) = o(N)$ for every nonzero intersective polynomial $P$. In \cite{sar1}, S\'{a}rk\"{o}zy showed by a variant of the Hardy--Littlewood circle method that $r_{y^2}(N) = O(N ( \log \log N)^{2/3}/ (\log N)^{1/3})$, a bound which certainly suffices to conclude $r_{y^2}(N) = o(N)$; in \cite{sar3}, similarly better-than-$o(N)$ bounds on $r_P(N)$ are claimed for $P = y^2-1$ and $P = y^k$ for integers $k > 2$. In contrast, by ergodic methods, $r_P(N) = o(N)$ was proved by Furstenberg in the case $P(y) = y^2$ in \cite{diag77}, generalized to the case\footnote{This is equivalent to the case $m=1$ of Theorem~\ref{finitary poly Sz}.} that $P \in \Z[y] \setminus \{0\}$ satisfies $P(0) = 0$ in \cite{furstenbergbook}, and in short order maximally generalized to the case that $P \neq 0$ is intersective by \cite{kmf}; the softness of the ergodic approach trades quantitative strength for generality.
	
	Recently, the situation has developed more. Stronger-than-$o(N)$ bounds for $r_P(N)$ were shown for any intersective polynomial $P \neq 0$ in, for example, \cite{rice}. On the other side, it has been established that the Furstenberg--S\'{a}rk\"{o}zy phenomenon has more general scope, extending to so-called generalized polynomials and to (generalized) polynomials with arguments taken from certain rarefied sets of integers. Recalling that $[\cdot ]$ denotes the greatest integer function, the bound $r_P(N) = o(N)$ is known for, e.g., $P(y) = [[\sqrt{2}y]\sqrt{3}y^2]^3[\sqrt{5} y]y$ (see \cite[Proposition 2.5]{bh}) or $P(y) = [7.9y^3 + \sqrt{2}y^2 + 3y]$ (see \cite[Proposition 6.18]{bhs}). For examples where the argument of a (generalized) polynomial is from a prescribed IP set, see \cite{bergfm} and \cite{bhm}. The authors are not aware of any quantitative bounds for these kinds of results.

	When working with general finite commutative rings, a common theme is that the tradeoff between quantitative strength (broadly understood) and generality of scope is rather harsh. We have seen an example of this already; see the second remark under Theorem~\ref{ring poly Sz} above, which explains a difference with the finite field situation. As we discuss our principal results below, this theme will develop in the background, manifesting, for example, in the second remark under Theorem~\ref{main thm for intro} or in the difference between the scope of the formulations of Corollary~\ref{cor: zero density} and Corollaries~\ref{cor: one config cor},~\ref{cor: one config cor, all same set},~and~\ref{cor: one config cor, simpler}. As the reader will see, our principal results are quite broad, and the resulting picture is complex.
	
	
	Let us discuss now the possibility of obtaining a certain zero density version of Theorem~\ref{ring poly Sz}. Consider the following question. We say polynomials $P_1,\ldots, P_m \in \Z[y]$ are \emph{essentially distinct} if for all distinct $i, j$, $P_i - P_j$ is not a constant polynomial.
	\begin{que}\label{naive zero density problem}
		For which families $\mathbf P = \{P_1,\ldots, P_m\} \subset \Z[y]$ of nonconstant, essentially distinct polynomials is it true that there exists $\gamma \in (0,1)$ depending only on $\mathbf P$ such that for any finite commutative ring $R$ and any subset $A \subset R$ of cardinality at least $|R|^{1-\gamma}$, $A$ contains a nontrivial configuration $\{x,x+P_1(y),\ldots, x+P_m(y)\}$ for some $x,y \in R$?
	\end{que}
	We collect some remarks on the formulation of this question before seeking to address it:
	\begin{enumerate}
		\item The hypothesis of essential distinctness rules out some families $\mathbf P$ that for trivial reasons are not candidates. For example, given a finite commutative ring $R$, one can easily define a large subset $A \subset R$ that avoids nontrivial configurations $\{x,x+2\}$ (and hence also avoids nontrivial configurations $\{x,x+P(y),x+P(y)+2\}$ for any integer polynomial $P$). 
		\item Even after assuming essential distinctness, there can be another issue that is related to ring torsion. For many choices of $\mathbf P$, an obvious obstacle to $\mathbf P$ being an answer to Question~\ref{naive zero density problem} is the existence of rings with characteristic from a small prescribed set. For example, every configuration $\{x,x+3y,x+5y^2\}$ is degenerated in characteristic 3 or 5. To address this issue, one could make one of several naive modifications to the question, such as, for example, restricting the conclusion to only those finite commutative rings with sufficiently large characteristic depending on $\mathbf P$, which clearly addresses the issue with families like $\mathbf P = \{3y,5y^2\}$ and is sufficient for the sake of discussion. Below this remark, we will give examples of families $\mathbf P$ which are not answers to Question~\ref{naive zero density problem}; none of these non-answers would become answers on account of this modification (or a similar one) pertaining to torsion.
	\end{enumerate}
	
	Our discussion above of some of the developments surrounding the conjectures of Erd\H{o}s and Tur\'{a}n was not merely for historical reasons: the mentioned counterexample sets $A_N \subset \{1,\ldots, N\}$ from \cite{sasp} can be interpreted as sets modulo $(k+1)N$ to obtain large subsets of $\mathbb{Z}_{(k+1)N}$ free of nontrivial configurations\footnote{Interpreting the set $A_N$ modulo $(k+1)N$ instead of, say, modulo $N$ ensures that the only possible $\Z_{(k+1)N}$ configurations $\{x,x+y,\ldots, x+ky\}$ it could contain would also be configurations over integers, which are already avoided by construction; in other words, this trick ensures that there cannot be any configurations $\{x,x+y,\ldots,x+ky\}$ that exist because of the ``wraparound'' aspect of $\Z_{(k+1)N}$.} $\{x,x+y,\ldots, x+ky\}$, showing that, for any integer $k \geq 2$, the family $\mathbf P = \{y,2y,\ldots, ky\}$ is not an answer to Question~\ref{naive zero density problem}.
	
	Turning to other families of polynomials yields other kinds of non-answers. For a fixed prime $p$ and positive integer $k$, consider the ring $R = \Z_{p^k}$. The set $A := \{x \in \Z_{p^k} : x \equiv 0 \bmod p \}$ contains no configurations of the form $\{x,x+y,x+y^2+1\}$ and has cardinality $|R|^{(k-1)/k}$. Since $\lim_{k\to\infty} (k-1)/k = 1$, the family\footnote{Note that this family is not jointly intersective, so even the third statement in Theorem~\ref{ring poly Sz equiv versions} does not hold, never mind the ``zero density'' content of Question~\ref{naive zero density problem}. However, a variant of the third statement nonetheless holds; see Corollary~\ref{cor: one config cor, simpler}.} $\mathbf P = \{y,y^2+1\}$ is not an answer to Question~\ref{naive zero density problem}. 
	
	Consider one last example. In a discussion with Alan Loper, he drew our attention to the ring which we will presently define and pointed out that many of its elements square to zero. Let $p$ be prime, and let $R$ be the quotient ring formed by taking $\Z_p[x_1,\ldots, x_k]$ modulo the ideal generated by all elements of the form $x_ix_j$, where $i,j \in \{1,\ldots, k\}$ are possibly indistinct. Then $R = \{c_0 + \sum_{i=1}^k c_ix_i : c_0,\ldots, c_k \in \Z_p\}$. Let $A \subset R$ be the subset of elements such that $c_0 = 0$. We claim that $A$ contains no nontrivial configurations of the form $\{x, x+y^2\}$, where $x,y \in R$. Indeed, any two elements of $A$ differ by an element with constant term zero, so $y^2$ must have constant term zero. However, for every $y \in R$, by explicitly writing $y = c_0 + \sum_{i=1}^k c_ix_i$ and computing $y^2 = c_0^2 + \sum_{i=1}^k 2c_ic_0 x_i$, we see that if $y^2$ has constant term zero, then $y^2 = 0$. By definition, if $\{x,x+y^2\}$ is a nontrivial configuration, that excludes the situation that $y^2 = 0$, proving the claim. Moreover, $|A| = |R|^{k/(k+1)}$. In this example, $k$ is arbitrary, so again the family $\mathbf P = \{y^2\}$ is not an answer to Question~\ref{naive zero density problem}. Similar examples can be constructed for $\mathbf P = \{y^\ell\}$ for any integer $\ell > 2$.
	
	The discussion above indicates that many polynomial families are not answers to Question~\ref{naive zero density problem}. However, the situation changes if one assumes the family $\mathbf P$ appearing in Question~\ref{naive zero density problem} is \emph{independent}, i.e., the only integer linear combination of the $P_i$ that is constant is the trivial combination with each $P_i$ multiplied by zero, and if one restricts the conclusion to certain (rather wide) collections of rings $\mathcal R$. Independent families are, in some sense, opposite to ``degenerated'' families like $\{y,2y,\ldots,ky\}$. The following question is a version of Question~\ref{naive zero density problem} that accounts for this narrowing of focus.
	
	\begin{que}\label{restricted zero density problem}
		Fix an independent family $\mathbf P = \{P_1,\ldots, P_m\} \subset \Z[y]$. For which collections $\mathcal R$ of finite commutative rings is it true that there exists $\gamma \in (0,1)$ such that for every $R \in \mathcal R$, every subset $A \subset R$ of cardinality at least $|R|^{1-\gamma}$ contains a nontrivial configuration $\{x,x+P_1(y),\ldots, x+P_m(y)\}$ for some $x,y \in R$?
	\end{que}

	Relevant\footnote{There has been some recent progress on the situation in finite prime fields with certain special non-independent families, but it would take the discussion a bit afield, as such work, to our knowledge, does not give an answer to Question~\ref{restricted zero density problem}, but instead to some variant thereof dealing with (iterated) logarithmic factors rather than power savings. See, e.g., \cite{kucafurther} and \cite{leng}.} previous work is all recent and concentrated on finite fields. Initiating this line of inquiry, Bourgain and Chang showed in \cite{bc} an affirmative answer to Question~\ref{naive zero density problem} for the family $\mathbf P = \{y,y^2\}$ and the collection of finite prime fields $\mathcal R = \{\Z_p : p \geq C \text{ prime}\}$, where $C \in \N$ is some constant. Peluse in \cite{pel3term} and Dong, Li, and Sawin in \cite{dls} showed an affirmative answer for any independent family $\mathbf P = \{P_1,P_2\} \subset \Z[y]$ where $P_1(0) = P_2(0) = 0$ and the collection of all finite fields $\mathcal R = \{\mathbb{F}_q : q = p^n \text{ a prime power and } p \geq C\}$, again where $C \in \N$ is some constant. Finally, Peluse showed in \cite{peluse} an affirmative answer for any independent family $\mathbf P = \{P_1,\ldots, P_m\} \subset \Z[y]$ where $P_1(0) = \ldots = P_m(0) = 0$, also for all finite fields with sufficiently large characteristic. The approach in \cite{peluse} did not depend on algebraic geometry, which had played a significant role in \cite{dls} and her earlier work \cite{pel3term}. For recent developments over finite fields in the complementary situation when the characteristic is assumed to be small, see \cite{lisauermann}, \cite{ab1}, and \cite{ab2}.
	
	In this article, our main theorem is a multifaceted generalization of the main result in \cite{peluse} to the setting of finite commutative rings and for independent\footnote{In the multivariable case, the definition of an \emph{independent} family remains the same as above.} families of multivariable polynomials. We defer discussion of the method until Subsection~\ref{subsec: methods}.
	
	Let us prepare to state our main theorem now. Given a positive integer $N$, we denote by $\lpf N$ its least prime factor; when $N$ is the characteristic of a ring, $\lpf N$ has both an algebraic interpretation and an ergodic one related to \emph{asymptotic total ergodicity}, both of which we will explain in the next subsection. We say a complex-valued function is \emph{1-bounded} if it takes values in $\{z \in \C : |z| \leq 1\}$.
	\begin{theorem}\label{main thm for intro}
		Let $\mathbf P = \{ P_1,\ldots,P_{m}\} \subset \Z[y_1,\ldots,y_n]$ be an independent family of polynomials. There exist $C,\gamma > 0$ such that the following holds. For any finite commutative ring $R$ with characteristic $N$ satisfying $\lpf N > C$ and any 1-bounded functions $f_0,\ldots, f_{m} : R \to \C$,
		\begin{multline}\label{eqn in main thm}
			\Bigg| \frac{1}{|R|^{n+1}}\sum_{x,y_1,\ldots, y_n \in R} f_0(x)f_1(x+P_1(y_1,\ldots,y_n))\cdots f_{m}(x+P_{m}(y_1,\ldots,y_n)) \\ - \left( \frac{1}{|R|} \sum_{x\in R} f_0(x) \right) \dots \left( \frac{1}{|R|} \sum_{x\in R} f_m(x) \right) \Bigg| \ \leq \ \lpf{N}^{-\gamma}.
		\end{multline}
	\end{theorem}
	We collect some remarks below and then state corollaries of the theorem.
	\begin{enumerate}
		\item The combinatorial strength of Theorem~\ref{main thm for intro} comes both from the fact that the $f_i$ may be different and from the fact that the leftmost average in \eqref{eqn in main thm} is approximately the product of the integrals of $f_i$. One consequence of the first fact is that, in the formulation of the theorem, the $P_i$ need not satisfy $P_i(0) = 0$. A more important consequence of both facts is that, for subsets $A_0,\ldots, A_m$ of $R$ with $\prod_{i=0}^m (|A_i|/|R|)$ not too small, there are many $(x,y) \in R \times R^n$ such that $x \in A_0$, $x + P_1(y) \in A_1$, ..., and $x + P_m(y) \in A_m$, in fact roughly as many $(x,y)$ as statistically expected. See Proposition~\ref{prop: config count} immediately following these remarks and the subsequent discussion for a precise formulation of this result and several of its consequences.
		\item Let $R = \Z_{p^k}$ for a prime $p > 2$ and a positive integer $k$. The set $A = \{x \in \Z_{p^k} : x \equiv 0 \bmod p\}$ contains no configurations of the form $\{x,x+y,x+y^2+1\}$ and has cardinality $|A| = p^{k-1} = |R|^{(k-1)/k}$, whence, for $f_0 = f_1 = f_2 = 1_A$, the left-hand side of \eqref{eqn in main thm} equals $|R|^{-3/k}$. Therefore, at least for the family $\mathbf P = \{y,y^2 + 1\}$, it is not possible to strengthen the error term in \eqref{eqn in main thm} to $|R|^{-\gamma}$. Using a similar idea and the same family, by considering the set of multiples of $\lpf N$ in $\Z_N$, where $N$ is a very highly composite number\footnote{For example, one may take squarefree $N$ such that $N > \lpf{N}^4$.}, one can see that it is also impossible to obtain in \eqref{eqn in main thm} an error term like $N^{-\gamma}$. For context, note that if a finite commutative ring $R$ has characteristic $N$, then $|R| \geq N \geq \lpf N$. When restricting attention to finite fields, the better error term $|R|^{-\gamma}$ may be obtained, at least in the case $n = 1$ (see \cite{peluse}).
		\item The assumption of independence is in fact necessary for Theorem~\ref{main thm for intro} to hold. Thus, linear independence of polynomials is necessary and sufficient for the approximate ``statistical independence'' exhibited by \eqref{eqn in main thm}. To see the necessity, suppose that $c_1,\ldots, c_m \in \Z$, not all zero, are such that $\sum_{i=1}^m c_iP_i$ is constant. Let $R$ be a finite commutative ring with characteristic $N$ satisfying $\lpf N > \max\{|c_1|,\ldots, |c_m|\}$. Let $\chi$ be a nontrivial additive character\footnote{An additive character is a homomorphism $(R,+)$ to $\{z \in \C: |z| = 1\}$. In the context of finite fields, it is common to refer to ``additive characters'' and ``multiplicative characters'' to indicate which operation is respected by the character. Since we will not discuss multiplicative characters of rings, it is not strictly necessary for us to use the word ``additive'' in this article, but we will do so for clarity.} of $R$. Then, taking $f_0 = \chi^{-\sum_{i=1}^m c_i}$ and $f_i = \chi^{c_i}$ for $i \in \{1,\ldots, m\}$, we observe that, by assumption on $\lpf N$, at least one of $f_0,f_1,\ldots, f_m$ is a nontrivial additive character of $R$ and hence the left-hand side of \eqref{eqn in main thm} equals 1.
		\item It is not clear whether the assumption $\lpf{\mathrm{char}(R)} > C$ may be weakened, say, to $\mathrm{char}(R) > C$. As indicated in the second remark following Question~\ref{naive zero density problem}, it is necessary to restrict the torsion on the ring in order to hope for interesting results, but there is somewhat of a gap between the former and latter assumptions.
		\item The quite special case $\mathbf P = \{P(y)\} \subset \Z[y]$ of Theorem~\ref{main thm for intro} generalizes\footnote{Up to a change in notation, inequality (2.44) in \cite{bb}, which applies only to rings of the form $\zn$, is the claim which is generalized here.} the main theorem in \cite{bb}, and this case already requires the use of new facts about general finite commutative rings. As illustration, we prove the case $P(y) = y^2$ in Proposition~\ref{example: main thm base} in Subsection~\ref{subsec: example base case}.
	\end{enumerate}
	
	There are many combinatorial consequences of Theorem~\ref{main thm for intro}. We start with the main corollary, from which the others are derived:
	
	\begin{proposition}\label{prop: config count}
		Let $\ve \in (0,1]$. Let $\mathbf P = \{ P_1,\ldots,P_{m}\} \subset \Z[y_1,\ldots,y_n]$ be an independent family of polynomials. There exists $\gamma = \gamma(\mathbf P)$ such that, for any finite commutative ring $R$ with $\lpf{\mathrm{char}(R)}$ sufficiently large and any subsets $A_0,\ldots, A_m$ of $R$ such that
		\begin{equation}\label{intro eqn 6}
			\ve|A_0|\cdots |A_m| \ > \ |R|^{m+1} \cdot \lpf{\mathrm{char}(R)}^{-\gamma},
		\end{equation}
		the normalized number of nontrivial configurations
		\begin{equation}
			S \ := \ \frac{\# \{ (x,y) \in R\times R^{n} : (x,x+P_1(y),\ldots, x+P_m(y)) \in A_0 \times A_1 \times \cdots \times A_m \} }{|R|^{n+1}}	
		\end{equation}
		satisfies
		\begin{equation}
			(1-\ve)\frac{|A_0|\cdots |A_m|}{|R|^{m+1}} \ < \ S \ < \ (1+\ve)\frac{|A_0|\cdots |A_m|}{|R|^{m+1}}.
		\end{equation}
	\end{proposition}
	\begin{remark}
		The number $S$ is thus close to the probability that $m+1$ independently uniformly randomly chosen elements of $R$ respectively lie in $A_0, \ldots, A_m$---it is as though each of $x$, $x + P_1(y)$, ..., $x+P_m(y)$ is such.
	\end{remark}
	
	This corollary is proved in Section~\ref{sec: derivation of zero and positive density results}, and we label it as a Proposition because its proof requires some input besides Theorem~\ref{main thm for intro}. The proof is complicated by the fact that, given an independent family $\mathbf P = \{P_1,\ldots, P_m\} \subset \Z[y_1,\ldots, y_n]$ and subsets $A_0,\ldots, A_m \subset R$ of a ring, one should carefully estimate the proportion of configurations $(x,x+P_1(y),\ldots, x + P_m(y)) \in A_0 \times A_1 \times \cdots \times A_m$ that are degenerated, a task which is nontrivial over a general finite commutative ring, for reasons we will explore in the discussion following Proposition~\ref{prop: sufficient condition for zero density result}. This is quite necessary: As we saw in the example from Alan Loper below Question~\ref{naive zero density problem}, all of the configurations $(x,x+y^2) \in A \times A$, where $A$ is the same set as in the example, are degenerated. (This situation does not contradict Proposition~\ref{prop: config count} because the condition \eqref{intro eqn 6} is not satisfied for any choice of $\ve$.)
	
	From Proposition~\ref{prop: config count}, we deduce two kinds of corollaries. The first kind is ``universal'' in scope; that is, the results are valid for all finite commutative rings, provided $\lpf{\mathrm{char}(R)}$ is large enough. The second kind, being restricted to certain subcollections of finite commutative rings, is more limited in scope but is quantitatively stronger. It is this second kind of corollary that will answer Question~\ref{restricted zero density problem} by giving a ``zero density'' version of Theorem~\ref{ring poly Sz}. The existence of these two kinds of corollaries recalls the previous discussion of tradeoffs between scope and strength when working with rings.
	
	We now give three corollaries of the ``universal'' kind, which may be easier to parse than Proposition~\ref{prop: config count}:
	
	\begin{cor}\label{cor: one config cor}
		Let $\mathbf P = \{ P_1,\ldots,P_{m}\} \subset \Z[y_1,\ldots,y_n]$ be an independent family of polynomials. There exists $\gamma = \gamma(\mathbf P)$ such that, for any finite commutative ring $R$ with $\lpf{\mathrm{char}(R)}$ sufficiently large (depending on $\mathbf P$) and any subsets $A_0,\ldots, A_m \subset R$ such that
		\begin{equation}
			|A_0|\cdots |A_m| \ > \ |R|^{m+1} \cdot \lpf{\mathrm{char}(R)}^{-\gamma},
		\end{equation}
		there is a nontrivial configuration $(x,x+P_1(y),\ldots, x+P_m(y))$ contained in $A_0 \times A_1 \times \cdots \times A_m$ for some $(x,y) \in R \times R^n$. 
	\end{cor}
	\begin{proof}
		Take $\ve = 1$ in Proposition~\ref{prop: config count}, so that $S > 0$.
	\end{proof}
	\begin{cor}\label{cor: one config cor, all same set}
		Let $\mathbf P = \{ P_1,\ldots,P_{m}\} \subset \Z[y_1,\ldots,y_n]$ be an independent family of polynomials. There exists $\gamma = \gamma(\mathbf P)$ such that, for any finite commutative ring $R$ with $\lpf{\mathrm{char}(R)}$ sufficiently large (depending on $\mathbf P$), every set $A \subset R$ with  
		\begin{equation}
			|A| \ > \ |R|/ \lpf{\mathrm{char}(R)}^{\gamma}
		\end{equation}
		contains a nontrivial configuration $\{x,x+P_1(y),\ldots, x+P_m(y)\}$. 
	\end{cor}
	\begin{proof}
		Take $A_0 = \cdots = A_m$ in Corollary~\ref{cor: one config cor}. 
	\end{proof}
	
	\begin{cor}\label{cor: one config cor, simpler}
		Let $\delta \in (0,1)$, and let $\mathbf P = \{P_1(y),\ldots, P_m(y)\} \subset \Z[y_1,\ldots,y_n]$ be an independent family. For any finite commutative ring $R$ with $\lpf{\mathrm{char}(R)}$ sufficiently large (depending on $\mathbf P$ and $\delta$) and any subsets $A_0,\ldots, A_m \subset R$ such that
		\begin{equation} |A_0|\cdots |A_m| \ \geq \ \delta|R|^{m+1},
		\end{equation}
		there is a nontrivial configuration $(x,x+P_1(y),\ldots, x+P_m(y))$ contained in $A_0 \times A_1 \times \cdots \times A_m$ for some $(x,y) \in R \times R^n$.
	\end{cor}
	\begin{proof} Take $\ve = 1$ in Proposition~\ref{prop: config count} and, if necessary, increase the threshold on $\lpf{\mathrm{char}(R)}$ so that the assumption $|A_0|\cdots |A_m| \geq \delta |R|^{m+1}$ always implies \eqref{intro eqn 6}.
	\end{proof}
	\begin{remark}
		When $n = 1$, each $P_i(0) = 0$, and $A_0 = \dots = A_m$, the conclusion of Corollary~\ref{cor: one config cor, simpler} is implied by Theorem~\ref{ring poly Sz}.
	\end{remark}
	
	Let us consider the second kind of corollaries to Theorem~\ref{main thm for intro}. The following proposition provides a sufficient condition on $\mathcal R$ under which Question~\ref{restricted zero density problem} has an affirmative answer.
	
	\begin{prop}\label{prop: sufficient condition for zero density result}
		Let $\mathcal R$ be a collection of finite commutative rings, and suppose there exists $\alpha \in (0,1)$ such that $\lpf{\mathrm{char}(R)} \geq |R|^\alpha$ for any $R \in \mathcal R$.
		Then, for any independent family $\mathbf P = \{P_1,\ldots, P_m\} \subset \Z[y_1,\ldots, y_n]$, there exists $\gamma \in (0,1)$ such that, for any ring $R \in \mathcal R$ with $\lpf{\mathrm{char}(R)}$ sufficiently large (depending on $\mathbf P$), for any subsets $A_0,\ldots, A_m \subset R$ such that
		\begin{equation}
			|A_0|\cdots |A_m| \ \geq \ |R|^{(m+1)(1-\gamma)},
		\end{equation}
		there is a nontrivial configuration $(x,x+P_1(y),\ldots, x+P_m(y)) \in A_0 \times \cdots \times A_m$ for some $(x,y) \in R \times R^n$.
	\end{prop}
	\begin{proof} Take $\ve = 1$ in Proposition~\ref{prop: config count}, and write $\gamma'$ for the constant that is asserted to exist. Now take $\gamma := \frac{\alpha \gamma'}{m+1}$.
	\end{proof}
	
	In Corollary~\ref{cor: zero density} below, we will derive some corollaries of the second kind to Theorem~\ref{main thm for intro}.
	
	Before doing that, we make a small digression to illustrate some considerations behind Propositions~\ref{prop: config count}~and~\ref{prop: sufficient condition for zero density result}, each of which separately introduce some subtlety that must be addressed in the process of answering Question~\ref{restricted zero density problem} using Theorem~\ref{main thm for intro}. In particular, we will give context both to the problem of estimating how many configurations $(x,x+P_1(y),\ldots, x+P_m(y))$ are degenerated (apropos of Proposition~\ref{prop: config count}) and to the condition $\lpf{\mathrm{char}(R)} \geq |R|^\alpha$ appearing in Proposition~\ref{prop: sufficient condition for zero density result}. First, we show how to derive an answer to Question~\ref{restricted zero density problem} in the simpler case that $\mathbf P = \{y,y^2\}$ and $\mathcal R$ is the collection of finite prime fields $\Z_p$ by using a limited version of our Theorem~\ref{main thm for intro} that was proved in \cite{bc}. With this point of reference, we can then explain the two points of complication that arise in the general situation.
	
	It follows from Theorem 1.1 in \cite{bc} that there exists a constant $C > 0$ such that, for every sufficiently large prime $p$ and every subset $A \subset \Z_p$, 
	\begin{equation}\label{intro eqn 1}
		\left| \cavg{x,y}{\Z_p} 1_A(x)1_A(x+y)1_A(x+y^2) - (|A|/p)^3 \right| \ \leq \ C|A|^{3/2}p^{-8/5} \ \leq \ Cp^{-1/10},
	\end{equation}
	where the second inequality is meant to simplify the discussion and follows from the fact that $|A| \leq p$. Denoting by $M$ the cardinality of the set $S := \{(x,y) \in \Z_p \times \Z_p : x,x+y,x+y^2 \in A\}$ and by $M_1$ (resp. $M_2$) the cardinality of the subset of $S$ where $(x,x+y,x+y^2)$ is a nontrivial (resp. degenerated) configuration, \eqref{intro eqn 1} becomes
	\begin{equation}
		\left|M/p^{2} - (|A|/p)^{3} \right| \ \leq \ Cp^{-1/10}.
	\end{equation}
	Our goal is to find $\gamma \in (0,1)$ such that, for sufficiently large $p$, whenever $|A| \geq p^{1-\gamma}$, $M_1 > 0$. Since $M = M_1 + M_2$, it follows that
	\begin{equation}
		M_1/p^{2} \ \geq \ -M_2/p^{2} + (|A|/p)^{3} - Cp^{-1/10},
	\end{equation}
	so it would suffice to show that the right-hand side exceeds zero, i.e.,
	\begin{equation}
		M_2/p^{2} + Cp^{-1/10} \ < \ (|A|/p)^{3}.
	\end{equation}
	The configuration $(x,x+y,x+y^2)$ is degenerated if and only if $y = 0$ or $y = 1$. Therefore,
	\begin{equation}\label{intro eqn 2}
		M_2/p^{2} \ \leq \ 2|A|/p^2 \ \leq \ 2/p,
	\end{equation}
	whence
	\begin{equation}
		M_2/p^{2} + Cp^{-1/10} \ \leq \ 2/p + Cp^{-1/10} \ \leq \ p^{-1/11},
	\end{equation}
	where the right-most inequality holds for sufficiently large $p$. Taking $\gamma = 1/36$ and $|A| \geq p^{1-\gamma}$, we observe that
	\begin{equation}
		(|A|/p)^{3} \ > \ p^{-1/11},
	\end{equation}
	guaranteeing $M_1 > 0$, which affirmatively answers Question~\ref{restricted zero density problem} in this special case (namely $\mathbf P = \{y,y^2\}$ and $\mathcal R = \{\Z_p : p \geq C' \text{ prime}\}$ for some constant $C' \in \N$).
	
	%
	
	We now indicate two points where the general problem of affirmatively answering Question~\ref{restricted zero density problem} using Theorem~\ref{main thm for intro} becomes harder.
	
	First, the analogous assertion to \eqref{intro eqn 1}, which arises from applying Theorem~\ref{main thm for intro}, is weaker, since the bound on the right-hand side is a negative power of $\lpf{\mathrm{char}(R)}$, rather than\footnote{See the second remark under the theorem for a discussion of the impossibility of strengthening the bound.} a negative power of $|R|$. In the prime field case, $|R| = \mathrm{char}(R) = \lpf{\mathrm{char}(R)}$, but in a general finite commutative ring, there is almost no quantitatively useful relationship between $|R|$, $\mathrm{char}(R)$, $\lpf{\mathrm{char}(R)}$, and powers or (iterated) logarithms of these. Thus, when one desires to assert the existence of nontrivial configurations in sets whose cardinality is measured in terms of a power of $|R|$, introducing an (apparently) incommensurable term like $\lpf{\mathrm{char}(R)}$ complicates matters.
	
	The second point is related to another phenomenon peculiar to rings, namely the relative untamedness of the number of roots of a polynomial. Consider the inequality \eqref{intro eqn 2}. We estimated $M_2 \leq 2|A|$ by counting the exact number of $y \in \Z_p$ that degenerate a configuration $(x,x+y,x+y^2)$---the $2$ came from the two values $y= 0$ and $y = 1$. For more general configurations (still over finite fields), the number of degenerating $y$ values may be simply estimated. For example, if $\{P_1,P_2\} \subset \Z[y]$ is an independent family of polynomials with maximal degree $d$, then at most $\binom{3}{2} d$ values of $y$ degenerate the configuration $\{x,x+P_1(y),x+P_2(y)\}$, since each of the equations $P_1(y) = 0$, $P_2(y) = 0$, and $P_1(y) = P_2(y)$ has at most $d$ solutions over a finite field. Thus, over finite fields, we have an estimate $M_2 \leq C'|A|$, where $C'$ is a constant that depends on the number of polynomials in $\mathbf P$ and on their maximal degree, and it is safe to introduce this constant into the analysis, because it does not depend on $|R|$, $\mathrm{char}(R)$, and $\lpf{\mathrm{char}(R)}$.
	
	Now, over finite commutative rings, it remains difficult to count the values of $y$ that degenerate a given configuration, so we hope to use the same ``pairwise sum of bounds'' argument as above (and we also hope that the argument generalizes to multiple variables).
	
	Thus, given $P \in \Z[y_1,\ldots, y_n]$ and a finite commutative ring $R$, we need a nontrivial estimate on the number of $(y_1,\ldots, y_n) \in R^n$ such that $P(y_1,\ldots, y_n) = 0$.
	
	Outside of finite fields, this is a hard problem. Certainly the degree of $P$ will not suffice, as can be seen by observing that, in the example from Alan Loper below Question~\ref{naive zero density problem}, the mentioned set $A$ is the set of zeros of $P(y) = y^2$ and has cardinality a power of $|R|$ that approaches 1 as the parameter $k$ is increased. In light of this example, it seems a strong bound is objectively impossible to obtain in general.
	
	In our case, the desired bound, regardless of its quality, only needs to be valid for $R$ with sufficiently large $\lpf{\mathrm{char}(R)}$ (which, as we will see in Subsection~\ref{subsec: heuristic}, is the minimal additive order of any nonzero element), which helps somewhat. Even so, we are not aware of any estimate over general finite commutative rings, even in the univariate case. Thus, we develop our own---see Lemma~\ref{prop: bound on number of roots}, which we also record for its possible independent interest. The bound from this lemma depends on $|R|$ and $\lpf{\mathrm{char}(R)}$, as well as the degree of the polynomial and the number of variables.
	
	Now that various inputs to the argument which is purported to affirmatively answer Question~\ref{restricted zero density problem} seem to depend on both $|R|$ and $\lpf{\mathrm{char}(R)}$, we must confront the incommensurability of these quantities. If we restrict to certain collections $\mathcal R$ of finite commutative rings by stipulating that $\lpf{\mathrm{char}(R)}$ is at least some fixed positive power of $|R|$, everything becomes commensurable, and the argument can be carried out. This stipulation can be roughly understood as a requirement that the whole collection of rings not vary too much in terms of ``relative amount of torsion.'' This abstract-sounding condition yields several natural collections of rings (going beyond finite fields), which we describe in Corollary~\ref{cor: zero density} below. Moreover, a condition of this sort seems to be justified \emph{a priori} by the influence of torsion on questions about polynomial configurations (cf. the second remark below Question~\ref{naive zero density problem}) and by the wildness of torsion in general finite commutative rings.
	
	In preparation to formulate Corollary~\ref{cor: zero density}, let us recall two classical collections of finite commutative rings and define a third. Given a prime $p$ and positive integers $n$ and $r$, the Galois ring $R = GR(p^n,r)$ is the quotient ring $\Z_{p^n}[x]/(f(x))$, where $f \in \Z[x]$ is a monic polynomial of degree $r$ that is irreducible modulo $p$. Note that $R$ is a commutative ring with $|R| = p^{nr}$ and $\mathrm{char}(R) = p^n$. Moreover, it is known\footnote{See, e.g., \cite{ragh}, in which Raghavendran named these rings ``Galois rings'' and described their properties, both independently of Janusz, who did the same in \cite{janusz}.} that the definition of $R$ does not depend on the choice of $f$ among monic irreducible (mod $p$) polynomials of degree $r$. Every finite field is a Galois ring (since $\mathbb{F}_{p^r} = GR(p,r)$), as is every $\zn$ with $N$ a prime power (since $\mathbb{Z}_{p^n} = GR(p^n,1)$). Thus, Galois rings encompass two natural and quite distinct kinds of rings.
	
	A finite commutative ring is called a \emph{chain ring} if all its ideals form a chain under inclusion. Chain rings appear in various areas of mathematics\footnote{Finite commutative chain rings appear in geometry and in algebraic number theory respectively as coordinatizing rings of Pappian Hjelmslev planes (see \cite{klingenberg} and references in \cite{tv}) and as quotient rings of rings of integers in number fields (see \cite{krull}), and there has been recent interest in constructions of partial difference sets and relative difference sets---see, e.g., \cite{leungma}---as well as in codes over chain rings, particularly over Galois rings. Most information in this footnote we learned from \cite{clarkliang}.}. Every finite commutative chain ring is the quotient of $GR(p^n,r)[x]$ by the ideal generated by $p^{n-1} x^t$ and an Eisenstein polynomial $g(x) \in GR(p^n,r)[x]$ of degree $k$, i.e., $g(x) = x^k - \sum_{i=0}^{k-1} pa_{i}x^i$ with $a_0$ a unit of $GR(p^n,r)$, where $p$ is prime and $n$, $r$, $k$, and $t$ are positive integers satisfying
	\begin{equation}
		\begin{cases} t = k & n = 1, \\
			1 \leq t \leq k & n \geq 2.
		\end{cases}
	\end{equation}
	The numbers $p$, $n$, $r$, $k$, $t$ above are called the invariants of the given chain ring, which has cardinality $p^{r(nk-k+t)}$ and characteristic $p^n$. See \cite[pp. 307--349]{brmcdonald} for proofs of these claims. Taking $t = k = 1$, we see that Galois rings are chain rings.
	
	The notion of Galois ring may also be slightly generalized in another direction while still retaining nice enough properties for our purposes. Let $N > 1$ be an integer. We say a polynomial $f \in \Z[x]$ is \emph{irreducible over $\zn$} if $f(x)$ is not a unit\footnote{Let $f(x) = a_0 + \sum_{i=1}^k a_i x^i \in \Z[x]$. Then $f(x)$ is a unit in $\zn[x]$ if and only if $a_0$ is a unit modulo $N$ and each $a_i$ with $i > 0$ is nilpotent in $\zn$, i.e., each $a_i$ is a multiple of the product of all the distinct prime factors of $N$.} in $\zn[x]$ and $f(x)$ cannot be factored into a product of units in $\zn[x]$. Given a positive integer $N > 1$ and a polynomial $f \in \Z[x]$ which is monic and irreducible over $\zn$, define the \emph{pseudo-Galois ring} $PGR(N,f)$ as the quotient ring $\zn[x]/(f(x))$; then $PGR(N,f)$ has cardinality $N^{d}$ and characteristic $N$, where $d = \deg_{\zn}(f)$ is the degree of $f(x)$ as a polynomial over $\zn$. Every $\zn$ (with $N > 1$ an integer) is a pseudo-Galois ring. Every Galois ring is both a pseudo-Galois ring and a chain ring. The set $\{ a + b \sqrt 2 : a,b\in \Z_6\}$ endowed with usual addition and multiplication is the pseudo-Galois ring $PGR(6,x^2-2)$, which is not a chain ring. Among finite commutative rings, some, such as $\Z_{p^2}[x]/(px,x^2-px-p) = \{a+bx : a \in \{0,\ldots,p^2-1\}, b \in \{0,\ldots, p-1\}\}$, are chain rings but not pseudo-Galois rings, and others, such as $\Z_3 \times \Z_{15}$, are neither chain rings nor pseudo-Galois rings.
	
	We conclude this section with an explicit formulation of the main corollary of interest, which, in the special case that $A_0 = \cdots = A_m$, is an affirmative answer to Question~\ref{restricted zero density problem} for certain collections of rings, establishing a ``zero density'' version of Theorem~\ref{ring poly Sz} for independent families.
	
	\begin{cor}\label{cor: zero density}  Let $\mathcal R$ be one of the collections of finite commutative rings indicated below. Let $\mathbf P = \{P_1,\ldots, P_m\} \subset \Z[y_1,\ldots,y_n]$ be a family of independent polynomials. There exists $\gamma = \gamma(\mathbf P,\mathcal R) \in (0,1)$ such that, for every $R \in \mathcal R$ with $\lpf{\mathrm{char}(R)}$ sufficiently large, for any subsets $A_0,\ldots, A_m \subset R$ such that $|A_0|\cdots |A_m| \geq |R|^{(m+1)(1-\gamma)}$, there is a nontrivial configuration $(x,x+P_1(y),\ldots, x+P_m(y)) \in A_0 \times \cdots \times A_m$ for some $(x,y) \in R \times R^n$. Here is the list:
		\begin{enumerate}
			\item $\mathcal R_1^{(\ell)} = \{ \Z_{p^i} : p \text{ prime, } 1 \leq i \leq \ell\}$, where $\ell \in \N$,
			\item $\mathcal R_2^{(\ell)} = \{ \mathbb{F}_{p^i} : p \text{ prime, } 1 \leq i \leq \ell\}$, where $\ell \in \N$,
			\item $\mathcal R_3^{(\ell)} = \{GR(p^n,r) : p \text{ prime and } nr \leq \ell\}$, where $\ell \in \N$,
			\item $\mathcal R_4^{(\ell)} = \{ R : R \text{ a chain ring with invariants } p, n, r, k, t \text{ and } r(nk-k+t) \leq \ell \}$, where $\ell \in \N$,
			\item $\mathcal R_5^{(\alpha)} = \{ \Z_N : \lpf N \geq N^\alpha \}$, where $\alpha \in (0,1)$,
			\item $\mathcal R_6^{(\alpha)} = \{ PGR(N,f) : f \in \Z[x] \text{ monic irreducible over } \zn \text{ and } \lpf N \geq N^{\deg_{\zn}(f)\alpha} \}$, where $\alpha \in (0,1)$,
			\item $\mathcal R_7^{(\alpha)} = \{ R : \lpf{\mathrm{char}(R)} \geq |R|^\alpha\}$, where $\alpha \in (0,1)$.
		\end{enumerate}
	\end{cor}
	\begin{proof}
		Fix $\alpha \in (0,1)$. Proposition~\ref{prop: sufficient condition for zero density result} is precisely the assertion that the conclusion of the corollary holds for $\mathcal R_7^{(\alpha)}$. Obviously $\mathcal R_5^{(\alpha)} \subset R_6^{(\alpha)} \subset \mathcal R_7^{(\alpha)}$. Let $\ell = \lfloor 1/\alpha \rfloor$, i.e., the greatest integer that is at most $1/\alpha$. The rest follows since $\mathcal R_1^{(\ell)} \cup \mathcal R_2^{(\ell)} \subset \mathcal R_3^{(\ell)} \subset \mathcal R_4^{(\ell)} \subset R_7^{(\alpha)}$.
	\end{proof}

	\subsection{Heuristic: Why can one expect Theorem~\ref{main thm for intro} to hold?}\label{subsec: heuristic}
	
	As stated above, the short answer is ``asymptotic total ergodicity.'' We now briefly explain what this means.\footnote{One may consult the introduction of \cite{bb} for a more detailed discussion from a slightly different perspective.}
	
	Recall that a \emph{measure-preserving system} is a quadruple $\mathsf{X} = (X,\calb,\mu,T)$, where $(X,\calb,\mu)$ is a probability space and $T: X \to X$ is a measurable transformation which preserves the measure $\mu$ (i.e., $\mu(T^{-1}A) = \mu(A)$ for all $A \in \calb$), and we say $\mathsf{X}$ is \emph{invertible} if moreover $T$ is a bijection and $T^{-1}$ is a measurable transformation. A system $\mathsf{X} = (X,\calb,\mu,T)$ is \emph{ergodic} if every $T$-invariant set $A \in \calb$ satisfies either $\mu(A) = 0$ or $\mu(A) = 1$, and $\mathsf{X}$ is \emph{totally ergodic} if, for every $k \in \N$, the system $(X,\calb,\mu,T^k)$ is ergodic.
	
	Answering a question of the first author, Frantzikinakis and Kra showed in \cite{frakra} the following theorem.
	\begin{theorem}\label{thm: frakra}
		Let $(X,\calb, \mu,T)$ be a totally ergodic invertible measure-preserving system, and assume that $\{P_1(n),\ldots, P_m(n)\} \subset \Z[n]$ is an independent family of polynomials. Then, for functions $f_1,f_2,\ldots, f_m \in L^\infty(\mu)$,
		\begin{equation}
			\lim_{N\to \infty} \norm{\frac{1}{N} \sum_{n=0}^{N-1} f_1(T^{P_1(n)}x)f_2(T^{P_2(n)}x) \cdots f_m(T^{P_m(n)}x) - \prod_{i=1}^m \int f_i \ d\mu }_{L^2(\mu)}\ = \ 0.
		\end{equation}
	\end{theorem}
	Theorem~\ref{thm: frakra} epitomizes the phenomenon of joint ergodicity (see \cite{bereberg}) for totally ergodic actions along independent polynomials. Another example of this phenomenon is given by Theorem~1.2 from \cite{bfm}, which establishes a similar theorem for ergodic actions of countable fields with characteristic zero along independent polynomials.
	
	Given an integer $N > 1$, the \emph{rotation on $N$ points} is the measure-preserving system $(\zn, \calp(\zn), \mu, T)$, where $\calp(\zn)$ is the power set, $\mu$ is the counting measure normalized so that $\mu(\zn) = 1$, and $T$ is the map $n \mapsto n + 1$ modulo $N$. Clearly, the rotation on $N$ points is ergodic but not totally ergodic. For any positive integer $k$, it is easy to see that $(\zn, \calp(\zn), \mu, T^k)$ is not ergodic if and only if $\gcd(k,N) > 1$. Hence, the smallest value of $k$ that witnesses the lack of total ergodicity of the rotation on $N$ points is precisely $\lpf N$. One may therefore conjecture that, at least for rotations on $N$ points, a version of Theorem~\ref{thm: frakra} holds, provided that we introduce an error term which decays to zero as $\lpf N$ tends to infinity, hence the name ``asymptotic total ergodicity.''
	
	In \cite{bb}, we obtained such a result; up to a change in notation, we obtained the special case $m = 1$ (and $n = 1$) of Theorem~\ref{main thm for intro} with the conclusion restricted to rings of the form $\zn$. Theorem~\ref{main thm for intro} can be viewed as a finitary version of the phenomenon of joint ergodicity for polynomial actions involving an independent family. Having validated the asymptotic total ergodicity heuristic in a special case, we should, on account of Theorem~\ref{thm: frakra}, therefore expect\footnote{A version of Theorem~\ref{thm: frakra} in the case of an independent family of \emph{multivariate} polynomials also holds, though it is not explicitly mentioned in the literature.} Theorem~\ref{main thm for intro} to hold for general positive integers $m$ and $n$, at least for rings of the form $\zn$.
	
	Certainly, it does, as we show in this article. However, a natural question arises: Why should one expect the conclusion of Theorem~\ref{main thm for intro} to hold for finite commutative rings other than $\zn$? It turns out that the quantity $\lpf N$, when $N$ is the characteristic of a finite ring, has a nice algebraic interpretation that justifies this expectation. Recall that for any nonzero element $r$ in a finite commutative ring $R$ with characteristic $N$, if $M$ is the additive order of $r$ (so $Mr = 0$ and $M$ is the minimal positive integer with this property), then $M$ divides $N$. Thus, the smallest possible additive order of a nonzero element in $R$ is precisely $\lpf N$, and some elements, e.g., $\frac{N}{\lpf{N}}\cdot 1 \in R$, attain this order. Therefore, in contrast to $N = \mathrm{char}(R)$, the ``lightest'' possible torsion, $\lpf N$ is the ``worst'' possible torsion and hence a natural measurement of the ``failure of a finite ring to be totally ergodic,'' if such a notion could be formalized.
	
	Of course, this potential conceptual distinction between $\mathrm{char}(R)$ and $\lpf{\mathrm{char}(R)}$ is not apparent in a finite field, where every nonzero element has the same additive order, a prime number. Nonetheless, using the heuristic of asymptotic total ergodicity, one can predict the shape of the main results of \cite{bc,dls,pel3term,peluse}, the latter of which is most general among these four and is essentially\footnote{The error term in \cite{peluse} is better than the error term in Theorem~\ref{main thm for intro} by virtue of working over $\mathbb{F}_q$. The power of the asymptotic total ergodicity heuristic lies in guessing the shape, not the strength, of relevant results.} Theorem~\ref{main thm for intro} (when $n=1$) for finite fields with large enough characteristic. More importantly, if one notices the distinction between $\mathrm{char}(R)$ and $\lpf{\mathrm{char}(R)}$, then one can predict the main theorem in this article, Theorem~\ref{main thm for intro}, which generalizes the main results just cited.

	\subsection{Outline and discussion of the proof of Theorem~\ref{main thm for intro}}\label{subsec: methods}
	
	Our main theorem, Theorem~\ref{main thm for intro}, is a wide generalization of the main result in \cite{peluse} to the setting of finite commutative rings and for independent families of multivariable polynomials.
	
	Like many results which deal with multiple polynomial averages, we will make use of the \emph{PET induction} technique introduced in \cite{wmpet} and used, with a variety of modifications, to prove, e.g., the results in \cite{bl,blm,lei,peluse,prend}. In the setting of finite commutative rings, we must make several innovations, but the rough idea of the method is the same. Namely, a multiple polynomial average is controlled in terms of another multiple polynomial average which is in some respect simpler (i.e., ``more linear'') than the original. The measure of ``complexity'' of the involved polynomial families roughly depends on several aspects thereof, such as the degrees of the constituent polynomials and the number of distinct leading coefficients, and is inducted on. By an intricate series of moves, the complexity of the involved polynomial families is decreased (while at the same time the families themselves get larger). The simplification process is accomplished by judicious usage of discrete differentiation on the involved polynomials. At the end of this process, we find that the original multiple polynomial average of interest is reduced to a multiple \emph{linear} average, which is much more amenable to analysis. 
	
	In this article, roughly speaking, we manipulate the ``output'' of our PET induction with an adapted version of a technique from \cite{peluse} to obtain our final result. However, in sharp contrast to the situation in \cite{peluse}, where the entire PET induction argument is not spelled out because finite fields make it straightforward to execute, finite commutative rings introduce significant obstacles, all of which can be traced to the presence of \emph{zero divisors}. To handle zero divisors, we must significantly modify the PET induction technique in ways that we will describe below. The final deduction of an affirmative answer to Question~\ref{restricted zero density problem} also encounters new obstacles not visible in the finite field case, as we saw in the discussion following Theorem~\ref{main thm for intro}. Together, the additional complications require this article, and especially Section~\ref{sec: proof of prop: Us control}, to be lengthy.
	
	For technical reasons, we will actually prove the following more general version of Theorem~\ref{main thm for intro}. For convenience, we will reuse some notation from \cite{peluse} in the formulation below.
	\begin{theorem}\label{main thm}
		Let $n \geq 1$, $m_1 \geq 1$, and $m_2 \geq 0$ be integers, and let $\mathbf P = \{ P_1,\ldots,P_{m_1},Q_1,\ldots,Q_{m_2}\} \subset \Z[y_1,\ldots,y_n]$ be an independent family of polynomials. There exist $c,C,\gamma > 0$ such that the following holds. For any finite commutative ring $R$ with characteristic $N$ satisfying $\lpf N > C$, any 1-bounded functions $f_0,\ldots, f_{m_1}: R \to \C$, and any additive characters $\psi_1,\ldots,\psi_{m_2}$ of $R$,
		\begin{multline}
			\Bigg| \frac{1}{|R|^{n+1}}\sum_{(x,y)\in R\times R^n} f_0(x)f_1(x+P_1(y))\cdots f_{m_1}(x+P_{m_1}(y))\psi_1(Q_1(y))\cdots \psi_{m_2}(Q_{m_2}(y)) \\ - 1_{\Psi=1}  \left(\frac{1}{|R|} \sum_{x\in R} f_0(x) \right) \dots \left(\frac{1}{|R|} \sum_{x\in R} f_{m_1}(x) \right) \Bigg| \ \leq \ c\,\lpf{N}^{-\gamma},
		\end{multline}
		where $1_{\Psi=1}$ equals 1 if every $\psi_i$ is trivial and 0 otherwise.
	\end{theorem}
	\begin{remark}
		Theorem~\ref{main thm} implies Theorem~\ref{main thm for intro}; just take $m_2 = 0$ (so $\mathbf P$ does not contain any polynomials with the label $Q_i$) and, if necessary, use a choice of $f_i$ which absorbs the constant term of $P_i$. The proof of Theorem~\ref{main thm}, which proceeds by induction on $m_1$, makes use of both the $m_2 = 0$ and $m_2 > 0$ cases. We also note that the formulation of this theorem, its proof, and associated lemmas are significantly simplified by allowing the possibility of $m_2 = 0$, which means no additive characters $\psi_i$ are instantiated in the conclusion, the number $1_{\psi = 1}$ equals 1 because the statement ``every $\psi_i$ is trivial'' is vacuously true, and so on. The statements and proofs of key lemmas also benefit by not having separate formulations for the $m_2 = 0$ and $m_2 > 0$ cases. For clarity, we have added footnotes at certain points so that readers are not misled by consequences of this flavor.
	\end{remark}
	
	We prove Theorem~\ref{main thm} by induction on $m_1$; the aforementioned PET induction will be instrumental in a subargument of the induction step.
	
	Broadly speaking, the base case $m_1 = 1$ is proved in the same way as the corresponding base case in \cite{peluse}---via Fourier analysis. We will need a lemma establishing a bound on character sums over rings with certain polynomial arguments, viz.,
	\begin{equation} \label{intro eqn 5}
		\frac{1}{|R|^n} \sum_{y\in R^n} \prod_{j=1}^m \psi_j(Q_j(y)),
	\end{equation}
	where $\{Q_1,\ldots, Q_m\} \subset \Z[y_1,\ldots, y_n]$ is an independent family of polynomials and $\psi_1,\ldots, \psi_m$ are additive characters of a finite commutative ring $R$.
	
	In \cite{peluse}, the corresponding lemma is proved by appealing to the Weil bound for curves over finite fields (see, e.g., Theorem 3.2 in \cite{kowalskinotes}). In our case, we are not aware of any existing bound over finite commutative rings that performs the same function. Exponential sums over $\zn$ are well studied, and an analogue of the Weil bound exists for Galois rings (see \cite{kuheca}), and there is even a computation of the exact value of \eqref{intro eqn 5} over finite commutative rings with odd characteristic in the case $n = m = 1$ and $Q_1(y) = y^2$ (see \cite{szechtman}), but these do not cover the general situation. Thus, we will need to develop a suitable general estimate for our purposes; see Lemma~\ref{lem: multicharacter product of independent polynomials as arguments}. This bound (and the base case of Theorem~\ref{main thm}) is proved in Section~\ref{sec: base case}.
	
	Let us now consider the induction step. First, we introduce some notation. For any finite set $S$ with cardinality $|S|$ and function $f : S \to \C$, we write
	\begin{equation}
		\mathop{\mathbb{E}}_{x \in S} f(x) \ := \ \frac{1}{|S|} \sum_{x \in S} f(x).
	\end{equation}
	Let $R$ be a finite commutative ring. Given polynomials $P_1(y),\ldots, P_{m_1}(y) \in \Z[y_1,\ldots, y_n]$ and a vector $F = (f_0,\ldots, f_{m_1})$ of functions $R \to \C$, we will write $\Lambda_{P_1,\ldots,P_{m_1}}(F)$ or $\Lambda_{P_1,\ldots,P_{m_1}}(f_0,\ldots,f_{m_1})$ for the expression
	\begin{equation}
		\cavg{(x,y)}{R \times R^n} f_0(x) f_1(x+P_1(y)) \cdots f_{m_1}(x+P_{m_1}(y)),
	\end{equation}
	and, given a vector $\Psi = (\psi_1,\ldots, \psi_{m_2})$ of additive characters of $R$, we will write $\Lambda_{P_1,\ldots,P_{m_1}}^{Q_1,\ldots,Q_{m_2}}(F;\Psi)$ or $\Lambda_{P_1,\ldots,P_{m_1}}^{Q_1,\ldots,Q_{m_2}}(f_0,\ldots,f_{m_1};\psi_1,\ldots, \psi_{m_2})$ for the expression
	\begin{equation}\label{pointer 1} \cavg{(x,y)}{R \times R^n} f_0(x) f_1(x+P_1(y)) \cdots f_{m_1}(x+P_{m_1}(y))\psi_1(Q_1(y))\cdots \psi_{m_2}(Q_{m_2}(y)).
	\end{equation}
	
	Let $\mathbf P = \{P_1,\ldots, P_{m_1},Q_1,\ldots, Q_{m_2}\} \subset \Z[y]$ be an independent family of polynomials with zero constant term. The goal is to find $c,C,\gamma > 0$ such that, for any finite commutative ring $R$ with characteristic $N$ satisfying $\lpf N > C$, any vector $F = (f_0,\ldots, f_{m_1})$ of 1-bounded functions $R \to \C$, and any vector $\Psi = (\psi_1,\ldots,\psi_{m_2})$ of additive characters of $R$, one has
	\be
	\left| \Lambda_{P_1,\ldots,P_{m_1}}^{Q_1,\ldots,Q_{m_2}}(F;\Psi) - 1_{\Psi=1} \prod_{i=0}^{m_1} \rcavg x f_i(x) \right| \ \leq \  c\,\lpf{N}^{-\gamma}.
	\ee
	For some appropriate value of $C$, this will follow by the triangle inequality once we establish the following estimates for some $c_1,c_2, \gamma_1,\gamma_2 > 0$ and some (for now unspecified) vector $F'$ that is related to $F$: 
	\be\label{intro step 1 of induction arg}
	\left| \Lambda_{P_1,\ldots,P_{m_1}}^{Q_1,\ldots,Q_{m_2}}(F;\Psi) - \left(\rcavg{x} f_{m_1}(x)\right) \Lambda_{P_1,\ldots,P_{m_1-1}}^{Q_1,\ldots,Q_{m_2}}(F';\Psi) \right| \ \leq \ c_1\,\lpf{N}^{-\gamma_1}
	\ee
	and
	\be\label{intro step 2 of induction arg}
	\left| \left(\rcavg{x} f_{m_1}(x)\right) \Lambda_{P_1,\ldots,P_{m_1-1}}^{Q_1,\ldots,Q_{m_2}}(F';\Psi)
	- 1_{\Psi=1} \prod_{i=0}^{m_1} \rcavg x f_i(x)
	\right| \ \leq \ c_2\,\lpf{N}^{-\gamma_2}.
	\ee
	After using the 1-boundedness of $f_{m_1}$ to bound the common factor $\rcavg{x} f_{m_1}(x)$ by 1 in \eqref{intro step 2 of induction arg},  this second estimate will follow immediately by the induction hypothesis. We now turn to discuss the ideas behind the first estimate \eqref{intro step 1 of induction arg}.
	
	Let $f_{m_1}' = f_{m_1} - \rcavg x f_{m_1}(x)$; then $f_{m_1}'$ has zero integral and is certainly 1-bounded if we scale it by the factor $\frac 12$. Let $F' = (f_0,\ldots, f_{m_1-1})$ and $F'' = (f_0,f_1,\ldots, f_{m_1-1},\frac{1}{2} f_{m_1}')$. Then we can trivially rewrite the content to be bounded in \eqref{intro step 1 of induction arg} as follows:
	\be
	\Lambda_{P_1,\ldots,P_{m_1}}^{Q_1,\ldots,Q_{m_2}}(F;\Psi) - \left(\rcavg{x} f_{m_1}(x)\right) \Lambda_{P_1,\ldots,P_{m_1-1}}^{Q_1,\ldots,Q_{m_2}}(F';\Psi) \ = \ 2\Lambda_{P_1,\ldots,P_{m_1}}^{Q_1,\ldots,Q_{m_2}}(F'';\Psi),
	\ee
	where the factor of 2 just cancels the factor $\frac 12$. Thus, the first estimate would follow if $\left| \Lambda_{P_1,\ldots,P_{m_1}}^{Q_1,\ldots,Q_{m_2}}(F'';\Psi) \right|$ could be bounded by a negative power of $\lpf N$.
	
	What has changed when shifting our focus to the average $\Lambda_{P_1,\ldots,P_{m_1}}^{Q_1,\ldots,Q_{m_2}}(F'';\Psi)$ is that now we are not working with a totally arbitrary vector of 1-bounded functions, but instead with a vector of 1-bounded functions such that one of them has zero integral, namely the function $\frac{1}{2} f'_{m_1}$ in the last entry of $F''$.
	
	The strategy from here is as follows. First, using a PET induction argument, we bound
	\[\left|\Lambda_{P_1,\ldots,P_{m_1}}^{Q_1,\ldots,Q_{m_2}}(F'';\Psi) \right|\]
	in terms of the Gowers uniformity norm\footnote{See the definition and discussion around \eqref{defn: Gowers unif norm}.} $\norm{f_i}_{U^s}$ of the constituent $f_i$ and an error term of size a negative power of $\lpf N$, where $s \geq 2$ is an integer. See Subsection~\ref{subsec: bounding in terms of Us}.
	Second, we will then reduce $s$ using a technique due to Peluse in \cite{peluse}; this technique is the reason we need to work with averages like $\Lambda_{P_1,\ldots,P_{m_1}}^{Q_1,\ldots,Q_{m_2}}(F;\Psi)$ instead of just $\Lambda_{P_1,\ldots,P_{m_1}}(F)$. See Subsection~\ref{subsec: reducing the step} for more information about this method. Each step of this technique reduces $s$ by 1 at the cost of introducing some error terms. In the end, we will obtain a bound of the shape
	\begin{multline}
		\left| \Lambda_{P_1,\ldots,P_{m_1}}^{Q_1,\ldots,Q_{m_2}}(F'';\Psi) \right| \\ \leq \ b_1^{(1)} \left( \min\left\{\norm{f_0}_{U^1},\norm{f_1}_{U^1},\ldots, \norm{f_{m_1-1}}_{U^1}, \norm{\frac{1}{2}f_{m_1}'}_{U^1}\right\}\right)^{1/2} + b_3^{(1)}(N), 
	\end{multline}
	for some constant $b_1^{(1)}$ and some function $b_3^{(1)}(N)$ that depends on the characteristic $N$ in a complicated way. We have already arranged for $\norm{\frac{1}{2} f_{m_1}'}_{U^1} =0$. Thus, our bound is actually
	\be \label{pointer 2}
	\left| \Lambda_{P_1,\ldots,P_{m_1}}^{Q_1,\ldots,Q_{m_2}}(F'';\Psi) \right| \ \leq \ b_3^{(1)}(N).
	\ee
	The final step in the proof of Theorem~\ref{main thm} is to analyze the error terms that comprise $b_{3}^{(1)}(N)$ to show that this quantity is bounded by a negative power of $\lpf{N}$.
	
	In the rest of this section, let us sketch a few of the novel difficulties that arise in the PET induction argument. Without loss of generality (see Subsection~\ref{subsec: bounding in terms of Us}), we only need to bound $\Lambda_{P_1,\ldots,P_m}(f_0,\ldots, f_m)$ in terms of $\norm{f_i}_{U^s}$ and an error term of size a negative power of $\lpf N$. This bound is the content of Proposition~\ref{prop: Us control}, whose proof occupies  Section~\ref{sec: proof of prop: Us control}. As a basis for discussion, we first give a very rough sketch of the proof that omits the novelties, then add some details. The main prerequisites to the proof are two lemmas, Lemma~\ref{lem: Ud bound if invertible} and Lemma~\ref{linearization step}. The first lemma asserts that, over finite commutative rings, Gowers uniformity (semi)norms $\norm{\cdot}_{U^s}$ control certain linear averages, and the second lemma bounds a polynomial average $\Lambda_{P_1,\ldots, P_m}(f_0,\ldots, f_m)$ in terms of some other polynomial average $\Lambda_{Q_1,\ldots,Q_{m'}}(g_0,\ldots g_{m'})$, where the family $\{Q_1,\ldots, Q_{m'}\}$ is, as discussed above, ``less complex'' or ``more linear'' than the original family $\{P_1,\ldots, P_m\}$. (Always $m' \geq m$.) To prove Proposition~\ref{prop: Us control}, we apply the second lemma repeatedly until some average is obtained to which the first lemma may be applied, then synthesize the bounds arising from the second lemma. The fact that there even is some finite sequence of applications of the second lemma which engenders a linear situation is in essence the insight behind the PET induction procedure, which is explained in Subsection~\ref{subsec: PET algorithms} and makes use of the notion of \emph{weight sequence} introduced in Subsection~\ref{subsec: weight sequences}.
	
	To explain some of the complications that arise, let us consider a concrete example. Suppose, after some number of applications of Lemma~\ref{linearization step}, we only need to bound the linear average
	\begin{equation}\label{intro eqn 3}
		\cavg{(x,y)}{R \times R} g_0(x) g_1(x+a_1y) \cdots g_{m'}(x+a_{m'}y),
	\end{equation}
	where $g_i : R \to \C$ are some 1-bounded functions and $a_i$ are some integers which depend on the original family $\{P_1,\ldots, P_m\}$ and on some hidden integer parameters $h_1,\ldots, h_k$, where $k$ is the number of times Lemma~\ref{linearization step} was applied. Using Lemma~\ref{lem: Ud bound if invertible}, the average \eqref{intro eqn 3} can be bounded in terms of the Gowers uniformity norms of $g_i$, but only\footnote{See the discussion following Lemma~\ref{lem: Ud bound if invertible} for a relevant counterexample.} provided that all $a_i$ and $a_i-a_j$ for distinct $i,j$ are units in $R$.
	
	This requirement on the $a_i$ and their pairwise differences is the point from which other technicalities spawn. The $a_i$ depend on the $h_j$, whose exact values largely depend on the original functions $f_0,\ldots, f_m$. We say ``largely'' because the strength of Lemma~\ref{linearization step} is that some values of $h_j$ may be excluded at the cost of introducing an error term.
	
	For brevity, if $a \in \Z$ and $R$ is a ring of characteristic $N$, we will routinely follow the convention of writing $a$ for the ring element $(a \bmod N) \cdot 1$ and, depending on the situation, freely refer to such elements either as integers or as elements of $\zn$ (which is a subring of $R$). We recall that an integer $a$ is a unit in a ring $R$ of characteristic $N$ if and only if $\gcd(a,N) = 1$ (see Lemma~\ref{lem: what is a unit}). It follows that any nonzero integer $a$ such that $|a| < \lpf N$ is a unit in $R$. Thus, modulo $N$, there is an ``interval'' around 0 in which every nonzero integer is a unit. We can exploit this fact and our limited control over the $h_j$ from Lemma~\ref{linearization step} to ensure that the $a_i$ and their pairwise differences are all units. For this reason, we will need to introduce the notion of the \emph{$\zn$-height} of an integer $a$, which is the minimal nonnegative integer $H$ such that $a \bmod N \in \{0,1,\ldots, H\} \cup \{N-H,\ldots, N-1\}$.\footnote{Here and elsewhere, we use the convention that, if $a$ and $b$ are integers with $b < a$, then the set $\{a,a+1,\ldots, b\}$ is the empty set.} Thus, if we can guarantee that the $a_i$ have $\zn$-height less than $\lpf N/2$, then each of the integers $a_i$ and $a_i - a_j$ will have $\zn$-height less than $\lpf N$ and hence be units, provided they are nonzero.
	
	Although it is difficult to determine the exact values of the $a_i$, we only need to ensure they are nonzero and distinct. We use a notion of \emph{essential distinctness modulo $N$} to manage the process of applying Lemma~\ref{linearization step} so that this situation occurs.
	\begin{definition}\label{definition of essential distinctness intro}
		Let $m \geq 2$, $n$, and $N$ be positive integers. Let $\mathbf{P} = \{P_1(y),\ldots, P_m(y)\} \subset \zn[y_1,\ldots, y_n]$ be a family of polynomials. We say that $\mathbf{P}$ is \emph{essentially distinct}\footnote{In our proofs, it will always be clear which value of $N$ is relevant, so, for brevity, we generally say ``essentially distinct'' rather than ``essentially distinct modulo $N$.''} if for any distinct $i,j \in \{1,\ldots, m\}$, neither $y \mapsto P_i(y)$ nor $y \mapsto P_i(y) - P_j(y)$ is a constant function $\mathbb{Z}_N^n \to \zn$.
	\end{definition}
	
	In the non-``modulo $N$'' setup, such as in \cite{wmpet}, a family of nonconstant integer polynomials is essentially distinct if none of the polynomials is constant and no difference of two distinct polynomials in the family is constant. In such a setup, essential distinctness plays a natural role as a non-degeneracy property of a family of polynomials that is preserved by the PET induction procedure. Essential distinctness modulo $N$ plays a similar role in our setup; what is different here is that it is harder to check that this property is actually preserved by our induction procedure.
	
	Indeed, for integer polynomials, there is a correspondence between polynomials and polynomial functions, so it is convenient to define the notion of essential distinctness in terms of polynomials, even though the relevant property for subsequent ergodic arguments is the fact that the induced polynomial functions (induced by a polynomial or a difference of two distinct polynomials) are not constant functions.
	
	However, as further explained in Subsection~\ref{subsec: polynomials}, there is no nice correspondence between polynomials over $\zn$ and polynomial functions over $\zn$, since multiple polynomials may induce the same polynomial function; sometimes, it is not even straightforward to decide whether a given polynomial induces a constant map! Moreover, as in other PET induction arguments, the relevant property that we need to guarantee is the non-constant nature of certain induced polynomial functions. Thus, Definition~\ref{definition of essential distinctness intro} directly accounts for the induced polynomials. (Curiously, the quantity $\lpf N$ makes another appearance in the study of polynomials versus polynomial functions over $\zn$; see Corollary~\ref{cor: polyuniq}.)
	
	\begin{remark}
		The reader may wonder why we only work with polynomials that have coefficients in $\zn$, rather than general ring coefficients. The short answer is that our Lemma~\ref{linearization step} allows us to stick with integer/$\zn$ coefficients, and if we choose not to take advantage of this, the resulting situation could perhaps best be described as a can of worms. At a minimum, one would have to be able to estimate the number of roots of polynomials with arbitrary ring coefficients and decide whether two such polynomials induce the same function $R^n \to R$; the problems with ensuring invertibility of the $a_i$ are likely harder to manage, and so on.
	\end{remark}
	
	We now close our discussion of the PET induction step by mentioning one more complication. The value $k$, the number of times Lemma~\ref{linearization step} has to be applied (and thus the number of parameters $h_j$), must be known or bounded in advance and in a way that is uniform over rings with sufficiently large $\lpf{\mathrm{char}(R)}$. We introduce the technical notion of a \emph{permissible operation} in Subsection~\ref{subsec: PET algorithms} for the purpose of analyzing our PET induction scheme in relation to such considerations (see that section for more expository details). We use this knowledge in tandem with the notion of $\zn$-height to ensure invertibility of $a_i$ and $a_i - a_j$. Typical PET induction arguments do not need this kind of knowledge and thus are formulated more simply. See the discussion in Subsection~\ref{subsec: PET algorithms}.
	
	As mentioned above, our notion of weight sequence, introduced in Subsection~\ref{subsec: weight sequences}, gives the well-ordering according to which our PET induction scheme is organized. For a reader familiar with the notion of \emph{weight} (as in, e.g., \cite{IPpolySz} or \cite{lei}) or \emph{degree sequence} (as in, e.g., \cite{prend}) that plays the same role in a usual PET induction argument, the following comment may be useful: Averages involving different polynomial families with the same weight sequence may not behave uniformly enough from the ``modulo $N$'' perspective, so we cannot prove statements that are quantified over all polynomial families and depend solely on the weight sequence of a given family; the non-uniformity we are referring to is the fact that we have to restrict some conclusions to rings with large enough $\lpf{\mathrm{char}(R)}$, where the threshold may grow arbitrarily large depending on the polynomial family (even among families with the same given weight sequence). Informally, this means our PET induction arguments must be ``unraveled,'' i.e., run a controlled number of times and with attention paid to aspects that would normally be homogenized by induction.
	
	In summary, the PET induction step in our case requires us to introduce and work with several new notions (such as $\zn$-height and permissible operations) and overcome several obstacles (related to invertibility, essential distinctness modulo $N$, non-uniformity among polynomial families with the same weight sequence, and representations of polynomial functions modulo $N$).

	
	\subsection{Organization of the article}
	The article is organized as follows.
	
	Section~\ref{sec: preliminaries} covers preliminary material from (higher order) Fourier analysis and ring theory.
	
	In Section~\ref{sec: base case}, we obtain pertinent results on character sum evaluation over finite commutative rings and use them to prove the base case of Theorem~\ref{main thm}. This section includes a demonstration of a relatively straightforward special case of our main lemma on character sums. 
	
	Section~\ref{sec: induction step} describes the strategy of the proof of the induction step of Theorem~\ref{main thm} and contains all necessary lemmas and propositions, leaving the technical proof by PET induction of the key Proposition~\ref{prop: Us control} for a later section.
	
	In Section~\ref{sec: proof of main thm}, we prove Theorem~\ref{main thm}, which implies Theorem~\ref{main thm for intro}.
	
	In Section~\ref{sec: derivation of zero and positive density results}, we prove Proposition~\ref{prop: config count}, the main tool for deriving the corollaries of Theorem~\ref{main thm} that have been explained above.
	
	The goal of Section~\ref{sec: proof of prop: Us control} is to prove Proposition~\ref{prop: Us control}. We first prove two key lemmas, Lemmas~\ref{lem: Ud bound if invertible}~and~\ref{linearization step}. Then, we give an ad hoc proof of a special case of Proposition~\ref{prop: Us control}. From our discussion of this proof, it will become clear that Lemma~\ref{linearization step} is insufficient as stated and needs to be considerably sharpened---the discussion of the relevant issues above can be viewed as a preview. Next, we develop the notions of $\zn$-height, essential distinctness modulo $N$, weight sequences, and permissible operations and use them to sharpen Lemma~\ref{linearization step} into Lemma~\ref{lem: PET inductive step}. Finally, we execute the PET induction argument to prove Proposition~\ref{prop: Us control}.
	
	An additional complication arises in the proof of Proposition~\ref{prop: Us control} in the multivariable situation. Section~\ref{sec: fixing a matrix} contains a straightforward, but tedious matrix manipulation that avoids this complication, while Section~\ref{sec: PET examples} collects this and other examples of certain obstacles that our PET induction argument is designed to avoid.
	
	\section*{Acknowledgments}
	We thank Borys Kuca for pointing out a significant simplification of the proof of Lemma~~\ref{lem: Gowers lowering}. We thank an anonymous referee for careful reading and for numerous helpful suggestions.

	\section{Preliminaries}\label{sec: preliminaries}
	Fix a finite set $S$ with cardinality $|S|$ and functions $f,g : S \to \C$. 
	We write
	\be
	\mathop{\mathbb{E}}_{x \in S} f(x) \ := \ \frac{1}{|S|} \sum_{x \in S} f(x),
	\ee
	and, as usual, the symbol $\mathop{\mathbb{E}}_{x,y \in S}$ abbreviates $\mathop{\mathbb{E}}_{x\in S} \mathop{\mathbb{E}}_{y \in S}$. We say $f$ is 1-bounded if $|f(x)| \leq 1$ for any $x \in S$. For any positive integer $p$, we set $\norm{f}_{L^p}^p := \rcavg x |f(x)|^p$, and we define the inner product
	\begin{equation}\label{inner product definition}
		\lr{f,g} \ := \ \rcavg{x} f(x) \ol{g}(x),
	\end{equation}
	which satisfies $\lr{f,f} \ = \ \norm{f}_{L^2}^2$.
	
	\subsection{Fourier-analytic preliminaries}
	
	Let $R$ be a finite commutative ring. Unless otherwise indicated, the facts stated in this subsection only make use of the fact that the additive group $(R,+)$ is a finite abelian group; the multiplicative structure of $R$ is not relevant.
	
	Let us first recall some classical facts from Fourier analysis. An additive character of $R$ is a homomorphism $\chi : (R,+) \to \{z \in \C: |z| = 1\}$. The set $\widehat{R}$ of additive characters of $R$ is an orthonormal basis for the $|R|$-dimensional $\mathbb{C}$-linear space of functions $f : R \to \mathbb{C}$ with the inner product $\lr{\cdot,\cdot}$ defined in \eqref{inner product definition}. Hence, the classical Fourier expansion of $f$ in this case is
	\begin{equation}\label{basisdecomp}
		f \ = \ \sum_{\chi \in \widehat{R}} \lr{f,\chi}\chi.
	\end{equation}
	
	We use the following normalization of the Fourier transform. The Fourier transform of $f$, written $\hat{f}$, is the function $\hat{f} : \widehat{R} \to \mathbb{C}$ given by $\hat{f}(\chi) = \lr{f,\overline{\chi}} = \cavg{x}{R} f(x)\chi(x)$. Thus $\lr{f,\chi} = \hat{f}(\overline{\chi})$; plugging this into~\eqref{basisdecomp} and reindexing the sum yields the Fourier inversion formula, valid for all $a \in R$:
	\begin{equation}\label{inversion formula} f(a) \ = \ \sum_{\chi \in \widehat{R}} \hat{f}(\chi) \chi(-a).
	\end{equation}
	
	Also, the classical Plancherel's theorem in our context is as follows: For any $f_1, f_2 : R \to \mathbb{C}$, we have $|R|\lr{\hat{f_1},\hat{f_2}} = \lr{f_1,f_2}$, where on the left-hand side we mean the inner product
	\begin{equation*} \lr{\hat{f_1},\hat{f_2}} := \cavg{\chi}{\widehat{R}} \hat{f_1}(\chi) \overline{\hat{f_2}(\chi)},
	\end{equation*}
	which is the same as the one defined above after identifying $\widehat{R}$ with $R$.
	
	The following lemma will be useful in the sequel.
	
	\begin{lemma} \label{config rearrangement product ver} Let $R$ be a finite commutative ring with characteristic $N$. Let $k \geq 0$ and $n \geq 1$ be integers, and let $R' = R^n$ with coordinatewise addition and multiplication. Let $c_0,\ldots,c_k \in \mathbb{Z}_N^n$ be nonzero in $R'$, writing $c_i = (c_{i,1},\ldots, c_{i,n})$ for each $i$. Define $S = \{(x_0,\ldots,x_k) \in (R')^{k+1} : \sum_{i=0}^k c_ix_i = 0_{R'}\}$. For any functions $f_0,\ldots,f_k : R' \to \C$, one has
		\be\label{fourier eqn 1}
		\sum_{(x_0,\ldots,x_k)\in S} f_0(x_0)\cdots f_k(x_k) = |R'|^{k+1}\cavg{X}{\widehat{R'}} \hat{f_0}(X^{c_0})\cdots \hat{f_k}(X^{c_k}),
		\ee
		where, if $X = (\chi_1,\ldots, \chi_n) \in \widehat{R'}$, then $X^{c_i}$ denotes $(\chi_1^{c_{i,1}},\ldots, \chi_n^{c_{i,n}}) \in \widehat{R'}$.
	\end{lemma}
	\begin{proof} Expand each $\hat{f_i}(X^{c_i})$ on the right-hand side of \eqref{fourier eqn 1} using the definition of Fourier transform, combine $\prod_{i=0}^k X^{c_i}(x_i) = X\left(\sum_{i=0}^k c_ix_i\right)$, then apply orthonormality of characters.
	\end{proof}

	Now, we prepare to define the Gowers uniformity (semi)norms on $R$. Let $f : R \to \C$. For any $h \in R$, define $\Delta_h f : R \to \C$ by $\Delta_hf(x) := f(x+h) \ol{f}(x)$ for all $x \in R$. For any $h_1,\ldots, h_s \in R$, define $\Delta_{h_1,\ldots, h_s} f : R \to \C$ by $\Delta_{h_1,\ldots,h_s} f = \Delta_{h_1}( \Delta_{h_2}( \cdots (\Delta_{h_s}f)\cdots ))(x)$. The order of $h_1,\ldots, h_s$ in the subscript of $\Delta_{h_1,\ldots,h_s}$ does not matter.
	
	For any positive integer $s$ and finite commutative ring $R$, we define the Gowers uniformity (semi)norm $\norm{\cdot}_{U^s}$ on the set of functions $f : R \to \C$ by
	\begin{equation}\label{defn: Gowers unif norm}
		\norm{f}_{U^s}^{2^s} := \rcavg{x,h_1,\ldots,h_{s}} \Delta_{h_1,\ldots,h_s}f(x),
	\end{equation}
	noting that $U^1$ is only a seminorm since $\norm{f}_{U^1} = \left| \rcavg x f(x) \right|$. These (semi)norms are nondecreasing in $s$: $\norm{f}_{U^1} \leq \norm{f}_{U^{2}} \leq \norm{f}_{U^3} \leq \ldots$ for every $f$. Moreover, we have an inductive relation
	\be
	\norm{f}_{U^{s+1}}^{2^{s+1}} = \rcavg h \norm{\Delta_h f}_{U^s}^{2^s}
	\ee
	for any $s \geq 1$. For more information about these (semi)norms, including why they are norms for $s > 1$, one may consult \cite[Section 1.3.3]{tao}.
	
	As one can find in, e.g., \cite{tao}, there is a rich theory of Gowers uniformity (semi)norms, including so-called inverse theorems. In this article, we do not use these tools. Let us now state a handful of elementary facts that we will use, which mostly concern how to rewrite or bound the $U^2$ norm of a function in terms of its Fourier transform.
	\begin{lemma}\label{lem: u2l4} Let $R$ be a finite commutative ring. Then, for each $f : R \to \C$, we have
		\begin{equation}
			\norm{f}_{U^2}^4 = |R|\norm{\hat{f}}_{L^4}^4.
		\end{equation}
	\end{lemma}
	\begin{proof} We apply Lemma~\ref{config rearrangement product ver} with $n = 1$. Indeed, setting $S = \{(x_1,x_2,x_3,x_4) \in R^4 : x_1-x_2-x_3+x_4 = 0 \}$, we observe that
		\ba
		|R|^3 \norm{f}_{U^2}^4 \ & = \sum_{x,h_1,h_2 \in R} \Delta_{h_1,h_2} f(x) \ = \ \sum_{x,h_1,h_2 \in R} f(x+h_1+h_2)\ol{f}(x+h_1)\ol{f}(x+h_2)f(x) \\
		& = \ \sum_{(x_1,x_2,x_3,x_4) \in S} f(x_1)\ol{f}(x_2)\ol{f}(x_3)f(x_4) \ = \ |R|^4 \cavg{\chi}{\hat{R}} \hat{f}(\chi) \hat{\ol{f}}(\chi^{-1})\hat{\ol{f}}(\chi^{-1})\hat{f}(\chi) \\
		& = \ |R|^4 \cavg{\chi}{\hat{R}} \left| \hat{f}(\chi) \right|^4 \ = \ |R|^4 \norm{\hat{f}}_{L^4}^4.
		\ea
	\end{proof}
	For reference, we record the following consequence of Lemma~\ref{lem: u2l4}.
	\begin{lemma}\label{simple inequality for U2 norm}
		Let $f: R \to \C$ be a 1-bounded function on a finite commutative ring $R$. Then
		\be
		\norm{f}_{U^2}^4 \ \leq \ \max_{\psi \in \widehat{R}} |\hat{f}(\psi)|^2.
		\ee
	\end{lemma}
	\begin{proof}
		We have
		\begin{multline}
			\norm{f}_{U^2}^4 \ = \ |R| \norm{\hat{f}}_{L^4}^4 \ = \  \sum_{\psi \in \widehat{R}} \left| \hat{f}(\psi)\right|^4 \\ \leq \ |R| \norm{\hat{f}}_{L^2}^2 \max_{\psi \in \widehat{R}}  |\hat{f}(\psi)|^2  \ = \ \norm{f}_{L^2}^2 \max_{\psi \in \widehat{R}} |\hat{f}(\psi)|^2  \ \leq \ \max_{\psi \in \widehat{R}} |\hat{f}(\psi)|^2,
		\end{multline}
		where in the second line we used Plancherel's theorem and the fact that $\norm{f}_{L^2}^2 \leq 1$ holds by the 1-boundedness of $f$.
	\end{proof}
	The previous facts did not use the ring structure of $R$, only the fact that its additive group is finite abelian. We now record two estimates which make use of an invertibility assumption:
	\begin{lemma}\label{Fourier dilate trivial bound} Let $R$ be a finite commutative ring with characteristic $N$. Let $n \geq 1$ be an integer, and let $R' = R^n$ with coordinatewise addition and multiplication. Let $f : R' \to \C$ be a 1-bounded function, and let $c = (c_1,\ldots, c_n) \in \mathbb{Z}_N^n$ be invertible in $R'$. The function $g : \widehat{R'} \to \C$ given by $g(X) := \hat{f}(X^c)$ satisfies
		\be
		\norm{g}_{L^2} \ \leq \ |R'|^{-1/2}
		\ee
		and
		\be
		\norm{g}_{L^4} \ \leq \ |R'|^{-1/4},
		\ee
		where, if $X = (\chi_1,\ldots, \chi_n) \in \widehat{R'}$, then $X^c$ denotes $(\chi_1^{c_1},\ldots, \chi_n^{c_n}) \in \widehat{R'}$.
	\end{lemma}
	\begin{proof} By the change of variables $X \mapsto X^{c^{-1}}$ and Plancherel's theorem, we observe that
		\begin{equation} \norm{g}_{L^2}^2 \ = \ \cavg{X}{\widehat{R'}} |\hat{f}(X^c)|^2 \ = \ \cavg{X}{\widehat{R'}} |\hat{f}(X)|^2 \ = \ \norm{\hat{f}}_{L^2}^2 \ = \ \frac{1}{|R'|} \norm{f}_{L^2}^2 \ \leq \ \frac{1}{|R'|}.
		\end{equation}
		The change of variables $X \mapsto X^{c^{-1}}$ on $\widehat{R'}$ is valid because it corresponds to the change of variables $x \mapsto c^{-1}x$ on $R'$. Similarly, by Lemma~\ref{lem: u2l4} and the triangle inequality,
		\be
		\norm{g}_{L^4}^4 = \cavg{X}{\widehat{R'}} |\hat{f}(X^c)|^4 = \cavg{X}{\widehat{R'}} |\hat{f}(X)|^4 = \frac{1}{|R'|}||f||_{U^2}^4 = \frac{1}{|R'|}\left| \cavg{x,h_1,h_2}{R'} \Delta_{h_1,h_2}f(x) \right| \leq |R'|^{-1},
		\ee
		as desired.
	\end{proof}
	
	\subsection{Ring-theoretic preliminaries}
	Recall that if $a \in \Z$ and $R$ is a ring of characteristic $N$, we will routinely follow the convention of writing $a$ for the ring element $(a \bmod N) \cdot 1$ and, depending on the situation, freely refer to such elements either as integers or as elements of $\zn$ (which is a subring of $R$). The following lemma is useful in that it guarantees that the set $\{1,\ldots, \lpf N - 1\} \cup \{N- \lpf N + 1, \ldots, N-1\}$ consists entirely of units in a finite ring with characteristic $N$.
	\begin{lemma}\label{lem: what is a unit}
		Let $R$ be a finite commutative ring with characteristic $N$. Let $m$ be a positive integer. If $\gcd(m,N) = 1$, then $m$ is a unit in $R$.
	\end{lemma}
	\begin{proof}
		It is clear that $1$ generates a cyclic subgroup of $(R,+)$ of order $N$. This subgroup is actually a subring of $R$ isomorphic to $\Z_N$. Then $m$ is a unit in $R$ if and only if $m$ is a unit in $\Z_N$, which holds if and only if $\gcd(m,N) = 1$.
	\end{proof}
	
	Since we work with general finite commutative rings, it may be tempting to use the fact\footnote{See, e.g., \cite[Section 3]{cohenroche} for a proof.} that every finite ring is a direct product of rings of prime power order. However, this fact would not help us in the overall proof of our main theorem. At best, it might strengthen some auxiliary bounds we use, such as the one in Proposition~\ref{prop: bound on number of roots}. Using hypothetical strengthened versions of such bounds would complicate our presentation for not enough benefit.
	
	In the course of making calculations over a finite commutative ring $R$, it will be useful to have $\Z_{\mathrm{char}(R)}$ as a direct summand of $R$, which is generally only possible\footnote{For example, if $\mathbb{F}_9 \times \Z_5$ were ring-theoretically isomorphic to a direct product of some ring $S$ with $\Z_{15}$, then $|S| = 3$ and hence $S$ would be $\Z_3$ with usual addition and multiplication modulo 3, in which case the failure of isomorphism can be noticed by, for example, observing that the range of the map $x \mapsto x^8$ has different cardinality over $\mathbb{F}_9 \times \Z_5$ vs. over $S \times \Z_{15}$. In contrast, $(\mathbb{F}_9 \times \Z_5,+)$ is clearly isomorphic to $\Z_3 \times \Z_3 \times \Z_5 \cong \Z_3 \times \Z_{15}$.} if we regard $R$ solely from an additive viewpoint. The following lemma describes the additive structure of a finite commutative ring in a useful way.
	
	\begin{lemma}\label{lem: structure of (R,+)}
		Let $R$ be a finite commutative ring with characteristic $N$. Then there is an isomorphism $\phi : (R,+) \to \Z_{b_1} \times \cdots \times \Z_{b_r}$, where $b_r = N$, the integers $b_i > 1$, $i \in \{1,\ldots, r-1\}$, are possibly indistinct prime powers which divide $N$, and $\phi(1) = (0,\ldots, 0,1)$. 
	\end{lemma}
	\begin{proof}
		By the fundamental theorem of finite abelian groups, there is an isomorphism $\phi' : (R,+) \to \Z_{b'_1} \times \dots \times \Z_{b'_{r'}}$, where the integers $b'_i > 1$ are possibly indistinct prime powers. We claim that $N = \lcm(b'_1,\ldots, b'_{r'})$. Indeed, it is clear from the isomorphism that $\lcm (b'_1,\ldots, b'_{r'})x = 0$ for every $x \in R$, which implies that $\mathrm{char}(R) \leq \lcm(b'_1,\ldots, b'_{r'})$. Moreover, this inequality must be sharp, as there is an element $s \in R$ such that $\lcm(b'_1,\ldots, b'_{r'})$ is the minimal positive integer $M$ such that $Ms = 0$. Indeed, for each prime $p$ dividing $\lcm(b'_1,\ldots, b'_{r'})$, there exists a positive integer $a_p$ and an index $i_p \in \{1,\ldots, r'\}$ such that $b'_{i_p} = p^{a_p}$ and $p^{a_p+1}$ does not divide $\lcm(b'_1,\ldots, b'_{r'})$. Let $I$ be a set of indices $i_p$ chosen in this way. Then the desired element $s$ satisfies $\phi'(s) = (c_1,\ldots, c_{r'})$, where
		\begin{equation*}c_i = \begin{cases} 1 & \text{ if } i \in I, \\ 0 & \text{ otherwise.} \end{cases}
		\end{equation*}
		Thus $N = \lcm(b'_1,\ldots, b'_{r'})$.
		
		With that shown, we will proceed with another argument; for clarity, what is indicated below by the notation $i_p$ and $I$ may be different from what is indicated by the same symbols above.
		
		Write $\phi'(1) = (e''_1,\ldots, e''_{r'})$. We now claim that, for each prime $p$ dividing $N$, there exists $i_p \in \{1,\ldots, r'\}$ such that $pb'_{i_p}$ does not divide $N$ and $e_{i_p}''$ is invertible modulo $b'_{i_p}$. Suppose, for a contradiction, that there is a prime $p$ dividing $N$ such that, for every $i \in \{1,\ldots, r'\}$, we have $pb'_i$ divides $N$ or $e_{i}''$ is not a unit modulo $b'_i$. Let $I = \{i \in \{1,\ldots, r'\} : pb'_i \text{ does not divide } N\}$ be the set of indices where $b'_i$ is the largest power of $p$ that appears among $b'_1, \ldots, b'_{r'}$. Clearly $I$ is nonempty. For all $i \in I$, the contrary assumption implies $e_{i}''$ is not a unit modulo $b'_i$ and hence is a multiple of $p$. However, this forces $(N/p)\phi'(1) = 0$, which contradicts that the characteristic of $R$ is $N$. This proves the claim.
		
		By the claim, $N$ is the product of $b_{i_p}'$ over all prime $p$ dividing $N$, which expresses $N$ as a product of pairwise coprime prime power factors. Moreover, $e_{i_p}'' \in \Z_{b_{i_p}'}^\times$ for each prime $p$ dividing $N$.
		
		Hence, by the Chinese remainder theorem, it follows that there is an isomorphism $\phi'' : (R,+) \to \Z_{b_1} \times \cdots \times \Z_{b_r}$, where $r \leq r'$, $b_r = N$, the integers $b_i > 1$, $i \in \{1,\ldots, r-1\}$, are possibly indistinct prime powers which divide $N$, and, writing $\phi''(1) = (e_1,\ldots, e_r)$, we have that $e_r$ is invertible modulo $N$.
		
		For each $j \in \{1,\ldots, r\}$, define the element $\delta^{(j)} = (\delta^{(j)}_1, \ldots, \delta^{(j)}_r) \in \Z_{b_1} \times \cdots \times \Z_{b_r}$ by
		\begin{equation*}
			\delta^{(j)}_i \ := \ \begin{cases} 1 & \text{ if } j = i, \\ 0 & \text{ otherwise,} \end{cases}  
		\end{equation*}
		and recall that $\phi''(1) = (e_1,\ldots, e_r)$. Then, using the fact that $e_r$ is invertible modulo $b_r$, we observe that $\{\delta^{(j)} : j \in \{1,\ldots, r-1\} \} \cup \{(e_1,\ldots, e_r)\}$ is a minimal set of generators of $\Z_{b_1} \times \cdots \times \Z_{b_r}$. Thus, by composing $\phi''$ with the isomorphism $\Z_{b_1} \times \cdots \times \Z_{b_r} \to \Z_{b_1} \times \cdots \times \Z_{b_r}$ defined by sending $\delta^{(j)}$ to itself for $j < r$ and $(e_1,\ldots, e_r)$ to $\delta^{(r)}$, we obtain an isomorphism $\phi: (R,+) \to \Z_{b_1} \times \cdots \times \Z_{b_r}$ which satisfies $\phi(1) = (0,\ldots, 0,1)$, as desired.
	\end{proof}

	\subsection{Polynomials}\label{subsec: polynomials}
	This article deals extensively with polynomials over various rings. In this subsection, we collect basic notions and results about such polynomials, starting with their definition.
	
	Fix a ring $R$ and a positive integer $n$. Let $\alpha = (\alpha_1,\ldots, \alpha_n) \in (\N \cup \{0\})^n$. A monomial $y_1^{\alpha_1}\cdots y_n^{\alpha_n}$, also written $y^\alpha$, is a formal product of the $n$ indeterminates $y_i$, $i \in \{1,\ldots, n\}$, which are assumed to commute with each other and with elements of $R$. The degree of the monomial $y^\alpha$ is defined to be $\sum_{i=1}^n \alpha_i$. A polynomial $P = P(y) = P(y_1,\ldots, y_n) \in R[y_1,\ldots, y_n]$ is a finite linear combination of monomials with coefficients in $R$. The degree of $P$ is the maximum of the degrees of its monomials with nonzero coefficients.
	
	A finite commutative ring with characteristic $N$ contains a unique subring isomorphic to $\Z_N$. Hence, several facts about polynomials over $\zn$ will be relevant to us.
	
	To begin, the technical distinction between polynomials in the ring $\zn[y_1,\ldots, y_n]$ and (multivariable) polynomial functions on $\zn$, if not treated carefully, may frustrate our efforts at a key juncture. First, we recall the situation in one variable and over $\Z$, which is simpler. A polynomial $P \in \Z[y_1]$ induces the polynomial function $\Z \to \Z$ given by $y_1 \mapsto P(y_1)$. Conversely, if a function $f : \Z \to \Z$ is known to be a polynomial function, then it is induced by exactly one polynomial $P \in \Z[y_1]$. Thus, if two polynomial functions are equal over $\Z$, then, helpfully, it follows that corresponding coefficients of their inducing polynomials in $\Z[y_1]$ are equal.
	
	However, in $\zn$, while it is true that a polynomial $P \in \zn[y_1]$ induces the polynomial function $\zn \to \zn$ given by $y_1 \mapsto P(y_1)$, a (single-variable) polynomial function on $\zn$ may be induced by a variety of (single-variable) polynomials over $\zn$, even of degree less than $N$. For example, $5{y_1}^3 - 6y_1 \equiv 14y_1 \bmod 15$ for all $y_1 \in \mathbb{Z}_{15}$, so the polynomial function $y_1 \mapsto 5{y_1}^3 - 6y_1$ on $\mathbb{Z}_{15}$ is induced by both the polynomial $5y_1^3 - 6y_1 \in \mathbb{Z}_{15}[y_1]$ and by the polynomial $14y_1 \in \mathbb{Z}_{15}[y_1]$ -- and indeed, by others as well.
	
	In our arguments below, we mostly will be working with polynomial functions on $R$ or $\zn$, and not directly with elements of the polynomial rings $R[y_1,\ldots, y_n]$ or $\zn[y_1,\ldots, y_n]$. However, we will encounter a situation with the following description: Two multivariable polynomial functions $f_1, f_2 : \mathbb{Z}_N^n \to \zn$, respectively induced by polynomials $P_1, P_2 \in \zn[y_1,\ldots, y_n]$, will be equal as functions, and we would like to conclude the equality of corresponding coefficients of $P_1$ and $P_2$. Thus, we will want this situation to be one where a multivariable polynomial function over $\zn$ corresponds to one polynomial in $\zn[y_1,\ldots,y_n]$, which will hold if some reasonable restriction on the degree is assumed. To this end, we state a theorem, a corollary, and a lemma.
	
	The theorem describes a canonical representation of a single-variable polynomial function modulo a positive integer. Such a theorem has been reproved several times (see, e.g., \cite{kempner} and \cite{singmaster}), so we do not prove it here. Then, the corollary specifies the ``reasonable'' degree restriction to which we have just now alluded. Finally, the lemma is the analogous result for multivariable polynomials that we actually use in this article.
	
	\begin{theorem}[{\cite[Theorem 10]{singmaster}}]\label{thm: singmaster}
		Let $N$ be a positive integer, and let $\ell$ be the least positive integer such that $N$ divides $\ell!$. Let $f$ be a polynomial function on $\zn$. Then $f$ is induced by a unique polynomial $F = \sum_{k=0}^{\ell-1} b_kx^k \in \zn[x]$ such that $0 \leq b_k < N/\gcd(k!,N)$ for each $k \in \{0,\ldots, \ell-1\}$. 
	\end{theorem}
	\begin{cor}\label{cor: polyuniq}
		Let $N$ be a positive integer. Let $P \in \zn[x]$ be a polynomial with degree less than $\lpf N$. Then the induced polynomial function $x\mapsto P(x)$ on $\zn$ cannot be represented by any other polynomial $Q \in \zn[x]$ with degree less than $\lpf N$.
	\end{cor}
	\begin{proof}
		Observe that $\gcd(k!,N) = 1$ for each $k < \lpf N$. Define $\ell$ as in Theorem~\ref{thm: singmaster}. If $\ell < \lpf N$, then $\gcd(\ell!,\lpf N) = 1$, so $N$ cannot divide $\ell!$. Therefore $\ell \geq \lpf N$. The result follows.
	\end{proof}
	
	\begin{lemma}\label{lem: mvpolyuniq}
		Let $n$ and $N$ be positive integers. Let $P(y_1,\ldots, y_n)$ and $Q(y_1,\ldots, y_n) \in \Z_N[y_1,\ldots, y_n]$ be polynomials, each of degree less than $\lpf N$, which satisfy $P(y_1,\ldots, y_n) \equiv Q(y_1,\ldots, y_n) \bmod N$ for all $y_1,\ldots, y_n \in \Z_N$. Then $P(y_1,\ldots, y_n) = Q(y_1,\ldots, y_n)$ as polynomials; i.e., corresponding coefficients are equal.
	\end{lemma}
	\begin{proof} We induct on $n$. When $n = 1$, the result follows by Corollary~\ref{cor: polyuniq}. Now suppose the result holds for $n \geq 1$. We show that it holds for $n + 1$. Let $P(y_1,\ldots, y_{n+1})$ and $Q(y_1,\ldots, y_{n+1}) \in \Z_N[y_1,\ldots, y_{n+1}]$ be two polynomials, each of degree less than $\lpf N$, which satisfy $P(y_1,\ldots, y_{n+1}) \equiv Q(y_1,\ldots, y_{n+1}) \bmod N$ for all $y_1,\ldots, y_{n+1} \in \Z_N$.
		
		Given indeterminates $y_1, \ldots, y_{n+1}$ and a vector $\beta$ of nonnegative integers $\beta_1, \ldots, \beta_{n + 1}$, write $y^\beta$ for the expression $y_1^{\beta_1} \cdots y_{n +1}^{\beta_{n + 1}}$ and $\underline{y}^{\underline{\beta}}$ for the expression $y_1^{\beta_1} \cdots y_{n}^{\beta_{n}}$. Let $d$ be the maximum of the degrees of $P$ and of $Q$. Let $S$ denote the set $\{ \beta \in \Z^{n+1}: 0 \leq \beta_1, \ldots, \beta_{n +1} \leq d\}$, and for each $i \in \{0,\ldots, d\}$, let $S_i$ denote the set $\{ \beta \in S : \beta_{n+1} = i\}$.
		
		Write $P(y_1, \dots, y_{n + 1}) := \sum_{\beta \in S} c_\beta y^\beta$ and $Q(y_1, \dots, y_{n + 1}) := \sum_{\beta \in S} c'_\beta y^\beta$. We wish to show that $c_\beta = c'_\beta$ for each $\beta \in S$.
		
		For each $a \in \Z_N$, let $P_a(y_1,\ldots, y_n) := P(y_1,\ldots, y_n, a) \in \Z_N[y_1, \dots, y_n]$ and $Q_a(y_1,\dots, y_n) := Q(y_1,\dots, y_n, a) \in \Z_N[y_1, \dots, y_n]$.
		
		By assumption on $P$ and $Q$, we have $P_a(y_1, \dots, y_n) \equiv Q_a(y_1, \ldots, y_n) \bmod N$ for all $y_1,\ldots, y_n \in \Z_N$, and both $P_a$ and $Q_a$ have degree less than $\lpf N$ since $P$ and $Q$ do. Hence, by the inductive hypothesis, $P_a = Q_a$ as polynomials in $\Z_N[y_1,\ldots, y_n]$. We will apply this conclusion in a moment. First, we observe that
		\begin{equation}
			P_a(y_1,\dots, y_n) = \sum_{i = 0}^d \sum_{\beta \in S_i} c_\beta a^i \underline{y}^{\underline{\beta}},
		\end{equation}
		and, letting $\underline{S}$ denote the set $\{ \underline{\alpha} = (\alpha_1, \dots, \alpha_n) \in \Z^{n} : 0 \leq \alpha_1, \ldots,\alpha_n \leq d \}$, it follows that
		\begin{equation}
			P_a(y_1,\dots, y_n) = \sum_{\underline{\alpha} \in \underline{S}} (c_{(\underline{\alpha},0)} + c_{(\underline{\alpha},1)} a + \dots + c_{(\underline{\alpha},d)} a^d) \underline{y}^{\underline{\alpha}},  
		\end{equation}
		where $(\underline{\alpha},i) := (\alpha_1,\ldots, \alpha_{n}, i)$. Similarly,
		\begin{equation}
			Q_a(y_1,\ldots, y_n) = \sum_{\underline{\alpha} \in \underline{S}} (c'_{(\underline{\alpha},0)} + c'_{(\underline{\alpha},1)} a + \dots + c'_{(\underline{\alpha},d)} a^d) \underline{y}^{\underline{\alpha}}.
		\end{equation}
		Since, for each $a \in \Z_N$, $P_a = Q_a$ as polynomials in $\Z_N[y_1,\ldots, y_{n}]$, it follows that for each $\underline{\alpha} \in \underline{S}$, 
		\begin{equation}
			c_{(\underline{\alpha},0)} + c_{(\underline{\alpha},1)} x + \dots + c_{(\underline{\alpha},d)} x^d \equiv c'_{(\underline{\alpha},0)} + c'_{(\underline{\alpha},1)} x + \dots + c'_{(\underline{\alpha},d)} x^d \bmod N,
		\end{equation}
		so by Corollary~\ref{cor: polyuniq}, valid since $d < \lpf N$, it follows that $c_{(\underline{\alpha},i)} \equiv c'_{(\underline{\alpha},i)} \bmod N$ for each $i \in \{0,\ldots, d\}$. The previous statement is true for all $\underline{\alpha} \in \underline{S}$, which suffices to show that $c_\beta = c'_\beta$ for each $\beta \in S$, as desired.
	\end{proof}
	
	Next, we consider the phenomenon of linear independence over $\zn$. It is a simple matter to ensure that independence of polynomials in $\Z[y_1,\ldots, y_n]$ translates to linear independence over $\zn$:
	\begin{prop}\label{prop: lindep over zn}
		Let $n \geq 1$ be an integer. Let $\{P_1,\ldots, P_m \}\subset \Z[y_1,\ldots, y_n]$ be an independent family. Then there exists $C_1 > 0$ such that, for any positive integer $M$ with $\lpf M > C_1$, the family $\{P_1,\ldots, P_m\}$, viewed as a family of polynomials in $\mathbb{Z}_M[y_1,\ldots, y_n]$, is linearly independent over $\mathbb{Z}_M$.
	\end{prop}
	\begin{proof} 
		Enumerate the set $\{\alpha^{(1)},\ldots, \alpha^{(W)}\}$ of $\alpha \in (\N \cup \{0\})^n$ such that the coefficient of $y^\alpha$ is nonzero for some $P_j(y)$.
		
		First, form the $W\times m$ matrix of coefficients of $P_1,\ldots,P_m$ such that, in the $i$th row, the $j$th entry is the coefficient of $y^{\alpha^{(i)}}$ in $P_j(y)$. Since $\{P_1,\ldots,P_m\}$ is an independent family, the columns of this matrix are linearly independent, and hence there exists a nonzero $m \times m$ minor $M'$. Let $C_1$ be a positive integer larger than any prime which divides $M'$. If $M$ satisfies $\lpf M > C_1$, then $M'$ is invertible mod $M$, so $\{P_1,\ldots,P_m\}$ is a linearly independent family over $\mathbb{Z}_M$.
	\end{proof}
	
	We will also need the following result.
	
	\begin{lemma}\label{lem: Bx-c=0 solutions}
		Let $N > 1$ be a positive integer. Let $B,C \in \Z$ be integers such that $B \not\equiv 0 \bmod N$. Then the number of solutions $x \in \Z_N$ to the equation $Bx-C \equiv 0 \bmod N$ is at most $N/\lpf N$.
	\end{lemma}
	\begin{proof}
		Factor $N = p_1^{a_1}\cdots p_\ell^{a_\ell}$ with $p_1 < \cdots < p_\ell$ and all $a_i$ positive. There exist nonnegative integers $b_1,\ldots b_\ell, c_1,\ldots, c_\ell$ and integers $b,c$ such that $B = bp_1^{b_1}\cdots p_\ell^{b_\ell}$, $C = cp_1^{c_1}\cdots p_\ell^{c_\ell}$, and $\gcd(b,N) = \gcd(c,N) = 1$. Without loss of generality, we may assume $b = c = 1$ because the invertibility of $b$ and $c$ modulo $N$ ensures that the number of solutions $x \in \Z_N$ to the equation $Bx - C \equiv 0 \bmod N$ is the same as the number of solutions to the equation $p_1^{b_1}\cdots p_\ell^{b_\ell} x - p_1^{c_1}\cdots p_\ell^{c_\ell} \equiv 0 \bmod N$.
		
		By the Chinese remainder theorem,
		\begin{equation*}
			\# \{ x \in \Z_N : Bx - C \equiv 0 \bmod N \} \ = \ \prod_{i=1}^\ell \# \{ x \in \Z_{p_i^{a_i}} : Bx - C \equiv 0 \bmod p_i^{a_i} \}.
		\end{equation*}
		
		By assumption, $B \not\equiv 0 \bmod N$, so there exists $j$ such that $b_j < a_j$. The integers $\prod_{i \neq j} p_i^{b_i}$ and $\prod_{i \neq j} p_i^{c_i}$ are invertible modulo $p_j^{a_j}$, so
		\begin{equation*}
			\#\{ x \in \Z_{p_j^{a_j}} : Bx - C \equiv 0 \bmod p_j^{a_j} \} = \# \{ x \in \Z_{p_j^{a_j}} : p_j^{b_j}x - p_j^{c_j} \equiv 0 \bmod p_j^{a_j} \}.
		\end{equation*}
		First, suppose $b_j > c_j$. Then $p_j^{b_j} x - p_j^{c_j} =  p_j^{c_j}(p_j^{b_j - c_j}x - 1) \equiv 0 \bmod p_j^{a_j}$ has no solutions $x \in \Z_{p_j^{a_j}}$. Second, suppose $b_j \leq c_j$. Then $p_j^{b_j} x - p_j^{c_j} =  p_j^{b_j}(x - p^{c_j-b_j}) \equiv 0 \bmod p_j^{a_j}$ has $p_j^{b_j}$ solutions $x \in \Z_{p_j^{a_j}}$. In either case, it follows that 
		\begin{equation*}
			\# \{ x \in \Z_N : Bx - C \equiv 0 \bmod N \} \ \leq \ \left(\prod_{i\neq j} p_i^{a_i} \right) p_j^{b_j} \ = \ \frac{N}{p_j^{a_j-b_j}} \ \leq \ \frac{N}{\lpf N},
		\end{equation*}
		as desired.
	\end{proof}

	%
	%
	%
	%
	%
	%
	%
	%
	%
	%
	%
	%
	%
	%
	%

	%
	%
	%
	%
	%
	%
	%
	%
	%
	%
	%
	%
	%
	%
	%
	%
	%
	%
	\section{The base case of Theorem~\ref{main thm}}\label{sec: base case}
	Our overall plan is to prove Theorem~\ref{main thm} by induction on $m_1$. In this section, we prove the base case:
	\begin{prop}[Theorem~\ref{main thm} in the case $m_1 = 1$]\label{main thm base}
		Let $m_2 \geq 0$ and $n \geq 1$ be integers, and let $\mathbf P = \{ P_1,Q_1,\ldots,Q_{m_2}\} \subset \Z[y_1,\ldots, y_n]$ be an independent family of polynomials with zero constant term. There exist $c,C,\gamma > 0$ such that the following holds. For any finite commutative ring $R$ with characteristic $N$ such that $\lpf N > C$, any vector $F = (f_0,f_1)$ of 1-bounded functions $R \to \C$, and any vector $\Psi = (\psi_1,\ldots,\psi_{m_2})$ of additive characters of $R$, one has
		\be
		\left| \Lambda_{P_1}^{Q_1,\ldots,Q_{m_2}}(F;\Psi) - 1_{\Psi=1} \rcavg x f_0(x) \rcavg x f_1(x)\right| \ \leq \  c\,\lpf{N}^{-\gamma},
		\ee
		where $1_{\Psi=1}$ equals 1 if every character in $\Psi$ is trivial and 0 otherwise.
	\end{prop}
	We will prove this proposition by Fourier analysis. We need a lemma establishing a bound on character sums over rings with certain polynomial arguments, viz.,
	\begin{equation}
		\frac{1}{|R|^n} \sum_{y\in R^n} \prod_{j=1}^m \psi_j(Q_j(y)),
	\end{equation}
	where $\{Q_1,\ldots, Q_m\} \subset \Z[y_1,\ldots, y_n]$ is an independent family of polynomials and $\psi_1,\ldots, \psi_m$ are additive characters of a finite commutative ring $R$.
	
	In \cite{peluse}, the corresponding lemma is proved by appealing to the Weil bound for curves (see, e.g., Theorem 3.2 in \cite{kowalskinotes}). In our case, we are not aware of any existing bound over finite commutative rings that performs the same function.  Thus, we will need to develop a suitable general estimate for our purposes. This is Lemma~\ref{lem: multicharacter product of independent polynomials as arguments} below.
	
	To illustrate some of the arguments involved in the proof of Proposition~\ref{main thm base}, we will first prove Proposition~\ref{main thm base} in the special case that $m_2 = 0$, $n = 1$, and $P_1(y_1) = y_1^2$. After proving this special case in Subsection~\ref{subsec: example base case}, we prove Proposition~\ref{main thm base} more generally in Subsections~\ref{subsec: some character sums}~and~\ref{subsec: proof of prop main thm base}, in which, having seen the skeleton of the argument, we are ready to handle any complexity that is introduced by working with a multivariable family $\mathbf P$ rather than just a single polynomial in one variable with low degree. The exposition in Subsection~\ref{subsec: example base case} is independent of the exposition in Subsections~\ref{subsec: some character sums}~and~\ref{subsec: proof of prop main thm base}.
	
	\subsection{Proof of a special case of Proposition~\ref{main thm base}}\label{subsec: example base case}
	In this subsection, we prove Proposition~\ref{main thm base} in the special case that $m_2 = 0$, $n = 1$, and $P_1(y_1) = y_1^2$. The organization of this subsection is straightforward: Proposition~\ref{example: prop: multiplicative character estimate} is used to prove Proposition~\ref{example: prop: distinct powers character estimate}, which is used to prove Proposition~\ref{example: main thm base}, which is precisely the special case we deal with.
	
	We begin with a simple character sum estimate in two variables. The following proposition is a special case of Proposition~\ref{prop: multiplicative character estimate}. We include the proof here both to make this section more self-contained and to ease the understanding of the general case by introducing the main ideas in this simpler case.

	\begin{proposition}\label{example: prop: multiplicative character estimate}
		Let $R$ be a finite commutative ring with characteristic $N$. Let $\chi \in \hat{R}$ be a nontrivial additive character. Then
		\be
		\left|  \rcavg{h_1,h_2}  \chi\left( h_1h_2 \right) \right| \ \leq \ \frac{1}{\lpf N}.
		\ee
	\end{proposition}
	\begin{proof}
		
		By Lemma~\ref{lem: structure of (R,+)}, there is an isomorphism $\phi : (R,+) \to \Z_{b_1} \times \cdots \times \Z_{b_r}$, where $b_r = N$, the integers $b_i > 1$, $i \in \{1,\ldots, r-1\}$, are possibly indistinct prime powers which divide $N$, and $\phi(1) = (0,\ldots, 0,1)$.
		
		For each $j \in \{1,\ldots, r\}$, define the element $\delta^{(j)} = (\delta^{(j)}_1, \ldots, \delta^{(j)}_r) \in \Z_{b_1} \times \cdots \times \Z_{b_r}$ by
		\begin{equation*}
			\delta^{(j)}_i \ := \ \begin{cases} 1 & \text{ if } i = j, \\ 0 & \text{ otherwise,} \end{cases}  
		\end{equation*}
		and let $g_j = \phi^{-1}(\delta^{(j)}) \in R$.
		Then $g_r = 1$ and each element $s \in R$ can be written in exactly one way as $\sum_{i=1}^r x_i'g_i$ for some $x_i' \in \Z_{b_i}$, $i \in \{1,\ldots,r\}$. Moreover, if $s = \sum_{i=1}^r x_i' g_i$ for some $x_i' \in \Z_{b_i}, i \in \{1,\ldots, r\}$, then for any $r$-tuple $\mathbf{x^{(1)}} = (x_1^{(1)},\ldots, x_r^{(1)}) \in \Z_{N}^r$ such that $x_i^{(1)} \equiv x_i' \bmod b_i$ holds for all $i$, it follows that $s = \sum_{i=1}^r x_i^{(1)} g_i$ as well. Hence, for each $s \in R$, there are exactly $\prod_{i=1}^r (N/b_i) = N^r / |R|$ $r$-tuples $\mathbf{x^{(1)}} = (x_1^{(1)},\ldots, x_r^{(1)}) \in \Z_{N}^r$ such that $s = \sum_{i=1}^r x_i^{(1)}g_i$. Thus, $r$-tuples over $\Z_N$ are in $N^r/|R|$-to-1 correspondence with elements of $R$, which justifies the equation
		\begin{equation}\label{example: lpfcharestimate eqn1}
			\left| \rcavg{h_1,h_2}\chi\left( h_1h_2 \right) \right| \ = \ \left| \cavg{\mathbf{x^{(1)}},\mathbf{x^{(2)}}}{\Z_N^r} \chi\left( \left(\sum_{i_1=1}^r x^{(1)}_{i_1} g_{i_1} \right) \left(\sum_{i_2=1}^r x^{(2)}_{i_2} g_{i_2} \right)\right) \right|.
		\end{equation}
		Let us now compute the right-hand side.
		
		Given a positive integer $n$, we use the shorthand $e_n(x) := e^{2\pi i x/n}$ for any real $x$. 
		The group of additive characters of $R$ is isomorphic to $(R,+)$. Since $\chi \in \hat{R}$ is nontrivial, it follows that there exist $a_i \in \Z_{b_i}$, $i \in \{1,\ldots, r\}$, not all zero, such that $\chi(\sum_{i=1}^r x_i g_i)  \ = \ \prod_{i=1}^r e_{b_i}(a_ix_i) \ = \ e_N\left( \sum_{i=1}^r \frac{N}{b_i} a_i x_i \right)$ for each $(x_1,\ldots, x_r) \in \Z_N^r$.
		
		To proceed, we will need an explicit expression for the argument of $\chi$ in \eqref{example: lpfcharestimate eqn1}. Since $\chi$ is an additive character, not a multiplicative one, in order to understand the product inside, we will express $g_{i_1}g_{i_2}$ as a linear combination of $g_i$. Thus, for each $i,j \in \{1,\ldots, r\}$, let $c_k^{(i,j)} \in \Z_{b_k}$, $k \in \{1,\ldots, r\}$, be such that $g_ig_j = \sum_{k=1}^r c_k^{(i,j)} g_k$. Then, for any $i,k \in \{1,\ldots, r\}$, we observe that, since $g_r = 1$,
		\begin{equation}\label{example: eqn explaining multiplication law}
			c_{k}^{(i,r)} \ = \ c_k^{(r,i)} \ = \ \begin{cases} 1 & \text{ if } i = k, \\ 0 & \text{ otherwise.} \end{cases}
		\end{equation}
		Thus, we can write
		\begin{align*}
			\left( \sum_{i_1=1}^r x_{i_1}^{(1)} g_{i_1} \right)\left( \sum_{i_2=1}^r x_{i_2}^{(2)} g_{i_2} \right) \ & =  \sum_{i_1,i_2=1}^r x_{i_1}^{(1)}x_{i_2}^{(2)} g_{i_1}g_{i_2} \\
			& = \  \sum_{i_1,i_2=1}^r x_{i_1}^{(1)}x_{i_2}^{(2)} \left( \sum_{k=1}^r c_{k}^{(i_1,i_2)} g_{k}\right)  \\ 
			& = \ \sum_{k=1}^r \left( \sum_{i_1, i_2 = 1}^r x_{i_1}^{(1)} x_{i_2}^{(2)} c_k^{(i_1,i_2)} \right) g_{k}.
		\end{align*}
		Hence, for each $\mathbf{x^{(2)}} \in \Z_N^r$, denoting by $F(\mathbf{x^{(2)}};j)$ the expression $\sum_{i=1}^r \frac{N}{b_i}a_i \sum_{i_2=1}^r x_{i_2}^{(2)}c_i^{(j,i_2)}$, we find
		\begin{align*}
			\cavg{\mathbf{x^{(1)}}}{\Z_N^r} \chi\left( \left( \sum_{i_1=1}^r x_{i_1}^{(1)} g_{i_1} \right)\left( \sum_{i_2=1}^r x_{i_2}^{(2)} g_{i_2} \right)  \right) \ & = \ \cavg{\mathbf{x^{(1)}}}{\Z_N^r} e_N\left( \sum_{i = 1}^r \frac{N}{b_i}a_i \left( \sum_{i_1,i_2 = 1}^r x_{i_1}^{(1)} x_{i_2}^{(2)} c_i^{(i_1,i_2)} \right) \right) \\
			& = \ \cavg{x_1^{(1)},\ldots,x_r^{(1)}}{\Z_N} \prod_{i_1=1}^r e_N\left( x_{i_1}^{(1)} F(\mathbf{x^{(2)}};i_1) \right) \\
			& = \ \prod_{j=1}^r  \cavg{x_j^{(1)}}{\Z_N} e_N\left( x_{j}^{(1)} F(\mathbf{x^{(2)}}; j) \right) \\
			& = \ \begin{cases} 1  \quad \text{ if } \displaystyle F(\mathbf{x^{(2)}}; j) \equiv 0 \bmod N \text{ for every } j \in \{1,\ldots, r\}, \\ 0 \quad \text{ otherwise.} \end{cases}
		\end{align*}
		It follows that
		\begin{align*}
			& \left|\cavg{\mathbf{x^{(1)}},\mathbf{x^{(2)}}}{\Z_N^r} \chi\left(  \left( \sum_{i_1=1}^r x_{i_1}^{(1)} g_{i_1} \right)\left( \sum_{i_2=1}^r x_{i_2}^{(2)} g_{i_2} \right) \right)\right| \\ & = \ \left|\cavg{\mathbf{x^{(2)}}}{\Z_N^r} \cavg{\mathbf{x^{(1)}}}{\Z_N^r}\chi\left( \left( \sum_{i_1=1}^r x_{i_1}^{(1)} g_{i_1} \right)\left( \sum_{i_2=1}^r x_{i_2}^{(2)} g_{i_2} \right)  \right)\right|
			\\ & = \ \frac{1}{N^{r}}\# \Big\{ \mathbf{x^{(2)}} \in \Z_N^{r} : \text{For every } j \in \{1,\ldots, r\}, \; F(\mathbf{x^{(2)}}; j) \equiv 0 \bmod N \Big\} \\
			& \leq \ \frac{1}{N^{r}}\# \Big\{ \mathbf{x^{(2)}} \in \Z_N^{r} : F(\mathbf{x^{(2)}};r) \equiv 0 \bmod N \Big\}.
		\end{align*}
		Let us estimate $\# \Big\{ \mathbf{x^{(2)}} \in \Z_N^{r} : F(\mathbf{x^{(2)}};r) \equiv 0 \bmod N \Big\}$ in a simple way. By \eqref{example: eqn explaining multiplication law}, it follows that
		\begin{equation*}
			F(\mathbf{x^{(2)}}; r) \ \equiv \ \sum_{i=1}^r a_i \frac{N}{b_i} \sum_{i_2 = 1}^r x_{i_2}^{(2)} c_{i}^{(r,i_2)} \ \equiv \ \sum_{i=1}^{r} a_i \frac{N}{b_i} x_{i}^{(2)} \bmod N.
		\end{equation*}
		Since $\chi$ is nontrivial, there is an index $i_0 \in \{1,\ldots, r\}$ such that $a_{i_0} \not\equiv 0 \bmod b_{i_0}$, whence $a_{i_0} N / b_{i_0} \not\equiv 0 \bmod N$. Then, for each of the $N^{r-1}$-many tuples $(x_1^{(2)},\ldots, x_{i_0-1}^{(2)}, x_{i_0+1}^{(2)}, \ldots, x_r^{(2)}) \in \Z_N^{r-1}$, there are at most $N/\lpf N$ solutions $x_{i_0}^{(2)} \in \Z_N$ to the equation $\sum_{i=1}^{r} a_i \frac{N}{b_i} x_{i}^{(2)} \equiv 0 \bmod N$ by Lemma~\ref{lem: Bx-c=0 solutions}, so
		\begin{equation*}
			\# \{ \mathbf{x^{(2)}} \in \Z_N^{r} : F(\mathbf{x^{(2)}}; r) \equiv 0 \bmod N \} \ \leq \ \frac{N^{r}}{\lpf N},
		\end{equation*}
		which completes the proof.
	\end{proof}
	
	Next, we use the previous proposition to prove a character sum estimate along $y^2$. Again, this estimate is a special case of a later one (namely Proposition~\ref{prop: distinct powers character estimate}), and we include it for self-containedness and for ease of understanding. 
	
	\begin{proposition}\label{example: prop: distinct powers character estimate}
		Let $R$ be a finite commutative ring with characteristic $N$. Suppose that $\lpf N >2$. Let $\chi \in \hat{R}$ be a nontrivial additive character. Then
		\begin{equation*}
			\left\vert \cavg{y}{R}  \chi(y^{2})\right\vert \ \leq \ \left(\frac{1}{\lpf N}\right)^{1/4}. 
		\end{equation*}
	\end{proposition}
	\begin{proof}
		Observe that by a change of variables $(y',y) \mapsto (y+h_1,y) \in R^2$ we have 
		\begin{equation}
			\left\vert \cavg{y}{R}  \chi(y^{2})\right\vert^2 \ = \ \rcavg{y'} \chi((y')^2) \ol{\rcavg{y} \chi(y^2)} \ = \ \rcavg{y,h_1} \chi((y+h_1)^2-y^2) \ = \ \rcavg{y,h_1} \chi(2h_1y+h_1^2)
		\end{equation}
		and, similarly, for each $h_1 \in R$,
		\begin{equation}
			\left\vert \rcavg{y} \chi(2h_1y+h_1^2) \right\vert^2 \ = \ \rcavg{y'} \chi(2h_1y'+h_1^2) \ol{\rcavg{y} \chi(2h_1y+h_1^2)} \ = \ \rcavg{y,h_2} \chi(2h_1h_2).  
		\end{equation}
		It follows by Cauchy--Schwarz that
		\begin{multline}
			\left\vert \rcavg{y} \chi(y^2) \right\vert^4 \ = \ \left\vert \rcavg{y,h_1} \chi(2h_1y+h_1^2) \right\vert^2 \\ \leq \ \rcavg{h_1} \left\vert \rcavg{y} \chi(2h_1y+h_1^2) \right\vert^2 \ = \ \rcavg{h_1} \rcavg{y,h_2} \chi(2h_1h_2) \ = \ \rcavg{h_1,h_2} \chi(2h_1h_2).
		\end{multline}
		Since $\lpf N > 2$, $2$ is a unit in $R$ by Lemma~\ref{lem: what is a unit} and hence $2h_1 \mapsto h_1$ is a valid change of variables, yielding
		\begin{equation}
			\left\vert \rcavg{y} \chi(y^2) \right\vert^4 \ \leq \ \rcavg{h_1,h_2} \chi(h_1h_2).
		\end{equation}
		By Proposition~\ref{example: prop: multiplicative character estimate}, we conclude that
		\begin{equation*}
			\left\vert \cavg{y}{R} \chi(y^{2})\right\vert \ \leq \ \left( \frac{1}{\lpf{N}} \right)^{1/4}.
		\end{equation*}
	\end{proof}
	
	Finally, we use Fourier analysis and the previous proposition to finish the proof of the special case. This nontrivially generalizes the corresponding special case of the main theorem in \cite{bb}; see the fifth remark under Theorem~\ref{main thm for intro} for more precise commentary.
	\begin{prop}\label{example: main thm base}
		There exist $c,C,\gamma > 0$ such that the following holds. For any finite commutative ring $R$ with characteristic $N$ such that $\lpf N > C$ and any 1-bounded functions $f_0,f_1: R \to \C$, one has
		\be
		\left| \rcavg{x,y} f_0(x)f_1(x+y^2) - \rcavg x f_0(x) \rcavg x f_1(x)\right| \ \leq \  c\,\lpf{N}^{-\gamma}.
		\ee
	\end{prop}
	\begin{proof}
		Let $C = 2$. Applying Lemma~\ref{example: prop: distinct powers character estimate}, we see that for any finite commutative ring $R$ with characteristic $N$ such that $\lpf N > C$ and any nontrivial additive character $\phi$ of $R$, we have
		\be
		\left| \cavg{y}{R} \phi(y^2)  \right| \ \leq \ \left(\frac{1}{\lpf N} \right)^{1/4}.
		\ee
		
		Fix a finite commutative ring $R$ with characteristic $N$ satisfying $\lpf N > C$ and 1-bounded functions $f_0,f_1 : R \to \C$, and let $f_1' = f_1 - \rcavg x f_1(x)$. By splitting $f_1 = f_1' + \rcavg x f_1(x)$, we observe that
		\begin{equation}\label{example: one-time-use eqn 1}
			\rcavg{x,y} f_0(x)f_1(x+y^2)
			\ = \ \left(\rcavg x f_0(x) \rcavg{x} f_1(x) \right) + \rcavg{x,y}f_0(x)f_1'(x+y^2).
		\end{equation}
		
		We claim that
		\be
		\left| \rcavg{x,y} f_0(x)f_1'(x+y^2) \right| \ \leq \ 2\cdot  \left(\frac{1}{\lpf N} \right)^{1/4}.
		\ee
		Indeed, by Fourier inversion, we can compute that
		\begin{align*}
			\rcavg{x,y} f_0(x)f_1'(x+y^2) \
			& = \  \sum_{\phi_0,\phi_1 \in \widehat{R}} \hat{f_0}(\phi_0) \hat{f_1'}(\phi_1) \left( \rcavg x \phi_0(x) \phi_1(x) \right) \left(\cavg{y}{R} \ol{\phi_1}(y^2) \right) \\
			& = \ \sum_{\phi \in \widehat{R}} \hat{f_0}(\phi) \hat{f_1'}(\ol{\phi}) \left(\cavg{y}{R} \phi(y^2) \right) \\
			& = \ \sum_{1 \neq \phi \in \widehat{R}} \hat{f_0}(\phi) \hat{f_1'}(\ol{\phi}) \left(\cavg{y}{R} \phi(y^2)  \right),
		\end{align*}
		where the first equality holds by the change of variables $x \mapsto -x$ in $R$, the second equality holds by orthogonality relations, and the last equality holds since $\hat{f_1'}(1) = \rcavg x f_1'(x) = 0$. By our choice of $C$, Cauchy--Schwarz, Plancherel's theorem, and the fact that $f_0$ is 1-bounded and $f_1'$ is 2-bounded, we conclude that
		\ba
		\left| \rcavg{x,y} f_0(x)f_1'(x+y^2) \right| \ & \leq \ \left(\frac{1}{\lpf N} \right)^{1/4} \sum_{1 \neq \phi \in \widehat{R}} \left| \hat{f_0}(\phi) \hat{f_1'}(\ol{\phi})\right| \\
		& \leq \ \left(\frac{1}{\lpf N} \right)^{1/4} |R| \lr{\hat{f_0},\hat{f_0}}^{1/2} \lr{\hat{f_1'},\hat{f_1'}}^{1/2} \\
		& = \ \left(\frac{1}{\lpf N} \right)^{1/4} \norm{f_0}_{L^2} \norm{f_1'}_{L^2} \\
		& \leq \ 2\cdot  \left(\frac{1}{\lpf N} \right)^{1/4}.
		\ea
		Combining the claim and the equality \eqref{example: one-time-use eqn 1} yields the result.
	\end{proof}
	
	\subsection{Character sums in finite commutative rings}\label{subsec: some character sums}
	The goal of this subsection is to prove Lemma~\ref{lem: multicharacter product of independent polynomials as arguments}, which asymptotically controls certain character averages with independent polynomial arguments in terms of the ring characteristic and the maximal degree of the polynomial family. Unlike the Weil bound, which enables the power savings (in terms of $|R|$) in the finite field case, our Lemma~\ref{lem: multicharacter product of independent polynomials as arguments} instead saves a factor in terms of $\lpf{\mathrm{char}(R)}$, a quantity which, though sufficient for our purposes, has no general relation to $|R|$.
	
	In the previous subsection, we saw that Proposition~\ref{example: prop: distinct powers character estimate}, when combined with some Fourier analysis, yielded the desired Proposition~\ref{example: main thm base}. Similarly, Lemma~\ref{lem: multicharacter product of independent polynomials as arguments} will yield the desired generalization, Proposition~\ref{main thm base}. 
	
	To prove the lemma, we first need to show Proposition~\ref{prop: multiplicative character estimate}, which establishes a bound on $\left| \rcavg{h_1,\ldots, h_m} \chi(h_1\cdots h_m)\right|$ for an additive character $\chi \in \hat{R}$ in terms of $m$ and the characteristic of the ring $R$. We begin with the prerequisite to the proof of Proposition~\ref{prop: multiplicative character estimate}, namely a technical proposition which generates our version of power savings.

	Let $r$ be a positive integer. Suppose $c_{k}^{(i,j)}$, $i,j,k \in \{1,\ldots, r\}$, are integers. For each $m \geq 2$, define the associated auxiliary function $\eta_m : \{1,\ldots,r\}^{2m-1} \to \Z$ by
	\begin{equation}\label{definition of auxiliary eta_m}
		\eta_m(k_1,\ldots,k_{m-1};i_1,\ldots, i_m) := c_{k_1}^{(i_1,i_2)} \prod_{j=3}^m c_{k_{j-1}}^{(k_{j-2},i_j)},
	\end{equation}
	using the convention that an empty product equals one. These functions satisfy the identity
	\begin{equation}\label{identity for auxiliary eta_m}
		\eta_m(k_1,\ldots,k_{m-2},i;r,i_2,\ldots, i_m) = \eta_{m-1}(k_1,\ldots,k_{m-2};r,i_2,\ldots,i_{m-1}) c_i^{(k_{m-2},i_m)}
	\end{equation}
	for each $m \geq 3$. These fearsome-looking functions are only used in Proposition~\ref{prop trivial estimate on linear system of equations mod N}, the purpose of which is to bound the cardinality of zero sets of certain polynomials arising in the proof of our exponential sum estimate, Proposition~\ref{prop: multiplicative character estimate}. The inductive proof of Proposition~\ref{prop trivial estimate on linear system of equations mod N} is reminiscent of the proof of the classical Schwarz--Zippel lemma.
	
	\begin{proposition}\label{prop trivial estimate on linear system of equations mod N} 
		Fix a positive integer $N > 1$, a positive integer $r$, positive integers $b_1,\ldots, b_r$ that are (non-1) divisors of $N$, and integers $a_i \in \Z_{b_i}$, $i \in \{1,\ldots, r\}$, not all zero. Let $c_{k}^{(i,j)}$, $i,j,k\in \{1,\ldots, r\}$, be integers such that
		\begin{equation*}
			c_k^{(r,i)} \ = \ c_k^{(i,r)} \ = \ \begin{cases} 1 & \text{ if } i = k, \\ 0 & \text{ otherwise,} \end{cases}
		\end{equation*}
		and define the associated auxiliary functions $\eta_m$, $m \geq 2$, as in \eqref{definition of auxiliary eta_m}.
		Then, for any integer $m \geq 2$,
		\begin{multline}
			\# \{ (x_1^{(2)},\ldots,x_r^{(2)},\ldots,x_1^{(m)},\ldots, x_r^{(m)}) \in \Z_N^{r(m-1)} :  \\ \sum_{i=1}^r a_i \frac{N}{b_i} \sum_{i_2,\ldots,i_m = 1}^r x_{i_2}^{(2)}\cdots x_{i_m}^{(m)}\sum_{k_1,\ldots,k_{m-2} = 1}^r \eta_m(k_1,\ldots, k_{m-2},i;r,i_2,\ldots,i_m) \equiv 0 \bmod N \} \\ \leq \ \frac{(m-1)N^{r(m-1)}}{\lpf N}.
		\end{multline}
	\end{proposition}
	\begin{proof}
		To simplify notation in this proof, we will write $\mathbf{x^{(j)}} = (x_1^{(j)},\ldots,x_r^{(j)})$ for each $j \in \{2,\ldots, m\}$. We proceed by induction on $m$. First suppose $m = 2$. We wish to show that
		\begin{equation*}
			\# \{ \mathbf{x^{(2)}} \in \Z_N^{r} : \sum_{i=1}^r a_i \frac{N}{b_i} \sum_{i_2 = 1}^r x_{i_2}^{(2)} c_{i}^{(r,i_2)} \equiv 0 \bmod N \} \ \leq \ \frac{N^{r}}{\lpf N}.
		\end{equation*}
		By assumption,
		\begin{equation*}
			c_i^{(r,i_2)} \ = \ \begin{cases} 1 & \text{ if } i_2 = i, \\ 0 & \text{ otherwise,} \end{cases}
		\end{equation*}
		so it follows that
		\begin{equation*}
			\sum_{i=1}^r a_i \frac{N}{b_i} \sum_{i_2 = 1}^r x_{i_2}^{(2)} c_{i}^{(r,i_2)} \ \equiv \ \sum_{i=1}^{r} a_i \frac{N}{b_i} x_{i}^{(2)} \bmod N.
		\end{equation*}
		By assumption on $a_i$, we may choose an index $i_0 \in \{1,\ldots, r\}$ such that $a_{i_0} \not\equiv 0 \bmod b_{i_0}$; hence $a_{i_0} N / b_{i_0} \not\equiv 0 \bmod N$. Then, for each of the $N^{r-1}$-many tuples $(x_1^{(2)},\ldots, x_{i_0-1}^{(2)}, x_{i_0+1}^{(2)}, \ldots, x_r^{(2)}) \in \Z_N^{r-1}$, there are at most $N/\lpf N$ solutions $x_{i_0}^{(2)} \in \Z_N$ to the equation $\sum_{i=1}^{r} a_i \frac{N}{b_i} x_{i}^{(2)} \equiv 0 \bmod N$ by Lemma~\ref{lem: Bx-c=0 solutions}, so
		\begin{equation*}
			\# \{ \mathbf{x^{(2)}} \in \Z_N^{r} : \sum_{i=1}^r a_i \frac{N}{b_i} \sum_{i_2 = 1}^r x_{i_2}^{(2)} c_{i}^{(r,i_2)} \equiv 0 \bmod N \} \ \leq \ \frac{N^{r}}{\lpf N},
		\end{equation*}
		which proves the base case.
		
		Now assume that $m \geq 3$ and that the result holds for $m-1$, i.e.,
		\begin{multline*}
			\# \{ (\mathbf{x^{(2)}},\ldots,\mathbf{x^{(m-1)}}) \in \Z_N^{r(m-2)} :  \\ \sum_{i=1}^r a_i \frac{N}{b_i} \sum_{i_2,\ldots,i_{m-1} = 1}^r x_{i_2}^{(2)}\cdots x_{i_{m-1}}^{(m-1)}\sum_{k_1,\ldots,k_{m-3} = 1}^r \eta_{m-1}(k_1,\ldots, k_{m-3},i;r,i_2,\ldots,i_{m-1}) \equiv 0 \bmod N \} \\ \leq \ \frac{(m-2)N^{r(m-2)}}{\lpf N}.
		\end{multline*}
		Given $\mathbf{x^{(2)}},\ldots,\mathbf{x^{(m-1)}} \in \Z_N^r$ and $i_m \in \{1,\ldots, r\}$, define the quantity
		\begin{multline*}
			B(\mathbf{x^{(2)}},\ldots,\mathbf{x^{(m-1)}}, i_m) := \\ \sum_{i=1}^r a_i \frac{N}{b_i} \sum_{i_2,\ldots,i_{m-1} = 1}^r x_{i_2}^{(2)}\cdots x_{i_{m-1}}^{(m-1)}\sum_{k_1,\ldots,k_{m-2} = 1}^r \eta_{m-1}(k_1,\ldots, k_{m-2};r,i_2,\ldots,i_{m-1})c_{i}^{(k_{m-2},i_m)}.
		\end{multline*}
		Then, in light of the identity \eqref{identity for auxiliary eta_m}, our goal is to show that
		\begin{multline*}
			\# \{ (\mathbf{x^{(2)}},\ldots,\mathbf{x^{(m)}}) \in \Z_N^{r(m-1)} : \sum_{i_m=1}^r x_{i_m}^{(m)} B(\mathbf{x^{(2)}},\ldots,\mathbf{x^{(m-1)}}, i_m) \equiv 0 \bmod N \} \\ \leq \ \frac{(m-1)N^{r(m-1)}}{\lpf N}.
		\end{multline*}
		By assumption,
		\begin{equation*}
			c_i^{(k_{m-2},r)} \ = \ \begin{cases} 1 & \text{ if } k_{m-2} = i, \\ 0 & \text{ otherwise,} \end{cases}
		\end{equation*}
		so we observe that
		\begin{multline*} B(\mathbf{x^{(2)}},\ldots,\mathbf{x^{(m-1)}}, r) \ = \\ \sum_{i=1}^r a_i \frac{N}{b_i} \sum_{i_2,\ldots,i_{m-1} = 1}^r x_{i_2}^{(2)}\cdots x_{i_{m-1}}^{(m-1)}\sum_{k_1,\ldots,k_{m-3} = 1}^r \eta_{m-1}(k_1,\ldots, k_{m-3},i;r,i_2,\ldots,i_{m-1}).
		\end{multline*}
		Let 
		\begin{equation*}
			\mathcal{T}_1 := \{ (\mathbf{x^{(2)}},\ldots,\mathbf{x^{(m-1)}}) \in \Z_N^{r(m-2)} :  B(\mathbf{x^{(2)}},\ldots,\mathbf{x^{(m-1)}}, r) \equiv 0 \bmod N  \}
		\end{equation*}
		and $\mathcal{T}_2 := \Z_N^{r(m-2)} \setminus \mathcal{T}_1$. By the induction hypothesis, $|\mathcal{T}_1| \leq  \frac{(m-2)N^{r(m-2)}}{\lpf N}$, and trivially we have $|\mathcal{T}_2| \leq N^{r(m-2)}$.
		
		On the one hand, suppose $(\mathbf{x^{(2)}},\ldots,\mathbf{x^{(m-1)}}) \in \mathcal{T}_1$. Then we trivially estimate that all tuples $\mathbf{x^{(m)}} \in \Z_N^r$ may be such that
		\begin{equation*}
			\sum_{i_m=1}^r x_{i_m}^{(m)} B(\mathbf{x^{(2)}},\ldots,\mathbf{x^{(m-1)}}, i_m) \equiv 0 \bmod N,
		\end{equation*}
		contributing at most $|\mathcal{T}_1| \cdot N^r \leq \frac{(m-2)N^{r(m-1)}}{\lpf N}$ tuples to the set
		\begin{equation*}
			\{ (\mathbf{x^{(2)}},\ldots,\mathbf{x^{(m)}}) \in \Z_N^{r(m-1)} :  \sum_{i_m=1}^r x_{i_m}^{(m)} B(\mathbf{x^{(2)}},\ldots,\mathbf{x^{(m-1)}}, i_m) \equiv 0 \bmod N \}.
		\end{equation*}
		On the other hand, suppose $(\mathbf{x^{(2)}},\ldots,\mathbf{x^{(m-1)}}) \in \mathcal{T}_2$. Then
		\begin{equation*}
			B(\mathbf{x^{(2)}},\ldots,\mathbf{x^{(m-1)}}, r) \not\equiv 0 \bmod N,
		\end{equation*}
		so it follows that, for each of the $N^{r-1}$-many tuples $(x_1^{(m)},\ldots, x_{r-1}^{(m)}) \in \Z_N^{r-1}$, there are at most $N/\lpf N$ solutions $x_{r}^{(m)} \in \Z_N$ to the equation
		\begin{equation*}
			\sum_{i_m=1}^r x_{i_m}^{(m)} B(\mathbf{x^{(2)}},\ldots,\mathbf{x^{(m-1)}}, i_m) \equiv 0 \bmod N
		\end{equation*}
		by Lemma~\ref{lem: Bx-c=0 solutions}, contributing at most $|\mathcal{T}_2| \cdot N^{r-1} \cdot N/\lpf{N} \ \leq \ N^{r(m-1)}/\lpf{N}$ tuples to the set
		\begin{equation*}
			\{ (\mathbf{x^{(2)}},\ldots,\mathbf{x^{(m)}}) \in \Z_N^{r(m-1)} : \sum_{i_m=1}^r x_{i_m}^{(m)} B(\mathbf{x^{(2)}},\ldots,\mathbf{x^{(m-1)}}, i_m) \equiv 0 \bmod N \}.
		\end{equation*}
		Each tuple in the previous set is such that $(\mathbf{x^{(2)}},\ldots,\mathbf{x^{(m-1)}})$ belongs to either $\mathcal{T}_1$ or $\mathcal{T}_2$, so between the two estimates, no tuple is missing. Adding the two estimates proves the induction hypothesis for $m$.
	\end{proof}

	\begin{proposition}\label{prop: multiplicative character estimate}
		Fix a positive integer $m$. Let $R$ be a finite commutative ring with characteristic $N$. Let $\chi \in \hat{R}$ be a nontrivial additive character. Then
		\be
		\left|  \rcavg{h_1,\ldots ,h_m}  \chi\left( \prod_{i=1}^m h_i \right) \right| \ \leq \ \frac{m-1}{\lpf N}.
		\ee
	\end{proposition}
	\begin{proof}
		The result is trivial when $m = 1$, so assume $m \geq 2$.
		
		By Lemma~\ref{lem: structure of (R,+)}, there is an isomorphism $\phi : (R,+) \to \Z_{b_1} \times \cdots \times \Z_{b_r}$, where $b_r = N$, the integers $b_i > 1$, $i \in \{1,\ldots, r-1\}$, are possibly indistinct prime powers which divide $N$, and $\phi(1) = (0,\ldots, 0,1)$.
		
		For each $j \in \{1,\ldots, r\}$, define the element $\delta^{(j)} = (\delta^{(j)}_1, \ldots, \delta^{(j)}_r) \in \Z_{b_1} \times \cdots \times \Z_{b_r}$ by
		\begin{equation*}
			\delta^{(j)}_i \ := \ \begin{cases} 1 & \text{ if } i = j, \\ 0 & \text{ otherwise,} \end{cases}  
		\end{equation*}
		and let $g_j = \phi^{-1}(\delta^{(j)}) \in R$.
		Then $g_r = 1$ and each element $s \in R$ can be written in exactly one way as $\sum_{i=1}^r x_i'g_i$ for some $x_i' \in \Z_{b_i}$, $i \in \{1,\ldots,r\}$. Moreover, if $s = \sum_{i=1}^r x_i' g_i$ for some $x_i' \in \Z_{b_i}, i \in \{1,\ldots, r\}$, then for any $r$-tuple $\mathbf{x^{(1)}} = (x_1^{(1)},\ldots, x_r^{(1)}) \in \Z_{N}^r$ such that $x_i^{(1)} \equiv x_i' \bmod b_i$ holds for all $i$, it follows that $s = \sum_{i=1}^r x_i^{(1)} g_i$ as well. Hence, for each $s \in R$, there are exactly $\prod_{i=1}^r (N/b_i) = N^r / |R|$ $r$-tuples $\mathbf{x^{(1)}} = (x_1^{(1)},\ldots, x_r^{(1)}) \in \Z_{N}^r$ such that $s = \sum_{i=1}^r x_i^{(1)}g_i$. Thus, $r$-tuples over $\Z_N$ are in $N^r/|R|$-to-1 correspondence with elements of $R$, which justifies the equation
		\begin{equation}\label{lpfcharestimate eqn1}
			\left| \rcavg{h_1,\ldots,h_m}\chi\left( \prod_{j=1}^m h_j \right) \right| \ = \ \left| \cavg{\mathbf{x^{(1)}},\ldots,\mathbf{x^{(m)}}}{\Z_N^r} \chi\left( \prod_{j=1}^m \sum_{i=1}^r x^{(j)}_i g_i  \right) \right|.
		\end{equation}
		Let us now compute the right-hand side.
		
		Recall that, given a positive integer $n$, we use the shorthand $e_n(x) := e^{2\pi i x/n}$ for any real $x$. 
		The group of additive characters of $R$ is isomorphic to $(R,+)$. Since $\chi \in \hat{R}$ is nontrivial, it follows that there exist $a_i \in \Z_{b_i}$, $i \in \{1,\ldots, r\}$, not all zero, such that $\chi(\sum_{i=1}^r x_i g_i)  \ = \ \prod_{i=1}^r e_{b_i}(a_ix_i) \ = \ e_N\left( \sum_{i=1}^r \frac{N}{b_i} a_i x_i \right)$ for each $(x_1,\ldots, x_r) \in \Z_N^r$.
		
		To proceed, we will need an explicit expression for the argument of $\chi$ in \eqref{lpfcharestimate eqn1}. Since $\chi$ is an additive character, not a multiplicative one, in order to understand the product inside, we will express $g_{i_1}g_{i_2}$ as a linear combination of $g_i$. Thus, for each $i,j \in \{1,\ldots, r\}$, let $c_k^{(i,j)} \in \Z_{b_k}$, $k \in \{1,\ldots, r\}$, be such that $g_ig_j = \sum_{k=1}^r c_k^{(i,j)} g_k$. Then, for any $i,k \in \{1,\ldots, r\}$, we observe that, since $g_r = 1$,
		\begin{equation}\label{eqn explaining multiplication law}
			c_{k}^{(i,r)} \ = \ c_k^{(r,i)} \ = \ \begin{cases} 1 & \text{ if } i = k, \\ 0 & \text{ otherwise.} \end{cases}
		\end{equation}
		Thus, recalling the auxiliary functions $\eta_m$, $m\geq 2$, as defined in \eqref{definition of auxiliary eta_m}, we write
		\begin{align*}
			\prod_{j=1}^m \sum_{i=1}^r x^{(j)}_i g_i \ & = \ \left( \sum_{i_1=1}^r x_{i_1}^{(1)} g_{i_1} \right)\left( \sum_{i_2=1}^r x_{i_2}^{(2)} g_{i_2} \right) \cdots \left( \sum_{i_m=1}^r x_{i_m}^{(m)} g_{i_m} \right) \\
			& = \ \left( \sum_{i_1,i_2=1}^r x_{i_1}^{(1)}x_{i_2}^{(2)} g_{i_1}g_{i_2}\right) \cdots \left( \sum_{i_m=1}^r x_{i_m}^{(m)} g_{i_m} \right) \\
			& = \ \left( \sum_{i_1,i_2=1}^r x_{i_1}^{(1)}x_{i_2}^{(2)} \left( \sum_{k_1=1}^r c_{k_1}^{(i_1,i_2)} g_{k_1}\right) \right) \cdots \left( \sum_{i_m=1}^r x_{i_m}^{(m)} g_{i_m} \right) \\
			& = \ \left( \sum_{i_1,i_2,i_3=1}^r x_{i_1}^{(1)}x_{i_2}^{(2)}x_{i_3}^{(3)}  \sum_{k_1=1}^r c_{k_1}^{(i_1,i_2)} g_{k_1}g_{i_3} \right) \cdots \left( \sum_{i_m=1}^r x_{i_m}^{(m)} g_{i_m} \right) \\
			& = \ \left( \sum_{i_1,i_2,i_3=1}^r x_{i_1}^{(1)}x_{i_2}^{(2)}x_{i_3}^{(3)}  \sum_{k_1,k_2=1}^r c_{k_1}^{(i_1,i_2)}c_{k_2}^{(k_1,i_3)} g_{k_2} \right) \cdots \left( \sum_{i_m=1}^r x_{i_m}^{(m)} g_{i_m} \right),
		\end{align*}
		and, continuing in this way, we obtain
		\begin{align*}		
			\prod_{j=1}^m \sum_{i=1}^r x^{(j)}_i g_i
			\ & = \ \sum_{i_1,\ldots, i_m = 1}^r x_{i_1}^{(1)} \cdots x_{i_m}^{(m)} \sum_{k_1,\ldots, k_{m-1} = 1}^r c_{k_1}^{(i_1,i_2)}c_{k_2}^{(k_1,i_3)} \cdots c_{k_{m-1}}^{(k_{m-2},i_m)} g_{k_{m-1}} \\
			& = \ \sum_{k_{m-1}=1}^r \left( \sum_{i_1,\ldots, i_m = 1}^r x_{i_1}^{(1)} \cdots x_{i_m}^{(m)} \sum_{k_1,\ldots, k_{m-2} = 1}^r \eta_m(k_1,\ldots,k_{m-1};i_1,\ldots,i_m) \right) g_{k_{m-1}} \\
			& = \ \sum_{i=1}^r \left( \sum_{i_1,\ldots, i_m = 1}^r x_{i_1}^{(1)} \cdots x_{i_m}^{(m)} \sum_{k_1,\ldots, k_{m-2} = 1}^r \eta_m(k_1,\ldots,k_{m-2},i;i_1,\ldots,i_m) \right) g_{i}.
		\end{align*}
		Hence, for each $\mathbf{x^{(2)}}, \ldots, \mathbf{x^{(m)}} \in \Z_N^r$, abbreviating the expression 
		\begin{equation}
			F(\mathbf{x^{(2)}},\ldots, \mathbf{x^{(m)}};j) \ := \ 
			\sum_{i = 1}^r \frac{N}{b_i}a_i \sum_{i_2,\ldots, i_m = 1}^r x_{i_2}^{(2)} \cdots x_{i_m}^{(m)} \sum_{k_1,\ldots, k_{m-2} = 1}^r \eta_m(k_1,\ldots,k_{m-2},i;j,i_2,\ldots,i_m),
		\end{equation}
		we obtain
		\begin{align*}
			& \cavg{\mathbf{x^{(1)}}}{\Z_N^r} \chi\left( \prod_{j=1}^m \sum_{i=1}^r x^{(j)}_i g_i  \right) \\ & = \ \cavg{\mathbf{x^{(1)}}}{\Z_N^r} e_N\left( \sum_{i = 1}^r \frac{N}{b_i}a_i \left( \sum_{i_1,\ldots, i_m = 1}^r x_{i_1}^{(1)} \cdots x_{i_m}^{(m)} \sum_{k_1,\ldots, k_{m-2} = 1}^r \eta_m(k_1,\ldots,k_{m-2},i;i_1,\ldots,i_m) \right) \right) \\
			& = \ \cavg{x_1^{(1)},\ldots,x_r^{(1)}}{\Z_N} \prod_{i_1=1}^r e_N\left( x_{i_1}^{(1)} F(\mathbf{x^{(2)}},\ldots, \mathbf{x^{(m)}};i_1) \right) \\
			& = \ \prod_{j=1}^r  \cavg{x_j^{(1)}}{\Z_N} e_N\left( x_{j}^{(1)} F(\mathbf{x^{(2)}},\ldots, \mathbf{x^{(m)}};j) \right) \\
			& = \ \begin{cases} 1  \quad \text{ if } F(\mathbf{x^{(2)}},\ldots, \mathbf{x^{(m)}};j) \equiv 0 \bmod N \text{ for every } j \in \{1,\ldots, r\}, \\ 0 \quad \text{ otherwise.} \end{cases}
		\end{align*}
		It follows that
		\begin{align*}
			& \left| \cavg{\mathbf{x^{(1)}},\ldots,\mathbf{x^{(m)}}}{\Z_N^r} \chi\left( \prod_{j=1}^m \sum_{i=1}^r x^{(j)}_i g_i  \right)\right| \ = \ \left|\cavg{\mathbf{x^{(2)}},\ldots,\mathbf{x^{(m)}}}{\Z_N^r} \cavg{\mathbf{x^{(1)}}}{\Z_N^r}\chi\left( \prod_{j=1}^m \sum_{i=1}^r x^{(j)}_i g_i  \right)\right|
			\\ & = \ \frac{1}{N^{r(m-1)}}\# \Big\{ (\mathbf{x^{(2)}},\ldots,\mathbf{x^{(m)}}) \in \Z_N^{r(m-1)} :  \forall j \in \{1,\ldots, r\}, F(\mathbf{x^{(2)}},\ldots, \mathbf{x^{(m)}};j) \equiv 0 \bmod N \Big\} \\
			& \leq \ \frac{1}{N^{r(m-1)}}\# \Big\{ (\mathbf{x^{(2)}},\ldots,\mathbf{x^{(m)}}) \in \Z_N^{r(m-1)} : F(\mathbf{x^{(2)}},\ldots, \mathbf{x^{(m)}};r) \equiv 0 \bmod N \Big\} \\
			& \leq \ \frac{m-1}{\lpf N}
		\end{align*}
		by Proposition~\ref{prop trivial estimate on linear system of equations mod N}. 
	\end{proof}
	
	
	We need one more intermediate result, which makes use of van der Corput-style differencing. To this end, we need to define the following notation. Let $R_1,\ldots, R_n$ be finite commutative rings, and let $f : R_1 \times \dots \times R_n \to \C$ be a function. Given $i \in \{1,\ldots, n\}$ and $h_i \in R_i$, we define the discrete derivative $\Delta^{(i)}_{h_i} f : R_1 \times \cdots \times R_n \to \C$ by $\Delta^{(i)}_{h_i} f(y_1,\ldots, y_n) := f(y_1,\ldots, y_{i-1},y_i+h_i,y_{i+1},\ldots, y_n) \ol{f}(y_1,\ldots, y_n)$. For $i,i' \in \{1,\ldots, n\}$ and $h_i \in R_i$ and $h_{i'} \in R_{i'}$, note that $\Delta^{(i)}_{h_i}\Delta^{(i')}_{h_{i'}} f = \Delta^{(i')}_{h_{i'}}\Delta^{(i)}_{h_{i}} f$.
	
	\begin{proposition}\label{prop: distinct powers character estimate}
		Let $R$ be a finite commutative ring with characteristic $N$. Let $n \geq 1$ be an integer. Let $S = \{\alpha_1,\ldots, \alpha_{|S|}\}$ be a finite subset of $(\N \cup \{0\})^n \setminus\{(0,\ldots,0)\}$, and write $\alpha_k = (\alpha^{(k)}_1,\ldots, \alpha^{(k)}_n)$ for each $k \in \{1,\ldots, |S|\}$. Let $\chi_1,\ldots, \chi_{|S|} \in \hat{R}$ be nontrivial additive characters. Let $d$ be a positive integer such that $d < \lpf N$ and $d \geq \max_{k \in \{1,\ldots, |S|\}} \sum_{i=1}^n \alpha^{(k)}_i$. Then
		\begin{equation*}
			\left\vert \cavg{y}{R^n} \prod_{k=1}^{|S|} \chi_{k}(y^{\alpha_k})\right\vert \ \leq \ \left(\frac{d-1}{\lpf N}\right)^{2^{-d}}. 
		\end{equation*}
	\end{proposition}
	\begin{proof}
		We claim for every integer $m \geq 1$, function $f : R^n \to \C$, and indices $i_1,\ldots, i_m \in \{1,\ldots, n\}$ that
		\begin{equation*}
			\left\vert \cavg{y}{R^n} f(y) \right\vert^{2^{m}} \ \leq \  \cavg{y}{R^n} \rcavg{h_1,\ldots,h_{m}} \Delta_{h_m}^{(i_m)} \cdots \Delta_{h_1}^{(i_1)}f(y).
		\end{equation*}
		This follows by induction on $m$. Indeed, suppose $m = 1$. Then, after applying Cauchy--Schwarz and changing variables $(y_{i_1},y_{i_1}') \in R^2$ to $(y_{i_1}+h_{1},y_{i_1}) \in R^2$, we observe that
		\begin{align*}
			\left\vert \cavg{y}{R^n} f(y) \right\vert^{2} \ & \leq \ \cavg{y_i}{R, \ i \neq i_1} \left\vert \rcavg{y_{i_1}} f(y) \right\vert^2 \\
			& = \ \cavg{y_i}{R, \ i \neq i_1} \rcavg{y_{i_1},y_{i_1}'} f(y_1,\ldots, y_{i_1-1},y_{i_1},y_{i_1+1},\ldots, y_n)\ol{f}(y_1,\ldots, y_{i_1-1},y_{i_1}',y_{i_1+1},\ldots, y_n)\\
			& = \ \cavg{y}{R^n} \rcavg{h_1} \Delta^{(i_1)}_{h_1} f(y).
		\end{align*}
		Suppose the result holds for $m \geq 1$. We now show it holds for $m+1$. First observe that
		\begin{align*}
			\left\vert \rcavg{y} f(y) \right\vert^{2^{m+1}} \ & \leq \ \left\vert \cavg{y}{R^n} \rcavg{h_1,\ldots,h_{m}} \Delta_{h_m}^{(i_m)} \cdots \Delta_{h_1}^{(i_1)}f(y) \right\vert^2\\
			& \ \leq \cavg{y_i}{R, \ i \neq i_{m+1}} \rcavg{h_1,\ldots,h_m} \left| \rcavg{y_{i_{m+1}}}  \Delta_{h_m}^{(i_m)} \cdots \Delta_{h_1}^{(i_1)}f(y)\right|^2
		\end{align*}
		by, respectively, the induction hypothesis and Cauchy--Schwarz. To finish the induction, we reuse the change of variables trick from the base case:
		\begin{align*}
			\left| \rcavg{y_{i_{m+1}}}  \Delta_{h_m}^{(i_m)} \cdots \Delta_{h_1}^{(i_1)}f(y)\right|^2 \ & = \ \rcavg{y_{i_{m+1}},h_{m+1}}\Delta_{h_{m+1}}^{(i_{m+1})} \Delta_{h_m}^{(i_m)} \cdots \Delta_{h_1}^{(i_1)}f(y).
		\end{align*}
		In preparation to apply the claim, we need to define a few parameters first. Let 
		\begin{align*}
			d_1 \ & := \ \max\{\alpha_1^{(k)} : k \in \{1,\ldots, |S| \} \}, \\
			d_2 \ & := \ \max\{\alpha_2^{(k)} : k \in \{1,\ldots, |S| \} \text{ such that } \alpha_1^{(k)} = d_1 \}, \\
			d_3 \ & := \ \max\{\alpha_3^{(k)} : k \in \{1,\ldots, |S| \} \text{ such that } \alpha_1^{(k)} = d_1, \alpha_2^{(k)} = d_2 \}, \\
			& \vdots \\
			d_n \ & := \ \max\{\alpha_n^{(k)} : k \in \{1,\ldots, |S| \} \text{ such that } \alpha_1^{(k)} = d_1, \ldots, \alpha_{n-1}^{(k)} = d_{n-1} \}.
		\end{align*}
		Then each $d_i \leq d$.
		Applying the claim to $m = \sum_{i=1}^n d_i$, the function $f(y) = \prod_{k = 1}^{|S|} \chi_k(y^{\alpha_k})$, and 
		\begin{align*} i_1 \ = \ \cdots \ = \ i_{d_1} \ & = \ 1, \\
			i_{d_1+1} \ = \ \cdots \ = \ i_{d_1+d_2} \ & = \ 2, \\
			i_{d_1+d_2+1} \  = \ \cdots \ = \ i_{d_1+d_2+d_3} \ & = \ 3, \\
			& \vdots \\
			i_{d_1+\dots + d_{n-1}+1} \ = \ \cdots \ = \ i_{m} \ & = \ n,
		\end{align*}
		we obtain
		\begin{multline}\label{label me 1}
			\left\vert \cavg{y}{R^n} \prod_{k = 1}^{|S|} \chi_k(y^{\alpha_k})\right\vert^{2^{m}} \ \leq \ \cavg{y}{R^n} \rcavg{h_1,\ldots,h_{m}} \Delta_{h_m}^{(i_m)} \cdots \Delta_{h_1}^{(i_1)}\prod_{k = 1}^{|S|} \chi_k(y^{\alpha_k}) \\ = \ \cavg{y}{R^n} \rcavg{h_1,\ldots,h_{m}} \prod_{k = 1}^{|S|} \Delta_{h_m}^{(i_m)} \cdots \Delta_{h_1}^{(i_1)}\chi_k(y^{\alpha_k}).
		\end{multline}
		Let us simplify the right-hand side of \eqref{label me 1}, step by step. For each $k$, we first claim that
		\begin{equation*}\Delta_{h_{d_1}}^{(i_{d_1})} \cdots \Delta_{h_1}^{(i_1)}\chi_k(y^{\alpha_k}) \ = \ \begin{cases} \chi_{k}\left((d_1!h_1\cdots h_{d_1})y_2^{\alpha^{(k)}_2}\cdots y_n^{\alpha_n^{(k)}}\right) & \text{ if } \alpha_1^{(k)} = d_1, \\
				1 & \text{ otherwise.} \end{cases}
		\end{equation*}
		Indeed, if $P(y) \in R[y_1,\ldots, y_n]$ is a polynomial and $\chi \in \hat{R}$ is an additive character, then $\Delta_{h}^{(1)}\chi(P(y)) = \chi(P(y_1+h,y_2,\ldots, y_n)-P(y_1,\ldots, y_n))$. Recall that $i_1 = \cdots = i_{d_1} = 1$. Thus, the sequence of operators $\Delta_{h_{d_1}}^{(i_{d_1})} \cdots \Delta_{h_1}^{(i_1)} = \Delta_{h_{d_1}}^{(1)} \cdots \Delta_{h_1}^{(1)}$ takes $d_1$-many finite differences of just the $y_1$ argument of the $y^{\alpha_{k}}$, which will obviously annihilate any $y^{\alpha_k} = y_1^{\alpha_1^{(k)}}(y_2^{\alpha_2^{(k)}} \cdots y_n^{\alpha_n^{(k)}})$ such that $\alpha_1^{(k)} < d_1$; when $\alpha_1^{(k)} = d_1$, it is easy to check that the claimed formula holds.
		
		Second, for $k$ such that $\alpha_1^{(k)} = d_1$, we similarly observe that
		\begin{multline*}\Delta_{h_{d_1+d_2}}^{(i_{d_1+d_2})} \cdots \Delta_{h_{d_1+1}}^{(i_{d_1+1})}\chi_{k}\left((d_1!h_1\cdots h_{d_1})y_2^{\alpha^{(k)}_2}\cdots y_n^{\alpha_n^{(k)}}\right) \\ = \ \begin{cases} \chi_{k}\left((d_1!h_1\cdots h_{d_1})(d_2!h_{d_1+1}\cdots h_{d_1+d_2})y_3^{\alpha^{(k)}_3}\cdots y_n^{\alpha_n^{(k)}}\right) & \text{ if } \alpha_2^{(k)} = d_2, \\
				1 & \text{ otherwise.} \end{cases}
		\end{multline*}
		We proceed in this way until we have exhausted the variables $y_1,\ldots, y_n$. In the process, we trivialize all but one character; indeed, there is exactly one $\alpha \in S$, say $\alpha_{k_0}$, such that $\alpha_i^{(k_0)} = d_i$ for each $i$. Thus, we eventually conclude that 
		\begin{multline}
			\left\vert \cavg{y}{R^n} \prod_{k = 1}^{|S|} \chi_k(y^{\alpha_k})\right\vert^{2^{m}} \ \leq \
			\cavg{y}{R^n} \rcavg{h_1,\ldots,h_{m}} \chi_{k_0}\left( \left(\prod_{i=1}^n d_i!\right) \left(\prod_{j=1}^m h_j\right)\right) \\ = \ \rcavg{h_1,\ldots,h_{m}} \chi_{k_0}\left( \left(\prod_{i=1}^n d_i!\right) \left(\prod_{j=1}^m h_j\right)\right) \ = \ \rcavg{h_1,\ldots,h_{m}} \chi_{k_0}\left( \prod_{j=1}^m h_j\right),
		\end{multline}
		where the change of variables $h_1 \mapsto h_1 \left(\prod_{i=1}^n d_i!\right)^{-1}$ used in the last equality is justified by the following fact: Each $d_i \leq d < \lpf N$, so $d_i!$ is a unit in $R$ by Lemma~\ref{lem: what is a unit}.
		
		By Proposition~\ref{prop: multiplicative character estimate}, we conclude that
		\begin{equation*}
			\left\vert \cavg{y}{R^n} \prod_{k = 1}^{|S|} \chi_k(y^{\alpha_k})\right\vert \ \leq \ \left( \frac{m-1}{\lpf{N}} \right)^{2^{-m}}.
		\end{equation*}
		The function $g: [1,\infty) \to \R$ given by $g(x) = ( \frac{x-1}{\lpf N} )^{2^{-x}}$ is increasing on $(1, \lpf N + 1)$. By assumption, $1 \leq m \leq d < \lpf N$, so $g(m) \leq g(d)$, which proves the proposition.
	\end{proof}
	
	We now state and prove the important lemma in this section.
	\begin{lemma}\label{lem: multicharacter product of independent polynomials as arguments}
		Let $m, n \geq 1$ be integers. Let $\mathbf P = \{Q_1,\ldots,Q_m\} \subset \Z[y_1,\ldots, y_n]$ be an independent family of polynomials with zero constant term. Let $d = \max_{1\leq j \leq m} \deg(Q_j)$. Then there exists $C=C(\mathbf P)$ with the following property. Namely, for any finite commutative ring $R$ with characteristic $N$ such that $\lpf N > C$ and any characters $\psi_1,\ldots,\psi_m \in \hat{R}$, not all of which are trivial, one has
		\be
		\left| \cavg{y}{R^n} \prod_{j=1}^m \psi_j(Q_j(y)) \right| \ \leq \ \left(\frac{d-1}{\lpf N} \right)^{2^{-d}}.
		\ee
	\end{lemma}
	\begin{proof}
		By Proposition~\ref{prop: lindep over zn}, there exists $C_1$ depending on $\mathbf P$ such that whenever $M$ is a positive integer with $\lpf M > C_1$, the family $\{Q_1,\ldots,Q_m\}$, viewed as a family of polynomials in $\mathbb{Z}_M[y_1,\ldots, y_n]$, is linearly independent over $\mathbb{Z}_M$.
		
		Set $C := \max \{ 2,d,C_1\}$, and suppose that $R$ is a finite commutative ring with characteristic $N$ such that $\lpf{N} > C$.
		
		By Lemma~\ref{lem: structure of (R,+)}, $(R,+)$ is isomorphic to $\Z_{b_1} \times \cdots \times \Z_{b_{r-1}} \times \Z_{b_r}$, where $b_r = N$ and $b_i > 1$, $i < r$, are possibly indistinct prime powers dividing $N$. Then, endowing $\hat{R}$ with pointwise multiplication, we obtain the character group of $R$, which is isomorphic to the same product of cyclic groups.
		
		Let $\psi_1,\ldots, \psi_m$ be characters of $R$, not all of which are trivial. For each $j \in \{1,\ldots, m\}$, write $\psi_j$ as $(a^{(j)}_1,\ldots, a^{(j)}_r) \in \Z_{b_1} \times \cdots \times \Z_{b_r}$. In other words, $\psi_j((x_1,\ldots, x_r)) = \prod_{i=1}^r e_{b_i}(a_i^{(j)}x_i)$.
		
		Let $S'$ be the set of $\alpha = (\alpha_1,\ldots, \alpha_n) \in (\N \cup \{0\})^n$ such that there exists $j \in \{1,\ldots, m\}$ such that the coefficient of $y^\alpha$ in $P_j(y)$ is nonzero.
		
		For each $j \in \{1,\ldots, m\}$, write $Q_j(y) = \sum_{\alpha \in S'} c_\alpha^{(j)} y^\alpha$, where $c_\alpha^{(j)} \in \Z$, noting that some $c_\alpha^{(j)}$ could be zero. Since characters are additive homomorphisms to $\C$, we rewrite
		\begin{equation*}
			\prod_{j=1}^m \psi_i(Q_i(y)) \ = \ \prod_{\alpha \in S'} \left( \psi_1^{c^{(1)}_\alpha}\cdots \psi_m^{c^{(m)}_\alpha}\right)(y^\alpha).
		\end{equation*}
		Next, we claim that there exists $\alpha \in S'$ such that $\psi_1^{c^{(1)}_\alpha}\cdots \psi_m^{c^{(m)}_\alpha}$ is not the trivial character. Indeed, suppose not. It would then follow that, for each $\alpha \in S'$ and $i \in \{1,\ldots, r\}$, 
		\begin{equation}\label{eqn: need to label 1}
			\sum_{j=1}^m c_\alpha^{(j)}a_i^{(j)} \ \equiv \ 0 \bmod b_i.
		\end{equation}
		Enumerate $S' = \{\alpha_1,\ldots, \alpha_{|S'|}\}$. Now, if we form the $|S'|\times m$ matrix $A$ of coefficients of $P_1,\ldots,P_m$ such that within each row, there exists $k \in \{1,\ldots, S'\}$ such that the $j$th entry in the row is the coefficient of $y^{\alpha_k}$ in $P_j$, then rewriting \eqref{eqn: need to label 1} in terms of matrix $A$, we observe that for each $i \in \{1,\ldots, r\}$,
		\begin{equation*}
			\underbrace{\begin{pmatrix}
					c_{\alpha_{|S'|}}^{(1)} & \cdots & c_{\alpha_{|S'|}}^{(m)} \\
					\vdots & \ddots & \vdots \\
					c_{\alpha_1}^{(1)} & \cdots & c_{\alpha_1}^{(m)} \end{pmatrix}}_{A} \begin{pmatrix} a_i^{(1)} \\ \vdots \\ a_i^{(m)} \end{pmatrix} \ \equiv \ \begin{pmatrix} 0 \\ \vdots \\ 0 \end{pmatrix} \bmod b_i.
		\end{equation*}
		By assumption, not all $\psi_1,\ldots, \psi_m$ are trivial. Hence, for at least one $i \in \{1,\ldots, r\}$, the vector $(a_i^{(1)},\ldots, a_i^{(m)})^T$ is nonzero. However, this is a contradiction since the matrix $A$ has $\Z_{b_i}$-linearly independent columns, which follows from the fact that $\lpf{b_i} \geq \lpf{N} > C \geq C_1$.
		
		Thus, the claim is proved. Let $S$ be the set of $\alpha \in S'$ such that $\psi_1^{c^{(1)}_\alpha}\cdots \psi_m^{c^{(m)}_\alpha}$ is not the trivial character. The claim showed that $S$ is nonempty. Moreover, since each $Q_j$ has constant term zero, $(0,\ldots, 0) \not\in S'$, so $S$ does not contain that vector either. Re-enumerating $S'$ if necessary, we have $S = \{\alpha_1,\ldots, \alpha_{|S|}\}$. For each $k \in \{1,\ldots, |S|\}$, define the additive character $\chi_k := \psi_1^{c_{\alpha_k}^{(1)}}\cdots \psi_m^{c_{\alpha_k}^{(m)}}$. Then
		\begin{equation*}
			\left| \cavg{y}{R^n} \prod_{j=1}^m \psi_i(Q_i(y)) \right| \ = \ \left| \cavg{y}{R^n} \prod_{k=1}^{|S|} \chi_k(y^{\alpha_k}) \right|.
		\end{equation*}
		Write $\alpha_k = (\alpha^{(k)}_1,\ldots, \alpha^{(k)}_n)$ for each $k$. Observe that $d = \max_{1\leq j\leq m} \deg(Q_j)$ satisfies \\$\max_{k \in \{1,\ldots, |S|\}} \sum_{i=1}^n \alpha^{(k)}_i \leq d \leq C < \lpf N$, so by Proposition~\ref{prop: distinct powers character estimate}, it follows that
		\begin{equation*}
			\left| \cavg{y}{R^n} \prod_{k=1}^{|S|} \chi_k(y^{\alpha_k}) \right| \ \leq \ \left(\frac{d-1}{\lpf N} \right)^{2^{-d}},
		\end{equation*}
		as desired.
	\end{proof}
	Finally, we record a consequence of the previous lemma, which we will use to derive a corollary of our main theorem. This result is a bound on the number of roots of an integer polynomial over a general finite commutative ring.
	\begin{proposition}\label{prop: bound on number of roots}
		Let $n \geq 1$ be a positive integer, and let $P(y) \in \Z[y_1,\ldots, y_n]$ have degree at least one. There exist $c, C > 0$ and $\ve \in (0,1)$, each depending only on $P$, such that the following holds. For any finite commutative ring $R$ with characteristic $N$ such that $\lpf N > C$, the number of $y = (y_1,\ldots, y_n) \in R^n$ such that $P(y) = 0_R$ is at most $|R|^{n-1} + \frac{c|R|^n}{\lpf{N}^\ve}$.
	\end{proposition}
	\begin{proof}
		Applying Lemma~\ref{lem: multicharacter product of independent polynomials as arguments} with $\mathbf P = \{P\}$, we see that there exists $C = C(P)$ such that, for any finite commutative ring $R$ with characteristic $N$ such that $\lpf N > C$ and any nontrivial additive character $\psi \in \hat{R}$, one has
		\begin{equation}
			\left| \cavg{y}{R^n} \psi(P(y)) \right| \ \leq \ \left( \frac{d-1}{\lpf N} \right)^{2^{-d}},
		\end{equation}
		where $d = \deg(P)$.
		
		Let $R$ be a finite commutative ring with characteristic $N$ such that $\lpf N > C$. Set $\ve = 2^{-d}$ and $c = (d-1)^{2^{-d}}$. 
		
		By orthonormality of characters, we observe that the number of zeros of $P(y)$ in $R^n$ is given by
		\begin{equation}
			\sum_{y \in R^n} \cavg{\psi}{\hat{R}} \psi(P(y)),
		\end{equation}
		whose magnitude is at most
		\begin{align*}
			\left| \cavg{\psi}{\hat{R}} \sum_{y \in R^n} \psi(P(y)) \right| \ & \leq \ |R|^{n-1} + |R|^{n-1} \sum_{\psi \in \hat{R}, \psi \neq 1} \left| \cavg{y}{R^n} \psi(P(y)) \right| \\
			& \leq \ |R|^{n-1}\left(1 + \sum_{\psi \in \hat{R}, \psi \neq 1} \left(\frac{d-1}{\lpf N}\right)^{2^{-d}}\right) \\
			& = \ |R|^{n-1}\left(1 +  (|R|-1)\left(\frac{d-1}{\lpf N}\right)^{2^{-d}}\right) \\
			& \leq \ |R|^{n-1} + \frac{c|R|^n}{\lpf{N}^\ve}.
		\end{align*}
	\end{proof}
	
	\subsection{Proof of Proposition~\ref{main thm base}}\label{subsec: proof of prop main thm base}
	Using Fourier analysis, we can upgrade Lemma~\ref{lem: multicharacter product of independent polynomials as arguments} to prove Proposition~\ref{main thm base}, which we restate here for convenience.
	\begin{proprep}[\ref{main thm base}]
		Let $m_2 \geq 0$ and $n \geq 1$ be integers, and let $\mathbf P = \{ P_1,Q_1,\ldots,Q_{m_2}\} \subset \Z[y_1,\ldots, y_n]$ be an independent family of polynomials with zero constant term. There exist $c,C,\gamma > 0$ such that the following holds. For any finite commutative ring $R$ with characteristic $N$ such that $\lpf N > C$, any vector $F = (f_0,f_1)$ of 1-bounded functions $R \to \C$, and any vector $\Psi = (\psi_1,\ldots,\psi_{m_2})$ of additive characters of $R$, one has
		\be
		\left| \Lambda_{P_1}^{Q_1,\ldots,Q_{m_2}}(F;\Psi) - 1_{\Psi=1} \rcavg x f_0(x) \rcavg x f_1(x)\right| \ \leq \  c\,\lpf{N}^{-\gamma},
		\ee
		where $1_{\Psi=1}$ equals 1 if every character in $\Psi$ is trivial and 0 otherwise.
	\end{proprep}
	\begin{proof} Let $d = \max\{\deg(P_1),  \max_{1\leq j \leq m_2} \deg(Q_j)\}$.
		Applying Lemma~\ref{lem: multicharacter product of independent polynomials as arguments}, let $C = C(\mathbf P)$ be such that, for any finite commutative ring $R$ with characteristic $N$ such that $\lpf N > C$ and any additive characters $\phi,\psi_1,\ldots,\psi_{m_2}$ of $R$, not all of which are trivial, we have
		\be
		\left| \cavg{y}{R^n} \phi(P_1(y)) \prod_{i=1}^{m_2} \psi_i(Q_i(y)) \right| \ \leq \ \left(\frac{d-1}{\lpf N} \right)^{2^{-d}}.
		\ee
		
		Fix a finite commutative ring $R$ with characteristic $N$ satisfying $\lpf N > C$ and vectors $F$ and $\Psi$ as in the statement of this proposition, and let $f_1' = f_1 - \rcavg x f_1(x)$ and $F' := (f_0,f_1')$. By splitting $f_1 = f_1' + \rcavg x f_1(x)$, we observe that
		\begin{equation}\label{one-time-use eqn 1}
			\Lambda_{P_1}^{Q_1,\ldots, Q_{m_2}}(F;\Psi) \ = \ \left(\rcavg x f_0(x) \rcavg{x} f_1(x) \right) \cavg{y}{R^n} \prod_{i=1}^{m_2} \psi_i(Q_i(y)) + \Lambda_{P_1}^{Q_1,\ldots, Q_{m_2}}(F';\Psi).
		\end{equation}
		
		First, we claim that
		\be
		\left| \left( \rcavg{x} f_0(x) \rcavg{x} f_1(x) \right) \cavg{y}{R^n} \prod_{i=1}^{m_2} \psi_i(Q_i(y)) - 1_{\Psi=1} \rcavg{x} f_0(x) \rcavg{x} f_1(x) \right| \ \leq \ \left(\frac{d-1}{\lpf N} \right)^{2^{-d}}.
		\ee
		Indeed, if $\psi_i$ is trivial for every $i \in \{1,\ldots, m_2\}$, then $\cavg{y}{R^n} \prod_{i=1}^{m_2} \psi_i(Q_i(y)) = 1$ and $1_{\Psi=1} = 1$; otherwise, $1_{\Psi=1} = 0$ and $\left| \cavg{y}{R^n} \prod_{i=1}^{m_2} \psi_i(Q_i(y)) \right| \leq \left(\frac{d-1}{\lpf N} \right)^{2^{-d}}$ by our choice of $C$. Thus
		\be
		\left|\rcavg{y} \prod_{i=1}^{m_2} \psi_i(Q_i(y)) - 1_{\Psi=1} \right| \ \leq \ \left(\frac{d-1}{\lpf N} \right)^{2^{-d}},
		\ee
		which implies the claim since $f_0$ and $f_1$ are 1-bounded.
		
		Second, we claim that
		\be
		\left| \Lambda_{P_1}^{Q_1,\ldots, Q_{m_2}}(F';\Psi) \right| \ \leq \ 2\cdot  \left(\frac{d-1}{\lpf N} \right)^{2^{-d}}.
		\ee
		Indeed, by Fourier inversion, we can compute that
		\begin{align*}
			\Lambda_{P_1}^{Q_1,\ldots, Q_{m_2}}(F';\Psi) \
			& = \ \sum_{\phi_0,\phi_1 \in \widehat{R}} \hat{f_0}(\phi_0) \hat{f_1'}(\phi_1) \left( \rcavg x \phi_0(x) \phi_1(x) \right) \left(\cavg{y}{R^n} \ol{\phi_1}(P_1(y)) \prod_{i=1}^{m_2} \psi_i(Q_i(y)) \right) \\
			& = \ \sum_{\phi \in \widehat{R}} \hat{f_0}(\phi) \hat{f_1'}(\ol{\phi}) \left(\cavg{y}{R^n} \phi(P_1(y)) \prod_{i=1}^{m_2} \psi_i(Q_i(y)) \right) \\
			& = \ \sum_{1 \neq \phi \in \widehat{R}} \hat{f_0}(\phi) \hat{f_1'}(\ol{\phi}) \left(\cavg{y}{R^n} \phi(P_1(y)) \prod_{i=1}^{m_2} \psi_i(Q_i(y)) \right),
		\end{align*}
		where the first equality holds by the change of variables $x \mapsto -x$ in $R$, the second equality holds by orthogonality relations, and the last equality holds since $\hat{f_1'}(1) = \rcavg x f_1'(x) = 0$. By our choice of $C$, Cauchy--Schwarz, Plancherel's theorem, and the fact that $f_0$ is 1-bounded and $f_1'$ is 2-bounded, we conclude that
		\ba
		\left| \Lambda_{P_1}^{Q_1,\ldots, Q_{m_2}}(F';\Psi) \right| \ & \leq \ \left(\frac{d-1}{\lpf N} \right)^{2^{-d}} \sum_{1 \neq \phi \in \widehat{R}} \left| \hat{f_0}(\phi) \hat{f_1'}(\ol{\phi})\right| \\
		& \leq \ \left(\frac{d-1}{\lpf N} \right)^{2^{-d}} |R| \lr{\hat{f_0},\hat{f_0}}^{1/2} \lr{\hat{f_1'},\hat{f_1'}}^{1/2} \\
		& = \ \left(\frac{d-1}{\lpf N} \right)^{2^{-d}} \norm{f_0}_{L^2} \norm{f_1'}_{L^2} \\
		& \leq \ 2\cdot  \left(\frac{d-1}{\lpf N} \right)^{2^{-d}}.
		\ea
		Combining the two claims and the equality \eqref{one-time-use eqn 1} yields the result.
	\end{proof}
	
	
	
	\section{The induction step of Theorem~\ref{main thm}}\label{sec: induction step}

	This section concerns the induction step in the proof of Theorem~\ref{main thm}. As a reminder, we are inducting on $m_1$, and the case $m_1 = 1$ has already been shown in Proposition~\ref{main thm base}.
	
	For a sketch of the argument, the reader may consult the part of Subsection~\ref{subsec: methods} that begins after \eqref{pointer 1} and continues until the sentence after \eqref{pointer 2}.
	
	
	\subsection{Approximate control by some Gowers $U^s$ norm}\label{subsec: bounding in terms of Us}
	Let $R$ be a finite commutative ring with characteristic $N$. One of the first steps we will take in the proof of Theorem~\ref{main thm} is to bound, for 1-bounded functions $f_0,\ldots, f_{m_1} : R \to \C$ and additive characters $\psi_1,\ldots,\psi_{m_2} \in \widehat{R}$, the average
	\[\Lambda_{P_1,\ldots,P_{m_1}}^{Q_1,\ldots,Q_{m_2}}(f_0,\ldots,f_{m_1};\psi_1,\ldots, \psi_{m_2})\]
	in terms of the Gowers norm (of potentially high degree) of the constituent $f_i$ and an error term of size a negative power of $\lpf N$. To this end, we will prove an intermediate proposition (Proposition~\ref{prop: Us control}) in Section~\ref{sec: proof of prop: Us control}. 
	\begin{prop}\label{prop: Us control}
		Let $m \geq 2$ and $n \geq 1$ be integers. Let $\mathbf{P} = \{P_1(y),\ldots, P_m(y)\} \subset \Z[y_1,\ldots,y_n]$ be an independent family of polynomials with zero constant term. 
		Then there exist $\lambda \in (0,1]$, $\ve, C_0 > 0$, and $s \in \N$ with $s \geq 2$, each depending only on $\mathbf P$, along with a constant $C_1$ that depends on $\mathbf P$ and $\ve$, such that, for any finite commutative ring $R$ with characteristic $N$ satisfying $\lpf N > C_1$ and any 1-bounded functions $f_0,\ldots, f_m : R \to \C$,
		\be
		\left| \Lambda_{P_1,\ldots,P_m}(f_0,\ldots, f_m)\right| \ \leq \ \frac{C_0}{\lpf{N}^{\ve}} + 2^n\min_{0\leq j \leq m} \norm{f_j}_{U^s}^\lambda.
		\ee
	\end{prop}
	Proposition~\ref{prop: Us control} may be viewed as a generalization of Proposition~2.2 in \cite{peluse}, which ``can be proven by carrying out the argument in Sections 3--5 of \cite{prend} almost word-for-word, but in the finite field setting instead of the integer setting.''
	
	However, the same cannot be said of our Proposition~\ref{prop: Us control}, for a simple reason: unlike a field, a finite commutative ring can have zero divisors, which complicate matters significantly. In Subsection~\ref{subsec: methods} we sketched some of the difficulties that arise; in Section~\ref{sec: proof of prop: Us control}, we will explain these more precisely.
	
	Assuming the previous proposition, we can immediately derive the statement we will use in the proof of Theorem~\ref{main thm}:
	\begin{prop}\label{prop: real Us control}
		Let $m_1 \geq 2$, $m_2 \geq 0$, and $n \geq 1$ be integers. Let \\ $\mathbf P = \{P_1,\ldots, P_{m_1},Q_1,\ldots, Q_{m_2}\} \subset \Z[y_1,\ldots,y_n]$ be an independent family of polynomials with zero constant term. Then there exist $\lambda \in (0,1]$, $\ve, C_0 > 0$, and $s \in \N$ with $s\geq 2$, each depending only on $\mathbf P$, along with a constant $C_1$ that depends on $\mathbf P$ and $\ve$, such that, for any finite commutative ring $R$ with characteristic $N$ satisfying $\lpf N > C_1$, any vector $F = (f_0,\ldots, f_{m_1})$ of 1-bounded functions $R \to \C$, and any vector $\Psi = (\psi_1,\ldots,\psi_{m_2})$ of additive characters of $R$, one has
		\be
		\left| \Lambda_{P_1,\ldots,P_{m_1}}^{Q_1,\ldots,Q_{m_2}}(F;\Psi)\right| \ \leq \ \frac{C_0}{\lpf{N}^{\ve}} + 2^n\min_{0\leq i \leq m_1} \norm{f_i}_{U^s}^\lambda.
		\ee
	\end{prop}
	\begin{proof}
		Observe that
		\begin{multline}
			\Lambda_{P_1,\ldots,P_{m_1}}^{Q_1,\ldots,Q_{m_2}}(F;\Psi) \ = \ \cavg{(x,y)}{R\times R^n} f_0(x) \prod_{i=1}^{m_1} f_i(x+P_i(y))\prod_{j=1}^{m_2} \ol{\psi_j(x)} \psi_j(x+Q_j(y))  \\ = \ \Lambda_{P_1,\ldots,P_{m_1},Q_1,\ldots,Q_{m_2}}\left(f_0\prod_{j=1}^{m_2} \ol{\psi_j}, f_1,\ldots, f_{m_1},\psi_1,\ldots,\psi_{m_2}\right).
		\end{multline}
		Applying Proposition~\ref{prop: Us control}, the result follows if $\norm{f_0\prod_{j=1}^{m_2} \ol{\psi_j}}_{U^s} = \norm{f_0}_{U^s}$. Indeed, for any $h \in R$, one has
		\be
		\Delta_h\left( f_0\prod_{j=1}^{m_2} \ol{\psi_j} \right) \ = \ \Delta_h\left( f_0\right) \Delta_h\left(\prod_{j=1}^{m_2} \ol{\psi_j} \right) \ = \ \Delta_h(f_0) \cdot d_h,
		\ee
		where $d_h = \prod_{j=1}^{m_2} \ol{\psi_j(h)}$ is a complex number of modulus one; thus
		\be
		\norm{f_0\prod_{j=1}^{m_2} \ol{\psi_j}}_{U^s}^{2^s} \ = \ \rcavg{h} \norm{\Delta_h\left( f_0\prod_{j=1}^{m_2} \ol{\psi_j} \right)}_{U^{s-1}
		}^{2^{s-1}}
		\ = \ \rcavg{h} \norm{\Delta_h f_0}_{U^{s-1}
		}^{2^{s-1}} \ = \ \norm{f_0}_{U^{s}}^{2^s}
		\ee
		by (seminorm) homogeneity.
	\end{proof}
	
	\subsection{From $U^s$ control to $U^{s-1}$ control} \label{subsec: reducing the step}
	After having bounded the averages $\Lambda_{P_1,\ldots,P_{m_1}}^{Q_1,\ldots,Q_{m_2}}(F;\Psi)$ in terms of a Gowers norm of potentially high degree $s$, we will reduce $s$ using an adaptation of an inductive argument in \cite{peluse}. The first lemma that we will state is borrowed with minor notational changes; the second lemma is our upgraded version of the first; the third lemma is modified for our situation for technical reasons. This third lemma is the one which actually reduces $s$ by 1, at the cost of proliferating error terms. In the proof of Theorem~\ref{main thm}, we will show how to manage the error terms that will result from applying the third lemma $s-1$ times.
	
	For these next few lemmas, we will consider functions $\xi : R \times R' \to \C$, where $R$ and $R'$ are finite commutative rings. Given such a function, recall that we define the discrete derivative $\Delta_{h_1,\ldots,h_s}^{(1)}\xi : R \times R' \to \C$ in the following way:
	\be
	\Delta_{h}^{(1)}\xi(x,y) := \xi(x+h,y)\ol{\xi}(x,y)
	\ee
	for $h \in R$ and
	\be
	\Delta_{h_1,\ldots,h_s}^{(1)}\xi(x,y) := \Delta_{h_1}^{(1)}( \Delta_{h_2}^{(1)}( \cdots (\Delta_{h_s}^{(1)}f)\cdots ))(x,y)
	\ee
	for $h_1,\ldots, h_s \in R$.
	
	\begin{lemma}\label{lem: pel 5.1}
		Let $m$ be a positive integer, and let $R$ be a finite commutative ring. Let $\xi_1',\ldots, \xi_m' : R^2 \to \C$ be 1-bounded functions. Set 
		\be g'(x) := \rcavg y \prod_{i=1}^m \xi_i'(x,y),
		\ee
		and, for every integer $t \geq 0$ and $h_1,\ldots, h_t \in R$, set
		\be
		g_{h_1,\ldots, h_t}'(x) := \begin{cases} \rcavg y \prod_{i=1}^m \Delta_{h_1,\ldots,h_t}^{(1)} \xi_i'(x,y) & \text{ if } t > 0, \\
			g'(x) & \text{ otherwise.}
		\end{cases}
		\ee
		Then $\norm{g'}_{U^s}^{2^{2s-2}} \leq \rcavg{h_1,\ldots,h_{s-2}} \norm{g_{h_1,\ldots, h_{s-2}}'}_{U^2}^4$ for all $s \geq 2$.
	\end{lemma}
	\begin{proof} The proof is as in \cite[Lemma 5.1]{peluse}. An additive change of variables is used at one point; over a finite ring $R$, this step is valid, with no need to worry about invertibility of any element.
	\end{proof}
	
	We apply the previous lemma to derive the lemma that we will actually use:
	
	\begin{lemma}\label{lem: pel 5.1, use case}
		Let $m$ and $n$ be positive integers, let $R$ be a finite commutative ring, and let $R' = R^n$ with coordinatewise addition and multiplication. Let $\xi_1,\ldots, \xi_m : R \times R' \to \C$ be 1-bounded functions. Define the function $g : R \to \C$ by 
		\be g(x) := \cavg{y}{R'} \prod_{i=1}^m \xi_i(x,y),
		\ee
		and, for every integer $t \geq 0$ and $h_1,\ldots, h_t \in R$, define $g_{h_1,\ldots,h_t} : R \to \C$ by
		\be
		g_{h_1,\ldots, h_t}(x) := \begin{cases} \cavg{y}{R'} \prod_{i=1}^m \Delta_{h_1,\ldots,h_t}^{(1)} \xi_i(x,y) & \text{ if } t > 0, \\
			g(x) & \text{ otherwise.}
		\end{cases}
		\ee
		Then $\norm{g}_{U^s}^{2^{2s-2}} \leq \rcavg{h_1,\ldots,h_{s-2}} \norm{g_{h_1,\ldots, h_{s-2}}}_{U^2}^4$ for all $s \geq 2$.
	\end{lemma}
	\begin{proof} For each $i \in \{1,\ldots,m\}$, define the 1-bounded functions $\xi_i' : R' \times R' \to \C$ by $\xi_i'(x,y) := \xi_i(x_1,y)$, where $x = (x_1,\ldots, x_n)$. Similarly, define $g' : R' \to \C$ by $g'(x) = g(x_1)$. 
		
		For every integer $t \geq 0$ and $h_1,\ldots, h_t \in R'$, define the function $g_{h_1,\ldots, h_t}' : R' \to \C$
		\be
		g_{h_1,\ldots, h_t}'(x) := \begin{cases} \cavg{y}{R'} \prod_{i=1}^m \Delta_{h_1,\ldots,h_t}^{(1)} \xi_i'(x,y) & \text{ if } t > 0, \\
			g'(x) & \text{ otherwise.}
		\end{cases}
		\ee
		Applying Lemma~\ref{lem: pel 5.1} to the ring $R'$, we conclude that, for all integers $s \geq 2$,
		\begin{equation}
			\norm{g'}_{U^s}^{2^{2s-2}} \leq \cavg{h_1,\ldots,h_{s-2}} {R'}\norm{g_{h_1,\ldots, h_{s-2}}'}_{U^2}^4.
		\end{equation}
		To finish the proof, it suffices to show that
		\begin{equation}
			\norm{g'}_{U^s}^{2^s} \ = \ \norm{g}_{U^s}^{2^s}
		\end{equation}
		and that
		\begin{equation}
			\cavg{h_1,\ldots,h_{s-2}} {R'} \norm{g_{h_1,\ldots, h_{s-2}}'}_{U^2}^4 \ = \ \cavg{h_{1,1},\ldots,h_{s-2,1}} {R}\norm{g_{h_{1,1},\ldots, h_{s-2,1}}}_{U^2}^4.
		\end{equation}
		For the first claim, writing $x = (x_1,\ldots, x_n)$ and $h_i = (h_{i,1},\ldots, h_{i,n})$, we observe that $\Delta_{h_1}g'(x) = \Delta_{h_{1,1}} g(x_1)$, which implies that
		\begin{multline}
			\norm{g'}_{U^s}^{2^s} \ = \ \cavg{x,h_1,\ldots, h_s}{R'} \Delta_{h_1,\ldots, h_s} g'(x) \ = \ \cavg{x,h_1,\ldots, h_{s}}{R'} \Delta_{h_{1,1},\ldots, h_{s,1}} g(x_1) \\ = \ \cavg{x_1,h_{1,1},\ldots, h_{s,1}}{R} \Delta_{h_{1,1},\ldots, h_{s,1}} g(x_1) \ = \ \norm{g}_{U^s}^{2^s}.
		\end{multline}
		The proof of the second claim relies on the same observation as the first but is more tedious to write. Note that the final formulation of this lemma dispenses with the $h_{1,1}$-type notation, which is only needed to write out the proof unambiguously.
	\end{proof}
	
	As described above, the following lemma, despite its lengthy formulation, simply reduces the degree $s$ of the Gowers norm that bounds the average $\Lambda_{P_1,\ldots,P_{m_1}}^{Q_1,\ldots,Q_{m_2}}(F;\Psi)$. The idea of the proof is to pass to a dual function instead (see, e.g., \eqref{example of dual function} and following) and use Lemma~\ref{lem: pel 5.1, use case} to analyze its $U^s$ norm on the level of $U^2$ norms of its discrete derivatives. See also Remark~\ref{rem: explanation of lem: Gowers lowering}.

	\begin{lemma}\label{lem: Gowers lowering} Let $m_1 \geq 2$, $m_2 \geq 0$, and $n \geq 1$ be integers. Suppose that $\mathbf P = \{P_1,\ldots, P_{m_1},Q_1,\ldots, Q_{m_2} \} \subset \Z[y_1,\ldots,y_n]$ is an independent family of polynomials with zero constant term. Write $P_0 (y) = 0$. For each $k \in \{0,\ldots,m_1\}$, denote by $k'$ the number $k+1 \bmod (m_1+1)$, fix a bijection $\sigma_k : \{1,\ldots, m_1-1\} \to \{0,\ldots, m_1\} \setminus \{k,k'\}$, and define the family $\mathbf{R}_{k} := \{R_{k,1},\ldots, R_{k,m_1-1},S_{k,1},\ldots, S_{k,m_2+1}\}$, where $R_{k,i} := P_{\sigma_k(i)} - P_{k'}$ and
		\be
		S_{k,j} \ = \ \begin{cases} Q_j & 1 \leq j \leq m_2, \\ P_k - P_{k'} & j = m_2+1. \end{cases}
		\ee
		Assume the following two conditions hold:
		\begin{enumerate}
			\item There exist a positive function $b_3 : \N \to \R$, positive real numbers $b_1,b_2, b_4$, and an integer $t \geq 2$ such that, for any finite commutative ring $R$ with characteristic $N$ satisfying $\lpf{N} > b_4$, for any vector $F = (f_0,f_1,\ldots,f_{m_1})$ of 1-bounded functions $R \to \C$ and any vector $\Psi = (\psi_1,\ldots, \psi_{m_2})$ of additive characters of $R$, one has
			\be
			\left| \Lambda_{P_1,\ldots, P_{m_1}}^{Q_1,\ldots,Q_{m_2}}(F;\Psi) \right| \ \leq \ b_1 \min_{0\leq i \leq m_1} \norm{f_i}_{U^t}^{b_2} + b_3(N).
			\ee
			\item For every $k \in \{0,\ldots,m_1\}$, there exist $c_{1,k} > 0$ and a positive function $\theta_k : \N \to \R$ such that, for any finite commutative ring $R$ with characteristic $N$ satisfying $\lpf{N} > c_{1,k}$, for any vector $G = (g_0,g_1,\ldots,g_{m_1-1})$ of 1-bounded functions $R \to \C$ and any vector $\Phi = (\phi_1,\ldots, \phi_{m_2+1})$ of additive characters of $R$, one has 
			\be
			\left| \Lambda_{R_{k,1},\ldots, R_{k,m_1-1}}^{S_{k,1}\ldots,S_{k,m_2+1}}(G;\Phi) - 1_{\Phi=1} \prod_{i=0}^{m_1-1} \rcavg x g_i(x) \right| \ \leq \ \theta_k(N).
			\ee
		\end{enumerate}
		Then there exist $c_1' > 0$ and a positive function $\Theta : \N \to \R$ depending only on $\mathbf P$ (and not on $b_2,b_3,b_4$, or $t$) such that, for any finite commutative ring $R$ with characteristic $N$ satisfying $\lpf {N} > \max\{c_1',b_4\}$, for any vector $F = (f_0,f_1,\ldots,f_{m_1})$ of 1-bounded functions $R \to \C$ and any vector $\Psi = (\psi_1,\ldots, \psi_{m_2})$ of additive characters of $R$, one has
		\begin{equation}
			\left| \Lambda_{P_1,\ldots, P_{m_1}}^{Q_1,\ldots,Q_{m_2}}(F;\Psi) \right| \ \leq \ b_1^{1/2} \min_{0 \leq i \leq m_1} \norm{f_i}_{U^{t-1}}^{b_2/2^t} + \Theta(N)^{2b_2/4^t} + b_3(N)^{1/2}.
		\end{equation}
		
		Moreover, one may take $c_1'=\max_{0 \leq k \leq m_1} c_{1,k}$; writing $\theta(N) = \max_{0\leq k \leq m_1} \theta_k(N)$, one may also take $\Theta(N) = 2\theta(N)+\theta(N)^2$.
	\end{lemma}
	\begin{remark}\label{rem: explanation of lem: Gowers lowering}
		Let us explain the role of the two conditions in Lemma~\ref{lem: Gowers lowering} in our inductive proof of Theorem~\ref{main thm}. In the inductive step, we apply the lemma repeatedly.
		
		The first condition is about Gowers seminorm control. The first time Lemma~\ref{lem: Gowers lowering} is applied, this condition will hold by Proposition~\ref{prop: real Us control}. After that, it will hold by the previous application of Lemma~\ref{lem: Gowers lowering}.
		
		The second condition will hold by the induction hypothesis of Theorem~\ref{main thm}.
	\end{remark}
	
	\begin{proof} By hypothesis, there exist $c_{1,0},c_{1,1},\ldots, c_{1,m_1},b_1,b_2,b_4 > 0$, positive functions $b_3,\theta_0,\theta_1,\ldots,\theta_{m_1}:\N\to\R$, and an integer $t \geq 2$ such that, for any finite commutative ring $R$ with characteristic $N$ satisfying $\lpf N > \max\{b_4,\max_{0 \leq k \leq m_1} c_{1,k}\}$, for any vectors $F = (f_0,\ldots, f_{m_1})$ and $G = (g_0,\ldots, g_{m_1-1})$ of 1-bounded functions $R \to \C$ and vectors $\Psi = (\psi_1,\ldots, \psi_{m_2})$ and $\Phi = (\phi_1,\ldots, \phi_{m_2+1})$ of additive characters of $R$, one has
		\be\label{eqn: ut min bound}
		\left| \Lambda_{P_1,\ldots, P_{m_1}}^{Q_1,\ldots,Q_{m_2}}(F;\Psi) \right| \ \leq \ b_1\min_{0\leq i \leq m_1} \norm{f_j}_{U^t}^{b_2} + b_3(N)
		\ee
		and, for $k \in \{0,\ldots,m_1\}$,
		\be\label{"the assumption"}
		\left| \Lambda_{R_{k,1},\ldots, R_{k,m_1-1}}^{S_{k,1}\ldots,S_{k,m_2+1}}(G;\Phi) - 1_{\Phi=1} \prod_{i=0}^{m_1-1} \rcavg x g_i(x) \right| \ \leq \ \theta_k(N).
		\ee
		
		
		Set $c_1'=\max_{0 \leq k \leq m_1} c_{1,k}$ and define the function $\theta(N) = \max_{0 \leq k \leq m_1} \theta_k(N)$. Let $R$ be a finite commutative ring with characteristic $N$ such that $\lpf N > \max\{c_1',b_4\}$.
		Let $F = (f_0,f_1,\ldots,f_{m_1})$ be a vector of 1-bounded functions $R \to \C$ and $\Psi = (\psi_1,\ldots, \psi_{m_2})$ be a vector of additive characters of $R$. Let $R' = R^n$ with coordinatewise addition and multiplication.
		
		Let us first bound the term $\Lambda_{P_1,\ldots,P_{m_1}}^{Q_1,\ldots,Q_{m_2}} (F; \Psi)$ in one of several ways. Fix $k \in \{0,\ldots, m_1\}$, and define the dual function
		\be\label{example of dual function} 
		g(x) \ := \ \cavg{y}{R'} \prod_{\substack{i=0 \\ i \neq k}}^{m_1} f_i(x+P_i(y)-P_{k}(y)) \prod_{j=1}^{m_2} \psi_j(Q_j(y)),
		\ee
		so that, after a change of variables, we have the identity
		\begin{equation}
			\Lambda_{P_1,\ldots,P_{m_1}}^{Q_1,\ldots,Q_{m_2}} (F; \Psi) \ = \ \cavg{x}{R} f_{k}(x)g(x).  
		\end{equation}
		By applying this identity, Cauchy--Schwarz, and \eqref{eqn: ut min bound}, we observe that
		\begin{align}
			\left| \Lambda_{P_1,\ldots,P_{m_1}}^{Q_1,\ldots,Q_{m_2}} (F; \Psi) \right|^2 \ & = \ \left| \cavg{x}{R} f_{k}(x)g(x) \right|^2 \nonumber \\
			& \leq \ \cavg{x}{R} \left|g(x)\right|^2 \nonumber \\
			& = \ \cavg{x}{R} \ol{g}(x) g(x) \nonumber \\
			& = \ \cavg{(x,y)}{R\times R'} \ol{g}(x+P_{k}(y)) \prod_{\substack{i = 0 \\ i \neq k}}^{m_1} f_i(x+P_i(y)) \prod_{j=1}^{m_2} \psi_j(Q_j(y))  \nonumber \\
			& \leq \ b_1 \norm{g}_{U^t}^{b_2} + b_3(N). \label{random thing 0.5}
		\end{align}
		
		Let us bound $\norm{g}_{U^t}$. Applying Lemma~\ref{lem: pel 5.1, use case} with $m := m_1+m_2$ and 
		\be
		\xi_i(x,y):= \begin{cases}
			f_i(x+P_i(y)-P_{k}(y)) & \text{ if } 1 \leq i \leq m_1 \\
			\psi_{i-m_1}(Q_{i-m_1}(y)) & \text{ if } m_1 + 1 \leq i \leq m_1+m_2\end{cases}
		\ee
		if $k = 0$ and
		\be
		\xi_i(x,y):= \begin{cases}
			f_0(x - P_{k}(y)) & \text{ if } i = k \\
			f_i(x+P_i(y)-P_{k}(y)) & \text{ if } 1 \leq i \leq m_1 \text{ and } i \neq k \\
			\psi_{i-m_1}(Q_{i-m_1}(y)) & \text{ if } m_1 + 1 \leq i \leq m_1+m_2\end{cases}
		\ee
		otherwise, one obtains
		\be\label{random thing 1}
		\norm{g}_{U^t}^{2^{2t-2}} \ \leq \ \rcavg{h_1,\ldots, h_{t-2}} \norm{g_{h_1,\ldots, h_{t-2}}}_{U^2}^4.
		\ee
		Let $h_1,\ldots,h_{t-2} \in R$ and let $\phi_{m_2+1}$ be an additive character of $R$. For each $j \in \{1,\ldots, m_2\}$, set $\phi_j = \psi_j$ if $t = 2$ and $\phi_j = 1$ if $t > 2$. Write $I = \{0,\ldots, m_1\} \setminus \{k\}$. Then 
		\be
		g_{h_1,\ldots, h_{t-2}}(x) \ = \ \cavg{y}{R'} \prod_{i \in I} \Delta_{h_1,\ldots,h_{t-2}} f_i(x+P_i(y)-P_{k}(y)) \prod_{j=1}^{m_2} \phi_j(Q_j(y)).
		\ee
		Recall that $k' = k + 1 \bmod (m_1+1)$ and that $\sigma_{k}$ is a bijection $\{1,\ldots, m_1-1\} \to I \setminus \{k'\}$. We obtain the identity 
		\begin{align}
			& \widehat{g}_{h_1,\ldots,h_{t-2}}(\phi_{m_2+1}) \nonumber 
			\\ &  = \ \cavg{(x,y)}{R\times R'} \phi_{m_2+1}(x) \prod_{i \in I} \Delta_{h_1,\ldots, h_{t-2}} f_i(x+P_i(y)-P_{k}(y)) \prod_{j=1}^{m_2} \phi_j(Q_j(y)) \nonumber \\
			& = \ \cavg{(x,y)}{R\times R'} \phi_{m_2+1}(x+P_{k}(y)-P_{k'}(y)) \prod_{i \in I} \Delta_{h_1,\ldots, h_{t-2}} f_i(x+P_i(y)-P_{k'}(y)) \prod_{j=1}^{m_2} \phi_j(Q_j(y)) \nonumber \\
			& = \ \Lambda_{R_{k,1},\ldots,R_{k,m_1-1}}^{S_{k,1},\ldots,S_{k,m_2+1}}(G_{k}; \Phi) \label{random fourier identity for g}
		\end{align}
		for the vector $G_{k} = (g_{k,0},g_{k,1},\ldots, g_{k,m_1-1})$ of 1-bounded functions given by
		\be
		g_{k,i} \ := \ \begin{cases} \phi_{m_2+1}\Delta_{h_1,\ldots,h_{t-2}}f_{k'} & i = 0 \\ \Delta_{h_1,\ldots, h_{t-2}} f_{\sigma_{k}(i)} & i \neq 0 \end{cases}
		\ee
		and the vector $\Phi := (\phi_1,\ldots,\phi_{m_2+1})$. By the second condition (this will be the induction hypothesis in the proof of Theorem~\ref{main thm}, see Remark~\ref{rem: explanation of lem: Gowers lowering}), we have
		\be \left| \widehat{g}_{h_1,\ldots,h_{t-2}}(\phi_{m_2+1})  - 1_{\Phi=1}\prod_{i=0}^{m_1-1} \rcavg x g_{k,i}(x) \right| \ \leq \ \theta_k(N).\label{random application of an assumption}
		\ee
		Analyzing the estimate
		\be
		\left| 1_{\Phi=1}\prod_{i=0}^{m_1-1} \rcavg x g_{k,i}(x) \right| \ \leq \ \min_{0 \leq i \leq m_1-1} \left| \rcavg x g_{k,i}(x)\right|
		\ee
		and applying \eqref{random fourier identity for g} and \eqref{random application of an assumption}, we conclude that
		\be
		\left| \widehat{g}_{h_1,\ldots,h_{t-2}}(\phi_{m_2+1}) \right| \ \leq \ \min_{i \in I \setminus \{k'\}} \left| \rcavg x \Delta_{h_1,\ldots,h_{t-2}}f_i(x)\right| + \theta_k(N)
		\ee		
		for any $\phi_{m_2+1} \in \hat{R}$. It follows that for any $h_1,\ldots,h_{t-2} \in R$ and any $i \in \{0,\ldots, m_1\} \setminus \{k,k'\}$, 
		\be\label{random thing 2}
		\max_{\phi \in \widehat{R}} \left| \widehat{g}_{h_1,\ldots,h_{t-2}}(\phi) \right| \ \leq \ \norm{\Delta_{h_1,\ldots,h_{t-2}}f_i}_{U^1} + \theta(N).
		\ee
		
		Thus, by \eqref{random thing 1}, Lemma~\ref{simple inequality for U2 norm}, \eqref{random thing 2}, and the 1-boundedness of the $f_i$, we deduce that, for any $i \in \{0,\ldots, m_1\} \setminus \{k,k'\}$,
		\begin{align*}
			\norm{g}_{U^t}^{2^{2t-2}} \ &
			\leq \ \rcavg{h_1,\ldots,h_{t-2}} \norm{g_{h_1,\ldots,h_{t-2}}}_{U^2}^4 \\
			& \leq \ \rcavg{h_1,\ldots,h_{t-2}} \max_{\phi\in \widehat{R}} \left| \widehat{g}_{h_1,\ldots,h_{t-2}}(\phi)\right|^2 \\
			& \leq \  \rcavg{h_1,\ldots,h_{t-2}} \norm{\Delta_{h_1,\ldots,h_{t-2}}f_i}_{U^1}^2 + 2\theta(N)+\theta(N)^2 \\
			& = \  \rcavg{h_1,\ldots,h_{t-2}} \rcavg{x,h_{t-1}} \Delta_{h_1,\ldots,h_{t-2}}f_i(x+h_{t-1})\ol{\Delta_{h_1,\ldots,h_{t-2}}f_i(x)} + 2\theta(N)+\theta(N)^2 \\
			& = \  \rcavg{x,h_1,\ldots,h_{t-1}} \Delta_{h_1,\ldots,h_{t-1}} f_i(x) + 2\theta(N)+\theta(N)^2 \\
			& = \  \norm{f_i}_{U^{t-1}}^{2^{t-1}} + 2\theta(N)+\theta(N)^2.
		\end{align*}
		Since $a^{r} + b^{r} \geq (a+b)^{r}$ whenever $a,b \geq 0$ and $r \in (0,1]$, we conclude that
		\begin{equation}\label{random thing 3}
			\norm{g}_{U^t} \ \leq \ \min_{i \in I\setminus \{k'\}} \norm{f_i}_{U^{t-1}}^{2^{1-t}} +  \Theta(N) ^{2^{2-2t}}
		\end{equation}
		for $\Theta(N) := 2\theta(N)+\theta(N)^2$. By \eqref{random thing 0.5} and \eqref{random thing 3},
		\begin{equation} \label{random thing 4}
			\left| \Lambda_{P_1,\ldots,P_{m_1}}^{Q_1,\ldots,Q_{m_2}} (F; \Psi) \right| \ \leq \ b_1^{1/2}\norm{g}_{U^t}^{b_2/2} + b_3(N)^{1/2} \ \leq \ b_1^{1/2} \min_{i \in I \setminus \{k'\}} \norm{f_i}_{U^{t-1}}^{b_2/2^t} + \Theta(N)^{2b_2/4^t} + b_3(N)^{1/2}.
		\end{equation}
		Finally, the initial choice of $k \in \{0,\ldots, m_1\}$ was arbitrary and the inequality \eqref{random thing 4} only refers to $i$ belonging to $I \setminus \{k'\} = \{0,\ldots, m_1 \} \setminus \{k,k'\}$, so repeating this argument for values of $k$ from some subset $I' \subset \{0,\ldots, m_1\}$ such that $\{0,\ldots, m_1\} = \cup_{k \in I'} \{0,\ldots m_1\} \setminus \{k,(k+1) \bmod (m_1+1) \}$ finishes the proof.
	\end{proof}
	
	\section{Proof of Theorem~\ref{main thm}}\label{sec: proof of main thm}
	We now prove the main theorem. The only proposition we use that has not been fully proven is Proposition~\ref{prop: real Us control}, which depends on Proposition~\ref{prop: Us control}, whose proof is deferred to Section~\ref{sec: proof of prop: Us control}.
	
	\begin{thmrep}[\ref{main thm}]
		Let $n \geq 1$, $m_1 \geq 1$, and $m_2 \geq 0$ be integers, and let $\mathbf P = \{ P_1,\ldots,P_{m_1},Q_1,\ldots,Q_{m_2}\} \subset \Z[y_1,\ldots,y_n]$ be an independent family of polynomials. There exist $c,C,\gamma > 0$ such that the following holds. For any finite commutative ring $R$ with characteristic $N$ satisfying $\lpf N > C$, any 1-bounded functions $f_0,\ldots, f_{m_1}: R \to \C$, and any additive characters $\psi_1,\ldots,\psi_{m_2}$ of $R$, one has
		\begin{equation}
			\left| \Lambda_{P_1,\ldots,P_{m_1}}^{Q_1,\ldots,Q_{m_2}}(f_0,\ldots, f_{m_1};\psi_1,\ldots,\psi_{m_2}) - 1_{\Psi=1} \prod_{i=0}^{m_1} \rcavg x f_i(x) \right| \ \leq \  c\,\lpf{N}^{-\gamma},
		\end{equation}
		where $1_{\Psi=1}$ equals 1 if every $\psi_i$ is trivial and 0 otherwise.
	\end{thmrep}
	\begin{proof}Without loss of generality, each polynomial in $\mathbf P$ may be assumed to have zero constant term; this comes from the freedom to choose $f_0,\ldots, f_{m_1}$.
		
		We proceed by induction on $m_1$. When $m_1 = 1$, Proposition~\ref{main thm base} provides the base case.
		
		Now, let $n \geq 1$, $m_1 \geq 2$ and $m_2 \geq 0$. The induction hypothesis asserts the following: For any integers $m_1' \geq 1$ and $m_2' \geq 0$ with $m_1'<m_1$ and any independent family $\mathbf R = \{R_1,\ldots, R_{m_1'},S_1,\ldots, S_{m_2'}\} \subset \Z[y_1,\ldots, y_n]$ of polynomials with zero constant term, there exist $c(\mathbf {R}),C(\mathbf{R}),\gamma(\mathbf{R}) > 0$ such that, for any finite commutative ring $R$ with characteristic $N$ satisfying $\lpf N > C(\mathbf{R})$, any vector $G = (g_0,\ldots, g_{m_1'})$ of 1-bounded functions $R \to \C$, and any vector $\Phi = (\phi_1,\ldots,\phi_{m_2'})$ of additive characters of $R$, one has
		\be
		\left| \Lambda_{R_1,\ldots, R_{m_1'}}^{S_1,\ldots, S_{m_2'}}(G;\Phi) - 1_{\Phi=1} \prod_{i=0}^{m_1'} \rcavg x g_i(x) \right| \ \leq \ c(\mathbf{R})\,\lpf{N}^{-\gamma(\mathbf{R})}.
		\ee
		
		Let $\mathbf P = \{P_1,\ldots, P_{m_1},Q_1,\ldots, Q_{m_2}\} \subset \Z[x]$ be an independent family of polynomials with zero constant term. The goal is to find $c,C,\gamma > 0$ such that, for any finite commutative ring $R$ with characteristic $N$ satisfying $\lpf N > C$, any vector $F = (f_0,\ldots, f_{m_1})$ of 1-bounded functions $R \to \C$, and any vector $\Psi = (\psi_1,\ldots,\psi_{m_2})$ of additive characters of $R$, one has
		\be
		\left| \Lambda_{P_1,\ldots,P_{m_1}}^{Q_1,\ldots,Q_{m_2}}(F;\Psi) - 1_{\Psi=1} \prod_{i=0}^{m_1} \rcavg x f_i(x) \right| \ \leq \  c\,\lpf{N}^{-\gamma}.
		\ee
		
		Let $k \in \{0,\ldots,m_1\}$. Define the family $\mathbf{R}_k$ as in the statement of Lemma~\ref{lem: Gowers lowering}. Then $\mathbf{R}_k$ is an independent family of polynomials with zero constant term since $\mathbf P$ is such. By the induction hypothesis, there exist $c(\mathbf{R}_k),C(\mathbf{R}_k),\gamma(\mathbf{R}_k) > 0$ such that, for any finite commutative ring $R$ with characteristic $N$ satisfying $\lpf N > C(\mathbf{R}_k)$, any vector $G = (g_0,g_1,\ldots,g_{m_1-1})$ of 1-bounded functions $R \to \C$, and any vector $\Phi = (\phi_1,\ldots, \phi_{m_2+1})$ of additive characters of $R$, one has 
		\be
		\left| \Lambda_{R_{k,1},\ldots, R_{k,m_1-1}}^{S_{k,1}\ldots,S_{k,m_2+1}}(G;\Phi) - 1_{\Phi=1} \prod_{i=0}^{m_1-1} \rcavg x g_i(x) \right| \ \leq \ c(\mathbf{R}_k)\,\lpf{N}^{-\gamma(\mathbf{R}_k)}.
		\ee
		
		Define $\mathbf{P}' = \{P_1,\ldots, P_{m_1-1},Q_1,\ldots,Q_{m_2}\}$. Since $\mathbf P$ is an independent family of polynomials with zero constant term, $\mathbf{P}'$ is also such. Moreover, by the induction hypothesis, there exist $c(\mathbf{P}'), C(\mathbf{P}'), \gamma(\mathbf{P}') > 0$ such that, for any finite commutative ring $R$ with characteristic $N$ satisfying $\lpf N > C(\mathbf{P}')$, any vector $G = (g_0,\ldots, g_{m_1-1})$ of 1-bounded functions $R \to \C$, and any vector $\Phi = (\phi_1,\ldots,\phi_{m_2})$ of additive characters of $R$, one has
		\be
		\left| \Lambda_{P_1,\ldots, P_{m_1-1}}^{Q_1,\ldots, Q_{m_2}}(G;\Phi) - 1_{\Phi=1} \prod_{i=0}^{m_1-1} \rcavg x g_i(x) \right| \ \leq \ c(\mathbf{P}')\,\lpf{N}^{-\gamma(\mathbf{P}')}.
		\ee
		
		By assumption, $\mathbf P$ is an independent family of polynomials with zero constant term. Thus, by Proposition~\ref{prop: real Us control}, there exist $\lambda \in (0,1]$, $\ve, C_0 > 0$, and $s \in \N$ with $s \geq 2$, each depending only on $\mathbf P$, along with a constant $C_1$ depending on $\mathbf P$ and $\ve$, such that, for any finite commutative ring $R$ with characteristic $N$ satisfying $\lpf N > C_1$, any vector $F = (f_0,\ldots, f_{m_1})$ of 1-bounded functions $R \to \C$, and any vector $\Psi = (\psi_1,\ldots,\psi_{m_2})$ of additive characters of $R$, one has
		\be
		\left| \Lambda_{P_1,\ldots,P_{m_1}}^{Q_1,\ldots,Q_{m_2}}(F;\Psi)\right| \ \leq \ \frac{C_0}{\lpf{N}^{\ve}} + 2^n\min_{0\leq i \leq m_1} \norm{f_i}_{U^s}^\lambda.
		\ee
		
		We intend to apply Lemma~\ref{lem: Gowers lowering} repeatedly, in such a way that its first condition will be satisfied by different settings of $b_i$ at each application and that its second condition will always be satisfied by the same setting of $c_{1,k}$ and $\theta_k$ for $k \in \{0,\ldots,m_1\}$. Thus, notice that, on setting, for each $k \in \{0,\ldots,m_1\}$, $c_{1,k} := C(\mathbf{R}_k)$ and
		\be
		\theta_k(N) \ := \ c(\mathbf{R}_k)\,\lpf{N}^{-\gamma(\mathbf{R}_k)},
		\ee
		the second condition is met.
		
		Set $b_1^{(s)} := 2^n$, $b_2^{(s)} := \lambda$, $b_3^{(s)}(N) := C_0\, \lpf{N}^{-\ve}$, and $b_4^{(s)} := C_1$. Then the first condition of the lemma is met for $t = s$. Hence, there exist $c_1' > 0$ and a positive function $\Theta : \N \to \R$ depending only on $\mathbf P$ (and not on $b_1^{(s)},b_2^{(s)},b_3^{(s)},b_4^{(s)}$, or $s$), such that, for any finite commutative ring $R$ with characteristic $N$ satisfying $\lpf N > \max\{c_1',b_4^{(s)}\}$, for any vector $F = (f_0,f_1,\ldots,f_{m_1})$ of 1-bounded functions $R \to \C$ and any vector $\Psi = (\psi_1,\ldots, \psi_{m_2})$ of additive characters of $R$, one has
		\begin{equation}
			\left| \Lambda_{P_1,\ldots, P_{m_1}}^{Q_1,\ldots,Q_{m_2}}(F;\Psi) \right| \ \leq \ (b_1^{(s)})^{1/2} \min_{0 \leq i \leq m_1} \norm{f_i}_{U^{s-1}}^{b_2^{(s)}/2^s} + \Theta(N)^{2b_2^{(s)}/4^s} + b_3^{(s)}(N)^{1/2}.
		\end{equation}
		It may happen that $s = 2$, in which case we would stop here. However, for the sake of explanation, let us assume that $s > 2$.
		
		Set
		\ba
		b_1^{(s-1)} \ & := \ (b_1^{(s)})^{1/2}, \\ b_2^{(s-1)} \ & := \ b_2^{(s)}/2^s, \\
		b_3^{(s-1)}(N) \ & := \ \Theta(N)^{2b_2^{(s)}/4^s} + b_3^{(s)}(N)^{1/2}, \\
		b_4^{(s-1)} \ & := \ \max\{c_1',b_4^{(s)}\}.
		\ea
		Then the first condition of Lemma~\ref{lem: Gowers lowering} is met for $t = s-1$. Applying the lemma, we note that neither $c_1'$ nor $\Theta$ in the conclusion depend on the new settings of $b_i$, so we conclude that for any finite commutative ring $R$ with characteristic $N$ satisfying $\lpf N > \max\{c_1',b_4^{(s-1)}\}$, for any vector $F = (f_0,f_1,\ldots,f_{m_1})$ of 1-bounded functions $R \to \C$ and any vector $\Psi = (\psi_1,\ldots, \psi_{m_2})$ of additive characters of $R$, one has
		\begin{equation}
			\left| \Lambda_{P_1,\ldots, P_{m_1}}^{Q_1,\ldots,Q_{m_2}}(F;\Psi) \right| \ \leq \ (b_1^{(s-1)})^{1/2} \min_{0 \leq i \leq m_1} \norm{f_i}_{U^{s-2}}^{b_2^{(s-1)}/2^{s-1}} + \Theta(N)^{2b_2^{(s-1)}/4^{s-1}} + b_3^{(s-1)}(N)^{1/2}.
		\end{equation}
		Either $s-2 = 1$ or $s-2 > 1$. In the latter case, the following settings are suggested by the previous display:
		\ba
		b_1^{(s-2)} \ & := \ (b_1^{(s-1)})^{1/2}, \\
		b_2^{(s-2)} \ & := \ b_2^{(s-1)}/2^{s-1}, \\
		b_3^{(s-2)}(N) \ & := \ \Theta(N)^{2b_2^{(s-1)}/4^{s-1}} + b_3^{(s-1)}(N)^{1/2}, \\
		b_4^{(s-2)} \ & := \ \max\{c_1',b_4^{(s-1)}\}.
		\ea
		We continue in this way to apply Lemma~\ref{lem: Gowers lowering} to conclude the following claim:
		\begin{cla}\label{cla: a claim}
			For any finite commutative ring $R$ with characteristic $N$, if $\lpf N > b_4^{(1)}$, then for any vector $F = (f_0,f_1,\ldots,f_{m_1})$ of 1-bounded functions $R \to \C$ and any vector $\Psi = (\psi_1,\ldots, \psi_{m_2})$ of additive characters of $R$, one has
			\be\left| \Lambda_{P_1,\ldots, P_{m_1}}^{Q_1,\ldots,Q_{m_2}}(F;\Psi) \right| \ \leq \ b_1^{(1)}\min_{0\leq i\leq m_1} \norm{f_i}_{U^1}^{b_2^{(1)}} + b_3^{(1)}(N),
			\ee
			where
			\ba
			& b_1^{(\ell-1)} = \begin{cases} 
				(b_1^{(\ell)})^{1/2} & \text{ if } \ell \in \{2,\ldots, s\}, \\
				2^n & \text{ if } \ell = s+1,
			\end{cases} \\
			& b_2^{(\ell-1)} = \begin{cases} 
				b_2^{(\ell)}/2^\ell & \text{ if } \ell \in \{2,\ldots, s\}, \\
				\lambda & \text{ if } \ell = s+1,
			\end{cases} \\
			& b_3^{(\ell-1)}(N) = \begin{cases} 
				\Theta(N)^{2b_2^{(\ell)}/4^\ell} + b_3^{(\ell)}(N)^{1/2} & \text{ if } \ell \in \{2,\ldots, s\}, \\
				C_0\, \lpf{N}^{-\ve} & \text{ if } \ell = s+1,
			\end{cases} \\ 
			& b_4^{(\ell-1)} = \begin{cases} 
				\max\{c_1',C_1\} & \text{ if } \ell \in \{2,\ldots, s\}, \\
				C_1 & \text{ if } \ell = s+1.
			\end{cases}
			\ea
		\end{cla}
		Claim~\ref{cla: a claim} will be applied in a situation where one of the $f_i$ will have zero integral, hence zero $U^1$ norm, so there is no need to compute $b_1^{(1)}$, and it will be helpful to bound $b_3^{(1)}(N)$. As $b_3^{(\ell-1)}(N)$ is defined in terms of $b_2^{(\ell)}$ and $b_3^{(\ell)}(N)$, we first observe that, for every $\ell \in \{1,\ldots, s\}$,
		\begin{equation}
			b_2^{(\ell)} \ = \ \lambda/ 2^{(1/2)(s-\ell)(1+s+\ell)}.
		\end{equation}
		
		Next, we claim by induction that, for every $k \in \{2,\ldots, s\}$,
		\begin{equation}\label{eqn: first bound on b31}
			b_3^{(1)}(N) \ \leq \ \sum_{\ell = 2}^k \Theta(N)^{b_2^{(\ell)}/2^{3(\ell-1)}} + b_3^{(k)}(N)^{2^{1-k}}.
		\end{equation}
		Indeed, when $k = 2$, we have equality in \eqref{eqn: first bound on b31}, and assuming \eqref{eqn: first bound on b31} holds for $k \in \{2,\ldots, s-1\}$, we observe that
		\begin{equation}
			b_3^{(1)}(N) \ \leq \ \sum_{\ell = 2}^k \Theta(N)^{b_2^{(\ell)}/2^{3(\ell-1)}} + b_3^{(k)}(N)^{2^{1-k}}
		\end{equation}
		and, since $a^{r} + b^{r} \geq (a+b)^{r}$ whenever $a,b \geq 0$ and $r \in (0,1]$,
		\begin{equation}
			b_3^{(k)}(N)^{2^{1-k}} \ = \ \left(\Theta(N)^{2b_2^{(k+1)}/4^{k+1}} + b_3^{(k+1)}(N)^{1/2}\right)^{2^{1-k}} \ \leq \ \Theta(N)^{b_2^{(k+1)}/2^{3k} } + b_3^{(k+1)}(N)^{2^{-k}},
		\end{equation}
		which proves the claim. We would like $b_3^{(1)}(N)$ to be bounded by a negative power of $\lpf{N}$. By the claim, we have
		\begin{equation}
			b_3^{(1)}(N) \ \leq \ \sum_{\ell = 2}^s \Theta(N)^{b_2^{(\ell)}/2^{3(\ell-1)}} + b_3^{(s)}(N)^{2^{1-s}} \ = \ \sum_{\ell = 2}^s \Theta(N)^{b_2^{(\ell)}/2^{3(\ell-1)}} + (C_0\, \lpf{N}^{-\ve})^{2^{1-s}}.
		\end{equation}
		
		Let us bound this expression for appropriate values of $N$. First, consider $\Theta(N)$. There exists $k_0 \in \{0,\ldots,m_1\}$ and $C' > 0$ such that $\lpf{N} > C'$ implies
		\be
		\max_{k\in\{0,\ldots,m_1\}} c(\mathbf{R}_k)\,\lpf{N}^{-\gamma(\mathbf{R}_k)} \ = \ c(\mathbf{R}_{k_0})\,\lpf{N}^{-\gamma(\mathbf{R}_{k_0})} \ \leq \ 1.
		\ee
		Let $\gamma' = \gamma(\mathbf{R}_{k_0})$. Then, for $N$ such that $\lpf N > C'$, one has 
		\be
		\Theta(N) \ = \ 2\theta_{k_0}(N)+\theta_{k_0}(N)^2 \ \leq \ 3 \theta_{k_0}(N) \ = \ 3c(\mathbf{R}_{k_0}) \,\lpf{N}^{-\gamma'}
		\ee
		and hence
		\begin{multline}
			b_3^{(1)}(N) \ \leq \ \sum_{\ell = 2}^s \Theta(N)^{b_2^{(\ell)}/2^{3(\ell-1)}} + (C_0\, \lpf{N}^{-\ve})^{2^{1-s}} \\ \leq \ \sum_{\ell = 2}^s (3c(\mathbf{R}_{k_0}) \,\lpf{N}^{-\gamma'})^{b_2^{(\ell)}/2^{3(\ell-1)}} + (C_0\, \lpf{N}^{-\ve})^{2^{1-s}}.
		\end{multline}
		Let $\ve' \ := \ \min_{2 \leq \ell \leq s} b_2^{(\ell)}/2^{3(\ell-1)} \in (0,1]$, $\gamma'' := \min\{\gamma' \ve', \ve \cdot 2^{1-s}\} > 0$, and $c'' :=  \sum_{\ell = 2}^s (3c(\mathbf{R}_{k_0}))^{b_2^{(\ell)}/2^{3(\ell-1)}} + C_0^{2^{1-s}}$. We conclude that
		\be
		b_3^{(1)}(N) \ \leq \ c''\, \lpf{N}^{-\gamma''}
		\ee
		whenever $\lpf N > C'$.
		
		Now we will prove the required statement. Let $c := 2c''+c(\mathbf{P}')$, $C := \max\{C(\mathbf{P}'),b_4^{(1)}, C'\}$, and $\gamma := \min\{\gamma(\mathbf{P}'),\gamma''\}$.
		
		Let $R$ be a finite commutative ring with characteristic $N$ satisfying $\lpf N > C$. Let $F = (f_0,\ldots, f_{m_1})$ be a vector of 1-bounded functions $R \to \C$ and let $\Psi = (\psi_1,\ldots, \psi_{m_2})$ be a vector of additive characters of $R$. Let $F' = (f_0,\ldots, f_{m_1-1})$. Let $f_{m_1}' = f_{m_1} - \rcavg x f_{m_1}(x)$. Then $\norm{f_{m_1}'}_{U^1} = 0$ and $F'' := (f_0,f_1,\ldots, f_{m_1-1},\frac{1}{2} f_{m_1}')$ is a vector of 1-bounded functions. By the claim,
		\ba
		\left| \Lambda_{P_1,\ldots,P_{m_1}}^{Q_1,\ldots,Q_{m_2}}(F;\Psi) - \left(\rcavg{x} f_{m_1}(x)\right) \Lambda_{P_1,\ldots,P_{m_1-1}}^{Q_1,\ldots,Q_{m_2}}(F';\Psi) \right| \ & = \ 2 \left| \Lambda_{P_1,\ldots,P_{m_1}}^{Q_1,\ldots,Q_{m_2}}(F'';\Psi)\right| \\
		& \leq \ 2b_3^{(1)}(N) \\
		& \leq \ 2 c''\,\lpf{N}^{-\gamma''}.
		\ea
		Moreover, since $f_{m_1}$ is 1-bounded, we have
		\ba
		\left| \left(\rcavg{x} f_{m_1}(x)\right) \Lambda_{P_1,\ldots,P_{m_1-1}}^{Q_1,\ldots,Q_{m_2}}(F';\Psi)
		- 1_{\Psi=1} \prod_{i=0}^{m_1} \rcavg x f_i(x)
		\right| \ & \leq \ \left|  \Lambda_{P_1,\ldots,P_{m_1-1}}^{Q_1,\ldots,Q_{m_2}}(F';\Psi)
		- 1_{\Psi=1} \prod_{i=0}^{m_1-1} \rcavg x f_i(x)
		\right| \\
		& \leq \ c(\mathbf{P}')\,\lpf{N}^{-\gamma(\mathbf{P}')}
		\ea
		by the induction hypothesis. Thus,
		\be
		\left| \Lambda_{P_1,\ldots,P_{m_1}}^{Q_1,\ldots,Q_{m_2}}(F;\Psi) - 1_{\Psi=1} \prod_{i=0}^{m_1} \rcavg x f_i(x) \right| \ \leq \ 2 c''\,\lpf{N}^{-\gamma''} + c(\mathbf{P}')\,\lpf{N}^{-\gamma(\mathbf{P}')} \ \leq \ c\,\lpf{N}^{-\gamma},
		\ee
		as desired.
	\end{proof}
	
	\section{Combinatorial consequences of Theorem~\ref{main thm for intro}} \label{sec: derivation of zero and positive density results}
	
	The combinatorial consequences of Theorem~\ref{main thm for intro} all follow from Proposition~\ref{prop: config count}, which we prove in this section. First, given an independent family $\mathbf P = \{P_1,\ldots, P_m\} \subset \Z[y_1,\ldots, y_n]$ and subsets $A_0,\ldots, A_m$ of a finite commutative ring $R$, we need to estimate the number of degenerated configurations $(x,x+P_1(y),\ldots, x+P_m(y)) \in A_0 \times A_1 \times \cdots \times A_m$. As a reminder, $(x,x+P_1(y),\ldots, x+P_m(y))$ is degenerated if the set $\{0,P_1(y),\ldots, P_m(y)\} \subset R$ has cardinality less than $m+1$.
	
	\begin{prop}\label{prop: bound on number of degen configs}
		Let $\mathbf P = \{P_1,\ldots, P_m\} \subset \Z[y_1,\ldots, y_n]$ be an independent family. There exist $C',\beta > 0$ such that, for any finite commutative ring $R$ with characteristic $N$ satisfying $\lpf N > C'$, for any subsets $A_0,\ldots, A_m \subset R$, the number of $(x,y) \in R \times R^n$ such that $(x,x+P_1(y),\ldots, x+P_m(y)) \in A_0 \times A_1 \times \cdots \times A_m$ is a degenerated configuration is at most $|R|^{n+1} \cdot \lpf{N}^{-\beta}$. 
	\end{prop}
	\begin{proof}
		Set $P_0(y) = 0$. Denote by $I$ the set of ordered pairs $(i,j) \in \{0,\ldots, m\}^2$ such that $i < j$. Let $(i,j) \in I$. Applying Proposition~\ref{prop: bound on number of roots}, let $c_{i,j}, C_{i,j}, \ve_{i,j} > 0$ be such that, for any finite commutative ring $R$ with $\lpf{\mathrm{char}(R)} > C_{i,j}$, the number of $y \in R^n$ such that $P_i(y) - P_j(y) = 0$ is at most $|R|^{n-1} + \frac{c_{i,j}|R|^n}{\lpf{\mathrm{char}(R)}^{\ve_{i,j}}}$.
		
		Let $c$ (resp. $C, \ve$) be the maximum of $c_{i,j}$ (resp. the maximum of $C_{i,j}$, the minimum of $\ve_{i,j}$) over the set of $(i,j) \in I$. Let $\beta = \ve/2$ and let $C'$ be the maximum of $C$ and the least positive integer $C''$ such that $(c+1)\binom{m+1}{2} M^{-\beta} \leq 1$ for all integers $M \geq C''$.  
		
		Let $R$ be a finite commutative ring with characteristic $N$ satisfying $\lpf N > C'$, and let $A_0,\ldots, A_m \subset R$. The number of $y \in R^{n}$ such that $\{0,P_1(y),\ldots, P_m(y)\} \subset R$ is a set of cardinality less than $m+1$ is at most
		\begin{equation}
			\sum_{(i,j) \in I} |R|^{n-1} + \frac{c_{i,j}|R|^n}{\lpf{N}^{\ve_{i,j}}} \ \leq \ \binom{m+1}{2}\left( |R|^{n-1} +  \frac{c |R|^n}{\lpf {N}^\ve} \right).
		\end{equation}
		Thus, the number of $(x,y) \in R \times R^n$ such that $(x,x+P_1(y),\ldots, x+P_m(y))$ is a degenerated configuration contained in $A_0 \times \cdots \times A_m$ is at most
		\begin{equation}
			|A_0| \cdot \binom{m+1}{2}\left( |R|^{n-1} +  \frac{c |R|^n}{\lpf {N}^\ve} \right) \ \leq \ \binom{m+1}{2} |R|^{n+1} \left(\frac{1}{|R|} + \frac{c}{\lpf{N}^{\ve}} \right)  \ \leq \ |R|^{n+1} \cdot \lpf{N}^{-\beta}.
		\end{equation}
	\end{proof}
	Let us now prove Proposition~\ref{prop: config count}.
	\begin{proprep}[\ref{prop: config count}]
		Let $\ve \in (0,1]$. Let $\mathbf P = \{ P_1,\ldots,P_{m}\} \subset \Z[y_1,\ldots,y_n]$ be an independent family of polynomials. There exists $\gamma = \gamma(\mathbf P)$ such that, for any finite commutative ring $R$ with $\lpf{\mathrm{char}(R)}$ sufficiently large and any subsets $A_0,\ldots, A_m$ of $R$ such that 
		\begin{equation}
			\ve|A_0|\cdots |A_m| \ > \ |R|^{m+1} \cdot \lpf{\mathrm{char}(R)}^{-\gamma},
		\end{equation}
		the normalized number of nontrivial configurations
		\begin{equation}
			S \ := \ \frac{\# \{ (x,y) \in R\times R^{n} : (x,x+P_1(y),\ldots, x+P_m(y)) \in A_0 \times A_1 \times \cdots \times A_m \} }{|R|^{n+1}}	
		\end{equation}
		satisfies
		\begin{equation}\label{ccc eqn 1}
			(1-\ve)\frac{|A_0|\cdots |A_m|}{|R|^{m+1}} \ < \ S \ < \ (1+\ve)\frac{|A_0|\cdots |A_m|}{|R|^{m+1}}.
		\end{equation}
	\end{proprep}
	\begin{proof}
		By Proposition~\ref{prop: bound on number of degen configs}, there exist $C',\beta > 0$ such that, for any finite commutative ring $R$ with characteristic $N$ satisfying $\lpf N > C'$, for any subsets $A_0,\ldots, A_m \subset R$, the number of $(x,y) \in R \times R^n$ such that $(x,x+P_1(y),\ldots, x+P_m(y)) \in A_0 \times A_1 \times \cdots \times A_m$ is a degenerated configuration is at most $|R|^{n+1} \cdot \lpf{N}^{-\beta}$.
		
		By Theorem~\ref{main thm for intro}, there exist $C'', \gamma' > 0$ such that, for any finite commutative ring $R$ with characteristic $N$ satisfying $\lpf{N} > C''$, for any 1-bounded functions $f_0,\ldots,f_m : R \to \C$, one has
		\begin{multline}\label{working out zero density criterion eqn 1}
			\Bigg| \frac{1}{|R|^{n+1}}\sum_{x,y_1,\ldots, y_n \in R} f_0(x)f_1(x+P_1(y_1,\ldots,y_n))\cdots f_{m}(x+P_{m}(y_1,\ldots,y_n)) \\ - \left( \frac{1}{|R|} \sum_{x\in R} f_0(x) \right) \dots \left( \frac{1}{|R|} \sum_{x\in R} f_m(x) \right) \Bigg| \ \leq \ \lpf{N}^{-\gamma'}.
		\end{multline}
		Let $\gamma = \min\{\gamma', \beta\}/2$, and let $C$ be the maximum of $C'$, $C''$, and the least positive integer $C'''$ such that $2 M^{-\gamma} \leq 1$ for all integers $M \geq C'''$.
		
		Suppose $R$ is a finite commutative ring with characteristic $N$ satisfying $\lpf{N} > C$, and let $A_0,\ldots, A_m \subset R$ satisfy
		\begin{equation}\label{ccc eqn 3}
			\ve|A_0|\cdots |A_m| \ > \ |R|^{m+1} \cdot \lpf{N}^{-\gamma}.
		\end{equation}
		Applying \eqref{working out zero density criterion eqn 1} with the setting $f_i = 1_{A_i}$, $i \in \{0,\ldots, m\}$, and denoting by $M$ the cardinality of the set $\{(x,y) \in R\times R^n : (x,x+P_1(y),\ldots, x+P_m(y)) \in A_0 \times \cdots \times A_m \}$ and by $M_1$ (resp. $M_2$) the cardinality of the subset where $(x,\ldots, x+P_m(y))$ is a nontrivial (resp. degenerated) configuration, we observe that
		\begin{equation}
			\left|M/|R|^{n+1} - \prod_{i=0}^{m} |A_i|/|R| \right| \ \leq \ \lpf{N}^{-\gamma'}.
		\end{equation}
		Since $M = M_1 + M_2$, it follows that
		\begin{equation}\label{ccc eqn 2}
			S \ = \ M_1/|R|^{n+1} \ \geq \ -M_2/|R|^{n+1} + \prod_{i=0}^{m} |A_i|/|R| - \lpf{N}^{-\gamma'},
		\end{equation}
		so for the first inequality in \eqref{ccc eqn 1}, it would suffice to show that the right-hand side of \eqref{ccc eqn 2} exceeds $(1-\ve)\frac{|A_0|\cdots |A_m|}{|R|^{m+1}}$, which holds if and only if
		\begin{equation}\label{ccc eqn 4}
			\ve \prod_{i=0}^{m} |A_i|/|R| \ > \ M_2/|R|^{n+1} + \lpf{N}^{-\gamma'}.
		\end{equation}
		By Proposition~\ref{prop: bound on number of degen configs}, we have
		\begin{equation}
			M_2/|R|^{n+1} \ \leq \ \lpf{N}^{-\beta},
		\end{equation}
		so that
		\begin{equation}
			M_2/|R|^{n+1} + \lpf{N}^{-\gamma'} \ \leq \ 2\ \lpf{N}^{-\min\{\beta,\gamma'\}} \ \leq \ \lpf{N}^{-\gamma}
		\end{equation}
		by choice of $\gamma$ and $C$. This inequality and the assumption \eqref{ccc eqn 3} imply \eqref{ccc eqn 4}.
		
		The second inequality in \eqref{ccc eqn 1} is proven similarly, but it is more straightforward.
	\end{proof}

	\section{A PET induction argument}\label{sec: proof of prop: Us control}
	The goal of this section is to prove Proposition~\ref{prop: Us control} via a PET induction argument.
	
	Here is a very rough sketch of the proof of Proposition~\ref{prop: Us control} that omits all the complications unique to our setting. The main prerequisites to the proof are two lemmas, which we state later. The first lemma (Lemma~\ref{lem: Ud bound if invertible}) asserts that, over finite commutative rings, Gowers uniformity (semi)norms  $\norm{\cdot}_{U^s}$ control certain linear averages, and the second lemma (Lemma~\ref{linearization step}) relates a given polynomial average to another polynomial average with lower complexity. To prove Proposition~\ref{prop: Us control}, we apply the second lemma repeatedly until some average is obtained to which the first lemma may be applied, then synthesize the bounds arising from the second lemma.
	
	The rest of this section is organized as follows.
	
	We first prove Lemma~\ref{lem: Ud bound if invertible} in Subsection~\ref{subsec: proof of lem: Ud bound if invertible} and Lemma~\ref{linearization step} in  Subsection~\ref{subsec: proof of linearization step}.
	
	Next, in Subsection~\ref{subsec: proof of prop: Us control in case y, y2}, we give an ad hoc proof of Proposition~\ref{prop: Us control} in the special case that $\mathbf P = \{y,y^2\} \subset \Z[y]$. This proof will hint at the ways that we handle the presence of zero divisors in a general commutative ring. In the rest of Subsection~\ref{subsec: proof of prop: Us control in case y, y2}, with the benefit of the specificity provided by the ad hoc proof, we describe potential complications of the general proof of Proposition~\ref{prop: Us control}.
	
	Afterwards, we will collect and develop the technical notions that we use to properly formulate the argument in the $\{y,y^2\}$ and general cases. The notions of $\zn$-height and essential distinctness modulo $N$, although already introduced in either the introduction or in the ad hoc proof, are further developed in Subsections~\ref{subsec: zn height}~and~\ref{subsec: essential distinctness}. Weight sequences are introduced in Subsection~\ref{subsec: weight sequences}, and permissible operations in Subsection~\ref{subsec: PET algorithms}. As will be described in the analysis at the end of Subsection~\ref{subsec: proof of prop: Us control in case y, y2}, one of the two main lemmas, Lemma~\ref{linearization step}, must be rewritten and modified in terms of these technical notions; in Subsection~\ref{subsec: proof of lem: PET inductive step} we prove Lemma~\ref{lem: PET inductive step}, the upgraded version of Lemma~\ref{linearization step} suitable for our actual proof of Proposition~\ref{prop: Us control}.
	
	Finally, we prove Proposition~\ref{prop: Us control} in Subsection~\ref{subsec: proof of prop: Us control}.
	
	\subsection{Gowers norms control certain linear averages in rings}\label{subsec: proof of lem: Ud bound if invertible}
	The goal of this subsection is to prove Lemma~\ref{lem: Ud bound if invertible}. We first need to define some notation and prove an auxiliary lemma.
	
	Given a finite commutative ring $R$ and a vector $(a_0,a_1) \in R^2$, write $(a_0,a_1)^\perp = (a_1,-a_0)$, and for a positive integer $n$, given $(a_1,\ldots, a_n)$ and $(b_1,\ldots, b_n) \in R^n$, write
	\begin{equation} (a_1,\ldots, a_n) \cdot (b_1,\ldots,b_n) = \sum_{i=1}^n a_ib_i \in R.
	\end{equation}
	The following lemma is an ``invertible'' version of \cite[Lemma 5.2]{prend} in the ring setting.
	\begin{lemma}\label{lem: Ud estimate} Let $d \geq 2$ and $n \geq 1$ be integers. Let $R$ be a finite commutative ring with characteristic $N$. Let $R' = R^n$ with coordinatewise addition and multiplication. Let $g_0,g_1,\ldots, g_d : R' \to \C$ be 1-bounded functions. Let $\vec{a_2}, \ldots, \vec{a_d} \in (\mathbb{Z}_N^n)^2 \subset (R')^2$ be such that, for each $i \in \{2,\ldots, d\}$, both entries of $\vec{a_i}$ are invertible in $R'$ and such that $\vec{a_i} \cdot \vec{a_j}^\perp$ is invertible in $R'$ whenever $i \neq j$. Then we have
		\be
		\left| \cavg{z_0,z_1}{R'} g_0(z_0)g_1(z_1)\prod_{i=2}^d g_i(\vec{a_i} \cdot \vec{z})\right| \ \leq \ ||g_d||_{U^d},
		\ee
		writing $\vec{z} = (z_0,z_1) \in (R')^2$. For clarity, in the formulation and proof of this lemma, the only dot products that are written with the symbol $\cdot$ are taken with vector entries in $R'$.
	\end{lemma}
	\begin{proof}
		We induct on $d$. Consider the base case $d = 2$, and write $\vec{a_2} = (a_0,a_1)$. Recall that if $X = (\chi_1,\ldots, \chi_n) \in \widehat{R'}$ and $c = (c_1,\ldots, c_n) \in \mathbb{Z}_N^n$, then $X^{c}$ denotes $(\chi_1^{c_{1}},\ldots, \chi_n^{c_{n}}) \in \widehat{R'}$. By Lemma~\ref{config rearrangement product ver} and H\"{o}lder's inequality, one obtains
		\ba
		\left| \sum_{z_0,z_1 \in R'} g_0(z_0)g_1(z_1)g_2(\vec{a_2} \cdot \vec{z})\right| \ & = \ \left| \sum_{\substack{z_0,z_1,z_2 \in R' \\ a_0z_0 + a_1z_1 - z_2 = 0_{R'}}} g_0(z_0)g_1(z_1)g_2(z_2) \right| \\
		& = \ |R'|^3 \left| \cavg{X}{\widehat{R'}} \hat{g_0}(X^{a_0})\hat{g_1}(X^{a_1})\hat{g_2}(X^{-1}) \right|\\
		& \leq \ |R'|^3 \norm{g_0'g_1'\ol{\hat{\ol{g_2}}} }_{L^1} \\
		& \leq \ |R'|^3 \norm{g_0'}_{L^2} \norm{g_1'}_{L^4}\norm{\ol{\hat{\ol{g_2}}}}_{L^4},
		\ea
		where $g_i'(X) := \hat{g_i}(X^{a_i})$ for $i \in \{0,1\}$.
		Since $a_0, a_1 \in \mathbb{Z}_N^n$ are each invertible in $R'$, by Lemma~\ref{Fourier dilate trivial bound} one has $\norm{g_0'}_{L^2} \leq |R'|^{-1/2}$ and $\norm{g_1'}_{L^4} \leq |R'|^{-1/4}$. By Lemma~\ref{lem: u2l4}, one also has $\norm{\ol{\hat{\ol{g_2}}}}_{L^4} = \norm{\hat{\ol{g_2}}}_{L^4} = |R'|^{-1/4}\norm{\ol{g_2}}_{U^2} = |R'|^{-1/4}\norm{g_2}_{U^2}$. This proves the base case.
		
		Now suppose $d >2$, and still write $\vec{a_2} = (a_0,a_1)$. Define the function $G : (R')^2 \to \C$ by $G(\vec{z}) = g_0(z_0)g_1(z_1) \prod_{i=3}^d g_i(\vec{a_i} \cdot \vec{z})$. Then 
		\ba
		\sum_{z_0,z_1 \in R'} g_0(z_0)g_1(z_1)\prod_{i=2}^d g_i(\vec{a_i} \cdot \vec{z}) \ & = \  \sum_{z_0,z_1 \in R'} G(\vec{z})g_2(a_0z_0 + a_1z_1)  \\
		& = \  \sum_{\substack{z_0,z_1,z_2 \in R' \\ a_0z_0 + a_1z_1 - z_2 = 0_{R'}}} G(\vec{z})g_2(z_2)  \\
		& = \ \sum_{z_0,z_1,z_2\in R'} G(\vec{z})g_2(z_2) \cavg{X}{\widehat{R'}} X(a_0z_0+a_1z_1-z_2) \\
		& = \ \cavg{X}{\widehat{R'}} \sum_{z_0,z_1 \in R'} G(\vec{z})X(\vec{a_2} \cdot \vec{z}) \sum_{z_2 \in R'} g_2(z_2)X(-z_2) \\
		& = \ |R'| \cavg{X}{\widehat{R'}} H(X) g_2'(X),
		\ea
		where $H(X) := \sum_{z_0,z_1 \in R'} G(\vec{z})X(\vec{a_2}\cdot \vec{z})$ and $g_2'(X) := \hat{g_2}(X^{-1})$. By Cauchy--Schwarz,
		\begin{equation*}
			\left|    \sum_{z_0,z_1 \in R'} g_0(z_0)g_1(z_1)\prod_{i=2}^d g_i(\vec{a_i} \cdot \vec{z})\right\vert \ \leq \ |R'| \norm{H}_{L^2} \norm{g_2'}_{L^2}.
		\end{equation*}
		Note that $\norm{g_2'}_{L^2} \leq |R'|^{-1/2}$ holds by Lemma~\ref{Fourier dilate trivial bound}. Let us analyze the other norm. One can show that
		\be
		\norm{H}_{L^2}^2 \ = \ \sum_{\substack{\vec{z},\vec{w} \in (R')^2 \\ \vec{a_2} \cdot (\vec z - \vec w)=0_{R'}}} G(\vec z) \ol{G}(\vec w) = \sum_{\substack{\vec h \in (R')^2 \\ \vec{a_2} \cdot \vec{h} = 0_{R'}}} \sum_{\vec{z} \in (R')^2} \ol{G}(\vec z)G(\vec z + \vec h). 
		\ee
		Now fix $\vec h = (h_0,h_1) \in (R')^2$, and let $\tilde{g}_0 = \Delta_{h_0}g_0$, $\tilde{g}_1 = \Delta_{h_1}g_1$, and $\tilde{g}_i = \Delta_{\vec{a_i} \cdot \vec h} g_i$ for each $i \geq 3$. Then by the induction hypothesis, we have
		\be
		\left| \cavg{z_0,z_1}{R'} \ol{G}(\vec z)G(\vec z + \vec h) \right| \ = \ \left| \cavg{z_0,z_1}{R'} \tilde{g}_0(z_0) \tilde{g}_1(z_1) \prod_{i=3}^d \tilde{g}_i(\vec{a_i} \cdot \vec z) \right| \ \leq \ \norm{\Delta_{\vec{a_d} \cdot \vec h} g_d}_{U^{d-1}}.
		\ee
		We claim that the set $S = \{ \vec h \in (R')^2 : \vec{a_2} \cdot \vec{h} = 0_{R'} \}$ is in bijective correspondence with $R'$ via the map $\psi : R' \to S$ defined by $\psi(x) = (a_1x,-a_0x)$. First, $\psi$ is injective. Indeed, if $\psi(x) = \psi(y)$, then $(a_1x, -a_0x) = (a_1y, -a_0y)$, which implies $x = y$ since $a_0$ and $a_1$ are invertible in $R'$. Next, we have $|S| = |R'|$. Indeed, for each $x \in R'$, we observe that $(x, -a_0 \inv a_1 x) \in S$; moreover, if $x,y \in R'$ are such that $(x,y) \in S$, then the equation $y = -a_0\inv a_1 x$ uniquely determines $y$, proving the claim. With this bijection in hand, we can rewrite
		\be
		\sum_{\substack{\vec h \in (R')^2 \\ \vec{a_2} \cdot \vec{h} = 0_{R'}}} \norm{\Delta_{\vec{a_d} \cdot \vec h} g_d}_{U^{d-1}} = \sum_{x \in R'} \norm{\Delta_{(\vec{a_d} \cdot \vec{a_2}^\perp)x} g_d}_{U^{d-1}} = \sum_{x \in R'} \norm{\Delta_{x} g_d}_{U^{d-1}},
		\ee
		changing variables $x \mapsto (\vec{a_d} \cdot \vec{a_2}^\perp)^{-1}x$ in the second equality with the help of the fact that $\vec{a_d} \cdot \vec{a_2}^\perp$ is invertible in $R'$. By repeated applications of Cauchy--Schwarz, we observe that
		\be
		\left(\sum_{x \in R'} \norm{\Delta_{x} g_d}_{U^{d-1}} \right)^{2^{d-1}} \leq \ |R'|^{2^{d-1}-1} \sum_{x \in R'} \norm{\Delta_{x} g_d}_{U^{d-1}}^{2^{d-1}} \ = \ |R'|^{2^{d-1}}\norm{g_d}_{U^d}^{2^d}.
		\ee
		Combining the previous observations, we see that
		\begin{equation*}
			\norm{H}_{L^2} \ \leq |R'|^{3/2} \norm{g_d}_{U^d}
		\end{equation*}
		and hence conclude that
		\be
		\left| \sum_{z_0,z_1 \in R'} g_0(z_0)g_1(z_1)\prod_{i=2}^d g_i(\vec{a_i} \cdot \vec{z})\right| \ \leq \ |R'|^{1-\frac{1}{2}+\frac{3}{2}} \norm{g_d}_{U^d} = |R'|^{2}\norm{g_d}_{U^d},
		\ee
		completing the proof.
	\end{proof}
	Now we are ready to prove the following lemma, which asserts that Gowers norms control ``invertible'' linear averages.
	\begin{lemma}\label{lem: Ud bound if invertible}
		Let $d$ and $n$ be positive integers. Let $f_0,f_1,\ldots, f_d : R \to \C$ be 1-bounded functions on a finite commutative ring $R$ with characteristic $N$. Let $R' = R^n$ with coordinatewise addition and multiplication. Let $a^{(i)} \in \mathbb{Z}_N^n$, $i \in \{1,\ldots, d\}$, be invertible elements of $R'$ such that $a^{(i)}-a^{(j)}$ is invertible in $R'$ whenever $i \neq j$. Then
		\be \label{Ud bound in linear case}
		\left| \rcavg{x} \cavg{y}{R'} f_0(x) \prod_{i=1}^d  f_i(x+a^{(i)}\cdot y)\right| \ \leq \  ||f_d||_{U^d}.
		\ee
		For clarity, in the formulation and proof of this lemma, unless otherwise indicated, dot products are taken with vector entries in $R$.
	\end{lemma}
	\begin{proof}
		In this proof, we write $y = (y_1,\ldots, y_n) \in R'$ and $a^{(i)} = (a^{(i)}_1,\ldots, a^{(i)}_n) \in R'$.
		
		Suppose $d = 1$. By assumption, for every $i \in \{1,\ldots, n\}$, $a^{(1)}_i$ is a unit in $R$. Then, by, for example, the change of variables $y_1 \mapsto (a^{(1)}_{1})^{-1}(-x+y_1-\sum_{j\neq 1} a^{(1)}_j y_j)$, we obtain
		\be
		\left| \rcavg{x}\cavg{y}{R'} f_0(x)f_1(x+a^{(1)}\cdot y) \right| \ = \ \left| \rcavg{x} f_0(x) \rcavg{y_1} f_1(y_1) \right|
		\ \leq \ \norm{f_1}_{U^1}
		\ee
		by the 1-boundedness of $f_0$.
		
		Suppose $d \geq 2$. Observe that for any function $g : R' \to \C$, one has
		\be
		\cavg{y}{R'} g(y) = \cavg{y,z_0,z_1}{R'} g(y-z_0+z_1).
		\ee
		For each $x \in R$, applying this observation to the left-hand side of \eqref{Ud bound in linear case} with $g(y) = \prod_{i=1}^d  f_i(x+a^{(i)}\cdot y)$ and changing variables $x \mapsto x + a^{(1)}\cdot z_0$, one obtains
		\begin{multline}\label{maximum display}
			\left| \rcavg{x}\cavg{y}{R'} f_0(x) \prod_{i=1}^d  f_i(x+a^{(i)}\cdot y)\right| \\
			= \Bigg| \rcavg{x}\cavg{y,z_0,z_1}{R'}f_0(x+a^{(1)}\cdot z_0)f_1(x+a^{(1)}\cdot (y+z_1))  \prod_{i=2}^d f_i(x+a^{(i)}\cdot y+(a^{(1)}-a^{(i)})\cdot z_0 + a^{(i)}\cdot z_1)  \Bigg| \\
			\leq \max_{(x,y) \in R \times R'} \Bigg| \cavg{z_0,z_1}{R'}f_0(x+a^{(1)}\cdot z_0)f_1(x+a^{(1)}\cdot (y+z_1)) \prod_{i=2}^d f_i(x+a^{(i)} \cdot y+(a^{(1)}-a^{(i)})\cdot z_0+a^{(i)} \cdot z_1)  \Bigg|.
		\end{multline}
		Let $(x', y') \in R \times R'$ be such that the maximum is achieved in \eqref{maximum display}. Let $g_0(z) := f_0(x'+a^{(1)}\cdot z)$ and $g_1(z) := f_1(x'+a^{(1)}\cdot(y'+z))$, and for $i \geq 2$, let $g_i(z) := f_i(x'+a^{(i)}\cdot y'+1_{R'}\cdot z)$ and $\vec{a_i} := (a^{(1)}-a^{(i)},a^{(i)}) \in (R')^2$. By assumption, the $\vec{a_i}$ satisfy the hypotheses of Lemma~\ref{lem: Ud estimate}; thus, one has
		\ba
		\left| \rcavg{x} \cavg{y}{R'} f_0(x) \prod_{i=1}^d  f_i(x+a^{(i)}\cdot y)\right| \ & \leq \ \left| \cavg{z_0,z_1}{R'} g_0(z_0)g_1(z_1)\prod_{i=2}^d g_i(\vec{a_i} \cdot \vec{z}) \right| \\
		& \leq \  ||g_d||_{U^d},
		\ea
		where $\vec{z} = (z_0,z_1) \in (R')^2$ and the dot product here is taken with vector entries in $R'$ to match the notation in the applied lemma. Finally, we observe that $||g_d||_{U^d(R')} = ||f_d||_{U^d(R)}$, which can be verified in a trivial but cumbersome way from the definition of the $U^d$ norm.
	\end{proof}
	
	Let us exhibit an example which demonstrates that the invertibility hypotheses are necessary in Lemma~\ref{lem: Ud bound if invertible}. Let $R = \mathbb{Z}_6$, and define the 1-bounded functions $f_0,f_1 : R \to \C$ by
	$f_1(0) = f_1(3) = e^{\pi i/4}$, $f_1(1) = f_1(4) = e^{-\pi i /4}$, $f_1(2) = f_1(5) = e^{3\pi i /8}$ and $f_0(x) = 1/f_1(x)$ for each $x \in R$. Then
	\begin{equation}
		\rcavg{x,y} f_0(x) f_1(x+3y) \ = \ \rcavg{x,y} 1 \ = \ 1,
	\end{equation}
	but $\norm{f_1}_{U^1} = \left| \frac{e^{\pi i/4} + e^{-\pi i /4} + e^{3\pi i /8}}{3} \right| < 1$. Moreover, there is even no way to salvage the lemma by allowing the bound on the right-hand side of \eqref{Ud bound in linear case} to be $\norm{f_d}_{U^s}$, where $s > d$. For the function $f_1$ above, some of its discrete derivatives can be computed as follows:
	\begin{equation}
		\begin{array}{ |l||c|c|c| } 
			\hline
			& x= 0, 3 & x = 1, 4 & x = 2, 5 \\
			\hline
			\Delta_1 f_1(x) & -i & -1 & -i \\ 
			\hline
			\Delta_{1,1} f_1(x) & -i & i & 1 \\ 
			\hline
			\Delta_{1,1,1} f_1(x) & -1 & -i & -i \\ 
			\hline
			\Delta_{1,1,1,1} f_1(x) & i & 1 & -i \\ 
			\hline
			\Delta_{1,1,1,1,1} f_1(x) & -i & -i & -1 \\ 
			\hline
			\Delta_{1,1,1,1,1,1} f_1(x) & 1 & -i & i \\ 
			\hline
			\Delta_{1,1,1,1,1,1,1} f_1(x) & -i & -1 & -i \\ 
			\hline
		\end{array}
	\end{equation}
	Note that $\Delta_{1,1,1,1,1,1,1} f_1 = \Delta_1 f_1$.  
	Then, for example, to show that $\norm{f_1}_{U^2} < 1$, we estimate
	\begin{multline}
		\norm{f_1}_{U^2}^{2^2} \ = \ \rcavg{x,h_1,h_2} \Delta_{h_1,h_2} f_1(x) \ = \ \frac{1}{|R|^2} \left(\rcavg{x} \Delta_{1,1}f_1(x)\right) + \frac{1}{|R|^3} \sum_{\substack{x,h_1,h_2 \in R \\ (h_1,h_2) \neq (1,1) }} \Delta_{h_1,h_2}f_1(x) \\
		\leq \ \frac{1}{|R|^2} \left| \rcavg{x} \Delta_{1,1}f_1(x)\right| + \frac{1}{|R|^3} \sum_{\substack{x,h_1,h_2 \in R \\ (h_1,h_2) \neq (1,1) }} \left| \Delta_{h_1,h_2}f_1(x)  \right| \\
		\leq \ \frac{1}{|R|^2} \left| \rcavg{x} \Delta_{1,1}f_1(x)\right| + \frac{|R|(|R|^2-1)}{|R|^3} \ < \ \frac{1}{|R|^2} + \frac{|R|(|R|^2-1)}{|R|^3},
	\end{multline}
	where the first inequality holds by the triangle inequality, the second inequality holds trivially since the discrete derivative is 1-bounded, and the third inequality holds since $\Delta_{1,1}f_1$ has constant magnitude 1 but is not itself constant. Similarly, since $\Delta_{1,1,1} f_1$ has constant magnitude 1 but is not constant, it follows that $\norm{f_1}_{U^3} < 1$ as well, and so on.
	
	\subsection{A van der Corput-type estimate}\label{subsec: proof of linearization step}
	The goal of this subsection is to prove Lemma~\ref{linearization step}, which can be viewed as a finite ring analogue of the classical van der Corput inequality (see, e.g., Theorem 2.2 and the following remark in \cite{ertau}). We first prove two auxiliary lemmas, the first of which is morally a special case of \cite[Lemma 3.1]{prend}. We include the proof here for completeness.
	\begin{lemma} \label{lem: vdc inequality}
		Let $n$ be a positive integer, and let $R$ be a finite commutative ring with characteristic $N$. Let $g : R' \to \C$ be a function on the ring $R' = R^n$ with coordinatewise addition and multiplication. Let $\mathcal H \subset \mathbb{Z}_N^n$ be nonempty. For $h = (h_1,\ldots,h_n) \in \mathbb{Z}_N^n$, let $r_\mathcal{H}(h)$ denote the number of pairs $(h',h'') \in \mathcal{H}^2$ such that $h'_i-h''_i \equiv h_i \bmod N$ for each $i$. Then
		\be
		\left| \cavg{y}{R'} g(y) \right|^2 \ \leq \ \frac{1}{|\mathcal H|^2} \sum_{h \in \mathbb{Z}_N^n} r_\mathcal{H}(h)  \cavg{y}{R'} g(y+h)\ol{g}(y).
		\ee
	\end{lemma}
	\begin{proof}
		Observe that
		\[\cavg{y}{R'} g(y) = \frac{1}{|\mathcal H|} \cavg{y}{R'} \sum_{h\in\mathcal H} g(y+h).\]
		Thus, by Cauchy--Schwarz, we obtain
		\begin{align*}
			\left| \cavg{y}{R'} g(y) \right|^2 \ & = \ \frac{1}{|\mathcal H|^2} \left| \cavg{y}{R'} \sum_{h\in\mathcal H} g(y+h) \right|^2 \\
			& \leq \ \frac{1}{|\mathcal H|^2} \cavg{y}{R'} \left| \sum_{h\in\mathcal H} g(y+h) \right|^2 \\
			& = \ \frac{1}{|\mathcal H|^2} \cavg{y}{R'}  \sum_{h,h'\in\mathcal H} g(y+h)\ol{g}(y+h') \\
			& = \ \frac{1}{|\mathcal H|^2} \cavg{y}{R'}  \sum_{h,h'\in\mathcal H} g(y+h-h')\ol{g}(y) \\
			& = \ \frac{1}{|\mathcal H|^2} \sum_{h\in\mathbb{Z}_N^n} r_\mathcal{H}(h) \cavg{y}{R'} g(y+h)\ol{g}(y),
		\end{align*}
		as desired.
	\end{proof}
	
	\begin{lemma}\label{vdc step} Let $n$ and $H$ be positive integers. Let $R$ be a finite commutative ring with characteristic $N$, and let $R' = R^n$ with coordinatewise addition and multiplication. Let $g : R \times R' \to \C$ be a 1-bounded function. Let $\mathcal{H}_1 \subset \mathbb{Z}_N^n$. If the set $(\{0,1,\ldots, H-1\}\cup \{N-(H-1),N-(H-1)+1,\ldots, N-1\})^n \setminus \mathcal{H}_1$ is nonempty, then it contains an $h$ such that
		\be
		\rcavg{x} \left| \cavg{y}{R'} g(x,y) \right|^2 \ \leq \ 2^n\left(\frac{|\mathcal{H}_1|}{H^n} + \left| \cavg{(x,y)}{R\times R'} g(x,y+h)\ol{g}(x,y)\right|\right).
		\ee 
	\end{lemma}
	\begin{proof}
		Let $\mathcal H = \{1,2,\ldots, H\}^n \subset \mathbb{Z}_N^n$. For each $x \in R$, apply Lemma~\ref{lem: vdc inequality} for this $\mathcal H$ and the function $g_x(y) = g(x,y)$ to obtain
		\begin{align}
			\rcavg{x} \left| \cavg{y}{R'} g(x,y) \right|^2 \ & \leq \ \rcavg{x} \frac{1}{|\mathcal{H}|^2} \sum_{h \in \mathbb{Z}_N^n} r_\mathcal{H}(h)  \cavg{y}{R'} g_x(y+h)\ol{g_x}(y) \nonumber \\
			& \leq \ \frac{1}{|\mathcal H|} \sum_{h \in \mathbb{Z}_N^n} \frac{r_\mathcal{H}(h)}{|\mathcal H|} \left| \cavg{(x,y)}{R \times R'} g(x,y+h)\ol{g}(x,y) \right|. \label{eqn: vdc step eqn 1}
		\end{align}
		Let us make some simplifications.
		First, the summand in \eqref{eqn: vdc step eqn 1} is zero for all $h = (h_1,\ldots, h_n)$ outside of $S_0 := (\{0,1,\ldots, H-1\}\cup \{N-(H-1),N-(H-1)+1,\ldots, N-1\})^n$, for if some $h_j$ is outside $\{0,1,\ldots, H-1\}\cup \{N-(H-1),N-(H-1)+1,\ldots, N-1\}$, then we must have $H \leq N/2$ and $h_j \in \{H,H+1,\ldots, N-H\} \subset \Z_N$, but, in $\Z_N$, no element of $\{H,H+1,\ldots, N-H\}$ is a difference of two elements in $\{1,\ldots,H\}$, hence $r_\mathcal{H}(h) = 0$.
		
		Second, set $S_1 := S_0 \cap \mathcal{H}_1$ and $S_2 :=  S_0 \setminus \mathcal{H}_1$. We crudely bound $|S_1|$ and $|S_2|$. We have $|S_1| \leq |\mathcal{H}_1| \leq 2^n |\mathcal{H}_1|$. Moreover, since $S_2 = (\{0,1,\ldots,H-1\}\cup \{N-(H-1),\ldots,N-1\})^n \setminus \mathcal{H}_1$, we also have $|S_2| \leq (2(H-1)-1)^n) \leq 2^nH^n = 2^n|\mathcal{H}|$. 
		Third, for any $h \in \mathbb{Z}_N^n$ we denote $\left|\cavg{(x,y)}{R\times R'} g(x,y+h)\ol{g}(x,y)\right|$ by $G(h)$. Hence, \eqref{eqn: vdc step eqn 1} simplifies to
		\begin{align}
			\rcavg{x} \left| \cavg{y}{R'} g(x,y) \right|^2 \ & \leq \ \frac{1}{|\mathcal H|} \sum_{h \in \mathbb{Z}_N^n} \frac{r_\mathcal{H}(h)}{|\mathcal H|} G(h) \nonumber \\
			& = \ \frac{1}{|\mathcal H|}\sum_{h \in S_1}  \frac{r_\mathcal{H}(h)}{|\mathcal H|} G(h) + \frac{1}{|\mathcal H|}\sum_{h \in S_2} \frac{r_\mathcal{H}(h)}{|\mathcal H|} G(h) \nonumber \\
			& \leq \ \frac{1}{|\mathcal H|}\sum_{h \in S_1}  G(h) + \frac{1}{|\mathcal H|}\sum_{h \in S_2} G(h) \label{eqn: vdc step eqn 2} \\
			& \leq \ \frac{|S_1|}{|\mathcal H|} + \frac{|S_2|}{|\mathcal H|}\max_{h \in S_2} G(h) \label{eqn: vdc step eqn 3} \\
			& \leq \ 2^n \left( \frac{|\mathcal{H}_1|}{H^n} + \max_{h \in S_2} G(h) \right) \nonumber,
		\end{align}
		where in \eqref{eqn: vdc step eqn 2} we have used the trivial estimate $r_\mathcal{H}(h)/|\mathcal{H}| \leq 1$ and in \eqref{eqn: vdc step eqn 3} we have used the bound $G(h) \leq 1$, valid since $g$ is 1-bounded. The result follows since $S_2$ is assumed nonempty.
	\end{proof}
	To simplify notation, for $z \in \C$, we let $\mathcal{C}z$ denote the complex conjugate $\overline{z}$; we use this notation only when $z$ is some complex-valued function with a long argument.
	\begin{lemma} \label{linearization step}
		Let $m$, $n$, and $H$ be positive integers. Let $R$ be a finite commutative ring with characteristic $N$, and let $R' = R^n$ with coordinatewise addition and multiplication. Let $P_1,\ldots,P_m : R' \to R$ be functions. Let $f_0,f_1,\ldots, f_m : R \to \C$ be 1-bounded.  Let $\mathcal{H} \subset \mathbb{Z}_N^n$ be a subset.  If the set $(\{0,1,\ldots, H-1\}\cup \{N-(H-1),N-(H-1)+1,\ldots, N-1\})^n \setminus \mathcal{H}$ is nonempty, then it contains an $h$ such that
		\begin{multline}
			\left| \Lambda_{P_1,\ldots,P_m}(f_0,\ldots,f_m) \right| \\ \leq \ 2^{n/2} \left[ \left( \frac{|\mathcal{H}|}{H^n} \right)^{1/2} + \left| \cavg{(x,y)}{R\times R'} \prod_{i=1}^m \prod_{\omega \in \{0,1\}}\mathcal{C}^{\omega+1} f_i(x + P_i(y+\omega h)-P_1(y))\right|^{1/2} \right].
		\end{multline}
		For clarity, $\omega+1$ in the exponent of $\mathcal C$ is one of the integers 1 or 2, whereas $\omega$ in the expression $\omega h$ refers to the ring element $0_{R'}$ or $1_{R'}$. Moreover, $\omega h$ is thus the product in $R'$ of $\omega$ and $h$, not the dot product with vector entries in $R$.
	\end{lemma}
	\begin{remark}
		The $h$ depends on the functions $f_i$. To see this, notice that in the proof of Lemma~\ref{vdc step}, the $h$ is chosen to maximize the quantity $G(h)$, which depends on the quantity $g(x,y)$, which, as we will see at the beginning of the proof of Lemma~\ref{linearization step}, depends on the functions $f_i$. 
	\end{remark}
	\begin{proof}
		Set $g(x,y) := \prod_{i=1}^m f_i(x+P_i(y))$. By applying Cauchy--Schwarz, the 1-boundedness of $f_0$, and Lemma~\ref{vdc step}, one obtains for some $h \in (\{0,1,\ldots, H-1\}\cup \{N-(H-1),N-(H-1)+1,\ldots, N-1\})^n \setminus \mathcal{H}$ that
		\ba
		\left| \Lambda_{P_1,\ldots,P_m}(f_0,\ldots,f_m) \right|^2 \ & = \ \left| \rcavg{x} f_0(x) \cavg{y}{R'} g(x,y) \right|^2 \\
		& \leq \ \rcavg{x} |f_0(x)|^2 \left| \cavg{y}{R'} g(x,y) \right|^2 \\
		& \leq \ \rcavg{x} \left| \cavg{y}{R'} g(x,y) \right|^2 \\
		& \leq \ 2^n \left( \frac{|\mathcal{H}|}{H^n}+\left| \cavg{(x,y)}{R\times R'} g(x,y+h)\ol{g}(x,y) \right| \right).
		\ea
		The result follows after noting that $(a+b)^{1/2} \leq a^{1/2} + b^{1/2}$ holds for nonnegative real numbers $a$ and $b$ and changing variables $x \mapsto x - P_1(y)$.
	\end{proof}
	
	\subsection{Ad hoc proof of Proposition~\ref{prop: Us control} in the case $\mathbf P = \{y,y^2\}$}\label{subsec: proof of prop: Us control in case y, y2}
	We first recall the notion of $\zn$-height and prove a basic lemma about it.
	\begin{definition}
		Let $a \in \Z$ and $N \in \N$. The \emph{$\zn$-height} of $a$ is the minimal nonnegative integer $H$ such that $a \bmod N \in \{0,1,\ldots,H\}\cup\{N-H,\ldots, N-1\}$.
		Given a real number $x$, define $\norm{x}_{\R / \Z} := \min\{ |n - x| : n \in \Z\}$ to be the distance to the closest integer. Equivalently, the $\zn$-height of $a$ is the minimal nonnegative integer $H$ such that $\norm{a/N}_{\R/\Z} \leq H/N$.
	\end{definition}
	The following lemma confirms that $\zn$-heights of sums or products of integers can be bounded in a simple way.
	\begin{lemma}\label{lem: manipulating 0-closeness mod N}
		Let $N$ be a positive integer. For $i \in \{1,2\}$, let $a_i$ and $H_i$ be integers such that the $\zn$-height of $a_i$ is at most $H_i$. Then the $\zn$-height of $a_1+a_2$ is at most $H_1+H_2$, and the $\zn$-height of $a_1a_2$ is at most $H_1H_2$. 
	\end{lemma}
	\begin{proof}
		For $i\in\{1,2\}$, let $n_i$ be an integer such that $|n_i - a_i/N| = \norm{a_i/N}_{\R /\Z}$. By the triangle inequality,
		\begin{equation}
			\norm{(a_1+a_2)/N}_{\R/\Z} \ \leq \ |n_1+n_2 - (a_1+a_2)/N| \ \leq \ \norm{a_1/N}_{\R/\Z} + \norm{a_2/N}_{\R/\Z} \ \leq \ (H_1+H_2)/N.
		\end{equation}
		Similarly,
		\begin{multline}
			\norm{a_1a_2/N}_{\R/\Z} \ \leq \ |-Nn_1n_2 + n_2a_1 +n_1a_2 - a_1a_2/N| \ = \ N\left| n_1n_2 - n_2a_1/N - n_1a_2/N +a_1a_2/N^2\right| \\
			= \ N \left| n_1-a_1/N \right|\left| n_2-a_2/N \right| \ = \ N \norm{a_1/N}_{\R/\Z}\norm{a_1/N}_{\R/\Z} \ \leq \ H_1H_2/N,
		\end{multline}
		as desired.
	\end{proof}
	
	Now we give an ad hoc proof of Proposition~\ref{prop: Us control} in the special case that $\mathbf P = \{y,y^2\} \subset \Z[y]$.
	\begin{prop}\label{example: prop: Us control}
		There exist $\lambda \in (0,1]$, $C_0, C_1 > 0$, and $s \in \N$ with $s \geq 2$ such that, for any finite commutative ring $R$ with characteristic $N$ satisfying $\lpf N > C_1$ and any 1-bounded functions $f_0,f_1, f_2 : R \to \C$, one has
		\be\label{example: eqn: negative power of lpf N}
		\left| \rcavg{x,y} f_0(x)f_1(x+y)f_2(x+y^2)\right| \ \leq \ \frac{C_0}{\lpf{N}^{1/2}} + 2\min_{0\leq j \leq 2} \norm{f_j}_{U^s}^\lambda.
		\ee
	\end{prop}
	
	\begin{proof}
		The constant $C_1$ will be chosen later.
		
		Let $R$ be a finite commutative ring with characteristic $N$ satisfying $\lpf N > C_1$, and let $f_0,f_1, f_2 : R \to \C$ be 1-bounded functions.
		
		The desired inequality \eqref{example: eqn: negative power of lpf N} holds if, for each $j \in \{0,1,2\}$,
		\be\label{example: eqn: negative power of lpf N, nonminimum version}
		\left| \rcavg{x,y} f_0(x)f_1(x+y)f_2(x+y^2) \right| \ \leq \ \frac{C_0}{\lpf{N}^{1/2}} + 2\norm{f_j}_{U^s}^\lambda.
		\ee
		We will first show \eqref{example: eqn: negative power of lpf N, nonminimum version} in the case $j = 2$.
		
		Applying Lemma~\ref{linearization step} with the settings $m=2$, $n=1$, $H = H_1$ to be decided later, $P_1(y) = y$, $P_2(y) = y^2$, and $\mathcal H = \mathcal H_1$ to be decided later, there exists $h_1 \in (\{0,1\ldots, H_1-1\} \cup\{N-(H_1-1),N-(H_1-1)+1,\ldots, N-1\}) \setminus \mathcal{H}_1$ such that
		\begin{multline}\label{example: basic application of lemma about linearization 1}
			\left| \rcavg{x,y} f_0(x)f_1(x+y)f_2(x+y^2) \right| \\ \leq \ 2^{1/2} \left[\left(\frac{|\mathcal H_1|}{H_1}\right)^{1/2}   + \left| \cavg{x,y}{R} \prod_{i=1}^2 \prod_{\omega \in \{0,1\}}\mathcal{C}^{\omega+1} f_i(x + P_i(y+\omega h_1)-P_1(y))\right|^{1/2} \right].
		\end{multline}
		
		Let $a_1 = P_1(h_1)-P_1(0)$; then
		\be
		\ol{f_1}(x+P_1(y)-P_1(y))f_1(x+P_1(y+h_1)-P_1(y)) = \Delta_{a_1}f_1(x+(P_1-P_1)(y)) = \Delta_{a_1}f_1(x).
		\ee
		Set $g_0 := \Delta_{a_1}f_1$. Set $g_{1} := \ol{f_{2}}$, $g_{2} := f_{2}$, $Q_{1}(y) := y^2 - y$, and $Q_{2}(y):= (y+h_1)^2 - y$. It is easy to check that
		\be
		g_0(x)g_1(x+Q_1(y)) g_{2}(x+Q_{2}(y)) = \prod_{i=1}^2 \prod_{\omega\in\{0,1\}} \mathcal{C}^{\omega+1}f_i(x+P_i(y+\omega h_1) -P_1(y)).
		\ee
		By construction, $g_{2} = f_2$, the $g_i$ are 1-bounded, and \eqref{example: basic application of lemma about linearization 1} implies
		\begin{multline}\label{example: intermediate eqn in PET inductive step}
			\left| \rcavg{x,y} f_0(x)f_1(x+y)f_2(x+y^2)\right| \\ \leq \ 2^{1/2} \left[ \left(\frac{|\mathcal H_1|}{H_1}\right)^{1/2} + \left| \rcavg{x,y} g_0(x)g_1(x+Q_1(y))g_2(x+Q_2(y))\right|^{1/2}\right].
		\end{multline}
		Applying Lemma~\ref{linearization step} with the settings $m=2$, $n=1$, $H = H_2$ to be decided later, $P_1 = Q_1$, $P_2 = Q_2$, and $\mathcal H = \mathcal H_2$ to be decided later, there exists $h_2 \in (\{0,1\ldots, H_2-1\} \cup\{N-(H_2-1),N-(H_2-1)+1,\ldots, N-1\}) \setminus \mathcal{H}_2$ such that
		\begin{multline}\label{example: basic application of lemma about linearization 2}
			\left| \rcavg{x,y} g_0(x)g_1(x+Q_1(y))f_2(x+Q_2(y)) \right| \\ \leq \ 2^{1/2} \left[\left(\frac{|\mathcal H_2|}{H_2}\right)^{1/2}   + \left| \cavg{x,y}{R} \prod_{i=1}^2 \prod_{\omega \in \{0,1\}}\mathcal{C}^{\omega+1} g_i(x + Q_i(y+\omega h_2)-Q_1(y))\right|^{1/2} \right].
		\end{multline}
		Set $g_0' := \ol{g_1}$, $g_1' := g_1$, $g_2' := \ol{g_2}$, $g_3' := g_2$, and 
		\begin{align*}
			Q_1'(y) \ & := \ Q_1(y+h_2)-Q_1(y) \ = \ 2h_2y +h_2^2 - h_2, \\ Q_2'(y) \ & := \ Q_2(y)-Q_1(y) \ = \ 2h_1y+h_1^2, \\ Q_3'(y) \ & := \ Q_2(y+h_2) - Q_1(y) \ = \ 2(h_1+h_2)y + (h_1+h_2)^2-h_2.
		\end{align*} 
		Then we have $g_3' = g_2 = f_2$, the $g_i'$ are 1-bounded, and 
		\be
		g_0'(x)g_1'(x+Q_1'(y))g_{2}'(x+Q_{2}'(y))g_3'(x+Q_3'(y)) = \prod_{i=1}^2 \prod_{\omega\in\{0,1\}} \mathcal{C}^{\omega+1}g_i(x+Q_i(y+\omega h_2) -Q_1(y)),
		\ee
		so by \eqref{example: basic application of lemma about linearization 2}, we obtain
		\begin{multline}\label{example: intermediate eqn in PET inductive step 2}
			\left| \rcavg{x,y} g_0(x)g_1(x+Q_1(y))g_2(x+Q_2(y))\right| \\ \leq \ 2^{1/2} \left[ \left(\frac{|\mathcal H_2|}{H_2}\right)^{1/2} + \left| \rcavg{x,y} g_0'(x)g_1'(x+Q_1'(y))g_2'(x+Q_2'(y))g_3'(x+Q_3'(y))\right|^{1/2}\right].
		\end{multline}
		Since $(a+b)^{1/2} \leq a^{1/2} + b^{1/2}$ for nonnegative real numbers $a$ and $b$, \eqref{example: intermediate eqn in PET inductive step} and \eqref{example: intermediate eqn in PET inductive step 2} imply
		\begin{multline}
			\left| \rcavg{x,y} f_0(x)f_1(x+y)f_2(x+y^2)\right| \\ \leq \ 2^{1/2} \left(\frac{|\mathcal H_1|}{H_1}\right)^{1/2} + 2^{3/4} \left(\frac{|\mathcal H_2|}{H_2}\right)^{3/4} + 2^{3/4} \left| \rcavg{x,y} g_0'(x)g_1'(x+Q_1'(y))g_2'(x+Q_2'(y))g_3'(x+Q_3'(y))\right|^{1/4}.
		\end{multline}
		We wish to bound the expression 
		\begin{equation} \left| \rcavg{x,y} g_0'(x)g_1'(x+Q_1'(y))g_2'(x+Q_2'(y))g_3'(x+Q_3'(y))\right|\end{equation} using Lemma~\ref{lem: Ud bound if invertible}. The constant term of each polynomial argument can be absorbed by using a version of $g'_i$ with a translated argument, but the linear term of each argument cannot be absorbed in this way. Thus, to satisfy the conditions of the lemma, it would suffice if each of the elements $2h_1$, $2h_2$, $2(h_1+h_2)$, and $2(h_1-h_2)$ were invertible modulo $N$. A direct way to ensure we are in this situation is to choose $H_1$, $H_2$, $\mathcal H_1$, and $\mathcal H_2$ in such a way that $2h_1$, $2h_2$, $2(h_1+h_2)$, and $2(h_1-h_2)$ belong to \begin{equation}
			S_{\rm{target}} := \{1,\ldots, \lpf N - 1\} \cup \{N-(\lpf N - 1), N - (\lpf N - 1) + 1, \ldots, N- 1\},
		\end{equation} because any element of $\zn$ that belongs to this set must be coprime to $N$ and hence a unit by Lemma~\ref{lem: what is a unit}.
		
		We now observe the following facts. First, $S_{\rm{target}}$ is precisely the set of nonzero elements in $\zn$ with $\zn$-height at most $\lpf N - 1$. Second, by Lemma~\ref{lem: manipulating 0-closeness mod N}, if $h_1, h_2$ have $\zn$-height at most $\lfloor \frac{\lpf N - 1}{4}\rfloor$, then $2h_1$ and $2h_2$ have $\zn$-height at most $2\lfloor \frac{\lpf N - 1}{4}\rfloor \leq \lpf N - 1$ and $2(h_1+h_2)$ and $2(h_1-h_2)$ have $\zn$-height at most $4\lfloor \frac{\lpf N - 1}{4}\rfloor \leq \lpf N - 1$. It follows\footnote{We also need to make some mild assumptions like $H_1-1 > 0$ and $H_2 - 1 > 0$ to ensure nonemptiness of the sets that $h_1$ and $h_2$ are drawn from, but to simplify the exposition of this example, we ignore this detail here and below.} that we may take $H_1 = H_2 = \lfloor \frac{\lpf N - 1}{4} \rfloor$ to ensure $h_1$ and $h_2$ are such that all of $2h_1$, $2h_2$, $2(h_1+h_2)$, and $2(h_1-h_2)$ have $\zn$-height at most $\lpf N -1$. 
		
		Thus, each of $2h_1$, $2h_2$, $2(h_1+h_2)$, and $2(h_1-h_2)$ will belong to $S_{\rm{target}}$ (and hence be invertible modulo $N$) as soon as they are nonzero modulo $N$. To this end, we may impose some condition on $\lpf N$ and then select $\mathcal H_1$ and $\mathcal H_2$ appropriately to exclude any particular values of $h_1$ and $h_2$ that would cause some of these four terms to be zero. Assuming that $\lpf N > 2$, it follows that 2 is a unit and not a zero divisor. This in turn allows us to choose\footnote{This choice interacts with the omitted detail in the previous footnote in a distracting way. We handle this nonemptiness issue properly in the general proof of Proposition~\ref{prop: Us control}.} $\mathcal H_1 = \{0\}$ and $\mathcal H_2 = \{0,h_1,-h_1\}$.
		
		Thus, by Lemma~\ref{lem: Ud bound if invertible},
		\begin{equation}\label{example: applying lem: Ud bound if invertible 1}
			\left| \rcavg{x,y} g_0'(x)g_1'(x+Q_1'(y))g_2'(x+Q_2'(y))g_3'(x+Q_3'(y))\right| \ \leq \ \norm{g_3'}_{U^3} \ = \ \norm{f_2}_{U^3}.
		\end{equation}
		
		For sufficiently large $\lpf N$, we have $\frac{1}{H_1} \leq \frac{16}{\lpf N}$, whence
		\begin{multline}
			\left| \rcavg{x,y} f_0(x)f_1(x+y)f_2(x+y^2)\right| \\ \leq \ 2^{1/2} \left(\frac{16}{\lpf N}\right)^{1/2} + 2^{3/4} \left(\frac{48}{\lpf N}\right)^{3/4} + 2^{3/4} \norm{f_2}_{U^3}^{1/4} \\ \leq \ \frac{32^{1/2} + 96^{3/4}}{\lpf{N}^{1/2}} + 2\norm{f_2}_{U^3}^{1/4}.
		\end{multline}
		Thus, we need to choose $C_1$ at least so that $\frac{1}{H_1} \leq \frac{16}{\lpf N}$ holds. This proves \eqref{example: eqn: negative power of lpf N, nonminimum version} in the case $j = 2$.
		
		Let us sketch the proof of \eqref{example: eqn: negative power of lpf N, nonminimum version} in the case $j = 1$; the case $j = 0$ is similar. First observe
		\begin{equation}
			\rcavg{x,y} f_0(x)f_1(x+y)f_2(x+y^2) \ = \ \rcavg{x,y} f_2(x)f_0(x-y^2)f_1(x+y-y^2)
		\end{equation}
		by the change of variables $x \mapsto x - y^2$. Applying Lemma~\ref{linearization step} with the settings $m = 2$, $n = 1$, $H$ to be decided later, $P_1(y) = -y^2$, $P_2(y) = y-y^2$, and $\mathcal H$ to be decided later, there exists $h \in (\{0,1,\ldots, H-1\} \cup \{N-(H-1),\ldots, N-1\}) \setminus \mathcal H$ such that
		\begin{multline}\label{example: unnamed eqn 1}
			\left|\rcavg{x,y} f_2(x)f_0(x-y^2)f_1(x+y-y^2) \right| \\
			\leq \ 2^{1/2}\left[ \left(\frac{|\mathcal H|}{H}\right)^{1/2}   + \left| \cavg{x,y}{R} \ol{f_0}(x) f_0(x-2hy-h^2) \ol{f_1}(x+y)f_1(x+(1-2h)y+h-h^2)\right|^{1/2} \right].
		\end{multline}
		Again, we wish to apply Lemma~\ref{lem: Ud bound if invertible} to conclude that
		\begin{equation}\label{example: applying lem: Ud bound if invertible 2}
			\left| \cavg{x,y}{R} \ol{f_0}(x) f_0(x-2hy-h^2) \ol{f_1}(x+y)f_1(x+(1-2h)y+h-h^2)\right| \ \leq \ \norm{f_1}_{U^3}.
		\end{equation}
		To satisfy the conditions of the lemma, it would suffice if each of the elements $1$, $-2h$, $1-2h$, and $1+2h$ were invertible modulo $N$. One way to ensure we are in this situation is to require $\lpf N > 2$ and then choose $H$ and $\mathcal H$ in such a way that $-2h$, $1-2h$, and $1+2h$ belong to $S_{\rm{target}}$. Thus, if we choose $H = \lfloor \frac{\lpf N - 2}{2} \rfloor$, then $h$ has $\zn$-height at most $H$, so $-2h$ has $\zn$-height at most $2H \leq \lpf N - 2$ and $1-2h$ and $1+2h$ each have $\zn$-height at most $\lpf N -1$. Thus, provided that $-2h$, $1-2h$, and $1+2h$ are nonzero, they each belong to $S_{\rm{target}}$ and hence are invertible modulo $N$. Hence, we may take $\mathcal H = \{0,\pm 1/2\}$, where $1/2$ denotes the multiplicative inverse of 2 modulo $N$, which exists since $\lpf N > 2$. Thus, after using \eqref{example: applying lem: Ud bound if invertible 2} and \eqref{example: unnamed eqn 1} to derive a bound in terms of $\norm{f_1}_{U^3}$ and $\lpf N$, there is some numerical fiddling, as in the case $j = 2$ above, to determine conditions on $C_1$ and the exact shape of the resulting inequality, whose discussion we omit here.
	\end{proof}
	Let us highlight several aspects of the previous proof that become more cumbersome when we work with a general family of linearly independent polynomials $\mathbf P = \{P_1(y),\ldots, P_m(y)\} \subset \Z[y]$, rather than just $\{y,y^2\}$. In general, we need to work with such families of polynomials in $n$ indeterminates, but for the sake of exposition in the next few paragraphs, we will take $n = 1$.
	
	The first issue is straightforward. Namely, it may not be clear to all readers that, given $\mathbf P$, there is always a way to apply Lemma~\ref{linearization step} a certain number of times so as to, like in the case $j = 2$ above, create a chain of inequalities (viz. \eqref{example: intermediate eqn in PET inductive step} and \eqref{example: intermediate eqn in PET inductive step 2}), the last of which includes an average of functions with linear polynomial arguments in $y$. This issue may be handled using a PET induction scheme, which ensures that such a method always exists. This is why we will need to introduce weight sequences.
	
	There are still additional issues. Suppose we have a PET induction scheme according to which we can create a chain of inequalities, the last of which is in terms of the magnitude of
	\begin{equation}\label{example: final average}
		\rcavg{x,y} g_0(x)g_1(x+a_1y+b_1)g_2(x+a_2y+b_2)\cdots g_{k}(x+a_ky+b_k)
	\end{equation}
	(for some 1-bounded $g_i$ and some $k$), which we will in these few paragraphs call ``the final average.'' 
	
	The second issue is that, as the ``complexity'' of the family $\mathbf P$ increases, it quickly becomes impractical to determine the exact linear coefficients $a_i$ appearing in the final average, such as those coefficients $-2h$, $1$, and $1-2h$ in \eqref{example: applying lem: Ud bound if invertible 2}.
	
	As a result, it becomes more difficult to choose appropriate conditions on $\lpf N$ and to choose $\mathcal H_1$, $\mathcal H_2$, and so on, to exclude values of $h_1$, $h_2$, and so on, in such a way as to ensure that, in the final average, no $a_i$ is a zero divisor and no difference of two $a_i$ is a zero divisor.
	
	To address these difficulties, we introduce some technical notions whose only purpose is to manage the repeated application of Lemma~\ref{linearization step} in such a way that the $a_i$ in the final average are invertible with invertible differences. These notions are $\zn$-height and essential distinctness.
	
	There is a third issue, which cannot be seen from the example of $\{y,y^2\}$, and the simplest example that demonstrates it is likely something like $\{y^3,y^3 + y^2\}$. We delay the discussion of this issue until the necessary terminology is developed in Subsection~\ref{subsec: PET algorithms}.
	
	Finally, there is a fourth issue, which can only occur when working with polynomials in $n > 1$ indeterminates: Sometimes the final average is not immediately suitable for applying Lemma~\ref{lem: Ud bound if invertible} even after managing $\zn$-height and essential distinctness. We discuss this issue in Section~\ref{sec: fixing a matrix}. 
	
	\subsection{$\zn$-height}\label{subsec: zn height}
	We introduced the notion of $\zn$-height for an integer in the previous subsection before the proof of Proposition~\ref{example: prop: Us control}. Let us develop the notion a little more.
	\begin{definition}
		Let $N$, $n$ be positive integers. Let $P$ belong to $\Z[y_1,\ldots, y_n]$ or $\zn[y_1,\ldots, y_n]$. Define the \emph{$\zn$-height} of $P$ as the maximum of the $\zn$-heights of the coefficients of $P$. For a finite subset $\mathbf P = \{P_1,\ldots, P_m\}$ of $\Z[y_1,\ldots, y_n]$ or $\zn[y_1,\ldots, y_n]$, define the \emph{$\zn$-height} of $\mathbf P$ as the maximum of the $\zn$-heights of the $P_j$. 
	\end{definition}
	We will need to understand the effect of Lemma~\ref{linearization step} on the $\zn$-height of a family of polynomials. The following bound will be sufficient for our purposes.
	\begin{lemma}\label{lem: P(y+h)-Q(y) height bound} Let $N$, $n$, $H_1$, $H_2$, and $d$ be positive integers such that $d > 1$. Let $P(y),Q(y) \in \Z_N[y_1,\ldots, y_n]$ be polynomials with $2 \leq \deg_{\zn}(P) \leq d$ and $\deg_{\zn}(Q) \leq d$. Suppose the $\zn$-height of $\{P,Q\}$ is at most $H_1$. Let $h = (h_1,\ldots, h_n) \in \Z^n$ be such that each $h_i$ has $\zn$-height at most $H_2$. Then the $\zn$-height of $P(y+h)-Q(y)$ is at most $\frac 12 (d+1)^{2dn}H_1H_2^d$. 
	\end{lemma}
	\begin{proof} Write $P(y) = \sum_{\alpha \in S} c_\alpha y^\alpha$, where $S = \{\alpha \in \Z^n : 0 \leq \alpha_1,\ldots, \alpha_n \leq d \}$. Write $P(y+h) = P(y_1+h_1,\ldots, y_n + h_n) = \sum_{\alpha\in S} c''_\alpha y^\alpha$ and $P(y+h)-Q(y) = \sum_{\alpha\in S} c'''_\alpha y^\alpha$. Fix $\alpha \in S$, and define $S' := \{\beta \in S : \beta_i \geq \alpha_i\}$. For each $\beta \in S'$, let $c'_\beta$ denote the coefficient of $y^\alpha$ that appears when expanding the expression $c_\beta (y+h)^\beta$. The coefficient $c''_\alpha$ is then precisely $\sum_{\beta \in S'} c'_\beta$. Let us compute $c'_\beta$ for $\beta \in S'$. We have
		\begin{multline}
			c_\beta(y+h)^\beta \ = \ c_\beta \prod_{i=1}^n (y_i+h_i)^{\beta_i} \ = \ c_\beta \prod_{i=1}^n \left(\binom{\beta_i}{\alpha_i}h_i^{\beta_i-\alpha_i}y_i^{\alpha_i} +\sum_{\substack{0 \leq k \leq \beta_i \\ k \neq \alpha_i}} \binom{\beta_i}{k} h_i^{\beta_i-k} y_i^k \right),
		\end{multline}
		so that, setting $\delta_i = \beta_i - \alpha_i$, we have
		\begin{equation}
			c'_\beta \ = \ c_\beta \prod_{i=1}^n \binom{\beta_i}{\alpha_i} h_i^{\beta_i-\alpha_i} \ = \ c_\beta \prod_{i=1}^n \binom{\beta_i}{\delta_i} h_i^{\delta_i}.
		\end{equation}
		Now we will apply Lemma~\ref{lem: manipulating 0-closeness mod N}. First, by assumption on the $\Z_N$-height of $P$, we know that $c_\beta$ has $\zn$-height at most $H_1$. 
		Next, for each $i$, using crude bounds everywhere, we have $\delta_i \leq \beta_i \leq d$, which implies that $\binom{\beta_i}{\delta_i} \leq \beta_i^{\beta_i} \leq d^d$. Thus, for each $i$, $\binom{\beta_i}{\delta_i}$ is an integer of $\zn$-height at most $d^d$. Finally, we know that $\sum_{i=1}^n \delta_i \leq \sum_{i=1}^n \beta_i \leq d$, and each $h_i$ has $\zn$-height at most $H_2$ by assumption. Hence, applying the lemma, we conclude that $c'_\beta$ has $\zn$-height at most $d^{dn} H_1H_2^d$. Using a crude bound, we see there are at most $(n+1)^d$ coefficients $c'_\beta$ since $|S'| \leq |S| \leq (d+1)^n$. Hence, using the fact that $c''_\alpha = \sum_{\beta \in S'} c'_\beta$ and applying the lemma again, we conclude that $c''_\alpha$ has $\zn$-height at most $(d+1)^nd^{dn}H_1H_2^d$.
		
		Moreover, the coefficient $c_0$ of $y^\alpha$ in $Q$ also trivially has $\zn$-height at most $d^{dn} H_1H_2^d$ by assumption on $Q$. Hence, applying Lemma~\ref{lem: manipulating 0-closeness mod N} again, $c'''_\alpha = c''_\alpha+c_0$ has $\zn$-height at most $((d+1)^n+1)d^{dn}H_1H_2^d$. This argument is valid for each $\alpha \in S$.
		
		Finally, we observe that
		\begin{equation}
			((d+1)^n+1)d^{dn} \ \leq \ \frac 12 (d+1)^{2dn}
		\end{equation}
		since $d \geq 2$ and $n \geq 1$.
	\end{proof}

	In the proof of Lemma~\ref{lem: PET inductive step}, we will need another technical result about $\zn$-height, which we record here.
	\begin{lemma}\label{lem: bound on translates violating ess dist}
		Let $N$, $H$, and $n$ be positive integers. Let $c \in \Z$. Let $P(y), Q(y) \in \zn[y_1,\ldots, y_n]$ be two polynomials such that $1 \leq \deg_{\zn}(P) < \lpf N$ and the nonzero coefficients of $P$ are invertible modulo $N$. Then the number of $h=(h_1,\ldots, h_n) \in \mathbb{Z}_N^n$ such that $P(y+h) = Q(y) + c$ in $\zn[y_1,\ldots,y_n]$ and each $h_i$ has $\zn$-height at most $H-1$ is at most $(2H-1)^{n-1}$.
	\end{lemma}
	\begin{remark}
		We will apply this lemma in the case $\deg_{\zn}(P) > 1$.
	\end{remark}
	\begin{proof}
		The conclusion is trivial to verify when $\deg_{\zn}(P) = 1$, so we now assume $\deg_{\zn}(P) > 1$.
		
		Let $S$ be the set of $\alpha = (\alpha_1,\ldots,\alpha_n)\in \Z^n$ such that the coefficient of $y^\alpha = y_1^{\alpha_1}\cdots y_n^{\alpha_n}$ in $P(y)$ is nonzero modulo $N$, and write $P(y) = \sum_{\alpha \in S} c_\alpha y^\alpha$. By assumption on the degree of $P$, the set $S' := \{ \alpha \in S : \sum_{i=1}^n \alpha_i \geq 2\}$ is nonempty. It follows that there exists $j \in \{1,\ldots, n\}$ such that \begin{equation} S''_j \ := \ \{\alpha \in S' : \alpha_j > 0 \text{ and } \forall \beta = (\beta_1,\ldots, \beta_n) \in S, \; \beta_j \leq \alpha_j \} \end{equation} is nonempty. Let $\alpha \in S''_j$ be such that \begin{equation}
			\sum_{i=1}^n \alpha_i \ = \ \max\{d \in \N : \exists \beta \in S''_j, \; d = \sum_{i=1}^n \beta_i\},
		\end{equation}
		and write $\alpha' := (\alpha_1,\ldots, \alpha_{j-1},\alpha_j-1,\alpha_{j+1},\ldots,\alpha_n)$.
		
		We claim that the coefficient of $y^{\alpha'}$ in $P(y+h)$ is of the form $c_\alpha \binom{\alpha_j}{1}h_j + F$ for some quantity $F \in \zn$ not depending on $h_j$. Let $S''' := \{ \beta \in S : \beta_j \geq \alpha_j - 1 \text{ and } \beta_i \geq \alpha_i$ for $i \neq j\}$, and, for each $\beta \in S'''$, denote by $c'_{\beta}$ the coefficient of $y^{\alpha'}$ that appears when expanding the expression $c_\beta (y+h)^{\beta}$. Thus, the coefficient of $y^{\alpha'}$ in $P(y+h)$ is precisely $\sum_{\beta \in S'''} c'_{\beta}$.
		
		By construction, we have $\alpha \in S''_j$, which implies that any $\beta \in S'''$ satisfies $\beta_j \leq \alpha_j$. Hence, any $\beta \in S'''$ has either $\beta_j = \alpha_j - 1$ or $\beta_j = \alpha_j$. Thus, on the one hand, suppose $\beta \in S'''$ satisfies $\beta_j = \alpha_j - 1$. Then 
		\begin{equation}
			c_\beta(y+h)^\beta \ = \ c_\beta \left(y_j^{\alpha_j-1} + \sum_{k=1}^{\alpha_j-1} \binom{\alpha_j-1}{k} h_j^ky_j^{\alpha_j-1-k} \right) \prod_{i\neq j} (y_i + h_i)^{\beta_i},
		\end{equation}
		so that $c'_\beta = c_\beta F_\beta$, where $F_\beta$ is some polynomial expression in $h_i$, $i \neq j$.
		
		On the other hand, suppose $\beta \in S'''$ satisfies $\beta_j = \alpha_j$. Then
		\begin{equation}
			c_\beta(y+h)^\beta \ = \ c_\beta \left(\binom{\alpha_j}{1}h_j y_j^{\alpha_j-1} + \sum_{\substack{0 \leq k \leq \alpha_j, \\ k \neq 1}} \binom{\alpha_j}{k} h_j^ky_j^{\alpha_j-k} \right) \prod_{i\neq j} (y_i + h_i)^{\beta_i},
		\end{equation}
		so that $c'_\beta = c_\beta \binom{\alpha_j}{1} h_j + c_\beta F_\beta$, where $F_\beta$ is again some polynomial expression in $h_i$, $i \neq j$.
		
		To prove the claim, it remains only to show that the only $\beta \in S'''$ such that $\beta_j = \alpha_j$ is just $\alpha$. To see this, observe that $\beta \in S''_j$ and hence must satisfy $\sum_{i=1}^n \beta_i \leq \sum_{i=1}^n \alpha_i$ by the choice of $\alpha$. On the other hand, the fact that $\beta$ belongs to $S'''$ means that $\beta_i \geq \alpha_i$ for $i \neq j$, which, together with the assumption $\beta_j = \alpha_j$, implies that $\sum_{i=1}^n \alpha_i \leq \sum_{i=1}^n \beta_i$, yielding $\sum_{i=1}^n \alpha_i = \sum_{i=1}^n \beta_i$. Finally, we know that $\beta_{i} \geq \alpha_i$ for $i \neq j$. Supposing to the contrary that some $\beta_{i'} > \alpha_{i'}$ for $i' \neq j$, we would conclude from $\sum_{i\neq i'} \beta_i \geq \sum_{i \neq i'} \alpha_i$ that $\sum_{i=1}^n \beta_i > \sum_{i=1}^n \alpha_i$, which is a contradiction. Hence all $\beta_i = \alpha_i$ and $\beta = \alpha$, proving the claim.
		
		Thus, by the claim, the coefficient of $y^{\alpha'}$ in $P(y+h)$ is of the form $c_\alpha \alpha_j h_j + F$ for some $F \in \zn$ not depending on $h_j$, but possibly depending on $h_i$, $i \neq j$. Let $c'$ be the (possibly zero) coefficient of $y^{\alpha'}$ in $Q(y) + c$. By assumption, $c' \equiv c_\alpha \alpha_j h_j + F \bmod N$. Hence, fixing $h_i \in \{0,1,\ldots, H-1\}\cup \{N-(H-1),\ldots,N-1\}$ for every $i \neq j$, we observe that $h_j$ is determined by the equation $h_j \equiv (c' - F)/(c_\alpha \alpha_j) \bmod N$, where, by assumption, $c_\alpha$ is invertible modulo $N$ and the fact that $\alpha_j \leq \deg_{\zn}(P) < \lpf N$ implies that $\alpha_j$ is invertible modulo $N$. The result follows.
	\end{proof}

	\subsection{Essential distinctness modulo $N$} \label{subsec: essential distinctness} 
	As in previous instances of PET induction, we use the concept of essential distinctness of a family of polynomials as a kind of non-degeneracy condition that is preserved under differencing the family by applying Lemma~\ref{linearization step}. We recall the definition we use:
	
	\begin{definition}
		Let $m$, $n$, and $N$ be positive integers. Let $\mathbf{P} = \{P_1(y),\ldots, P_m(y)\} \subset \zn[y_1,\ldots, y_n]$ be a family of polynomials. We say that $\mathbf{P}$ is \emph{essentially distinct} if for any distinct $i,j \in \{1,\ldots, m\}$, neither $y \mapsto P_i(y)$ nor $y \mapsto P_i(y) - P_j(y)$ is a constant function $\mathbb{Z}_N^n \to \zn$.
	\end{definition}
	
	As explained in Subsection~\ref{subsec: polynomials}, there is no nice correspondence between polynomials over $\zn$ and polynomial functions over $\zn$, since multiple polynomials may induce the same polynomial function. Moreover, as in other PET induction arguments, the relevant property that we need to guarantee is the non-constant nature of certain induced polynomial functions.
	
	This complication in the definition explains why we needed to develop the material in Subsection~\ref{subsec: polynomials} about polynomials vs. polynomial functions, which we now use to prove the following proposition, whose purpose is to simplify some of the bookkeeping in the proof of Proposition~\ref{prop: Us control}.
	
	\begin{proposition}\label{prop: R-essential distinctness lpf condition}
		Let $m \geq 2$ and $n \geq 1$ be integers, and let $\mathbf P = \{P_1(y),\ldots,P_m(y)\} \subset \Z[y_1,\ldots, y_n]$ be an independent family of polynomials with zero constant term. There exists $C^* > 0$ such that, for any finite commutative ring $R$ with characteristic $N$ satisfying $\lpf N > C^*$, the family $\{P_1,\ldots, P_m\}$, viewed as a finite subset of $\zn[y_1,\ldots, y_n]$, is essentially distinct.
	\end{proposition}
	\begin{proof}
		By Proposition~\ref{prop: lindep over zn}, let $C_1 > 0$ be such that, for any positive integer $M$ with $\lpf M > C_1$, the family $\{P_1,\ldots, P_m\}$, viewed as a family of polynomials in $\Z_M[y_1,\ldots, y_n]$, is linearly independent over $\Z_M$.
		
		Let $C^* = \max\{C_1,\deg(P_1),\ldots, \deg(P_m)\}$.
		
		Let $R$ be a finite commutative ring with characteristic $N$ satisfying $\lpf N > C^*$. Then the family $\{P_1,\ldots, P_m\}$ is linearly independent over $\Z_N$.
		
		Now, observe that for any $i \in \{1,\ldots, m\}$, $y \mapsto P_i(y)$ is not a constant function $\mathbb{Z}_N^n \to \zn$. Indeed, if it were, then since $P_i(0,\ldots, 0) = 0$, we would have $P_i(y_1,\ldots, y_n) \equiv 0 \bmod N$ for all $y_1,\ldots, y_n \in \zn$. Since $\deg_{\zn} (P_i) \leq \deg(P_i) \leq C^* < \lpf N$, by Lemma~\ref{lem: mvpolyuniq}, it would follow that $P_i$ is the zero polynomial in $\zn[y_1,\ldots, y_n]$, which cannot be a member of a family of linearly independent polynomials over $\Z_N$.
		
		Similarly, if $y \mapsto P_i(y) - P_j(y)$ were a constant function $\mathbb{Z}_N^n \to \zn$ for distinct $i,j \in \{1,\ldots, m\}$, then it would follow that $P_i(y_1,\ldots, y_n) \equiv P_j(y_1,\ldots, y_n) \bmod N$ for all $y_1,\ldots, y_n \in \zn$ and hence $P_i = P_j$ by Lemma~\ref{lem: mvpolyuniq}, which again contradicts the linear independence of $\{P_1,\ldots, P_m\}$ over $\Z_N$.
	\end{proof}

	\subsection{Weight sequences}\label{subsec: weight sequences}
	The goal of this subsection is to define weight sequences, which we use to manage the PET induction argument in the proof of Proposition~\ref{prop: Us control}.
	
	First, we need to define the notion of a weight sequence of a family of multivariable polynomials with coefficients in either $\Z$ or $\zn$.
	
	Fix a positive integer $n$ for the remainder of this discussion. Let $S$ be the set of all $n$-tuples of nonnegative integers, not all of which are zero. Define an order $\prec_n$ on $S$ by asserting that $\alpha \prec_n \beta$ if and only if either $\sum_{i=1}^n \alpha_i < \sum_{i=1}^n \beta_i$ or both $\sum_{i=1}^n \alpha_i = \sum_{i=1}^n \beta_i$ and there exists $r \in \{1,\ldots,n\}$ such that $\alpha_r < \beta_r$ and for every $i < r$, $\alpha_i = \beta_i$. It is easy to check that the ordered sets $(S,\prec_n)$ and $(\N, <)$ are order isomorphic, so we let $\pi : S \to \N$ be the unique bijective function\footnote{Here we use $\pi$ to denote the order isomorphism instead of the usual $f$.} such that $\pi(\alpha) < \pi(\beta)$ if and only if $\alpha \prec_n \beta$.
	
	For example, when $n = 2$, observe that 
	\begin{equation} (0,1) \prec_2 (1,0) \prec_2 (0,2) \prec_2 (1,1) \prec_2 (2,0) \prec_2 \ldots
	\end{equation}
	Then $\pi((0,1)) = 1$, $\pi((1,0)) = 2$,  $\pi((0,2)) = 3$, and so on. This order $\prec_2$ formalizes the idea that $y_2 = y_1^0y_2^1$ is lighter than $y_1 = y_1^1y_2^0$, which is lighter than $y_2^2 = y_1^0y_2^2$, which is lighter than $y_1y_2 = y_1^1y_2^1$, which is lighter than $y_1^2 = y_1^2y_2^0$, and so on. Defined below, the leading coefficient of a polynomial will be the one attached to the heaviest monomial with a nonzero coefficient.
	
	Let $P(y) \in \Z[y_1,\ldots, y_n]$ be nonconstant. Write $P(y) = c_{(0,\ldots,0)}+\sum_{\alpha \in S} c_\alpha y^\alpha$, and define $\beta$ to be the unique $\beta \in S$ such that $\pi(\beta) = \max\{ \pi(\alpha) : \alpha \in S \text{ and } c_\alpha \neq 0 \}$. We call $L(P) := c_\beta$ the \emph{leading coefficient} of $P$ and $w(P) := \pi(\beta)$ the \emph{weight} of $P$.
	
	Now fix a positive integer $N$, and suppose that $P \in \zn[y_1,\ldots, y_n]$ (resp. $\Z[y_1,\ldots, y_n]$) satisfies $\deg_{\zn}(P) > 0$. Analogously, define $\beta_N$ to be the unique $\beta_N \in S$ such that 
	\begin{equation}
		\pi(\beta_N) = \max\{ \pi(\alpha) : \alpha \in S \text{ and } c_\alpha \not\equiv 0 \bmod N\}.
	\end{equation} We call $L_{\zn}(P) := c_{\beta_N}$ the \emph{leading coefficient} (resp. \emph{$\zn$-leading coefficient}) of $P$ and \\ $w_{\zn}(P) := \pi(\beta_N)$ the \emph{weight} (resp. \emph{$\zn$-weight}) of $P$. The notions of ($\zn$-)weight and ($\zn$-)degree are compatible in the following sense: For $P$ and $Q \in \Z[y_1,\ldots, y_n]$, $\deg_{\zn}(P) \leq \deg_{\zn}(Q)$ if and only if $w_{\zn}(P) \leq w_{\zn}(Q)$. 
	
	Now fix a positive integer $m$. Let $\mathbf{P} := \{P_1(y),\ldots, P_m(y)\} \subset \Z[y_1,\ldots,y_n]$ be a family of nonconstant polynomials. For each positive integer $r$ and $N$, we define 
	\begin{equation}
		W_{r}(\mathbf{P}) := |\{ L(P_i) \in \Z : i \in \{1,\ldots, m\} \text{ such that } w(P_i) = r \}|.
	\end{equation} Finally, we define the \emph{weight sequence} of $\mathbf{P}$ by $W(\mathbf{P}) := (W_1(\mathbf{P}),W_2(\mathbf{P}),\ldots)$.
	
	Similarly, let $\mathbf{P} := \{P_1(y),\ldots, P_m(y)\} \subset \zn[y_1,\ldots,y_n]$ (resp. $\Z[y_1,\ldots,y_n]$). For each positive integer $r$ and $N$, define 
	\begin{equation} W_{\zn,r}(\mathbf{P}) := |\{ L_{\zn}(P_i) \in \zn : i \in \{1,\ldots, m\} \text{ such that } w_{\zn}(P_i) = r \}|. \end{equation} We define the \emph{weight sequence} (resp. \emph{$\zn$-weight sequence}) of $\mathbf P$ by
	\begin{equation} W_{\zn}(\mathbf{P}) := (W_{\zn,1}(\mathbf{P}),W_{\zn,2}(\mathbf P), \ldots).
	\end{equation}

	All notions involved in defining ($\zn$-)weight sequences depend on $\prec_n$, but we will suppress notating this, as we will always work in a context with a fixed $n$.
	
	Let us give an example. The family \begin{equation}
		\mathbf P := \{y_1^2 +3y_2^2, 8y_1^2, 2y_1^2 + y_1y_2, 7y_2^2 + y_1, 2y_1,6y_2 +2y_1, y_1, 4y_1 + 2\} \subset \mathbb{Z}[y_1,y_2]
	\end{equation}
	has weight sequence $(0,3,1,0,3,0,\ldots)$ and $\mathbb{Z}_7$-weight sequence $(0,3,0,0,2,0,\ldots)$. First recall that $(0,1) \prec_2 (1,0) \prec_2 (0,2) \prec_2 (1,1) \prec_2 (2,0) \prec_2 \ldots$.
	
	There are four polynomials in $\mathbf P$ of weight 2, namely $2y_1$, $6y_2 + 2y_1$, $y_1$, and $4y_1$, which respectively have leading coefficients 2, 2, 1, and 4; hence $W_2(\mathbf P) = 3$. There is one polynomial of weight 3, namely $7y_2^2+y_1$, so $W_3(\mathbf P) = 1$. There are three polynomials of weight 5, namely $y_1^2 +3y_2^2$, $8y_1^2$, and $2y_1^2 + y_1y_2$, which respectively have leading coefficients 1, 8, and 2; hence $W_5(\mathbf P) = 3$. There are no polynomials in $\mathbf P$ of any other weight, so $W_r(\mathbf P) = 0$ for $r \neq 2, 3, 5$.
	
	There are five polynomials in $\mathbf P$ of $\mathbb{Z}_7$-weight 2, namely $7y_2^2+y_1$, $2y_1$, $6y_2 + 2y_1$, $y_1$, and $4y_1$, which respectively have $\mathbb{Z}_7$-leading coefficients 1, 2, 2, 1, and 4; hence $W_{\mathbb{Z}_7,2}(\mathbf P) = 3$. There are three polynomials of $\mathbb{Z}_7$-weight 5, namely $y_1^2 +3y_2^2$, $8y_1^2$, and $2y_1^2 + y_1y_2$, which respectively have $\mathbb{Z}_7$-leading coefficients 1, 1, and 2; hence $W_{\mathbb{Z}_7,5}(\mathbf P) = 2$. There are no polynomials in $\mathbf P$ of any other $\mathbb{Z}_7$-weight, so $W_{\mathbb{Z}_7,r}(\mathbf P) = 0$ for $r \neq 2, 5$.
	
	Let $\mathcal{W}$ be the collection of sequences of nonnegative integers, not all of which are zero, and all but finitely many of which are zero. Every ($\zn$-)weight sequence belongs to $\mathcal{W}$.
	
	Here is a basic fact about weight sequences that will be useful in the proof of Proposition~\ref{prop: Us control}.
	\begin{proposition}\label{prop: weight sequence initial stability}
		Let $n$ and $m$ be positive integers. Let $\mathbf P = \{P_1(y),\ldots, P_m(y)\} \subset \Z[y_1,\ldots,y_n]$ be an independent family of polynomials with zero constant term such that $w(P_m) = \max_{1 \leq j \leq m} w(P_j)$. Define the family $\mathbf{P'} = \{P_1-P_m,\ldots, P_{m-1}-P_m, -P_m\} \subset \Z[y_1,\ldots, y_n]$. There exists a constant $C_{\rm{stable}} > 0$ such that whenever $N$ is a positive integer with $\lpf N > C_{\rm{stable}}$, we have
		\begin{enumerate}
			\item $w_{\zn}(P) = w(P)$ for any $P \in \mathbf P \cup \mathbf{P'}$,
			\item $W_{\zn}(\mathbf P) = W(\mathbf P)$,
			\item $W_{\zn}(\mathbf{P'}) = W(\mathbf{P'})$, and
			\item $\max_{1 \leq j \leq m} \deg_{\zn}(P_j) = \max_{1 \leq j \leq m} \deg(P_j)$.
		\end{enumerate} 
	\end{proposition}
	\begin{proof} Let $S \subset \Z$ consist of every integer that appears as a coefficient of some polynomial in $\mathbf P \cup \mathbf{P'}$. Let $H$ be an integer such that every $c \in S$ satisfies $|c| \leq H$.
		
		Now suppose $N > 2H$, and let $c, c' \in S$. Then $c \equiv 0 \bmod N$ if and only if $c = 0$, and $c - c' \equiv 0 \bmod N$ if and only if $c = c'$. The first equivalence implies that $w_{\zn}(P) = w(P)$ for any $P \in \mathbf P \cup \mathbf{P'}$, and the second equivalence implies that any two polynomials in $\mathbf P$ (or in $\mathbf{P'}$) have the same leading coefficient if and only if they have the same $\zn$-leading coefficient; hence, for any positive integer $r$, $W_{\zn,r}(\mathbf P) = W_r(\mathbf P)$ and $W_{\zn,r}(\mathbf{P'}) = W_r(\mathbf{P'})$. This shows that $W_{\zn}(\mathbf P) = W(\mathbf P)$ and $W_{\zn}(\mathbf{P'}) = W(\mathbf{P'})$.
		
		To conclude that we can, therefore, take $C_{\rm{stable}} = 2H$, the only thing left to observe is that, for $P \in \mathbf P$, $w_{\zn}(P) = w(P)$ implies $\deg_{\zn}(P) = \deg(P)$; hence $\max_{1 \leq j \leq m} \deg_{\zn}(P_j) = \max_{1 \leq j \leq m} \deg(P_j)$.
	\end{proof}

	\subsection{PET algorithms}\label{subsec: PET algorithms}
	The purpose of this subsection is to develop the notions of a \emph{permissible operation} and a \emph{PET algorithm}. We first describe why we introduce these notions.
	
	We have stated that, roughly speaking, the proof of Proposition~\ref{prop: Us control} proceeds by repeated application of Lemma~\ref{linearization step} according to a PET induction scheme managed by weight sequences. In the previous subsections, we introduced various notions whose only purpose will be to manage the repeated application of Lemma~\ref{linearization step} to guarantee some desirable property about the final inequality that we obtain (namely, that it is in terms of an average that is controlled by a Gowers uniformity norm of some relevant function). In this subsection, we develop a tool to analyze the PET induction scheme itself.
	
	In the first published instances of PET induction (for example, \cite{wmpet}), it was common to argue in the following manner:
	\begin{enumerate}
		\item Introduce a well-ordering (call it $\prec$) on finite families of polynomials $\mathbf P$, using so-called ``weights,'' ``weight matrices,'' ``degree sequences,'' or, in this article, weight sequences.
		\item State a proposition that some conclusion holds for all finite families of polynomials $\mathbf P$: ``For all $\mathbf P$, $P(\mathbf P)$ holds.'' 
		\item Prove this proposition by induction on the well-ordering. In this way, it is not necessary, given a family $\mathbf Q$, to determine exactly which families $\mathbf P \prec \mathbf Q$ are such that $P(\mathbf P)$ implies $P(\mathbf Q)$.
	\end{enumerate}
	Thus, suppose that $\mathbf P, \mathbf P_1, \mathbf P_2, \ldots, \mathbf P_k$ are finite families of polynomials such that
	\begin{align*}
		P(\mathbf P_1) \ & \Rightarrow \ P(\mathbf P_2), \\
		P(\mathbf P_2) \ & \Rightarrow \ P(\mathbf P_3), \\
		& \ \vdots\\
		P(\mathbf P_{k-1}) \ & \Rightarrow \ P(\mathbf P_k), \\
		P(\mathbf P_k) \ & \Rightarrow \ P(\mathbf P),
	\end{align*}
	and $P(\mathbf P_1)$ has been established as part of the base case. Then $P(\mathbf P)$ holds, of course.
	
	A consequence of using this approach is that, unless one works out a specific example like $\mathbf P = \{y,y^2\}$, the possible values of $k$ are unknown, as are the intermediate families $\mathbf P_i$.
	
	More recent applications of PET induction often keep track of this kind of information. For example, see \cite[Lemma 4.5]{chufrahost}, \cite[Section 5]{dfmks}, \cite[Section 4]{kkl}, or \cite[Section 4]{pel20}.
	
	Similarly, we will need some control on $k$. Thus, we structure the proof of Proposition~\ref{prop: Us control} so that relevant information about the chain of implications is not forgotten. We introduce permissible operations for this purpose.
	
	Preparing to give the formal definition soon, we think of a permissible operation as one of the implication arrows $\Rightarrow$ in such a diagram as above. Actually, for convenience, we will define the notion of permissible operation slightly more broadly, so it will include all valid implication arrows, along with a few invalid ones.
	
	Let us now describe our present situation with more detail, because the previous discussion omitted a certain complication in our setting.
	
	Loosely speaking, Proposition~\ref{prop: Us control} asserts that, given a certain kind of family of polynomials $\mathbf P = \{P_1,\ldots, P_m\} \subset \Z[y_1,\ldots, y_n]$, a certain inequality holds for any 1-bounded functions $f_0,\ldots, f_m: R \to \C$ on a finite commutative ring $R$ with characteristic $N$ such that $\lpf N$ is sufficiently large. To prove this proposition, we will apply Lemma~\ref{lem: PET inductive step} $k \in \N$ times.
	
	A potential obstacle to our proof arises from the following fact:
	\begin{equation}\label{claimed possible nonuniformity} \text{The value of } k \text{ depends not only on } \mathbf P, \text{ but also on } N \text{ and the functions } f_0,\ldots, f_m.
	\end{equation}
	See Subsection~\ref{subsec: claimed possible nonuniformity} for a detailed example demonstrating two possible values of $k$ for the family $\mathbf P = \{y^3, y^3 +y^2\}$, depending on the given $f_j$.
	
	In the end, the validity of the proof of Proposition~\ref{prop: Us control} relies on a certain kind of uniformity that is not apparent from the statement of Lemma~\ref{lem: PET inductive step}. In particular, the number $k$ should be bounded by a constant that only depends on $\mathbf P$, and not on $N$ or the functions $f_j$.
	
	In the remaining part of this subsection, let us define the notions of permissible operation and PET algorithm and show that $k$ is bounded in terms of $\mathbf P$.
	
	Fix a positive integer $n$ for the whole subsection. Recall that we denote by $\mathcal{W}$ the collection of sequences of nonnegative integers, not all of which are zero, and all but finitely many of which are zero.
	
	Let $\mathcal S$ be the set consisting of all the ordered pairs $(m,(a_i)_{i\in\N})$ such that $m$ is a positive integer, $(a_1,a_2,\ldots) \in \mathcal{W} \cup \{(0,0,\ldots)\}$, and $\sum_{i=1}^\infty a_i \leq m$. Let $\mathcal S_{\rm{degree}=1}$ be the set of pairs in $\mathcal S$ for which the sequence part is of the form $(i_1,\ldots,i_{n},0,0,\ldots)$ for some nonnegative integers $i_1, \ldots, i_n$, not all zero, and let $\mathcal S_{\rm{degree}=0}$ be the set of pairs in $\mathcal S$ for which the sequence part is $(0,0,\ldots)$. Denote $\mathcal S_{\rm{degree}=0} \cup \mathcal S_{\rm{degree}=1}$ by $\mathcal S_{\rm{degree}\leq 1}$.
	
	Fix a positive integer $N > 1$. Note that if $(a_1,a_2,\ldots)$ is the weight sequence for some essentially distinct family $\{P_1,\ldots,P_m\} \subset \zn[y_1,\ldots, y_n]$ of polynomials, then $(m,(a_1,a_2,\ldots)) \in \mathcal S \setminus \mathcal S_{\rm{degree}=0}$; moreover, $(m,(a_1,a_2,\ldots)) \in \mathcal S_{\rm{degree}=1}$ if and only if the degree of each $P_j$ is one.
	
	Let us now define, generally, \emph{a PET algorithm} in a way that is local to this article and is sufficiently broad to cover the full range of situations that arise in the proof of Proposition~\ref{prop: Us control} and for which we require a uniform bound.
	
	Given a pair in $\mathcal S \setminus \mathcal S_{\rm{degree}\leq 1}$, perform certain \emph{permissible operations}, one at a time, from a limited set of operations, to transform it into a pair belonging to $\mathcal S_{\rm{degree}=1}$. We will call such a sequence of operations a PET algorithm (applied to a specific pair in $\mathcal{S} \setminus \mathcal S_{\rm{degree}\leq 1}$, if such specification is necessary).
	
	First, we describe the permissible operations that may be performed on a given pair $a^* = (m,(a_i)_{i\in\N}) \in \mathcal S \setminus \mathcal S_{\rm{degree}=0}$, depending on whether $a^*$ satisfies the following condition:
	\begin{equation} \label{is the sequence a lonely one?}
		\text{There exists } s > n \text{ such that } a_s = 1 \text{ and } a_i = 0 \text{ for all } i \neq s.
	\end{equation}
	
	
	On the one hand, suppose $a^*$ satisfies \eqref{is the sequence a lonely one?}, so that $a^* \in \mathcal{S} \setminus \mathcal S_{\rm{degree}\leq 1}$. We assert that the only permissible operations are those that yield a pair of the form \begin{equation} b^* = (m',(i_1,\ldots,i_{s-1},0,0,\ldots)) \in \mathcal S \setminus \mathcal S_{\rm{degree}=0},
	\end{equation} where $m'$ and the $i_j$ are some nonnegative integers which satisfy $m' \in \{m,m+1,\ldots,2m\}$ and $i_1 + i_2 + \cdots + i_{s-1} \leq 2m$ and are such that not all $i_j = 0$; note that taking $m' = 2m$, $i_1 = 1$ and $i_2 = i_3 = \cdots = i_{s-1} = 0$ ensures $b^* \in \mathcal S \setminus \mathcal S_{\rm{degree}=0}$, so there is at least one permissible operation for such $a^*$. 
	
	On the other hand, suppose $a^*$ does not satisfy \eqref{is the sequence a lonely one?}. Then $a^* \in \mathcal S \setminus \mathcal S_{\rm{degree}=0}$, and there exists $s' \geq 0$ such that $a_{s'+1} > 0$ and, for each $i$ belonging to the possibly empty set $\{1,2,\ldots, s'\}$, $a_i = 0$. We assert that the only permissible operations are those that yield a pair of the form
	\begin{equation} b^* = (m',(i_1,i_2,\ldots,i_{s'},a_{s'+1}-1,a_{s'+2},a_{s'+3},\ldots)) \in \mathcal S \setminus \mathcal S_{\rm{degree}=0},
	\end{equation} where $m'$ and the $i_j$ are some nonnegative integers which satisfy $m' \in \{m,m+1,\ldots,2m\}$ and $i_1 + i_2 + \cdots + i_{s'} \leq 2m$; note that taking  $m' = m$ and (if $s' > 0$) $i_1=i_2 = \dots = i_{s'} = 0$ ensures that $b^* \in \mathcal S$, so again there is at least one permissible operation for such $a^*$.
	
	Next, we observe the following fact: The only permissible operation that yields a pair of the form $b^* = (m',(0,0,\ldots)) \in \mathcal S_{\rm{degree} = 0}$ is one that is applied to a pair $a^*$ with sequence part of the form $(1,0,0,\ldots)$; hence $a^* \in \mathcal S_{\rm{degree}=1}$. 
	
	We are ready to describe a PET algorithm. Let $(m,(a_i)_{i\in\N}) \in \mathcal S \setminus \mathcal S_{\rm{degree} \leq 1}$. Perform a permissible operation on the pair $(m,(a_i)_{i\in\N})$ to obtain a new pair $(m',(a_i')_{i\in\N})$, which, by the previous paragraph, must belong to either $\mathcal S_{\rm{degree}=1}$ or $\mathcal S \setminus \mathcal S_{\rm{degree}\leq 1}$. In the former case, stop. In the latter case, perform a permissible operation on $(m',(a_i')_{i\in\N})$ to obtain $(m'',(a_i'')_{i\in\N})$, check whether $(m'',(a_i'')_{i\in\N})$ belongs to $\mathcal S_{\rm{degree}=1}$ or $\mathcal S \setminus \mathcal S_{\rm{degree}\leq 1}$, and so on...
	
	Though we do not find this formalism useful in this article, a PET algorithm may be formalized as follows. Let $\mathcal G$ be the directed graph with vertex set $\mathcal S$ and a directed edge from $a^*$ to $b^*$ if and only if a permissible operation transforms $a^*$ to $b^*$. Then a PET algorithm is just a path on $\mathcal G$ from a vertex in $\mathcal S \setminus \mathcal{S}_{\rm{degree}\leq 1}$ to a vertex in $\mathcal{S}_{\rm{degree}=1}$.
	
	There are many PET algorithms which transform a given element of $\mathcal{S} \setminus \mathcal S_{\rm{degree}\leq 1}$ into an element of $\mathcal S_{\rm{degree}=1}$, because at every step, there are potentially multiple valid permissible operations which may be applied.
	
	Though not strictly logically necessary for our purposes, we first claim that any PET algorithm terminates after finitely many steps, or applications of a permissible operation. Let $(b_1,b_2,\ldots), (c_1,c_2,\ldots) \in \mathcal{W}$; we may define an order $\prec'$ on $\mathcal{W}$ by asserting that \\ $(b_1,b_2,\ldots) \prec' (c_1,c_2,\ldots)$ if and only if there exists $r \in \N$ such that $b_r < c_r$ and for all $s > r$, we have $b_s = c_s$. Notice that a permissible operation applied to $(m,(a_i)_{i\in\N})$ results in a new pair $(m',(a_i')_{i\in\N})$ whose sequence part satisfies $(a_i') \prec' (a_i)$. But $\mathcal{W}$ is a well-ordered set under $\prec'$, and, if continued forever, a PET algorithm would generate a strictly decreasing (according to $\prec'$) sequence in $\mathcal W$, which contradicts the definition of well-ordering. Thus, a PET algorithm terminates after finitely many steps, but we have not determined how many.
	
	We now intend to argue that any PET algorithm applied to $(m,(a_i)_{i\in\N})$ terminates after a number of steps that is bounded in terms of $m$ only, which will be convenient for our purposes. We give both an informal and a formal argument for this claim. The informal argument is the following: Among the finitely many possible permissible operations that may be applied to a given element of $\mathcal S$, there is clearly one that can delay eventual termination of the algorithm the longest; indeed, always take $m' = 2m$ and $i_{s-1} = 2m$ or (if $s' > 0$) $i_{s'} = 2m$, letting the other $i_j$ be zero. To some readers, it may be clear that this choice, repeated every time a permissible operation is chosen, will maximize the number of steps until the algorithm terminates, and it should also be clear that calculation of the exact number of steps required would both be solely in terms of $m$. For readers who may be unconvinced by this informal argument, we now give a formal argument.
	
	For each positive integer $d$, define $\mathcal{S}_d$ to be the subset of $\mathcal S$ consisting of all pairs $(m,(a_1,a_2,\ldots))$ such that $a_1 = a_2 = \cdots = a_d = 0$. We need a lemma:
	\begin{lemma}\label{lem: clearing up to d} Let $d$ be a positive integer. There exist nondecreasing functions $f, g : \N \to \N$ with the following property: For any $a^* = (m,(a_1,a_2,\ldots,)) \in \mathcal S \setminus \mathcal{S}_{\rm{degree}=0}$, the number of permissible operations required to transform $a^*$ into a pair of the form $(m',(0,\ldots,0,a_{d+1},a_{d+2},\ldots)) \in \mathcal{S}_d$ (for some $m'$) is at most $f(m)$, and any resulting $m'$ is at most $g(m)$. 
	\end{lemma}
	\begin{proof}
		We induct on $d$. First suppose $d=1$. Let $a^* = (m,(a_1,a_2,\ldots,)) \in \mathcal S \setminus \mathcal{S}_{\rm{degree}=0}$. We observe that it takes $a_1 \leq m$ permissible operations to transform $a^*$ into a pair of the form $(m',(0,a_2,a_3,\ldots)) \in \mathcal{S}_1$ for some $m'$. From the definition of a permissible operation, it is clear that $m' \leq 2^{a_1}m$, hence $m' \leq 2^{m}m$. Thus, in the base case we can take $f(m) := m$ and $g(m) := 2^mm$, which are both nondecreasing.
		
		Now, suppose $d \geq 1$ and that there exist nondecreasing functions $f, g : \N \to \N$ with the following property: For any $a^* = (m,(a_1,a_2,\ldots,)) \in \mathcal S \setminus \mathcal{S}_{\rm{degree}=0}$, the number of permissible operations required to transform $a^*$ into a pair of the form $(m',(0,\ldots,0,a_{d+1},a_{d+2},\ldots)) \in \mathcal{S}_d$ (for some $m'$) is at most $f(m)$, and any resulting $m'$ is at most $g(m)$.
		
		We must show that there exist nondecreasing functions $f',g' : \N \to \N$ with the following property: For any $a^* = (m,(a_1,a_2,\ldots,)) \in \mathcal S \setminus \mathcal{S}_{\rm{degree}=0}$, the number of permissible operations required to transform $a^*$ into a pair of the form $(m',(0,\ldots,0,a_{d+2},a_{d+3},\ldots)) \in \mathcal{S}_{d+1}$ (for some $m'$) is at most $f'(m)$, and any resulting $m'$ is at most $g'(m)$.
		
		Let $a^* = (m,(a_1,a_2,\ldots,)) \in \mathcal S \setminus \mathcal{S}_{\rm{degree}=0}$. By the induction hypothesis, the number of permissible operations required to transform $a^*$ into a pair of the form $a^*_1 = (m_1,(0,\ldots, 0,a_{d+1},a_{d+2},\ldots)) \in \mathcal{S}_d$ for some $m_1$ is at most $f(m)$, and also $m_1 \leq g(m)$. Suppose that $a_{d+1} > 0$.
		
		By applying one permissible operation to $a^*_1$, we obtain a pair $b^*_1 \not\in \mathcal{S}_{\rm{degree}=0}$ of the form $(m_1', (i_1,\ldots,i_d, a_{d+1}-1,a_{d+2},\ldots))$ for some nonnegative integers $m_1'$ and $i_1,\ldots, i_d$ satisfying $i_1+\cdots+i_d \leq 2m_1$ and $m_1' \leq 2m_1$. According to the induction hypothesis, by applying at most $f(m_1')$ permissible operations to $b^*_1$, we obtain a pair $a^*_2$ of the form \\ $(m_2,(0,\ldots,0,a_{d+1}-1,a_{d+2},\ldots)) \in \mathcal{S}_d$, where $m_2 \leq g(m_1')$.
		
		Now we repeat the argument of the previous paragraph as many times as is needed; for example, we illustrate the case when $a_{d+1}-1 > 0$, i.e., when $a_2^* \not\in  \mathcal{S}_{\rm{degree}=0}$ and it is necessary to apply the argument one more time. By applying one permissible operation to $a^*_2$, we obtain a pair $b^*_2 \not\in  \mathcal{S}_{\rm{degree}=0}$ of the form $(m_2', (i_1,\ldots,i_d, a_{d+1}-2,a_{d+2},\ldots))$ for some nonnegative integers $m_2'$ and $i_1,\ldots, i_d$ satisfying $i_1+\cdots+i_d \leq 2m_2$ and $m_2' \leq 2m_2$. According to the induction hypothesis, by applying at most $f(m_2')$ permissible operations to $b^*_2$, we obtain a pair $a^*_3$ of the form $(m_3,(0,\ldots,0,a_{d+1}-2,a_{d+2},\ldots)) \in \mathcal{S}_d$, where $m_3 \leq g(m_2')$.
		
		To summarize the result, for each $i \in \{1,\ldots, a_{d+1}\}$, $m_{i+1} \leq g(m_{i}')$ and $m_i' \leq 2m_i$; moreover, the total number of permissible operations required to obtain from $a_1^*$ the pair \\ $a_{a_{d+1}+1}^* = (m_{a_{d+1}+1},(0,\ldots,0,a_{d+2},\ldots)) \in \mathcal{S}_{d+1}$ is at most $\sum_{i=1}^{a_{d+1}} 1 + f(m_i') = a_{d+1} + \sum_{i=1}^{a_{d+1}} f(m_i')$. We will estimate this quantity in a moment. Note that the number of permissible operations required to obtain from $a^*$ the pair $a_1^*$ is $f(m)$; hence, in total, the number of permissible operations required to transform $a^*$ into a pair belonging to $\mathcal{S}_{d+1}$ is at most $f(m) + a_{d+1} + \sum_{i=1}^{a_{d+1}} f(m_i')$. We now estimate this quantity.
		
		We have $a_{d+1} \leq m$ by the assumption $a^* \in \mathcal S$. Moreover, for each $i \in \{1,\ldots, a_{d+1}\}$, we claim that $f(m_i') \leq f((2g)^{\circ i}(m))$, where, for a function $h : \N \to \N$, we denote repeated composition by $h^{\circ 0}(x) := x$, $h^{\circ 1}(x) := h(x)$, $h^{\circ 2}(x) = h(h(x))$, $h^{\circ 3}(x) = h(h(h(x)))$, and so on. Indeed, let us first show that for each $i \in \{1,\ldots, a_{d+1}\}$, we have $m_i' \leq (2g)^{\circ i}(m)$; the claimed inequality will then follow from the nondecreasing nature of $f$. In the case that $i = 1$, we have $m_1' \leq 2m_1 \leq 2g(m)$. Moreover, if $m_{i}' \leq (2g)^{\circ i}(m)$ holds for some $i \in \{1,\ldots, a_{d+1}-1\}$, then it follows that $m_{i+1}' \leq 2m_{i+1} \leq 2g(m_i')$, from which we conclude $m_{i+1}' \leq 2g((2g)^{\circ i}(m)) = (2g)^{\circ(i+1)}(m)$ by the nondecreasing nature of $g$.
		
		Thus, $f(m) + a_{d+1} + \sum_{i=1}^{a_{d+1}} f(m_i') \leq f(m) + m + \sum_{i=1}^{a_{d+1}} f((2g)^{\circ i}(m)) \leq m + \sum_{i=0}^{m} f((2g)^{\circ i}(m))$, so we may define $f' : \N \to \N$ by $f'(m) := m + \sum_{i=0}^{m} f((2g)^{\circ i}(m))$, which is nondecreasing as a sum of $m+2$ nondecreasing functions and which is defined independently of $a^*$.
		
		To complete the overall inductive argument, we should also define $g' : \N \to \N$ so that $m_{a_{d+1}+1} \leq g'(m)$, which ensures that the non-sequence part of $a_{a_{d+1}+1}^* \in \mathcal{S}_{d+1}$ is bounded above by $g'(m)$. Notice that $m_{a_{d+1}+1} \leq g(m_{a_{d+1}}')$, so it follows by the nondecreasing nature of $g$ that $m_{a_{d+1}+1} \leq g((2g)^{\circ a_{d+1}}(m))$, which is trivially at most \\ $(2g)^{\circ(a_{d+1}+1)}(m) \leq (2g)^{\circ(m+1)}(m)$. Hence we may define $g'$ by $g'(m) := (2g)^{\circ(m+1)}(m)$, which is defined independently of $a^*$.
	\end{proof}
	
	Now we are ready to prove the following proposition:
	
	\begin{prop} \label{alg stops}
		Let $a^* = (m,(a_i)_{i\in\N}) \in \mathcal S \setminus \mathcal S_{\rm{degree}\leq 1}$. There exists a constant $T = T(m)$ such that any PET algorithm applied to $a^*$ terminates after at most $T$ steps.
	\end{prop}
	\begin{proof}
		Let $d$ be the largest index such that $a_{d+1} = a_{d+2} = \cdots = 0$. By Lemma~\ref{lem: clearing up to d}, there exists $f : \N \to \N$ such that the number of permissible operations required to transform $a^*$ into a pair of the form $a_{\rm{cleared}} = (m',(0,0,\ldots)) \in \mathcal S_{\rm{degree} = 0}$ is at most $f(m)$. As we observed before the lemma, the only permissible operation that yields $a_{\rm{cleared}}$ is one that is applied to a pair $b^*$ with sequence part of the form $(1,0,0,\ldots)$, which belongs to $\mathcal S_{\rm{degree}=1}$. Thus, a PET algorithm would have stopped at $b^*$ or perhaps even earlier. Hence we may take $T(m) = f(m)$. 
	\end{proof}
	With some effort, an explicit estimate of $T$ can be made by unfolding the proof of Lemma~\ref{lem: clearing up to d}, though we will not need it here.
	
	\subsection{Proof of Lemma~\ref{lem: PET inductive step}}\label{subsec: proof of lem: PET inductive step}
	We are now ready to prove the upgraded version of Proposition~\ref{linearization step}, which incorporates all of the notions we have introduced in the previous subsections. As a reminder, for a real number $x$, $\lceil x \rceil$ is the least integer that is at least $x$.
	
	\begin{lemma} \label{lem: PET inductive step}
		Let $m$, $n$, $N$, $H \in \N$ be such that $\max\{2,m^2\} \leq H \leq \lceil N/2 \rceil$.
		Let $R$ be a finite commutative ring with characteristic $N$, and let $R' = R^n$ with coordinatewise addition and multiplication. Let $f_0,\ldots,f_m : R \to \C$ be 1-bounded functions. Let $\mathbf{P} = \{P_1(y),\ldots, P_m(y)\} \subset \zn[y_1,\ldots, y_n]$ be an essentially distinct family such that
		\begin{enumerate}
			\item the maximal degree $k := \max_{1\leq i\leq m} \deg_{\zn}(P_i)$ satisfies $1 < k < \lpf N$, 
			\item $w_{\zn}(P_m) = \max_{1 \leq i \leq m} w_{\zn}(P_i)$, and
			\item the $\zn$-height of $\mathbf P$ is at most $M < \lpf N$.
		\end{enumerate} Then there exist a positive integer $m' \leq 2m$, 1-bounded functions $g_0,g_1,\ldots, g_{m'} : R \to \C$ with $g_{m'} = f_m$, and an essentially distinct family $\mathbf Q = \{Q_1(y),\ldots,Q_{m'}(y)\} \subset \zn[y_1,\ldots,y_n]$ such that
		\begin{enumerate}
			\item the maximal degree $\max_{1\leq i\leq m'} \deg_{\zn}(Q_i)$ is at most $k$,
			\item $w_{\zn}(Q_{m'}) = \max_{1\leq i \leq m'} w_{\zn}(Q_i)$, 
			\item the $\zn$-height of $\mathbf Q$ is at most $(k+1)^{2kn}MH^k$,
			\item a permissible operation applied to $(m,W_{\zn}(\mathbf P))$ yields $(m',W_{\zn}(\mathbf Q))$, and
			\item \begin{equation}\label{eqn in PET inductive step}
				\left| \Lambda_{P_1,\ldots,P_m}(f_0,\ldots,f_m) \right| \ \leq \ 2^{n/2} \left[ \frac{2^{n/2}m}{H^{1/2}} + \left|\Lambda_{Q_1,\ldots,Q_{m'}}(g_0,\ldots,g_{m'})\right|^{1/2}\right] .
			\end{equation}
		\end{enumerate}	
	\end{lemma}
	\begin{remark}
		The family $\mathbf Q$ depends on the functions $f_i$. This dependence arises from the application of Lemma~\ref{linearization step} in the proof below. For an explanation of why Lemma~\ref{linearization step} introduces this dependence, see the remark following the formulation of that lemma. 
	\end{remark}
	\begin{proof}
		After potentially reindexing the first $m-1$ polynomials in $\mathbf P$ if necessary, we may, using the fact that $k > 1$, ensure $w_{\zn}(P_1) = \min \{w_{\zn}(P_i) : 1 \leq i \leq m\}$ and the existence of $\ell \in \{1,\ldots, m\}$ such that
		\begin{enumerate}[(i)]
			\item $\deg_{\zn}(P_i) = 1$ if and only if $i < \ell$, and
			\item $w_{\zn}(P_\ell) = \min\{ w_{\zn}(P_i) : \ell \leq i \leq m\}$.
		\end{enumerate} In the case that $\ell = 1$, we additionally ensure, if possible, that the leading coefficient of $P_1$ is distinct from the leading coefficient of $P_m$.
		
		Denote by $S$ the set $(\{0,1,\ldots, H-1\}\cup \{N-H+1,N-H+2,\ldots, N-1\})^n \subset \mathbb{Z}_N^n$. Since $H \leq \lceil N/2 \rceil$, we observe that $|S| = (2H-1)^n$. Let $\mathcal H$ be the set of $h \in S$ for which the family $\mathbf F_h = \{P_1(y),\ldots, P_m(y), P_\ell(y+h), \ldots, P_m(y+h)\} \subset \zn[y_1,\ldots, y_n]$ is not essentially distinct.
		First, we claim that $|\mathcal H| \leq m^2(2H-1)^{n-1}$. Indeed, let $h \in \mathcal H$. Since the family $\mathbf P$ is essentially distinct and $\mathbf F_h$ is not essentially distinct, it follows that for some $i \in \{\ell,\ldots, m\}$, $j \in \{1,\ldots,m\}$, and $c \in \mathbb{Z}_N$, the equation $P_i(y+h) \equiv P_j(y) + c \bmod N$ holds for all $y \in \mathbb{Z}_N^n$.
		
		We observe that $\deg_{\zn}(P_i(y+h)) \leq \deg_{\zn}(P_i(y)) \leq k < \lpf N$ and $\deg_{\zn}(P_j(y)+c) \leq k < \lpf N$. It follows by Lemma~\ref{lem: mvpolyuniq} that the corresponding coefficients of $P_i(y+h)$ and $P_j(y)+c$ must agree modulo $N$.
		
		Observe that $2 \leq \deg_{\zn}(P_i) \leq k < \lpf N$ and the nonzero coefficients of $P_i$ are invertible modulo $N$ by the assumption that the $\zn$-height of $P_i$ is at most $M < \lpf N$ and by Lemma~\ref{lem: what is a unit}. Applying Lemma~\ref{lem: bound on translates violating ess dist} with $P = P_i$ and $Q = P_j$, there are at most $(2H-1)^{n-1}$ choices of $h \in S$ such that $P_i(y+h) = P_j(y) + c$ as polynomials over $\zn$. There were at most $m$ choices for $P_j$ and at most $m$ choices for $P_i$, proving the claim.
		
		Next, the assumption that $H \geq \max\{2,m^2\}$ implies that $2H - 1 > m^2$. With the help of this inequality, we observe that $\mathcal H$ is such that $S \setminus \mathcal{H}$ is nonempty since $|S| = (2H-1)^n > (2H-1)^{n-1}m^2 \geq |\mathcal{H}|$. By Lemma~\ref{linearization step}, there thus exists $h \in S \setminus \mathcal{H}$ such that
		\begin{multline} \label{basic application of lemma about linearization}
			\left| \Lambda_{P_1,\ldots,P_m}(f_0,\ldots,f_m) \right| \\ \leq \ 2^{n/2} \left[\left(\frac{|\mathcal H|}{H^n}\right)^{1/2}   + \left| \cavg{(x,y)}{R\times R'} \prod_{i=1}^m \prod_{\omega \in \{0,1\}}\mathcal{C}^{\omega+1} f_i(x + P_i(y+\omega h)-P_1(y))\right|^{1/2} \right].
		\end{multline}
		Note that $\left(\frac{|\mathcal H|}{H^n}\right)^{1/2} \leq \frac{2^{(n-1)/2}m}{H^{1/2}} \leq \frac{2^{n/2}m}{H^{1/2}}$.
		
		Let us prepare now to define the family $\mathbf Q$ as one of two families $\mathbf Q_0$ or $\mathbf Q_1$. First, we define $\mathbf Q_0$ and verify that it satisfies some of the claimed conditions about $\mathbf Q$ in the formulation of this lemma; later, we will define $\mathbf Q_1$ and verify that it satisfies some of the same conditions. In the process, we will discover whether to set $\mathbf Q = \mathbf Q_0$ or $\mathbf Q = \mathbf Q_1$ depending on $\mathbf P$. At the end, we will verify any remaining claimed conditions about $\mathbf Q$ in the formulation of this lemma.
		
		For each $i$ in the (possibly empty) set $\{1,\ldots, \ell -1\}$, let $a_i = P_i(h)-P_i(0)$; then
		\be
		\ol{f_i}(x+P_i(y)-P_1(y))f_i(x+P_i(y+h)-P_1(y)) = \Delta_{a_i}f_i(x+(P_i-P_1)(y)).
		\ee
		For the same $i$, set $g_{i-1} := \Delta_{a_i}f_i$ and $Q_{i-1}(y) := P_i(y)-P_1(y)$. For each $j$ with $0 \leq j \leq m-\ell$, define the functions $g_{\ell + 2j-1} := \ol{f_{\ell+j}}$ and $g_{\ell + 2j} := f_{\ell+j}$ and the polynomials $Q_{\ell+2j-1} := P_{\ell+j} - P_1$ and $Q_{\ell+2j}(y):= P_{\ell+j}(y+h) - P_1(y)$. It is easy to check that
		\be
		g_0(x)g_1(x+Q_1(y))\cdots g_{m'}(x+Q_{m'}(y)) = \prod_{i=1}^m \prod_{\omega\in\{0,1\}} \mathcal{C}^{\omega+1}f_i(x+P_i(y+\omega h) -P_1(y)),
		\ee
		where $m' = 2m - \ell \in \{m, \ldots, 2m\}$. By construction, $g_{m'} = f_m$, the $g_i$ are 1-bounded, and \eqref{basic application of lemma about linearization} implies
		\begin{equation}\label{intermediate eqn in PET inductive step}
			\left| \Lambda_{P_1,\ldots,P_m}(f_0,\ldots,f_m) \right| \ \leq \ 2^{n/2} \left[ \frac{2^{n/2}m}{H^{1/2}} + \left|\Lambda_{Q_1,\ldots,Q_{m'}}(g_0,\ldots,g_{m'})\right|^{1/2}\right].
		\end{equation}
		
		Define the family $\mathbf Q_0 := \{Q_1,\ldots, Q_{m'}\} 
		\subset \zn[y_1,\ldots, y_n]$. This family is essentially distinct, since adding $P_1$ to each polynomial yields the family $\mathbf F_h$, which has been guaranteed to be essentially distinct by our choice of $h$.
		
		If $Q_i \in \mathbf Q_0$ is of the form $P_{j_0}-P_{j_1}$, then clearly the $\zn$-height of $Q_i$ is bounded by $2M \leq \frac 12 (k+1)^{2kn}MH^k$. Otherwise, applying Lemma~\ref{lem: P(y+h)-Q(y) height bound}, we have that the $\zn$-height of $Q_i$ is at most $\frac 12 (k+1)^{2kn} MH^k$. Thus, the $\zn$-height of $\mathbf Q_0$ is at most $\frac 12 (k+1)^{2kn}MH^k$.
		
		Using the fact that, for any $i$, both of the polynomials $P_i(y)$ and $P_i(y+h)$ have the same leading coefficient, we compute the weight sequence $W_{\zn}(\mathbf Q_0)$. We define $t \geq 1$ to be the smallest integer such that $W_{\zn,t}(\mathbf P) > 0$. As a result of the process of choosing $\ell$, we have that $t = w(P_1)$. For $s > t$, we have $W_{\zn,s}(\mathbf Q_0) = W_{\zn,s}(\mathbf P)$. Indeed, if $w_{\zn}(P_i) > t$ for some $i$, then both of the polynomials $P_i(y) - P_1(y)$ and $P_i(y+h)-P_1(y)$ have the same leading coefficient as $P_i$ does. Next, it is similarly clear that $W_{\zn,t}(\mathbf Q_0) = W_{\zn,t}(\mathbf P) - 1$, since we have subtracted $P_1(y)$ from the relevant polynomials (the ones with weight $t = w_{\zn}(P_1)$). Indeed, if $w_{\zn}(P_i) = t$ for some $i \neq 1$, then both of the polynomials $P_i(y)-P_1(y)$ and $P_i(y+h)-P_1(y)$ have the same leading coefficient, although this time it will be different from the leading coefficient of $P_i$. For both of these computations of entries in the weight sequence of $\mathbf Q_0$, it makes no difference whether, for a given $i$, both $P_i-P_1$ and $P_i(y+h)-P_1(y)$ are actually included in $\mathbf Q_0$, since it is only important to notice whether new leading coefficients are introduced by passing to this new family with more polynomials. Finally, since there are only $m' \leq 2m$ polynomials in the family $\mathbf Q_0$, it is clear that $W_{\zn,1}(\mathbf Q_0) = i_1, W_{\zn,2}(\mathbf Q_0) = i_2, \ldots, W_{\zn,t-1}(\mathbf Q_0) = i_{t-1}$ for some nonnegative integers $i_1,\ldots, i_{t-1}$ such that $i_1 + \cdots + i_{t-1} \leq 2m$. Hence
		\begin{equation}\label{pre-figured computation of weight sequence of Q_0}
			W_{\zn}(\mathbf Q_0) \ = \ (i_1,\ldots,i_{t-1},W_{\zn,t}(\mathbf P)-1,W_{\zn,t+1}(\mathbf P), \ldots).
		\end{equation}
		
		
		Set $r = \max_{1\leq i \leq m'} w_{\zn}(Q_i)$. Then one of the five following mutually exclusive conditions holds:
		\begin{enumerate}[(a)]
			\item $\ell \geq 2$,
			\item $\ell = 1$ and $w_{\zn}(P_m) > w_{\zn}(P_1)$,
			\item $\ell = 1$, $w_{\zn}(P_1) = w_{\zn}(P_2) = \ldots = w_{\zn}(P_m)$, and $L_{\zn}(P_1) \neq L_{\zn}(P_m)$,
			\item $\ell = 1$, $w_{\zn}(P_1) = w_{\zn}(P_2) = \ldots = w_{\zn}(P_m)$, $L_{\zn}(P_1) = \ldots = L_{\zn}(P_m)$, and $w_{\zn}(Q_{m'}) = r$,
			\item $\ell = 1$, $w_{\zn}(P_1) = w_{\zn}(P_2) = \ldots = w_{\zn}(P_m)$, $L_{\zn}(P_1) = \ldots = L_{\zn}(P_m)$, and $w_{\zn}(Q_{m'}) < r$.
		\end{enumerate}
		In the first four conditions, $w_{\zn}(Q_{m'}) = r$ holds. Indeed, if $\ell \geq 2$, then $\deg_{\zn}(P_m) > \deg_{\zn}(P_1)$, so $Q_{m'}(y) = P_m(y+h)-P_1(y)$ has weight $r$. If $\ell = 1$ and $w_{\zn}(P_m) > w_{\zn}(P_1)$, then again $Q_{m'}$ has weight $r$. If $\ell = 1$, $w_{\zn}(P_1) = w_{\zn}(P_2) = \ldots = w_{\zn}(P_m)$, and it was possible earlier to ensure that the leading coefficient of $P_1$ is distinct from the leading coefficient of $P_m$, then still $Q_{m'}$ has maximal weight $r$. If $\ell = 1$, $w_{\zn}(P_1) = w_{\zn}(P_2) = \ldots = w_{\zn}(P_m)$, the leading coefficients of $P_1,\ldots, P_m$ are all the same, and it happens that $w_{\zn}(Q_{m'}) = r$, then certainly $w_{\zn}(Q_{m'}) = r$. Thus, if one of (a), (b), (c), or (d) holds, we will take $\mathbf Q = \mathbf Q_0$ to ensure that the maximal weight condition holds; the maximal degree condition and the condition that a permissible operation applied to $(m,W_{\zn}(\mathbf P))$ yields $(m',W_{\zn}(\mathbf Q))$ will both be checked later.
		
		Suppose the fifth condition holds. Then $\ell = 1$, $w_{\zn}(P_1) = w_{\zn}(P_2) = \ldots = w_{\zn}(P_m)$, the leading coefficients of $P_1,\ldots, P_m$ are all the same, but $w_{\zn}(Q_{m'}) < r$.
		By definition, $r = \max_{1\leq i \leq m'} w_{\zn}(Q_i)$, and hence we may choose $i' \in \{1,\ldots, m'\}$ such that $w_{\zn}(Q_{i'}) = r$. Since $w_{\zn}(Q_{m'}) < r$, it follows that $i' \neq m'$.
		
		Define the family $\mathbf Q_1 := \{Q'_{1},Q'_2,\ldots,Q'_{m'}\} \subset \zn[y_1,\ldots,y_n]$ by setting $Q'_{i'}:= -Q_{i'}$ and $Q'_i := Q_i - Q_{i'}$ for $i \neq i'$. Changing variables $x \mapsto x - Q_{i'}(y)$, we observe that
		\[ \Lambda_{Q_1,\ldots,Q_{m'}}(g_0,\ldots, g_{m'}) = \Lambda_{Q'_1,\ldots, Q'_{i'-1},Q'_{i'},Q'_{i'+1}, \ldots, Q'_{m'}}(g_{i'},g_1,\ldots, g_{i'-1},g_0,g_{i'+1},\ldots,g_{m'}),\]
		which helps \eqref{intermediate eqn in PET inductive step} to imply
		\[
		\left| \Lambda_{P_1,\ldots,P_m}(f_0,\ldots,f_m) \right| \ \leq \ 2^{n/2} \left[ \frac{2^{n/2}m}{H^{1/2}} + \left|\Lambda_{Q'_1,\ldots,Q'_{m'}}(g'_0,\ldots,g'_{m'})\right|^{1/2}\right],
		\]
		where $g'_0 := g_{i'}$, $g'_{i'} := g_0$, and $g'_i := g_i$ for $i \neq 0, i'$ are 1-bounded functions $R \to \C$ with $g'_{m'} = g_{m'} = f_m$.
		Moreover, $\mathbf Q_1$ is an essentially distinct family of polynomials (since $\mathbf Q_0$ is) with $\zn$-height at most $2\cdot \frac{1}{2} (k+1)^{2kn}MH^k$ (since the $\zn$-height of $\mathbf Q_0$ is at most $\frac{1}{2} (k+1)^{2kn}MH^k$) such that $w_{\zn}( Q'_{m'}) = r = \max_{1\leq i \leq m'} w_{\zn}(Q'_i)$. Thus, if (e) holds, we will take $\mathbf Q = \mathbf Q_1$ to ensure that the maximal weight condition holds; we need only to check the maximal degree condition and the condition that a permissible operation applied to $(m,W_{\zn}(\mathbf P))$ yields $(m',W_{\zn}(\mathbf Q))$.
		
		First, let us check the maximal degree condition in the five cases (a) through (e). Whether we take $\mathbf Q = \mathbf Q_0$ or $\mathbf Q = \mathbf Q_1$, it is clear that the maximal degree of $\mathbf Q$ is at most $k = \max_{1\leq i \leq m} \deg_{\zn}(P_i)$, since each polynomial in $\mathbf Q$ is a sum of finitely many polynomials of the form $\pm P_i(y)$ or $\pm P_i(y+h)$.
		
		Finally, we check that a permissible operation applied to $(m,W_{\zn}(\mathbf P))$ yields $(m',W_{\zn}(\mathbf Q))$. In all cases, we have $m' \in \{m,\ldots, 2m\}$, so it only remains to check that the weight sequence of $\mathbf Q$ obeys the required restrictions, which themselves slightly vary based on the weight sequence of $\mathbf P$. Consider the following statement:
		\begin{equation} \label{reprise: is the sequence a lonely one?}
			\text{There exists } s > n \text{ such that } W_{\zn,s}(\mathbf P) = 1 \text{ and } W_{\zn,i}(\mathbf P) = 0 \text{ for all } i \neq s.
		\end{equation}
		Whether \eqref{reprise: is the sequence a lonely one?} holds changes the definition of a permissible operation applied to $(m,W_{\zn}(\mathbf P))$. We first observe that ``either (d) or (e) holds'' is equivalent to ``\eqref{reprise: is the sequence a lonely one?} holds''. 
		
		On the one hand, suppose that \eqref{reprise: is the sequence a lonely one?} does not hold. Then (a), (b), or (c) holds, so we are in the situation where $\mathbf Q = \mathbf Q_0$. Recalling our computation
		\eqref{pre-figured computation of weight sequence of Q_0} of the weight sequence of $\mathbf Q_0$ and comparing to the definition of permissible operation (note $s' = t-1$), we see that a permissible operation applied to $(m,W_{\zn}(\mathbf P))$ yields $(m',W_{\zn}(\mathbf Q))$.
		
		On the other hand, suppose that \eqref{reprise: is the sequence a lonely one?} does hold, so that (d) or (e) holds. Then $s = t$. In this situation, whether we took $\mathbf{Q} = \mathbf{Q}_0$ or $\mathbf{Q} = \mathbf{Q}_1$, we are assured that the weight sequence $W_{\zn}(\mathbf Q)$ is not the zero sequence. Indeed, let $Q \in \mathbf Q$. Since $\mathbf Q$ is essentially distinct, $y \mapsto Q(y)$ is not a constant function $\mathbb{Z}_N^n \to \zn$ and hence $\deg_{\zn}(Q) > 0$; thus, $W_{\zn,w_{\zn}(Q)}(\mathbf Q) \geq 1$. 
		
		In the case that (d) holds, we have $\mathbf Q = \mathbf Q_0$. Recalling our computation \eqref{pre-figured computation of weight sequence of Q_0} and noting by the previous paragraph that $W_{\zn}(\mathbf Q_0)$ is not the zero sequence, it follows that not all the $i_j$ are zero, so a permissible operation applied to $(m,W_{\zn}(\mathbf P))$ yields $(m',W_{\zn}(\mathbf Q))$.
		
		Finally, in the case that (e) holds (so $\mathbf Q = \mathbf Q_1$), we compute $W_{\zn}(\mathbf Q_1)$ by means of the weight sequence of $\mathbf Q_0$. Recall our computation \eqref{pre-figured computation of weight sequence of Q_0}, which asserts that
		\begin{equation}
			W_{\zn}(\mathbf Q_0) \ = \ (i_1,\ldots,i_{t-1},W_{\zn,t}(\mathbf P)-1,W_{\zn,t+1}(\mathbf P), \ldots).
		\end{equation}
		Since condition \eqref{reprise: is the sequence a lonely one?} holds with $s = t$, this computation initially simplifies to
		\begin{equation}
			W_{\zn}(\mathbf Q_0) \ = \ (i_1,\ldots,i_{t-1},0,0, \ldots).
		\end{equation}
		Next, note that $r \leq t-1$ and hence $i_{j} = 0$ for $j$ belonging to the possibly empty set $\{r+1,r+2,\ldots, t-1\}$, so actually
		\begin{equation}
			W_{\zn}(\mathbf Q_0) \ = \ (i_1,\ldots, i_{r},0,0,\ldots).
		\end{equation}
		Now, let us turn to the weight sequence of $\mathbf Q_1$. First, since $\mathbf Q_0$ has no polynomials of weight more than $r$, neither does $\mathbf Q_1$; hence $W_{\zn,j}(\mathbf Q_1) = 0$ for all $j > r$.
		
		Next, we observe that $W_{\zn,r}(\mathbf Q_1) = W_{\zn,r}(\mathbf Q_0) = i_{r}$. Indeed, if $\{a_1,a_2,\ldots, a_{i_{r}}\}$ is the set of leading coefficients of polynomials in $\mathbf Q_0$ of weight $r$, and if $a_1$ denotes the leading coefficient of $Q_{i'}$, then $\{a_2-a_1,a_3-a_1,\ldots, a_{i_{r}}-a_1, -a_1\}$ is the set of leading coefficients of polynomials in $\mathbf Q_1$ of weight $r$.
		
		As for computing $W_{\zn,j}(\mathbf Q_1)$ for $j < r$, we observe that the only polynomials in $\mathbf Q_1$ of weight less than $r$ must arise by subtracting $Q_{i'}$ from a polynomial in $\mathbf Q_0$ of weight $r$ which has the same leading coefficient as $Q_{i'}$, namely $a_1$. The number of such polynomials is at most $2m-i_{r}$, because there are at least $i_{r} - 1$ polynomials in $\mathbf Q_0$ of weight $r$ that have leading coefficient not equal to $a_1$ modulo $N$, and because there is at least one polynomial in $\mathbf Q_0$ of weight not equal to $r$, namely $Q_{m'}$. Thus there exist nonnegative integers $i'_1,i'_2,\ldots, i'_{r-1}$ such that $i'_1 +i'_2 +\cdots + i'_{r-1} \leq 2m - i_{r}$ and $W_{\zn,j}(\mathbf Q_1) = i'_j$ holds for each $j < r$. To summarize, we have
		\begin{equation}W_{\zn}(\mathbf Q_1) = (i'_1,i'_2,\ldots, i'_{r-1},i_{r},0,0,\ldots)
		\end{equation}
		with $i'_1+i'_2+\dots +i'_{r-1} + i_r \leq 2m$. Moreover, we already argued that $W_{\zn}(\mathbf Q_1)$ is not the zero sequence, though it is perhaps just as clear that $i_r$ is nonzero. In any case, a permissible operation applied to $(m,W_{\zn}(\mathbf P))$ yields $(m',W_{\zn}(\mathbf Q_1))$.
		
		This completes the proof in each of the exhaustive and mutually exclusive conditions (a) through (e).	
	\end{proof}

	\subsection{Proof of Proposition~\ref{prop: Us control}}\label{subsec: proof of prop: Us control}
	For the reader's convenience, we recall Proposition~\ref{prop: Us control}. 
	
	\begin{proprep}[\ref{prop: Us control}]
		Let $m \geq 2$ and $n \geq 1$ be integers. Let $\mathbf{P} = \{P_1(y),\ldots, P_m(y)\} \subset \Z[y_1,\ldots,y_n]$ be an independent family of polynomials with zero constant term. 
		Then there exist $\lambda \in (0,1]$, $\ve, C_0 > 0$, and $s \in \N$ with $s \geq 2$, each depending only on $\mathbf P$, along with a constant $C_1$ that depends on $\mathbf P$ and $\ve$, such that, for any finite commutative ring $R$ with characteristic $N$ satisfying $\lpf N > C_1$ and any 1-bounded functions $f_0,\ldots, f_m : R \to \C$,
		\be\label{eqn: negative power of lpf N}
		\left| \Lambda_{P_1,\ldots,P_m}(f_0,\ldots, f_m)\right| \ \leq \ \frac{C_0}{\lpf{N}^{\ve}} + 2^n\min_{0\leq j \leq m} \norm{f_j}_{U^s}^\lambda.
		\ee
	\end{proprep}
	
	\begin{proof}
		Define $k := \max_{1\leq j \leq m} \deg(P_j)$, and observe by linear independence that $k > 0$. We handle the $k = 1$ and $k > 1$ cases in different ways.
		
		First suppose $k = 1$. Then there exist integers $a_i^{(j)}$, $i \in \{1,\ldots, n\}$ and $j \in \{1,\ldots, m\}$, such that $P_j(y) = \sum_{i=1}^n a_i^{(j)} y_i$ for each $j \in \{1,\ldots, m\}$. Form the $n \times m$ integer matrix $A$ with entries given by $A_{i,j} := a_i^{(j)}$. We can put $A$ into reduced row echelon form using Gauss--Jordan elimination. Moreover, since $A$ was composed of $m$ linearly independent columns, the resulting matrix will be the $m \times m$ identity matrix with $n-m$ rows of zeros attached to the bottom.
		
		Let $R$ be a finite commutative ring with characteristic $N$. In the previous paragraph, we described the result of Gauss--Jordan elimination (over $\Q$) applied to the matrix $A$. However, to be useful to us, this process should be implemented in the ring $R$ as well, so that
		\begin{equation}
			\Lambda_{P_1,\ldots,P_m}(f_0,\ldots, f_m) \ = \ \cavg{x}{R}\cavg{y_1,\ldots, y_m}{R} f_0(x) f_1(x+y_1) \cdots f_m(x+y_m)
		\end{equation}
		for any given $f_j : R \to \C$. To justify this equality, first view $A$ as a matrix with entries in $R$. Elementary row operations on this matrix correspond to either switching the labels on two variables $y_{i_1}$ and $y_{i_2}$ for $i_1 \neq i_2$, changing variables $y_i \mapsto \alpha y_i$ by the scale factor $\alpha \in R$, or changing variables $y_{i_1} \mapsto y_{i_1} + \alpha y_{i_2}$, where $i_1 \neq i_2$.  Since the entries of $A$ are integers, it suffices to use only rational scale factors $\alpha = p/q$ to reduce $A$. As long as $p$ and $q$ are invertible in $R$, which can be guaranteed by requiring $\lpf{N}$ to be larger than $\max\{|p|,|q|\}$, either of the changes of variables $y_i \mapsto \frac{p}{q}y_i$ and $y_{i_1} \mapsto y_{i_1} + \frac pq y_{i_2}$ is valid in $R$. Moreover, only finitely many different values of $\alpha$ are needed to reduce $A$. Hence, there exists $C_1$ depending on $A$ (and hence $\mathbf P$) such that, for any finite commutative ring $R$ with characteristic $N$ satisfying $\lpf N > C_1$ and any 1-bounded functions $f_0,\ldots, f_m : R \to \C$, one has
		\begin{align*}
			\left| \Lambda_{P_1,\ldots,P_m}(f_0,\ldots, f_m)\right| \ & = \ \left|\cavg{x}{R}\cavg{y_1,\ldots, y_m}{R} f_0(x) f_1(x+y_1) \cdots f_m(x+y_m) \right| \\
			& = \ \left| \cavg{x}{R} f_0(x) \prod_{j=1}^m \cavg{y_j}{R} f_j(y_j)\right| \\
			& \leq \ \min_{0 \leq j \leq m} \norm{f_j}_{U^1}
		\end{align*}
		by the change of variables $y_j \mapsto y_j - x$  for $j \in \{1,\ldots, m\}$ and 1-boundedness, from which the claimed result follows by setting $s = 2$ and $\lambda = 1$, choosing $C_0$ and $\ve$ arbitrarily, and recalling that $\norm{f_j}_{U^1} \leq \norm{f_j}_{U^2}$.
		
		Now suppose $k > 1$, so that $(m,W(\mathbf P)) \in \mathcal{S}\setminus \mathcal S_{\rm{degree}\leq 1}$.
		
		By reordering $\mathbf P$ if necessary, without loss of generality, we may assume $w(P_m) = \max_{1\leq j \leq m} w(P_j)$. (It is without loss of generality since we have not fixed the functions $f_j$ yet.) Define the family $\mathbf{P'} = \{P_1-P_m,\ldots, P_{m-1}-P_m, -P_m\} \subset \Z[y_1,\ldots, y_n]$.
		
		Applying Proposition~\ref{prop: weight sequence initial stability}, there exists a constant $C_{\rm{stable}} > 0$ such that whenever $N$ is a positive integer with $\lpf N > C_{\rm{stable}}$, we have
		\begin{enumerate}
			\item $w_{\zn}(P) = w(P)$ for any $P \in \mathbf P \cup \mathbf{P'}$,
			\item $W_{\zn}(\mathbf P) = W(\mathbf P)$,
			\item $W_{\zn}(\mathbf{P'}) = W(\mathbf{P'})$, and
			\item $\max_{1 \leq j \leq m} \deg_{\zn}(P_j) = \max_{1 \leq j \leq m} \deg(P_j)$.
		\end{enumerate}
		
		By Proposition~\ref{prop: R-essential distinctness lpf condition}, there exists $C^* > 0$ such that if $N$ is a positive integer such that $\lpf N > C^*$, then each of the families $\mathbf P$ and $\mathbf{P'}$, viewed as finite subsets of $\zn[y_1,\ldots, y_n]$, is essentially distinct.

		By Proposition~\ref{alg stops}, there exists $T = T(m)$ such that any PET algorithm applied to $(m,W(\mathbf P))$ or to $(m,W(\mathbf P'))$ terminates after at most $T$ moves.
		
		Set $\lambda = 2^{-T}$. Let $C_{k,n} = (k+1)^{2kn}$. Let $\alpha = 2^{2^{2T}nm^2}$. Let $\ve > 0$ be such that $\ve < \frac{1}{2^T\alpha kT}$. Set $C_0 = 2^nTm$ and $s = 2^Tm$. Set
		\[
		C_1 = \max\left\{C_{\rm{stable}},k,C^*,(2m^2)^{(\ve)^{-1}},\left( 2+2^{2T}\right)^{(1-2^T\ve)^{-1}}, \left( 2\left(2^{kT^2+4}MC_{k,n}^T \right)^\alpha \right)^{(1-2^T\alpha kT\ve)^{-1}}
		\right\}.
		\]

		Let $R$ be a finite commutative ring with characteristic $N$ satisfying $\lpf N > C_1$, and let $f_0,\ldots, f_m : R \to \C$ be 1-bounded functions.
		
		The desired inequality \eqref{eqn: negative power of lpf N} holds if, for each $j \in \{0,1,\ldots, m\}$,
		\be\label{eqn: negative power of lpf N, nonminimum version}
		\left| \Lambda_{P_1,\ldots,P_m}(f_0,\ldots, f_m)\right| \ \leq \ \frac{C_0}{\lpf{N}^{\ve}} + 2^n\norm{f_j}_{U^s}^\lambda.
		\ee
		
		Let $M$ be the $\zn$-height of $\mathbf P$. 
		
		For simplicity, we will show \eqref{eqn: negative power of lpf N, nonminimum version} in the case $j = m$. The general case $j \neq m$ can be obtained by one of two modifications: either interchange $P_j$ with $P_m$ if $w(P_j) = w(P_m)$, or translate $x \mapsto x - P_m(y)$ if $w(P_j) < w(P_m)$. Regardless of the modification used, the $\zn$-height of the polynomial family which plays the role of $\mathbf{P}$ is at most $2M$. Therefore, we will showcase the proof of \eqref{eqn: negative power of lpf N, nonminimum version} when $j = m$ as though the $\zn$-height of $\mathbf P$ were actually $2M$, in order to be fully convinced of the argument's lack of dependence on $j \in \{0,\ldots, m\}$.
		
		We will show \eqref{eqn: negative power of lpf N, nonminimum version} by repeated applications of Lemma~\ref{lem: PET inductive step}. 
		
		Set $m_0 := m$. For each $j \in \{1,\ldots, m\}$, define $P_j^{(0)} \in \zn[y_1,\ldots, y_n]$ as $P_j$ with coefficients reduced modulo $N$. Let $\mathbf{P}_0 := \{P_1^{(0)},\ldots, P_m^{(0)}\} \subset \zn[y_1,\ldots,y_n]$. Define $f_0^{(0)} := f_0$ and $f_j^{(0)} := f_j$ for each $j \in \{1,\ldots,m\}$.
		
		We prepare to apply Lemma~\ref{lem: PET inductive step}. Let $H_1$ be an integer to be chosen later which satisfies $\max\{2,m_0^2\} \leq H_1 \leq \lceil N/2 \rceil$. The family $\mathbf{P}_0$ is essentially distinct since $C^* \leq C_1 < \lpf N$. Since we chose to ensure $C_1 \geq \max\{k,C_{\rm{stable}}\}$, it also follows that
		\[ \max_{1 \leq j \leq m} \deg_{\zn}(P^{(0)}_j) = \max_{1 \leq j \leq m} \deg_{\zn}(P_j) = \max_{1 \leq j \leq m} \deg(P_j) = k \leq C_1 < \lpf N \]
		and $\max_{1 \leq j \leq m} \deg_{\zn}(P^{(0)}_j) = k > 1$.
		Moreover, since $C_1 \geq C_{\rm{stable}}$, we have \[w_{\zn}(P^{(0)}_m) = w_{\zn}(P_m) = w(P_m) = \max_{1 \leq j \leq m} w(P_j) =  \max_{1 \leq j \leq m} w_{\zn}(P_j) = \max_{1 \leq j \leq m} w_{\zn}(P^{(0)}_j).\]
		Hence, provided that $2M < \lpf N$, it follows by Lemma~\ref{lem: PET inductive step} that there exist a positive integer $m_1 \leq 2^1m_0$, 1-bounded functions $f_0^{(1)},f_1^{(1)},\ldots, f_{m_1}^{(1)} : R \to \C$ with $f_{m_1}^{(1)} = f_m$, and an essentially distinct family of polynomials $\mathbf P_1 = \{P_1^{(1)},\ldots, P_{m_1}^{(1)} \} \subset \zn[y_1,\ldots, y_n]$ such that the maximal degree and weight conditions $k':= \max_{1\leq j \leq m_1} \deg_{\zn}(P_{j}^{(1)}) \leq k$ and $w_{\zn}(P_{m_1}^{(1)}) = \max_{1 \leq j \leq m_1} w_{\zn}(P_j^{(1)})$ hold, the $\zn$-height of $\mathbf P_1$ is at most $2C_{k,n}M H_1^k$, a permissible operation applied to $(m_0,W_{\zn}(\mathbf P_0))$ yields $(m_1,W_{\zn}(\mathbf P_1))$, and
		\be
		\left| \Lambda_{P_1^{(0)},\ldots,P_m^{(0)}}(f_0^{(0)},\ldots,f_m^{(0)}) \right| \ \leq \ 2^{n/2} \left[ \frac{2^{n/2}m_0}{H_1^{1/2}} + \left|\Lambda_{P_1^{(1)},\ldots,P_{m_1}^{(1)}}(f_0^{(1)},\ldots,f_{m_1}^{(1)})\right|^{1/2}\right].
		\ee
		This first application of Lemma~\ref{lem: PET inductive step} performs a permissible operation to $a^*=(m_0,W_{\zn}(\mathbf{P}_0)) \in \mathcal{S}\setminus \mathcal S_{\rm{degree}=1}$ to obtain the pair $b^* = (m_1,W_{\zn}(\mathbf{P}_1)) \in \mathcal{S} \setminus \mathcal{S}_{\rm{degree}=0}$ in accordance with a PET algorithm applied to $a^*$. If in fact $b^* \in \mathcal S_{\rm{degree}=1}$, such a PET algorithm would terminate here; moreover, since $k'$ would equal 1 in this case, we could not further apply Lemma~\ref{lem: PET inductive step} to bound $\left|\Lambda_{P_1^{(1)},\ldots,P_{m_1}^{(1)}}(f_0^{(1)},\ldots,f_{m_1}^{(1)})\right|$. For clarity, we illustrate what happens when $b^* \in \mathcal{S}\setminus \mathcal S_{\rm{degree} \leq 1}$, i.e., when $k' > 1$ and a PET algorithm for $a^*$ has not yet been completed, leaving open the possibility of applying Lemma~\ref{lem: PET inductive step} again to bound $\left|\Lambda_{P_1^{(1)},\ldots,P_{m_1}^{(1)}}(f_0^{(1)},\ldots,f_{m_1}^{(1)})\right|$. Afterwards, we summarize the inequality or inequalities we obtain when $b^*$ belongs either to $\mathcal S_{\rm{degree}=1}$ or to $\mathcal{S}\setminus \mathcal S_{\rm{degree} \leq 1}$, no matter how many steps of a PET algorithm for $a^*$, or applications of Lemma~\ref{lem: PET inductive step}, it requires.
		
		We prepare to apply Lemma~\ref{lem: PET inductive step}; the hypotheses are easier to check now than in the first application of Lemma~\ref{lem: PET inductive step} above. Let $H_2$ be an integer to be chosen later which satisfies $\max\{2,m_1^2\} \leq H_2 \leq \lceil N/2 \rceil$. The family $\mathbf P_1$ is essentially distinct because the previous application of Lemma~\ref{lem: PET inductive step} ensured it; for the same reason, the maximal degree $k'$ satisfies $1 < k' \leq k < \lpf N$ and the weight of $P^{(1)}_{m_1}$ is maximal in $\mathbf P_1$. Thus, provided that $2C_{k,n}MH_1^k < \lpf N$, by Lemma~\ref{lem: PET inductive step}, there exist a positive integer $m_2 \leq 2m_1 \leq 2^2m_0$, 1-bounded functions $f_0^{(2)},f_1^{(2)},\ldots, f_{m_2}^{(2)} : R \to \C$ with $f_{m_2}^{(2)} = f_{m_1}^{(1)} = f_m$, and an essentially distinct family of polynomials $\mathbf P_2 = \{P_2^{(2)},\ldots, P_{m_2}^{(2)} \} \subset \zn[y_1,\ldots,y_n]$
		such that the maximal degree and weight conditions $\max_{1\leq j \leq m_2} \deg_{\zn} (P_{j}^{(2)}) \leq k' \leq k$ and  $w_{\zn}(P_{m_2}^{(2)}) = \max_{1 \leq j \leq m_2} w_{\zn}(P_j^{(2)})$ hold, the $\zn$-height of $\mathbf P_2$ is at most \begin{equation} 2C_{k,n}(k'+1)^{2k'n} MH_1^kH_2^{k'} \ \leq \ 2C_{k,n}^2MH_1^kH_2^k,
		\end{equation} a permissible operation applied to $(m_1,W_{\zn}(\mathbf P_1))$ yields $(m_2,W_{\zn}(\mathbf P_2))$, and
		\be
		\left|\Lambda_{P_1^{(1)},\ldots,P_{m_1}^{(1)}}(f_0^{(1)},\ldots,f_{m_1}^{(1)})\right| \ \leq \ 2^{n/2} \left[ \frac{2^{n/2}m_1}{H_2^{1/2}} + \left|\Lambda_{P_1^{(2)},\ldots,P_{m_2}^{(2)}}(f_0^{(2)},\ldots,f_{m_2}^{(2)})\right|^{1/2}\right].
		\ee
		This second application of Lemma~\ref{lem: PET inductive step} performs a permissible operation on $(m_1,W_{\zn}(\mathbf{P}_1))$ to obtain the pair $(m_2,W_{\zn}(\mathbf{P}_2)) \in \mathcal{S} \setminus \mathcal{S}_{\rm{degree}=0}$ in continued accordance with a PET algorithm applied to $a^*$. Applying Lemma~\ref{lem: PET inductive step} repeatedly in this way (until we must stop because the maximal degree of some $\mathbf{P}_j$ is not greater than 1) follows a PET algorithm in the sense that the pairs $(m_1, W_{\zn}(\mathbf P_1))$, $(m_2, W_{\zn}(\mathbf P_2))$, and so on are intermediate steps in a PET algorithm applied to $a^* = (m_0, W_{\zn}(\mathbf P_0))$. Thus, we apply Lemma~\ref{lem: PET inductive step} until we reach the last step of whichever PET algorithm we are compelled to follow by the repeated applications of the lemma, which in the end yields a pair belonging to $\mathcal{S}_{\rm{degree}=1}$. 
		We conclude the following claim:
		\begin{cla}
			There exists a positive integer $D = D(a^*) \leq T$ such that, for every $i \in \{1,\ldots, D\}$, there exist a positive integer $m_{i} \leq 2^{i} m_0$, 1-bounded functions $f_0^{(i)},\ldots, f_{m_i}^{(i)} : R \to \C$ with $f_{m_{i}}^{(i)} = f_m$, and an essentially distinct family of polynomials $\mathbf{P}_{i} = \{P_1^{(i)},\ldots, P_{m_{i}}^{(i)}\} \subset \zn[y_1,\ldots,y_n]$ of $\zn$-height at most $2C_{k,n}^{i}M \left( \prod_{j=1}^{i} H_j \right)^k$ such that
			\be\label{eqn: claim nested counting estimate}
			\left|\Lambda_{P_1^{(i-1)},\ldots,P_{m_{i-1}}^{(i-1)}}(f_0^{(i-1)},\ldots,f_{m_{i-1}}^{(i-1)})\right| \ \leq \ 2^{n/2} \left[ \frac{2^{n/2}m_{i-1}}{H_{i}^{1/2}} + \left|\Lambda_{P_1^{(i)},\ldots,P_{m_{i}}^{(i)}}(f_0^{(i)},\ldots,f_{m_{i}}^{(i)})\right|^{1/2}\right].
			\ee 
			Moreover, the family $\mathbf{P}_D$ consists of polynomials of degree 1 since $(m_D,W_{\zn}(\mathbf{P}_D)) \in \mathcal {S}_{\rm{degree}=1}$.  
			All of this holds provided the following conditions are met:
			\begin{enumerate}
				\item For each $i \in \{1,\ldots, D\}$, given $m_{i-1}$, an integer $H_{i}$ can be chosen to satisfy $\max\{2,m_{i-1}^2\} \leq H_{i} \leq \lceil N/2 \rceil$.
				\item For each $i \in \{0,\ldots, D-1\}$, we have $2C_{k,n}^i M\left( \prod_{j=1}^i H_j \right)^k < \lpf N$.
			\end{enumerate}
		\end{cla}
		Analyzing the inequalities offered by equation \eqref{eqn: claim nested counting estimate} for each $i \in \{1,\ldots, D\}$, and using the fact that $(a+b)^{1/2} \leq a^{1/2} + b^{1/2}$ for nonnegative real numbers $a$ and $b$, it can be shown by induction that for each $\ell \in \{1,\dots, D\}$, one has
		\begin{equation}
			\left| \Lambda_{P_1,\ldots,P_m}(f_0,\ldots, f_m)\right| \ \leq \ 2^n \sum_{i=1}^{\ell}  \left(\frac{m_{i-1}}{H_{i}^{1/2}} \right)^{2^{-(i-1)}} + 2^{n(1-\frac{1}{2^\ell})} \left|\Lambda_{P_1^{(\ell)},\ldots,P_{m_\ell}^{(\ell)}}(f_0^{(\ell)},\ldots, f_{m_\ell}^{(\ell)}) \right|^{2^{-\ell}}.
		\end{equation}
		Hence, when $\ell = D$, crudely bounding $2^{n(1-\frac{1}{2^{D}})} \leq 2^n$ and $\left( \frac{m_{i-1}}{H_{i}^{1/2}}  \right)^{2^{-(i-1)}} \leq \left( \frac{m_{i-1}}{H_{i}^{1/2}}  \right)^{2^{-(D-1)}}$ for $i \in \{1,\dots, D\}$, we obtain
		\be\label{eqn: claim unnested counting estimate}
		\left| \Lambda_{P_1,\ldots,P_m}(f_0,\ldots, f_m)\right| \ \leq \ 2^n \sum_{i=1}^{D} \left(\frac{m_{i-1}}{H_{i}^{1/2}} \right)^{2^{-(D-1)}} + 2^n \left|\Lambda_{P_1^{(D)},\ldots,P_{m_D}^{(D)}}(f_0^{(D)},\ldots, f_{m_D}^{(D)}) \right|^{2^{-D}}.
		\ee
		By the claim, the $\zn$-height of $\mathbf{P}_D$ is at most $M' := 2C_{k,n}^{D}M \left( \prod_{j=1}^{D} H_j \right)^k$.
		Since $\mathbf P_D \subset \zn[y_1,\ldots, y_n]$ is a family of polynomials of degree 1, we may, for each $j \in \{1,\ldots, m_D\}$, define the polynomial $\tilde{P}_j(y) := \sum_{i=1}^n a_i^{(j)} y_i$, with $a_i^{(j)} \in \zn$, such that $\tilde{P}_j(y)$ and $P_j^{(D)}(y)$ differ by a constant. For each $j \in \{1, \ldots, m_D\}$, we also define the 1-bounded function $\tilde{f}_j(x) := f^{(D)}_{j}(x+P_j^{(D)}(y)-\tilde{P}_j(y))$, which is just a translate of $f^{(D)}_j$. In particular, it follows that $\norm{\tilde{f}_{m_D}}_{U^{m_D}} = \norm{f^{(D)}_{m_D}}_{U^{m_D}}$, which we will use in a moment. Now, since $\mathbf P_D$ is essentially distinct and has $\zn$-height at most $M'$, the family $\tilde{\mathbf{P}} := \{\tilde{P}_1,\ldots, \tilde{P}_{m_D} \} \subset \zn[y_1,\ldots, y_n]$ is also essentially distinct and has $\zn$-height at most $M'$.
		
		It will now be helpful to define the \emph{$\zn$-height} of a matrix $A$ with entries in $\zn$ as the smallest nonnegative integer $M$ such that the $\zn$-height of each entry in $A$ is at most $M$. 
		
		Form the $n \times m_D$ matrix $A$ with entries given by $A_{i,j} = a_i^{(j)} \in \zn$.  Since $\tilde{\mathbf{P}}$ is essentially distinct, every column of $A$ has a nonzero entry and no two columns of $A$ are identical. Provided that $(8M')^{2^{nm_D^2}} < N$, if we apply Proposition~\ref{prop: linearized family nonzeroing+} with $M_0 = M'$ and $m = m_D$, then we conclude that $A$ can be transformed by a series of elementary row operations (on the set of $n \times m_D$ matrices with entries in $\zn$) to a matrix $B$ of $\zn$-height at most $M'':= (8M')^{2^{nm_D^2}}$ such that, within each row of $B$, all entries are distinct and nonzero. All such elementary row operations can be viewed as being done on the set of $n \times m_D$ matrices with entries in $R$ since $\zn$ is a subring of $R$. Notate $B = (b_{i}^{(j)})$. Thus, we have
		\begin{align*}
			\Lambda_{P_1^{(D)},\ldots,P_{m_D}^{(D)}}(f_0^{(D)},\ldots, f_{m_D}^{(D)}) \ & = \ \Lambda_{\tilde{P}_1,\ldots, \tilde{P}_{m_D}}(f_0^{(D)},\tilde{f}_1,\ldots, \tilde{f}_{m_D}) \\
			& = \ \cavg{x}{R} \cavg{y}{R'}f_0^{(D)}(x) \prod_{i=1}^{m_D} \tilde{f}_i (x + b^{(i)}\cdot y), 
		\end{align*}
		where $R' = R^n$ with coordinatewise addition and multiplication, the dot product is with vector entries in $R$, and, for each $j \in \{1,\ldots, m_D\}$, $b^{(j)} := (b_1^{(j)},\ldots, b_n^{(j)})$ is the $j$th column of $B$. By construction, within each row of $B$, all entries are distinct and nonzero. This implies that $b^{(i)}$ and $b^{(i)} - b^{(j)}$ for $i, j \in \{1,\ldots, m_D\}$ with $i \neq j$ are vectors in $\mathbb{Z}_N^n$ with nonzero entries; moreover, since the $\zn$-height of $B$ is at most $M''$, then if the $H_j$ can additionally be chosen so that $2M'' < \lpf N$, it follows by Lemma~\ref{lem: what is a unit} that each entry of $b^{(i)}$ and $b^{(i)}-b^{(j)}$ is invertible modulo $N$. We conclude by Lemma~\ref{lem: Ud bound if invertible} that \be
		\left|\Lambda_{P_1^{(D)},\ldots,P_{m_D}^{(D)}}(f_0^{(D)},\ldots, f_{m_D}^{(D)}) \right| \ \leq \ \norm{\tilde{f}_{m_D}}_{U^{m_D}} \ = \ \norm{f^{(D)}_{m_D}}_{U^{m_D}} \ = \ \norm{f_m}_{U^{m_D}}
		\ee
		since $f_{m_D}^{(D)} = f_m$, which implies with the help of \eqref{eqn: claim unnested counting estimate} that
		\be
		\left| \Lambda_{P_1,\ldots,P_m}(f_0,\ldots, f_m)\right| \ \leq \ 2^n \sum_{i=1}^{D} \left(\frac{m_{i-1}}{H_{i}^{1/2}} \right)^{2^{-(D-1)}} + 2^n  \norm{f_m}_{U^{m_D}}^{2^{-D}}.
		\ee
		Moreover, we have
		\be
		\norm{f_m}_{U^{m_D}}^{2^{-D}} \ \leq \ \norm{f_m}_{U^{2^Dm}}^{2^{-D}} \ \leq \ \norm{f_m}_{U^{s}}^{2^{-D}} \ \leq \ \norm{f_m}_{U^{s}}^{\lambda}
		\ee
		since $m_D \leq 2^Dm \leq 2^Tm = s$ and the Gowers norms are nondecreasing and since $\norm{f_m}_{U^s} \leq 1$ and $1 \geq 2^{-D} \geq \lambda$. Thus, we have
		\be\label{eqn: intermediate eqn 1}
		\left| \Lambda_{P_1,\ldots,P_m}(f_0,\ldots, f_m)\right| \ \leq \ 2^n \sum_{i=1}^{D} \left(\frac{m_{i-1}}{H_{i}^{1/2}} \right)^{2^{-(D-1)}} + 2^n\norm{f_m}_{U^{s}}^{\lambda}.
		\ee
		
		Now, let us clarify the previously unspecified choices of $H_i$. This will propel us towards the conclusion. For each $i \in \{1,\ldots, D\}$, set $H_i = 2^{2i-2} \cdot \lceil \lpf{N}^{2^D\ve} \rceil$. First, observe that
		\be
		\frac{m_{i-1}}{H_i^{1/2}} \ \leq \ \frac{m}{\lceil \lpf{N}^{2^D\ve}\rceil^{1/2}} \ \leq \ \frac{m}{ \lpf{N}^{2^{D-1}\ve}},
		\ee
		which implies
		\be
		2^n \sum_{i=1}^{D} \left(\frac{m_{i-1}}{H_{i}^{1/2}} \right)^{2^{-(D-1)}} \ \leq \ \frac{2^nDm^{2^{-(D-1)}}}{\lpf N^\ve} \ \leq \ \frac{2^nTm}{\lpf N^\ve} \ = \  \frac{C_0}{\lpf N^\ve},
		\ee
		so \eqref{eqn: intermediate eqn 1} implies
		\be
		\left| \Lambda_{P_1,\ldots,P_m}(f_0,\ldots, f_m)\right| \ \leq \ \frac{C_0}{\lpf N^\ve} + 2^n\norm{f_m}_{U^{s}}^{\lambda},
		\ee
		which proves \eqref{eqn: negative power of lpf N, nonminimum version} in the case $j = m$.
		
		Finally, we summarize and check the conditions of the claim and of the subsequent applications of Proposition~\ref{prop: linearized family nonzeroing+} and Lemma~\ref{lem: Ud bound if invertible}:
		\begin{enumerate}
			\item For each $i \in \{1,\ldots, D\}$, we have $\max\{2,m_{i-1}^2\} \leq H_{i} \leq \lceil N/2 \rceil$.
			\item For each $i \in \{0,\ldots, D-1\}$, we have $2C_{k,n}^i M\left( \prod_{j=1}^i H_j \right)^k < \lpf N$.
			\item $(8M')^{2^{nm_D^2}} < N$, where $M' = 2C_{k,n}^{D}M \left( \prod_{j=1}^{D} H_j \right)^k$.
			\item $2M'' < \lpf N$, where $M'' = (8M')^{2^{nm_D^2}}$.
		\end{enumerate}
		It is clear that the fourth condition implies $M'' < \lpf N$, which clearly implies the second and third conditions. Thus, we check the first and fourth conditions to finish the proof.
		
		By the choice of $C_1$, which implies that $\lpf N > C_1 \geq (2m^2)^{(\ve)^{-1}}$, we have
		\be
		\max\{2,m_{i-1}^2\} \ \leq \ 2m_{i-1}^2 \ \leq\ 2^{2i-1} m^2\ \leq\ 2^{2i-2} \lpf{N}^\ve\ \leq\ 2^{2i-2} \lpf{N}^{2^D\ve} \ \leq \  H_i
		\ee
		for each $i$. Moreover, for any $i \in \{1,\ldots, D\}$, $H_i \leq \lceil N/2 \rceil$ is implied by the single inequality $H_D \leq \lceil N/2 \rceil$, which follows from $\frac{\lpf N - 1}{2} = \lfloor \frac{\lpf N}{2} \rfloor \leq \lceil N/2\rceil$ and the inequality
		\be
		\lceil\lpf{N}^{2^D\ve}\rceil \left(\frac{ \lpf{N}}{\lceil \lpf{N}^{2^D\ve}\rceil} -2^{2D-1}\right) \ \geq \ 1,
		\ee
		which holds if both factors are at least one. Indeed, obviously $\lceil\lpf{N}^{2^D\ve}\rceil \geq 1$; one also sees that $\frac{ \lpf{N}}{\lceil \lpf{N}^{2^D\ve}\rceil} -2^{2D-1} \geq 1$ holds since it is implied by the inequality
		\be
		\frac{ \lpf{N}^{1-2^D\ve}}{2} -2^{2D-1} \ \geq \ 1,
		\ee
		which holds by the choice of $C_1 \geq (2+2T)^{(1-2^T\ve)^{-1}}$ since $\lpf{N}^{1-2^D\ve} \geq \lpf{N}^{1-2^T\ve} \geq 2^{2T} + 2 \geq 2^{2D} + 2$. This verifies the first condition.
		
		As for the fourth condition, we first observe that
		\begin{equation}
			\left( \prod_{j=1}^D H_j \right)^k\ = \ 2^{kD(D-1)}\lceil \lpf{N}^{2^D\ve} \rceil^{kD},
		\end{equation}
		which implies that
		\begin{multline}\label{eqn: intermediate eqn 2}
			8M' \ = \ 2^{kD(D-1)+4}C_{k,n}^D M\lceil \lpf{N}^{2^D\ve} \rceil^{kD} \ \leq \ 2^{kD^2+4}C_{k,n}^D M\left( \lpf{N}^{2^D\ve} \right)^{kD} \\ \leq \ 2^{kT^2+4}C_{k,n}^T M\lpf{N}^{2^T\ve kT},
		\end{multline}
		where the first inequality holds since $\lceil x \rceil \leq 2x$ for any real $x>1$ and the second inequality holds since $D \leq T$. Then, we observe that 
		\begin{multline}
			2M'' \ = \ 2(8M')^{2^{nm_D^2}} \ \leq \ 2(8M')^{2^{n(2^Dm)^2}} \ \leq \ 2(8M')^{\alpha} \ \leq \ 2(2^{kT^2+4}C_{k,n}^T M\lpf{N}^{2^T\ve kT})^\alpha
			\\ = \ 2(2^{kT^2+4}C_{k,n}^T M)^\alpha \lpf{N}^{2^T\alpha kT\ve} \ < \ \lpf N,
		\end{multline}
		where the third inequality holds by \eqref{eqn: intermediate eqn 2} and the fourth inequality holds since $\lpf N > C_1 \geq \left( 2\left(2^{kT^2+4}MC_{k,n}^T \right)^\alpha \right)^{(1-2^T\alpha kT\ve)^{-1}}$. This completes the proof.
	\end{proof}

	\section{A matrix manipulation in the proof of Proposition~\ref{prop: Us control}}
	\label{sec: fixing a matrix}
	
	In this section, we prove Proposition~\ref{prop: linearized family nonzeroing+}, which is necessary for our proof of Proposition~\ref{prop: Us control} in the multivariable case $n > 1$. See Subsection~\ref{subsec: prop: linearized family nonzeroing+ might be necessary} for an example demonstrating the necessity.
	
	First, we prove two simple lemmas relating $\zn$-height and nonzeroness modulo $N$ and then prove the desired proposition.
	
	\begin{lemma}\label{lem: zero avoidance and zn height 2}
		Let $H$, $H'$, and $N$ be positive integers which satisfy $HH' < N$. Suppose $x, y \in \zn$ are nonzero and such that $x$ (resp. $y$) has $\zn$-height at most $H$ (resp. $H'$). Then $xy$ is nonzero modulo $N$. 
	\end{lemma}
	\begin{proof}
		Suppose to the contrary that $xy \equiv 0 \bmod N$. Then we would have $\gcd(x,N)\gcd(y,N) \equiv 0 \bmod N$. Since neither $x$ nor $y$ are zero modulo $N$, it follows that both $\gcd(x,N)$ and $\gcd(y,N)$ exceed 1, hence $\gcd(x,N)\gcd(y,N) > 1$. Finally, if we show that $\gcd(x,N)\gcd(y,N) < N$ as an inequality in integers, this would give a contradiction with $\gcd(x,N)\gcd(y,N) \equiv 0 \bmod N$.
		
		We claim that $\gcd(x,N) \leq H$. Since the height of $x$ is at most $H$ and $x$ is nonzero, we may choose an integer $x' \in \{1,\ldots, H\} \cup \{N-H,\ldots, N-1\}$ such that $x' \equiv x \bmod N$. Then at least one of $1 \leq x' \leq H$ or $N-H \leq x' \leq N-1$ holds. If the former holds, then $\gcd(x,N) = \gcd(x',N) \leq x' \leq H$. If the latter holds, then $1 \leq N - x' \leq H$, so that then $\gcd(x,N) = \gcd(x',N) = \gcd(N-x',N) \leq N-x' \leq H$. Hence $\gcd(x,N) \leq H$, and similarly $\gcd(y,N) \leq H'$. Thus $\gcd(x,N)\gcd(y,N) \leq HH' < N$ by assumption, completing the proof by contradiction.
	\end{proof}
	
	\begin{lemma}\label{lem: zero avoidance and zn height}
		Let $M$, $N$, and $H$ be positive integers which satisfy $M < H$ and $M+MH < N$. If $x, y \in \zn$ have $\zn$-height at most $M$ and $x$ is nonzero, then $x + Hy$ is nonzero. 
	\end{lemma} 
	\begin{proof}
		We observe that $-x \in \{1,\ldots, M\} \cup \{N-M,\ldots, N-1\}$ and $Hy \in \{0,H,\ldots, HM\} \cup \{N-HM,\ldots, N-H\}$. Then $x + Hy \equiv 0 \bmod N$ could only hold if $\{1,\ldots, M\} \cup \{N-M,\ldots, N-1\}$ intersects $\{0,H,\ldots, HM\} \cup \{N-HM,\ldots, N-H\}$ modulo $N$. Now, we observe the following:
		\begin{enumerate}
			\item $\{1,\ldots, M\}$ does not intersect $\{0,H,\ldots, HM\}$ modulo $N$ if $M < H$ and $HM < N$.
			\item $\{1,\ldots, M\}$ does not intersect $\{N-HM,\ldots, N-H\}$ modulo $N$ if $M < N$ and $M < N-HM$.
			\item $\{N-M,\ldots, N-1\}$ does not intersect $\{0,H,\ldots, HM\}$ modulo $N$ if $HM < N$ and $HM < N-M$.
			\item $\{N-M,\ldots, N-1\}$ does not intersect $\{N-HM,\ldots, N-H\}$ modulo $N$ if $N-HM > 0$ and $N-H < N-M$.
		\end{enumerate}
		Removing redundant conditions, we see that $x +Hy \equiv 0 \bmod N$ cannot hold if $M < H$ and $M+MH < N$. 
	\end{proof}
	
	The next proposition is algorithmic in nature, and its (admittedly tedious) proof is straightforward. We recall that the $\zn$-height of a matrix $A$ with entries in $\zn$ is the smallest nonnegative integer $M$ such that the $\zn$-height of each entry in $A$ is at most $M$. In the proof of Proposition~\ref{prop: Us control}, the proposition below is applied to a matrix with coefficients determined by an essentially distinct family of linear polynomials with bounded $\zn$-height. After applying the proposition, certain linear averages with parameters determined by the new matrix $B$ are controllable by way of Lemma~\ref{lem: Ud bound if invertible}.
	
	\begin{proposition}\label{prop: linearized family nonzeroing+}
		Let $N$, $M_0$, $n$, and $m$ be positive integers such that $M_0^{2^{nm^2}} 2^{3 \cdot 2^{nm^2}} < N$. Let $A = (a_{i,j})$ be an $n \times m$ matrix with entries in $\zn$ such that every column has a nonzero entry, no two columns are identical, and every entry has $\zn$-height at most $M_0$. Then $A$ can be transformed via elementary row operations into a matrix $B$ of $\zn$-height at most $(8M_0)^{2^{nm^2}}$ such that, within each row of $B$, all entries are distinct and nonzero.
	\end{proposition}
	\begin{proof}
		The result is trivial when $n = 1$, since we can take $B = A$. Suppose $n > 1$.
		
		First, we show that we can transform $A$ via elementary row operations into a matrix $A'$ which has all nonzero entries, while keeping track of the $\zn$-height. Moreover, we will ensure that no two columns of $A'$ are identical.
		
		If every entry of $A$ is already nonzero, then take $A' = A$. Otherwise, let $a_{i_1,j_1} \equiv 0 \bmod N$ be a zero entry. By assumption, the $j_1$th column of $A$ has a nonzero entry, say $a_{i_1',j_1}$. Let $C_1$ be a positive integer to be determined later. Add to $A$'s $i_1$th row $C_1$ times the $i_1'$th row of $A$ to obtain the matrix $A_1$. Notate $A_1 = (a_{i,j}^{(1)})$. By construction, $A_1$ agrees with $A$ on all entries not in the $i_1$th row, and in the $i_1$th row we have the relation $a_{i_1,j}^{(1)} \equiv a_{i_1,j} + C_1a_{i_1',j} \bmod N$ for all $j$. Next, provided that $2M_0 < C_1$ and $2M_0 + 2M_0C_1 < N$, we claim that every entry of $A$ in the $i_1$th row that was nonzero remains nonzero in $A_1$, that $a_{i_1,j_1}^{(1)}$ is nonzero, that the $\zn$-height of $A_1$ is at most $2C_1M_0$, that $A_1$ has the property that no two columns are identical, and that every column of $A_1$ has a nonzero entry.

		Applying Lemma~\ref{lem: zero avoidance and zn height} with $M = M_0$ and $H = C_1$, provided that $M_0 < C_1$ and $M_0+M_0C_1 < N$, it follows that every entry of $A$ in the $i_1$th row that was nonzero remains nonzero in $A_1$. 
		
		Moreover, applying Lemma~\ref{lem: zero avoidance and zn height 2} with $H = C_1$ and $H'= M_0$, we see that $a_{i_1,j_1}^{(1)} \equiv a_{i_1,j_1}+C_1a_{i_1',j_1} \equiv C_1a_{i_1',j_1} \bmod N$ is nonzero if $C_1M_0 < N$ since $a_{i_1',j_1}$ is nonzero and has $\zn$-height at most $M_0$, proving the second claim.
		
		For the third claim, each entry of $A_1$ in the $i_1$th row has $\zn$-height at most $2M_0C_1$, and each entry of $A_1$ outside the $i_1$th row has $\zn$-height at most $M_0 < 2M_0C_1$.
		
		Next, we argue that no two columns of $A_1$ are identical. If two columns of $A$ differ on two entries outside of the $i_1$th row, then the same columns in $A_1$ will still differ on these two entries, since $A_1$ agrees with $A$ outside of the $i_1$th row. On the other hand, if two columns of $A$ only differ on their entries $a_{i_1,j}$ and $a_{i_1,j'}$ in the $i_1$th row for some distinct $j, j'$, then we see that $a^{(1)}_{i_1,j} - a^{(1)}_{i_1,j'} \equiv (a_{i_1,j} - a_{i_1,j'}) + C_1(a_{i_1',j} - a_{i_1',j'}) \bmod N$. Applying Lemma~\ref{lem: zero avoidance and zn height} with $H = C_1$ and $M = 2M_0$, provided that $2M_0 < C_1$ and $2M_0 + 2M_0C_1 < N$, it follows that $a^{(1)}_{i_1,j} - a^{(1)}_{i_1,j'}$ is nonzero. Hence the $j$th and $j'$th columns of $A_1$ differ, and the fourth claim holds.
		
		To show the fifth claim, we observe that if the $j$th column of $A$ has a nonzero entry not in the $i_1$th row, then the $j$th column of $A_1$ has the same nonzero entry since $A_1$ and $A$ agree there, and if the $j$th column of $A$ has a nonzero entry in the $i_1$th row, then the corresponding entry in $A_1$ is still nonzero by a previous claim.
		
		In the previous argument, we used the inequalities $M_0 < C_1$, $M_0+M_0C_1 < N$, $C_1M_0 < N$, $2M_0 < C_1$, and $2M_0 + 2M_0C_1 < N$. Since $M_0$ and $C_1$ are positive integers, the inequalities $2M_0 < C_1$ and $2M_0 + 2M_0C_1 < N$ suffice to imply the others.
		
		Let us summarize the previous argument. Defining a positive integer $C_1$ whose exact value is still to be determined and using the facts that every column of $A$ has a nonzero entry, that no two columns of $A$ are identical, and that the $\zn$-height of $A$ is at most $M_0$, we performed one elementary row operation to obtain the matrix $A_1$, which, provided that $2M_0 < C_1$ and $2M_0 + 2M_0C_1 < N$, is such that every column has a nonzero entry, no two columns are identical, the $\zn$-height of its entries is at most $M_1 := 2M_0C_1$, and it has at least one zero entry fewer than $A$.
		
		It is easy to see that, if necessary, we could repeat this argument by choosing a zero entry of $A_1$, finding a nonzero entry in the same column, and adding a positive integer constant $C_2$ (whose exact value is to be determined later) times that row to the row containing the zero entry, and verifying that the resulting matrix $A_2$, while preserving the properties necessary to continue inductively, has at least one zero entry fewer than $A_1$ and now has $\zn$-height at most $M_2 := 2M_1C_2$, all provided that $2M_1 < C_2$ and $2M_1 + 2M_1C_2 < N$ hold.
		
		Suppose that it takes $D$ applications of the previous argument to obtain a matrix with no zero entries. Regardless of which row operations we actually performed, we must have $D \leq nm$ since $A$ is an $n \times m$ matrix. Let $C_1,\ldots, C_D$ be some positive integer constants whose exact value will be determined later. For each $i \in \{0,\ldots, D-1\}$, define $M_{i+1} := 2M_{i}C_{i+1}$. Thus, provided that for each $i \in \{0,\ldots, D-1\}$, the inequalities $2M_{i} < C_{i+1}$ and $2M_{i} + 2M_{i}C_{i+1} < N$ hold, we obtain a matrix $A' := A_D$ with $\zn$-height at most $M_D$, no zero entries, and no two columns identical.
		
		Now, we show that we can transform $A' = A_D = (a^{(D)}_{i,j})$ via elementary row operations into a matrix $B$ (while keeping track of the $\zn$-height) such that, within each row of $B$, all entries are distinct and nonzero. If $A'$ is already such that within each row, all entries are distinct, then take $B = A'$. Otherwise, let $i_{D+1} \in \{1,\ldots, n\}$ and $j_{D+1},j_{D+1}' \in \{1,\ldots m\}$ be indices such that $j_{D+1} \neq j_{D+1}'$ and $a^{(D)}_{i_{D+1},j_{D+1}} \equiv a^{(D)}_{i_{D+1},j_{D+1}'} \bmod N$. Since no two columns of $A_D$ are identical, there exists $i_{D+1}' \in \{1,\ldots, n\}$ such that $a^{(D)}_{i_{D+1}',j_{D+1}} \not\equiv a^{(D)}_{i_{D+1}',j'_{D+1}} \bmod N$. Add to the $A_D$'s $i_{D+1}$th row $C_{D+1}$ times the $i_{D+1}'$th row of $A_D$ to obtain the matrix $A_{D+1} = (a_{i,j}^{(D+1)})$. By construction, $A_{D+1}$ agrees with $A_D$ on all entries not in the $i_{D+1}$th row, and in the $i_{D+1}$th row we have the relation $a^{(D+1)}_{i_{D+1},j} \equiv a^{(D)}_{i_{D+1},j} + C_{D+1}a^{(D)}_{i_{D+1}',j} \bmod N$ for all $j$. Next, provided that $2M_D < C_{D+1}$ and $2M_D + 2M_DC_{D+1} < N$, we claim that every pair of distinct entries on the $i_{D+1}$th row of $A_D$ remains a pair of distinct entries on the same row of $A_{D+1}$, that $a^{(D+1)}_{i_{D+1},j_{D+1}} \not\equiv a^{(D+1)}_{i_{D+1},j_{D+1}'} \bmod N$, that every entry of $A_{D+1}$ is nonzero, that the $\zn$-height of $A_{D+1}$ is at most $2M_DC_{D+1}$, and that $A_{D+1}$ has the property that no two columns are identical.
		
		For the first claim, suppose that $a_{i_{D+1},j}^{(D)} \not\equiv a_{i_{D+1},j'}^{(D)} \bmod N$. To see that $a_{i_{D+1},j}^{(D+1)} \not\equiv a_{i_{D+1},j'}^{(D+1)} \bmod N$, apply Lemma~\ref{lem: zero avoidance and zn height} with $M = 2M_D$ and $H = C_{D+1}$ to $x = a^{(D)}_{i_{D+1},j}-a^{(D)}_{i_{D+1},j'}$ and $y = a^{(D)}_{i'_{D+1},j}-a^{(D)}_{i'_{D+1},j'}$, which is a valid application of the lemma if $2M_D < C_{D+1}$ and $2M_D + 2M_DC_{D+1} < N$.
		
		For the second claim, we know $a^{(D+1)}_{i_{D+1},j_{D+1}} - a^{(D+1)}_{i_{D+1},j_{D+1}'} \equiv C_{D+1}(a^{(D)}_{i_{D+1},j_{D+1}} - a^{(D)}_{i_{D+1},j_{D+1}'}) \bmod N$ and, provided that $2M_{D}C_{D+1} < N$, apply Lemma~\ref{lem: zero avoidance and zn height 2} with $H = C_{D+1}$ and $H' = 2M_D$ to conclude that $a^{(D+1)}_{i_{D+1},j_{D+1}} \not\equiv a^{(D+1)}_{i_{D+1},j_{D+1}'} \bmod N$.
		
		For the third claim, by construction, $A_{D+1}$ agrees with $A_D$ on every row that is not the $i_{D+1}$th row, so in particular $A_{D+1}$ has nonzero entries there. On the $i_{D+1}$th row, we have, for each $j$, the relation $a^{(D+1)}_{i_{D+1},j} \equiv a^{(D)}_{i_{D+1},j} + C_{D+1}a^{(D)}_{i_{D+1}',j} \bmod N$, from which it follows that, after applying Lemma~\ref{lem: zero avoidance and zn height} with $H = C_{D+1}$ and $M = M_D$ and provided that $M_D < C_{D+1}$ and $M_D + M_DC_{D+1} < N$, the entry $a^{(D+1)}_{i_{D+1},j}$ is nonzero.
		
		For the fourth claim, the $\zn$-height of entries of $A_{D+1}$ not on the $i_{D+1}$th row is at most $M_D \leq 2M_DC_{D+1}$ since those entries are the same as the corresponding entries in $A_D$, and the $\zn$-height of entries of $A_{D+1}$ on the $i_{D+1}$th row is at most $M_D + C_{D+1}M_D \leq 2M_DC_{D+1}$.
		
		For the fifth claim, if two columns of $A_{D}$ differ on a row other than the $i_{D+1}$th row, then the corresponding two columns still differ in $A_{D+1}$ because $A_{D+1}$ agrees with $A_D$ on that row, and, for $j \neq j'$, if the $j$th and $j'$th columns of $A_{D}$ differ on the $i_{D+1}$th row, then $a^{(D)}_{i_{D+1},j}$ and $a^{(D)}_{i_{D+1},j'}$ are a pair of distinct entries on the $i_{D+1}$th row of $A_D$, which by the first claim implies that $a^{(D+1)}_{i_{D+1},j}$ and $a^{(D+1)}_{i_{D+1},j'}$ are a distinct pair of entries in $A_{D+1}$. Since $A_D$ has the property that no two columns are identical, these two cases cover all possible pairs of columns in $A_{D+1}$.
		
		In the previous argument, we used the inequalities $2M_D < C_{D+1}$, $2M_D+2M_DC_{D+1} < N$, $2M_DC_{D+1} < N$, $M_D < C_{D+1}$, and $M_D + M_DC_{D+1} < N$. Since $M_D$ and $C_{D+1}$ are positive integers, the inequalities $2M_D < C_{D+1}$ and $2M_D + 2M_DC_{D+1} < N$ suffice to imply the others.
		
		Let us summarize the previous argument. Defining a positive integer $C_{D+1}$ whose exact value is still to be determined and using the facts that every entry of $A_D$ is nonzero, that no two columns of $A_D$ are identical, and that the $\zn$-height of $A_D$ is at most $M_D$, we performed one elementary row operation to obtain the matrix $A_{D+1}$, which, provided that $2M_D < C_{D+1}$ and $2M_D + 2M_DC_{D+1} < N$, is such that every entry is nonzero, no two columns are identical, the $\zn$-height of its entries is at most $M_{D+1} := 2M_DC_{D+1}$, and it has at least one pair of same-row identical entries fewer than $A_D$.
		
		It is easy to see that, if necessary, we could repeat this argument by choosing a pair of same-row identical entries of $A_{D+1}$, finding a same-row pair of distinct entries from the same two columns, and adding a positive integer constant $C_{D+2}$ (whose exact value is to be determined later) times the row containing the same-row pair of distinct entries to the row containing the chosen pair of same-row identical entries, and verifying that the resulting matrix $A_{D+2}$, while preserving the properties necessary to continue inductively, has at least one pair of same-row identical entries fewer than $A_{D+1}$ and now has $\zn$-height at most $M_{D+2} := 2M_{D+1}C_{D+2}$, all provided that $2M_{D+1} < C_{D+2}$ and $2M_{D+1} + 2M_{D+1}C_{D+2} < N$ hold.
		
		Suppose that it takes $D'$ applications of the previous argument to obtain a matrix with the property that, within a row, all entries are distinct. Regardless of which row operations we actually performed, we must have $D' \leq n\binom m 2$ since, within each of $n$ rows of $A_D$, there are $\binom m 2$ pairs of entries. Let $C_{D+1},\ldots, C_{D+D'}$ be some positive integer constants whose exact values will be determined later. For each $i \in \{0,\ldots, D'-1\}$, define $M_{D+i+1} := 2M_{D+i}C_{D+i+1}$. Thus, provided that for each $i \in \{0,\ldots, D'-1\}$, the inequalities $2M_{D+i} < C_{D+i+1}$ and $2M_{D+i} + 2M_{D+i}C_{D+i+1} < N$ hold, we obtain a matrix $B := A_{D+D'}$ with $\zn$-height at most $M_{D+D'}$ and such that within each row, all entries are distinct and nonzero.
		
		Let us now choose suitable values for $C_1,\ldots, C_{D+D'}$ so that all the requisite inequalities hold. To summarize the inequalities used in the preceding proof, we need to show that $2M_i < C_{i+1}$ and $2M_i + 2M_iC_{i+1} < N$ hold for each $i \in \{0,\ldots, D+D'-1\}$.
		
		For each $i \in \{0,1,\ldots, D+D'-1\}$, let $C_{i+1} := 2^{3\cdot 2^i - 1} M_0^{2^i}$. Since $M_{i+1} = 2M_iC_{i+1}$ holds for each $i \in \{0,\ldots, D+D'-1\}$, we may conclude that, for each $i \in \{0,\ldots, D+D'\}$, we have $M_{i} = M_0^{2^i} 2^{3\cdot 2^i - 3}$. Hence, for each $i \in \{0,\ldots, D+D'-1\}$, we have $2M_i = M_0^{2^i} 2^{3\cdot 2^i - 2} < M_0^{2^i} 2^{3\cdot 2^i - 1} = C_{i+1}$, proving one set of inequalities. Now, since $m \geq 1$, it follows that $D+D' \leq nm + n\binom m 2 = n \binom{m+1}{2} \leq nm^2$. Hence, for each $i \in \{0,\ldots, D+D'-1\}$, we have
		$2M_i + 2M_iC_{i+1} = 2M_{i} + M_{i+1} \leq 3M_{D+D'} < 4M_{D+D'} = M_0^{2^{D+D'}} 2^{3\cdot 2^{D+D'}-1} \leq M_0^{2^{nm^2}} 2^{3\cdot 2^{nm^2}} < N$, where the last inequality holds by an initial assumption on $N$, $M_0$, $n$, and $m$.
		
		Finally, to show the claimed bound on the $\zn$-height of $B$, crudely bound
		\begin{equation} M_{D+D'} \ = \ M_0^{2^{D+D'}} 2^{3 \cdot 2^{D+D'} -3} \ \leq \ M_0^{2^{nm^2}} 2^{3 \cdot 2^{nm^2}} \ = \ (8M_0)^{2^{nm^2}}.
		\end{equation} 
	\end{proof}

	\newpage
	\section{Some PET diagrams}\label{sec: PET examples}
	The purpose of this section is to give two examples that were promised earlier. Each example will be explained in its own subsection.
	
	Before we give either example, let us introduce a kind of diagram to simplify working with families of polynomials that are created when Lemmas~\ref{linearization step}~and~\ref{lem: PET inductive step} are applied repeatedly. These diagrams, which we call PET diagrams, will focus on the polynomial families and ignore the created inequalities and other parameters.
	
	As a reminder, the differencing procedure that Lemma~\ref{linearization step} and Lemma~\ref{lem: PET inductive step} perform is the following:
	Let $\{P_1(y),\ldots, P_m(y)\} \subset \zn[y_1,\ldots,y_n]$ be a family of polynomials with some restrictions that we do not repeat here, and suppose that $P_1,\ldots, P_{\ell-1}$ are of degree one for some $\ell \in \{1,\ldots, m\}$. Most of the time, we difference by $P_1$ and some parameter $h \in \mathbb{Z}_N^n$ to obtain the new family 
	\begin{multline}\label{review of differenced family}
		\{P_2-P_1, P_3-P_1,\ldots, P_{\ell - 1} - P_1, \\ P_\ell-P_1, P_\ell(y+h)-P_1(y),P_{\ell + 1} - P_1, P_{\ell+1}(y+h)-P_1(y), \ldots, P_m - P_1, P_m(y+h)-P_1(y)\}.
	\end{multline}
	In a certain case, which was properly treated in the proof of Lemma~\ref{lem: PET inductive step}, another family is created instead, but we do not review the details here because the situation does not occur in any of the examples in this section.
	
	For completeness, we should also mention that there are some restrictions on which polynomial $P_1$ we choose to difference by: if at least one $P_i$ has degree one, then $P_1$ must have degree one, and if all $P_i$ have degree greater than one, then $P_1$ must have minimal weight among them. This restriction is always followed in the examples below.
	
	Now, let us give an example. 
	
	The following PET diagram describes the relevant aspects of the $j = 2$ case of the proof of Proposition~\ref{example: prop: Us control}. It will be helpful to show the diagram first and then explain it:
	\begin{align*}
		& \{ \underline{y}, y^2\} \\
		\longrightarrow_1 \ & \{ y^2-y, (y+h_1)^2 - y\} \\ = \ & \{ \underline{y^2-y}, y^2 + (2h_1-1)y \} \\
		\longrightarrow_2 \ & \{ (y+h_2)^2 - (y+h_2) - (y^2-y), y^2 + (2h_1-1)y - (y^2- y),
		\\ & \; (y+h_2)^2 + (2h_1-1)(y+h_2) - (y^2 -y) \} \\ = \ & \{  2h_2y, 2h_1y, (2h_2+2h_1)y\}
	\end{align*}
	
	In that proof, Lemma~\ref{linearization step} is applied twice to obtain the inequalities \eqref{example: intermediate eqn in PET inductive step} and \eqref{example: intermediate eqn in PET inductive step 2}, which are then suitably combined. The inequality \eqref{example: intermediate eqn in PET inductive step} relates averages involving the families $\{P_1, P_2\} = \{y,y^2\}$
	and $\{Q_1, Q_2\} = \{P_2(y)-P_1(y),P_2(y+h_1)-P_1(y)\} = \{y^2-y,(y+h_1)^2-y\}$. This connection between the two families is encoded in the diagram as $\longrightarrow_1$, and we underline $P_1(y) = y$ in the first row of the diagram as a shorthand for the fact that the family $\{Q_1,Q_2\}$ in the second line is created by differencing by $P_1$. Note also that the differencing parameter $h_1$ is introduced in the step labeled by $\longrightarrow_1$. A final simplification that we make in these diagrams is to omit constant terms: In the expansion of the second line to the third line, the constant term $h_1^2$ is omitted from the second polynomial $y^2 + (2h_1-1)y$, because it is not relevant for us here. 
	
	Similarly, the inequality \eqref{example: intermediate eqn in PET inductive step 2} relates averages involving the families $\{Q_1, Q_2\} = \{y^2-y,(y+h_1)^2-y\}$ and $\{Q_1', Q_2', Q_3'\} = \{ Q_1(y+h_2)-Q_1(y), Q_2(y)-Q_1(y), Q_2(y+h_2) - Q_1(y)\} = \{2h_2y +h_2^2 - h_2, 2h_1y+h_1^2, 2(h_1+h_2)y + (h_1+h_2)^2-h_2\}$. This connection between the two families is encoded in the diagram as $\longrightarrow_2$, and we underline $Q_1(y) = y^2-y$ in the third row of the diagram as a shorthand for the fact that the family $\{Q'_1,Q'_2,Q'_3\}$ in the fourth line is created by differencing by $Q_1$. Note also that, in the expansion of the fourth/fifth line to the last line, again the constant terms are dropped from the diagram.
	
	In this example, we took (a part of) an existing proof of Proposition~\ref{prop: Us control} in a special case, removed all information not pertaining to polynomial families, and summarized the remaining information in a diagram. However, if we are careful, we can produce a PET diagram directly, without condensing it from an existing proof. This allows us to accomplish the goal of this section with much less writing. We will do so now, but there is one caveat we must be wary of.
	
	We already see that these diagrams focus on the finite sequence of polynomial families that arise as Lemma~\ref{lem: PET inductive step} is applied repeatedly to prove Proposition~\ref{prop: Us control} for a given family. To be precise, our diagrams suppress information about the ring $R$, its characteristic $N$, and the 1-bounded functions $f_i$ about which a bound is asserted by the conclusion of the proposition. However, as noted in the remark immediately following the formulation of Lemma~\ref{linearization step}, the key point is that such information can influence the values of the parameters $h_i$, which in turn can change which polynomial is underlined (i.e., is differenced by) at a given step.
	
	Consequently, at key junctures in the construction of a PET diagram, we can reintroduce some of this suppressed information. In this way, we can deduce the existence of the examples we seek. It will be important to reintroduce the suppressed information in a way that remains consistent with the hypotheses of Lemma~\ref{lem: PET inductive step}; only by satisfying this requirement can we be convinced that our PET diagram is actually evidence for an objective phenomenon that we could notice in a more tedious expanded proof of the special case of Proposition~\ref{prop: Us control} for a certain family or families.
	\newpage 
	\subsection{Why Proposition~\ref{prop: linearized family nonzeroing+} is necessary}\label{subsec: prop: linearized family nonzeroing+ might be necessary}
	
	In Section~\ref{sec: fixing a matrix}, it was claimed that Proposition~\ref{prop: linearized family nonzeroing+} is (in general) necessary in the proof of Proposition~\ref{prop: Us control} when $n > 1$.
	
	The following PET diagram describes the polynomial families that are created during the proof of Proposition~\ref{prop: Us control} in the case that $\mathbf P = \{y_1y_2,y_1\} \subset \Z[y_1,y_2]$.
	\begin{align*}
		& \{ y_1y_2, \underline{y_1}\} \\
		\longrightarrow_1 \ & \{ y_1y_2 - y_1, (y_1+h_1)(y_2+h_1')-y_1\} \\
		= \ & \{ \underline{y_1y_2 - y_1}, y_1y_2 + (h_1'-1)y_1 + h_1y_2 \} \\
		\longrightarrow_2 \ & \{ (y_1+h_2)(y_2+h_2')-(y_1+h_2)-(y_1y_2-y_1), y_1y_2 + (h_1'-1)y_1 + h_1y_2 - (y_1y_2-y_1), \\ & \ (y_1+h_2)(y_2+h_2') + (h_1'-1)(y_1+h_2) + h_1(y_2+h_2') - (y_1y_2 - y_1)\} \\ = \ & \{ h_2'y_1 + h_2y_2, h_1'y_1 + h_1y_2, (h_2'+h_1')y_1 + (h_2+h_1)y_2 \}
	\end{align*}
	As noted in the remark immediately following the formulation of Lemma~\ref{lem: PET inductive step}, the exact values of the parameters $(h_1,h_1') \in \mathbb{Z}_N^2$ and $(h_2,h_2') \in \mathbb{Z}_N^2$ depend on the arbitrary functions $f_i : R \to \C$, which have been suppressed in our diagram. Since the $f_i$ are arbitrary, it may happen that $h_2 \equiv -h_1 \bmod N$, where $N$ is the characteristic of $R$, in which case the final line of the diagram would be
	\begin{equation}\label{random poly fam 1}
		\{ h_2'y_1 + h_2y_2, h_1'y_1 + h_1y_2, (h_2'+h_1')y_1 \}.
	\end{equation}
	Indeed, the essential distinctness of the family $\{ y_1y_2 - y_1, y_1y_2 + (h_1'-1)y_1 + h_1y_2 \}$, the coincidence $h_2 \equiv -h_1 \bmod N$, and the essential distinctness of the family \eqref{random poly fam 1} are not in contradiction, as can be verified by writing out which polynomials are required to be nonconstant.
	
	But clearly \eqref{random poly fam 1} is a family of degree 1 polynomials in $y_1$ and $y_2$, not all of whose linear coefficients are nonzero.
	\newpage
	\subsection{Why Proposition~\ref{alg stops} is necessary} \label{subsec: claimed possible nonuniformity}
	In Subsection~\ref{subsec: PET algorithms}, we claimed that \eqref{claimed possible nonuniformity} holds. We prove this claim when $\mathbf P = \{y^3, y^3 + y^2\}$. Having seen the proof in this special case, the general truth of the claim will be obvious.
	
	First note that, by one of the assumptions of Lemma~\ref{lem: PET inductive step}, we must be working in the situation that $\lpf N > 3$, where $N$ is the characteristic of the ring $R$. Thus, when we see the number 3 below, it is not a zero divisor.
	
	We will produce two PET diagrams, which diverge at an early point, depending on whether a certain coincidence holds.
	
	In this first PET diagram, we assume that the coincidence $3h_1 \equiv 1 \bmod N$ occurs; the starred equality reflects the resulting simplification:
	\begin{align*}
		& \{ \underline{y^3}, y^3+y^2\}
		\\ \longrightarrow_1 \ & \{ (y+h_1)^3 - y^3, y^3 + y^2-y^3, (y+h_1)^3 + (y+h_1)^2 - y^3\}
		\\ = \ & \{ 3h_1y^2 + 3h_1^2y, y^2, (3h_1+1)y^2 +(3h_1^2+2h_1)y \}
		\\ \overset{*}{=} \ & \{ y^2 + h_1y, \underline{y^2}, 2y^2 + y \}
		\\ \longrightarrow_2 \ & \{ y^2 + h_1y - y^2, 2y^2 + y - y^2,
		\\ & \ (y+h_2)^2 + h_1(y+h_2) -y^2, (y+h_2)^2 - y^2, 2(y+h_2)^2 + (y+h_2)-y^2 \}
		\\ = \ & \{  h_1y, y^2 +y, (2h_2+h_1)y, 2h_2y, y^2 + (4h_2 + 1)y\}
		\\ \overset{\rm{reorder}}{=} \ & \{\underline{h_1y},2h_2y, (2h_2+h_1)y, y^2+y,y^2+(4h_2+1)y \}
		\\ \longrightarrow_3 \ & \{ 2h_2y-h_1y, (2h_2+h_1)y-h_1y, y^2+y-h_1y,y^2+(4h_2+1)y-h_1y,
		\\ & \ (y+h_3)^2+(y+h_3)-h_1y,(y+h_3)^2+(4h_2+1)(y+h_3)-h_1y
		\}
		\\ = \ & \{(2h_2-h_1)y,\underline{2h_2y}, y^2 + (1-h_1)y, y^2 + (4h_2-h_1+1)y,
		\\ \ & y^2 + (2h_3 + 1 - h_1)y, y^2 + (2h_3+4h_2+1-h_1)y\}
		\\ \longrightarrow_4 \ & \{ (2h_2-h_1)y-2h_2y, y^2 + (1-h_1)y-2h_2y, y^2 + (4h_2-h_1+1)y-2h_2y,
		\\ & \ y^2 + (2h_3 + 1 - h_1)y -2h_2y, y^2 + (2h_3+4h_2+1-h_1)y -2h_2y,
		\\ & \
		(y+h_4)^2 + (1-h_1)(y+h_4)-2h_2y, (y+h_4)^2 + (4h_2-h_1+1)(y+h_4)-2h_2y,
		\\ & \ (y+h_4)^2 + (2h_3 + 1 - h_1)(y+h_4)-2h_2y, (y+h_4)^2 + (2h_3+4h_2+1-h_1)(y+h_4)-2h_2y\}
		\\ = \ &
		\{ \underline{-h_1y}, y^2 + (1-h_1-2h_2)y, y^2 + (2h_2-h_1+1)y,
		\\ & \ y^2 + (2h_3 + 1 - h_1-2h_2)y, y^2 + (2h_3+2h_2+1-h_1)y,
		\\ & \
		y^2 + (2h_4+1-h_1-2h_2)y, y^2 + (2h_4+2h_2-h_1+1)y,
		\\ & \ y^2 + (2h_4+2h_3 + 1 - h_1-2h_2)y, y^2 + (2h_4+2h_3+2h_2+1-h_1)y\}
	\end{align*}
	\begin{align*}
		\longrightarrow_5 \ & \{y^2 + (1-h_1-2h_2)y+h_1y, y^2 + (2h_2-h_1+1)y+h_1y,
		\\ & \ y^2 + (2h_3 + 1 - h_1-2h_2)y+h_1y, y^2 + (2h_3+2h_2+1-h_1)y+h_1y,
		\\ & \
		y^2 + (2h_4+1-h_1-2h_2)y+h_1y, y^2 + (2h_4+2h_2-h_1+1)y+h_1y,
		\\ & \ y^2 + (2h_4+2h_3 + 1 - h_1-2h_2)y+h_1y, y^2 + (2h_4+2h_3+2h_2+1-h_1)y+h_1y,
		\\ & \ (y+h_5)^2 + (1-h_1-2h_2)(y+h_5)+h_1y, (y+h_5)^2 + (2h_2-h_1+1)(y+h_5)+h_1y,
		\\ & \ (y+h_5)^2 + (2h_3 + 1 - h_1-2h_2)(y+h_5)+h_1y, (y+h_5)^2 + (2h_3+2h_2+1-h_1)(y+h_5)+h_1y,
		\\ & \
		(y+h_5)^2 + (2h_4+1-h_1-2h_2)(y+h_5)+h_1y, (y+h_5)^2 + (2h_4+2h_2-h_1+1)(y+h_5)+h_1y,
		\\ && \llap{$(y+h_5)^2 + (2h_4+2h_3 + 1 - h_1-2h_2)(y+h_5)+h_1y, (y+h_5)^2 + (2h_4+2h_3+2h_2+1-h_1)(y+h_5)+h_1y\}$}
		\\ = \ & 
		\{\underline{y^2 + (1-2h_2)y}, y^2 + (2h_2+1)y,
		\\ & \ y^2 + (2h_3 + 1 -2h_2)y, y^2 + (2h_3+2h_2+1)y,
		\\ & \
		y^2 + (2h_4+1-2h_2)y, y^2 + (2h_4+2h_2+1)y,
		\\ & \ y^2 + (2h_4+2h_3 + 1 -2h_2)y, y^2 + (2h_4+2h_3+2h_2+1)y,
		\\ & \ y^2 + (2h_5+1-2h_2)y, y^2 + (2h_5+2h_2+1)y,
		\\ & \ y^2 + (2h_5+2h_3 + 1-2h_2)y, y^2 + (2h_5+2h_3+2h_2+1)y,
		\\ & \
		y^2 + (2h_5+2h_4+1-2h_2)y, y^2 + (2h_5+2h_4+2h_2+1)y,
		\\ & \ y^2 + (2h_5+2h_4+2h_3 + 1-2h_2)y, y^2 + (2h_5+2h_4+2h_3+2h_2+1)y \}
		\\ \longrightarrow_6 \ & 
		\{4h_2y,
		\\ & \ 2h_3y, (2h_3+4h_2)y,
		\\ & \
		2h_4y, (2h_4+4h_2)y,
		\\ & \ (2h_4+2h_3)y, (2h_4+2h_3+4h_2)y,
		\\ & \ 2h_5y, (2h_5+4h_2)y,
		\\ & \ (2h_5+2h_3)y, (2h_5+2h_3+4h_2)y,
		\\ & \
		(2h_5+2h_4)y, (2h_5+2h_4+4h_2)y,
		\\ & \ (2h_5+2h_4+2h_3)y, (2h_5+2h_4+2h_3+4h_2)y,
		\\ & \ 2h_6y, (2h_6+4h_2)y,
		\\ & \ (2h_6+2h_3)y, (2h_6+2h_3+4h_2)y,
		\\ & \
		(2h_6+2h_4)y, (2h_6+2h_4+4h_2)y,
		\\ & \ (2h_6+2h_4+2h_3)y, (2h_6+2h_4+2h_3+4h_2)y,
		\\ & \ (2h_6+2h_5)y, (2h_6+2h_5+4h_2)y,
		\\ & \ (2h_6+2h_5+2h_3)y, (2h_6+2h_5+2h_3+4h_2)y,
		\\ & \
		(2h_6+2h_5+2h_4)y, (2h_6+2h_5+2h_4+4h_2)y,
		\\ & \ (2h_6+2h_5+2h_4+2h_3)y, (2h_6+2h_5+2h_4+2h_3+4h_2)y \}
	\end{align*}

	Now, in this second PET diagram, we assume that $3h_1 \not\equiv 1 \bmod N$ and $3h_1 \not\equiv -1 \bmod N$; compare the right-hand side of the starred equality here to the right-hand side of the starred equality in the first diagram.
	\begin{align*}
		& \{ \underline{y^3}, y^3+y^2\}
		\\ \longrightarrow_1 \ & \{ (y+h_1)^3 - y^3, y^3 + y^2-y^3, (y+h_1)^3 + (y+h_1)^2 - y^3\}
		\\ = \ & \{ 3h_1y^2 + 3h_1^2y, y^2, (3h_1+1)y^2 +(3h_1^2+2h_1)y \}
		\\ \overset{*}{=} \ & \{ 3h_1y^2 + 3h_1^2y, \underline{y^2}, (3h_1+1)y^2 +(3h_1^2+2h_1)y \}
		\\ \longrightarrow_2 \ & \{ 3h_1y^2 + 3h_1^2y-y^2, (3h_1+1)y^2 +(3h_1^2+2h_1)y-y^2,
		\\ & \ 3h_1(y+h_2)^2 + 3h_1^2(y+h_2)-y^2, (y+h_2)^2-y^2, (3h_1+1)(y+h_2)^2 +(3h_1^2+2h_1)(y+h_2) - y^2 \}
		\\ = \ & \{ (3h_1-1)y^2 + 3h_1^2y, 3h_1y^2 +(3h_1^2+2h_1)y,
		\\ & \ (3h_1-1)y^2+ (6h_1h_2+3h_1^2)y, 2h_2y, 3h_1y^2 + ( 6h_1h_2 + 2h_2+ 3h_1^2+2h_1)y \}
		\\ \overset{\rm{reorder}}{=} \ & \{ \underline{2h_2y},(3h_1-1)y^2 + 3h_1^2y, 3h_1y^2 +(3h_1^2+2h_1)y,
		\\ & \ (3h_1-1)y^2+ (6h_1h_2+3h_1^2)y, 3h_1y^2 + (6h_1h_2 + 2h_2 + 3h_1^2+2h_1)y \}
		\\ \longrightarrow_3 \ & \{ (3h_1-1)y^2 + 3h_1^2y-2h_2y, 3h_1y^2 +(3h_1^2+2h_1)y-2h_2y,
		\\ & \ (3h_1-1)y^2+ (6h_1h_2+3h_1^2)y-2h_2y, 3h_1y^2 + (6h_1h_2 + 2h_2 + 3h_1^2+2h_1)y-2h_2y,
		\\ & \ (3h_1-1)(y+h_3)^2 + 3h_1^2(y+h_3)-2h_2y, 3h_1(y+h_3)^2 +(3h_1^2+2h_1)(y+h_3)-2h_2y,
		\\ && \llap{$(3h_1-1)(y+h_3)^2+ (6h_1h_2+3h_1^2)(y+h_3)-2h_2y, 3h_1(y+h_3)^2 + (6h_1h_2 + 2h_2 + 3h_1^2+2h_1)(y+h_3)-2h_2y\}$} 
		\\ \overset{**}{=} \ & \{ \underline{(3h_1-1)y^2 + (3h_1^2-2h_2)y}, 3h_1y^2 +(3h_1^2+2h_1-2h_2)y,
		\\ & \ (3h_1-1)y^2+ (6h_1h_2+3h_1^2-2h_2)y, 3h_1y^2 + (6h_1h_2 + 3h_1^2+2h_1)y,
		\\ & \ (3h_1-1)y^2 + (6h_1h_3 -2h_3+3h_1^2-2h_2)y, 3h_1y^2 +(6h_1h_3+3h_1^2+2h_1-2h_2)y,
		\\ & \ (3h_1-1)y^2+ (6h_1h_3-2h_3+6h_1h_2+3h_1^2-2h_2)y, 3h_1y^2 + (6h_1h_3+6h_1h_2 + 3h_1^2+2h_1)y\}
	\end{align*}
	
	Before continuing the calculation, we offer some comments about the validity of some of the steps above.
	
	In the $\longrightarrow_2$ step, we may choose to difference by $y^2$ for the following reasons: First, since Lemma~\ref{lem: PET inductive step} ensures the essential distinctness of the family $\{3h_1y^2 + 3h_1^2y, y^2, (3h_1+1)y^2 +(3h_1^2+2h_1)y\}$ that results from the $\longrightarrow_1$ step, it follows that $h_1$ cannot be zero modulo $N$, since that would force the first polynomial $3h_1y^2+3h_1^2y$ to be the zero polynomial. As a result, the first polynomial $3h_1y^2+3h_1^2y$ has degree two because the fact that $h_1 \not\equiv 0 \bmod N$ and the fact that 3 is not a zero divisor together imply that $3h_1 \not\equiv 0 \bmod N$. Moreover, the third polynomial $(3h_1+1)y^2 + (3h_1^2+2h_1)y$ is also of degree two since we assume $3h_1 \not\equiv -1 \bmod N$. Hence, the third polynomial does not degenerate to a linear polynomial, and we may thus choose to difference by $y^2$ in the $\longrightarrow_2$ step. 
	
	Similarly, since neither $3h_1$ nor $3h_1-1$ are zero modulo $N$, no polynomial on the right-hand side of the double-starred equality has degree less than two, whence we may difference using the underlined polynomial $P_1(y)=(3h_1-1)y^2 + (3h_1^2-2h_2)y$ to obtain
	\begin{align*}
		\longrightarrow_4 \ & \{3h_1y^2 +(3h_1^2+2h_1-2h_2)y-P_1(y),
		\\ & \ (3h_1-1)y^2+ (6h_1h_2+3h_1^2-2h_2)y-P_1(y), 3h_1y^2 + (6h_1h_2 + 3h_1^2+2h_1)y-P_1(y),
		\\ & \ (3h_1-1)y^2 + (6h_1h_3 -2h_3+3h_1^2-2h_2)y-P_1(y), 3h_1y^2 +(6h_1h_3+3h_1^2+2h_1-2h_2)y-P_1(y),
		\\ & \ (3h_1-1)y^2+ (6h_1h_3-2h_3+6h_1h_2+3h_1^2-2h_2)y-P_1(y),
		\\ & \ 3h_1y^2 + (6h_1h_3+6h_1h_2 + 3h_1^2+2h_1)y-P_1(y),
		\\ & \ (3h_1-1)(y+h_4)^2 + (3h_1^2-2h_2)(y+h_4)-P_1(y),
		\\ & \ 3h_1(y+h_4)^2 +(3h_1^2+2h_1-2h_2)(y+h_4)-P_1(y),
		\\ & \ (3h_1-1)(y+h_4)^2+ (6h_1h_2+3h_1^2-2h_2)(y+h_4)-P_1(y),
		\\ & \ 3h_1(y+h_4)^2 + (6h_1h_2 + 3h_1^2+2h_1)(y+h_4)-P_1(y),
		\\ & \ (3h_1-1)(y+h_4)^2 + (6h_1h_3 -2h_3+3h_1^2-2h_2)(y+h_4)-P_1(y),
		\\ & \ 3h_1(y+h_4)^2 +(6h_1h_3+3h_1^2+2h_1-2h_2)(y+h_4)-P_1(y),
		\\ & \ (3h_1-1)(y+h_4)^2+ (6h_1h_3-2h_3+6h_1h_2+3h_1^2-2h_2)(y+h_4)-P_1(y),
		\\ & \ 3h_1(y+h_4)^2 + (6h_1h_3+6h_1h_2 + 3h_1^2+2h_1)(y+h_4)-P_1(y) \} 
		\\ = \ & \{y^2 +2h_1y,
		\\ & \ 6h_1h_2y, y^2 + (6h_1h_2 +2h_1+2h_2)y,
		\\ & \ (6h_1h_3 -2h_3)y, y^2 +(6h_1h_3+2h_1)y,
		\\ & \ (6h_1h_3-2h_3+6h_1h_2)y,
		\\ & \ y^2 + (6h_1h_3+6h_1h_2 + 2h_1+2h_2)y,
		\\ & \ (6h_1h_4-2h_4)y,
		\\ & \ y^2 +(6h_1h_4+2h_1)y,
		\\ & \ (6h_1h_4-2h_4+6h_1h_2)y,
		\\ & \ y^2 + (6h_1h_4+6h_1h_2 +2h_1+2h_2)y,
		\\ & \ (6h_1h_4-2h_4+6h_1h_3 -2h_3)y,
		\\ & \ y^2 +(6h_1h_4+6h_1h_3+2h_1)y,
		\\ & \ (6h_1h_4-2h_4+6h_1h_3-2h_3+6h_1h_2)y,
		\\ & \ y^2 + (6h_1h_4+6h_1h_3+6h_1h_2 +2h_1+2h_2)y \} 
		\\ \overset{\rm{reorder}}{=} \ & \{ 6h_1h_2y, (6h_1h_3 -2h_3)y, (6h_1h_3-2h_3+6h_1h_2)y, (6h_1h_4-2h_4)y,
		\\ & \ (6h_1h_4-2h_4+6h_1h_2)y, (6h_1h_4-2h_4+6h_1h_3 -2h_3)y, (6h_1h_4-2h_4+6h_1h_3-2h_3+6h_1h_2)y,
		\\ & \ y^2 +2h_1y, y^2 + (6h_1h_2 +2h_1+2h_2)y, y^2 +(6h_1h_3+2h_1)y,  y^2 + (6h_1h_3+6h_1h_2 + 2h_1+2h_2)y,
		\\ & \ y^2 +(6h_1h_4+2h_1)y, y^2 + (6h_1h_4+6h_1h_2 +2h_1+2h_2)y,
		\\ & \ y^2 +(6h_1h_4+6h_1h_3+2h_1)y, y^2 + (6h_1h_4+6h_1h_3+6h_1h_2 +2h_1+2h_2)y \}
	\end{align*}
	We will not complete the second PET diagram, but we can still see how it will end. As we can see from the last line, the polynomial family that results from the $\longrightarrow_4$ step consists of seven polynomials of degree 1 and a nonzero number of polynomials of degree two. If we were to continue the diagram, we would need to difference at least seven more times, since, by essential distinctness, none of the linear polynomials have the same linear coefficient. In fact, it would take exactly eight more times before we obtain a polynomial family consisting only of linear polynomials.
	
	Thus, in the first PET diagram (where $3h_1 \equiv 1 \bmod N$), we see that it takes six applications of Lemma~\ref{lem: PET inductive step} to reach the end state, but from the second PET diagram (where $3h_1 \not\equiv \pm 1 \bmod N$), it is clear that it takes twelve applications of the same lemma. But either of these situations can happen, since the $f_i$, which are arbitrary, may be rigged to force either one. This proves the claim \eqref{claimed possible nonuniformity}.


\begin{thebibliography}{99}
		
		
		
		\bibitem[ABerg1]{ab1} E. Ackelsberg and V. Bergelson, Asymptotic total ergodicity for actions of $F[t]$ and Furstenberg--S\'{a}rk\"{o}zy-type theorems over finite fields and rings. https://arxiv.org/abs/2303.00100
		
		\bibitem[ABerg2]{ab2} E. Ackelsberg and V. Bergelson, Polynomial patterns in subsets of large finite fields of low characteristic. https://arxiv.org/abs/2303.00925
		
		\bibitem[Beh]{behr} F. A. Behrend, On sets of integers which contain no three terms in arithmetical progression, Proc. Nat. Acad. Sci. U.S.A. {\bf 32} (1946), 331--332.
		
		
		
		
		
		
		
		
		\bibitem[BereBerg]{bereberg} D. Berend\ and\ V. Bergelson, Jointly ergodic measure-preserving transformations, Israel J. Math. {\bf 49} (1984), no.~4, 307--314.
		
		\bibitem[Berg1]{ertau} V. Bergelson, Ergodic Ramsey theory---an update, in {\it Ergodic theory of ${\bf Z}^d$ actions (Warwick, 1993--1994)}, 1--61, London Math. Soc. Lecture Note Ser., 228, Cambridge Univ. Press, Cambridge.
		
		\bibitem[Berg2]{wmpet} V. Bergelson, Weakly mixing PET, Ergodic Theory Dynam. Systems {\bf 7} (1987), no.~3, 337--349.
		
		
		
		\bibitem[BergBest]{bb} V. Bergelson\ and\ A. Best, The Furstenberg-S\'{a}rk\"{o}zy theorem and asymptotic total ergodicity phenomena in modular rings, J. Number Theory {\bf 243} (2023), 615--645.
		
		\bibitem[BergFMcC]{bergfm} V. Bergelson, H. Furstenberg\ and\ R.~G. McCutcheon, IP-sets and polynomial recurrence, Ergodic Theory Dynam. Systems {\bf 16} (1996), no.~5, 963--974.
		
		\bibitem[BergH]{bh} V. Bergelson\ and\ I.~J. H\aa land, Sets of recurrence and generalized polynomials, in {\it Convergence in ergodic theory and probability (Columbus, OH, 1993)}, 91--110, Ohio State Univ. Math. Res. Inst. Publ., 5, De Gruyter, Berlin.
		
		\bibitem[BergHMcC]{bhm} V. Bergelson, I.~J. H\aa land\ and\ R.~G. McCutcheon, IP-systems, generalized polynomials and recurrence, Ergodic Theory Dynam. Systems {\bf 26} (2006), no.~4, 999--1019.
		
		\bibitem[BergHSon]{bhs} V. Bergelson, I.~J. H\aa land\ and\ Y. Son, An extension of Weyl's equidistribution theorem to generalized polynomials and applications, Int. Math. Res. Not. IMRN {\bf 2021}, no.~19, 14965--15018.
		
		\bibitem[BergLei]{bl} V. Bergelson\ and\ A. Leibman, Polynomial extensions of van der Waerden's and Szemer\'{e}di's theorems, J. Amer. Math. Soc. {\bf 9} (1996), no.~3, 725--753. 
		
		\bibitem[BergLeiLes]{bll} V. Bergelson, A. Leibman\ and\ E. Lesigne, Intersective polynomials and the polynomial Szemer\'{e}di theorem, Adv. Math. {\bf 219} (2008), no.~1, 369--388.
		
		\bibitem[BergLeiMcC]{blm} V. Bergelson, A. Leibman\ and\ R. McCutcheon, Polynomial Szemer\'{e}di theorems for countable modules over integral domains and finite fields, J. Anal. Math. {\bf 95} (2005), 243--296.
		
		\bibitem[BergMcC1]{IPpolySz} V. Bergelson\ and\ R.~G. McCutcheon, An ergodic IP polynomial Szemer\'{e}di theorem, Mem. Amer. Math. Soc. {\bf 146} (2000), no.~695, viii+106 pp.
		
		\bibitem[BergMcC2]{bm unif} V. Bergelson\ and\ R.~G. McCutcheon, Uniformity in the polynomial Szemer\'{e}di theorem, in {\it Ergodic theory of ${\bf Z}^d$ actions (Warwick, 1993--1994)}, 273--296, London Math. Soc. Lecture Note Ser., 228, Cambridge Univ. Press, Cambridge.
		
		\bibitem[BestFM]{bfm} A. Best\ and\ A. Ferr\'{e}~Moragues, Polynomial ergodic averages for certain countable ring actions, Discrete Contin. Dyn. Syst. {\bf 42} (2022), no.~7, 3379--3413.
		
		\bibitem[BlSis]{bloomsisask} T. F. Bloom and O. Sisask, Breaking the logarithmic barrier in Roth's theorem on arithmetic progressions. 2020 arXiv preprint.
		
		\bibitem[BC]{bc}J. Bourgain\ and\ M.-C. Chang, Nonlinear Roth type theorems in finite fields, Israel J. Math. {\bf 221} (2017), no.~2, 853--867.
		
		
		\bibitem[ChFraH]{chufrahost} Q. Chu, N. Frantzikinakis, and B. Host, Ergodic averages of commuting transformations with distinct degree polynomial iterates, Proc. Lond. Math. Soc. \textbf{102} (2011), 801--842.
		
		\bibitem[ClLi]{clarkliang} W.~E. Clark\ and\ J.~J. Liang, Enumeration of finite commutative chain rings, J. Algebra {\bf 27} (1973), 445--453.
		
		\bibitem[CoRo]{cohenroche} J. Cohen\ and\ A. Roche, Which finite rings have the most or fewest units?, Amer. Math. Monthly {\bf 129} (2022), no.~8, 717--733.
		
		
		
		
		
		
		
		
		
		
		
		
		
		
		
		
		
		
		
		
		\bibitem[DLS]{dls} D. Dong, X.~C. Li\ and\ W.~F. Sawin, Improved estimates for polynomial Roth type theorems in finite fields, J. Anal. Math. {\bf 141} (2020), no.~2, 689--705.
		
		\bibitem[DFMKS]{dfmks} S. Donoso, A. Ferr\'{e} Moragues, A. Koutsogiannis, and W. Sun, Decomposition of multi-correlation sequences and joint ergodicity, Ergod. Theory Dynam. Syst. \textbf{44} (2024), no. 2, 432--480.
		
		\bibitem[ElGi]{eg} J.~S. Ellenberg\ and\ D.~C. Gijswijt, On large subsets of $\mathbb{F}^n_q$ with no three-term arithmetic progression, Ann. of Math. (2) {\bf 185} (2017), no.~1, 339--343.
		
		\bibitem[ErTu]{et} P. Erd\H{o}s\ and\ P. Tur\'{a}n, On Some Sequences of Integers, J. London Math. Soc. {\bf 11} (1936), no.~4, 261--264.
		
		\bibitem[FraKra]{frakra} N. Frantzikinakis\ and\ B. Kra, Polynomial averages converge to the product of integrals, Israel J. Math. {\bf 148} (2005), 267--276. 
		
		\bibitem[F1]{diag77} H. Furstenberg, Ergodic behavior of diagonal measures and a theorem of Szemer\'{e}di on arithmetic progressions, J. Analyse Math. {\bf 31} (1977), 204--256.
		
		\bibitem[F2]{furstenbergbook} H. Furstenberg, {\it Recurrence in ergodic theory and combinatorial number theory}, Princeton University Press, Princeton, NJ, 1981.
		
		\bibitem[Go1]{g-sz2} W.~T. Gowers, A new proof of Szemer\'{e}di's theorem, Geom. Funct. Anal. {\bf 11} (2001), no.~3, 465--588.
		
		\bibitem[Go2]{g-sz1} W.~T. Gowers, A new proof of Szemer\'{e}di's theorem for arithmetic progressions of length four, Geom. Funct. Anal. {\bf 8} (1998), no.~3, 529--551.
		
		
		
		
		
		
		
		\bibitem[J]{janusz} G.~J. Janusz, Separable algebras over commutative rings, Trans. Amer. Math. Soc. {\bf 122} (1966), 461--479.
		
		\bibitem[KaMF]{kmf} T. Kamae\ and\ M. Mend\`es France, Van der Corput's difference theorem, Israel J. Math. {\bf 31} (1978), no.~3-4, 335--342. 
		
		\bibitem[Ke]{kempner} A. J. Kempner, Polynomials and their residue systems, Trans. Amer. Math. Soc. {\bf 22} (1921), no.~2, 240--266.
		
		\bibitem[Kl]{klingenberg} W.~P.~A. Klingenberg, Projektive und affine Ebenen mit Nachbarelementen, Math. Z. {\bf 60} (1954), 384--406.
		
		\bibitem[Ko]{kowalskinotes} E. Kowalski. \textit{Exponential sums over finite fields, I: elementary methods}. www.math.ethz.ch/\%7ekowalski/exponential-sums-elementary.pdf
		
		\bibitem[KKL]{kkl} N. Kravitz, B. Kuca, and J. Leng, Quantitative concatenation for polynomial box norms. https://arxiv.org/abs/2407.08636
		
		\bibitem[Kr]{krull} W. Krull, {\it Idealtheorie}, zweite, erg\"{a}nzte Auflage, Ergebnisse der Mathematik und ihrer Grenzgebiete, Band 46, Springer, Berlin, 1968.
		
		\bibitem[Kuca]{kucafurther} B. Kuca, Further bounds in the polynomial Szemer\'{e}di theorem over finite fields, Acta Arith. {\bf 198} (2021), no.~1, 77--108.
		
		\bibitem[KuHeCa]{kuheca} P.~V. Kumar, T. Helleseth\ and\ A.~R. Calderbank, An upper bound for Weil exponential sums over Galois rings and applications, IEEE Trans. Inform. Theory {\bf 41} (1995), no.~2, 456--468.
		
		\bibitem[Lei]{lei} A. Leibman, Convergence of multiple ergodic averages along polynomials of several variables, Israel J. Math. {\bf 146} (2005), 303--315.
		
		\bibitem[Leng]{leng} J. Leng, A Quantitative Bound for Szemer\'{e}di's Theorem for a Complexity One Polynomial Progression over $\mathbb{Z}/N\mathbb{Z}$, Discrete Anal. \textbf{2024}, 33 pp.
		
		\bibitem[LeungMa]{leungma} K.~H. Leung\ and\ S.~L. Ma, Constructions of partial difference sets and relative difference sets on $p$-groups, Bull. London Math. Soc. {\bf 22} (1990), no.~6, 533--539.
		
		\bibitem[LS]{lisauermann} A. Li and L. Sauermann, S\'{a}rk\"{o}zy's Theorem in Various Finite Field Settings, SIAM J. Discrete Math. \textbf{38} (2024), no.~2, 1409--1416.
		
		\bibitem[McD]{brmcdonald} B.~R. McDonald, {\it Finite rings with identity}, Pure and Applied Mathematics, Vol. 28, Marcel Dekker, Inc., New York, 1974. 
		
		
		
		
		\bibitem[Pe1]{pel20} S. Peluse, Bounds for sets with no polynomial progressions, Forum Math. Pi \textbf{8} (2020).
		
		\bibitem[Pe2]{peluse} S. Peluse, On the polynomial Szemer\'{e}di theorem in finite fields, Duke Math. J. {\bf 168} (2019), no.~5, 749--774.
		
		\bibitem[Pe3]{pel3term} S. Peluse, Three-term polynomial progressions in subsets of finite fields, Israel J. Math. {\bf 228} (2018), no.~1, 379--405.
		
		\bibitem[PePr1]{pelprend1} S. Peluse and S. Prendiville, A polylogarithmic bound in the nonlinear Roth theorem, Int. Math. Res. Not. \textbf{2022} (2020), no. 8, 5658--5684.
		
		\bibitem[PePr2]{pelprend2} S. Peluse and S. Prendiville, Quantitative bounds in the non-linear Roth theorem, Invent. Math. \textbf{238} (2024), no. 3, 865--903.
		
		\bibitem[Pr]{prend} S. Prendiville, Quantitative bounds in the polynomial Szemer\'{e}di theorem: the homogeneous case, Discrete Anal. {\bf 2017}, Paper No. 5, 34 pp. 
		
		\bibitem[Ra]{ragh} R. Raghavendran, Finite associative rings, Compositio Math. {\bf 21} (1969), 195--229.
		
		\bibitem[Ri]{rice} A. Rice, A maximal extension of the best-known bounds for the Furstenberg-S\'{a}rk\"{o}zy theorem, Acta Arith. {\bf 187} (2019), no.~1, 1--41.
		
		\bibitem[Ro]{roth} K.~F. Roth, On certain sets of integers, J. London Math. Soc. {\bf 28} (1953), 104--109.
		
		\bibitem[SalSp]{sasp} R. Salem\ and\ D.~C. Spencer, On sets of integers which contain no three terms in arithmetical progression, Proc. Nat. Acad. Sci. U.S.A. {\bf 28} (1942), 561--563. 
		
		\bibitem[Sar1]{sar1} A. S\'{a}rk\"{o}zy, On difference sets of sequences of integers. I, Acta Math. Acad. Sci. Hungar. {\bf 31} (1978), no.~1-2, 125--149.
		
		\bibitem[Sar2]{sar3} A. S\'{a}rk\"{o}zy, On difference sets of sequences of integers. III, Acta Math. Acad. Sci. Hungar. {\bf 31} (1978), no.~3-4, 355--386.
		
		
		
		
		
		\bibitem[Si]{singmaster} D. Singmaster, On polynomial functions $({\rm mod}$ $m)$, J. Number Theory {\bf 6} (1974), 345--352.
		
		\bibitem[Sze]{szechtman} F. Szechtman, Quadratic Gauss sums over finite commutative rings, J. Number Theory {\bf 95} (2002), no.~1, 1--13.
		
		\bibitem[Sz1]{sz1} E. Szemer\'{e}di, On sets of integers containing no four elements in arithmetic progression, Acta Math. Acad. Sci. Hungar. {\bf 20} (1969), 89--104.
		
		\bibitem[Sz2]{sz2} E. Szemer\'{e}di, On sets of integers containing no $k$ elements in arithmetic progression, Acta Arith. {\bf 27} (1975), 199--245. 
		
		\bibitem[Tao]{tao} T. Tao, {\it Higher order Fourier analysis}, Graduate Studies in Mathematics, 142, American Mathematical Society, Providence, RI, 2012. 
		
		\bibitem[TV]{tv} G. T\"{o}rner\ and\ F.~D. Veldkamp, Literature on geometry over rings, J. Geom. {\bf 42} (1991), no.~1-2, 180--200.
		
	\end{thebibliography}
\end{document}